\documentclass[10pt]{amsart}

\usepackage{amsmath,amssymb}
\usepackage{amscd}
\usepackage{amsfonts}
\numberwithin{equation}{section}
\theoremstyle{plain}

\newtheorem{thm}{Theorem}[subsection]

\newtheorem{prop}[thm]{Proposition}
\newtheorem{lem}[thm]{Lemma}

\newtheorem{cor}[thm]{Corollary}

\newtheorem{ex}[thm]{Example}

\theoremstyle{definition}
\newtheorem{defn}[thm]{Definition}
\newtheorem{rmk}[thm]{Remark}

\newcommand{\A}{\mathbb{A}}
\newcommand{\C}{\mathbb{C}}

\newcommand{\bK}{\mathbb{K}}

\newcommand{\Ksep}{K^{\mathrm{sep}}}

\newcommand{\Q}{\mathbb{Q}}

\newcommand{\Z}{\mathbb{Z}}

\newcommand{\cC}{\mathcal{C}}
\newcommand{\cCe}{\mathcal{C}_1}
\newcommand{\cCo}[1]{\mathcal{C}^{#1}}

\newcommand{\cD}{\mathcal{D}}

\newcommand{\cF}{\mathcal{F}}

\newcommand{\cT}{\mathcal{T}}

\newcommand{\ab}{\mathrm{ab}}

\newcommand{\Aut}{\mathrm{Aut}}

\newcommand{\End}{\mathrm{End}}

\newcommand{\Gal}{\mathrm{Gal}}

\newcommand{\GL}{\mathrm{GL}}

\newcommand{\Hom}{\mathrm{Hom}}
\newcommand{\id}{\mathrm{id}}

\newcommand{\Ker}{\mathrm{Ker}}

\newcommand{\Map}{\mathrm{Map}}
\newcommand{\Mat}{\mathrm{Mat}}

\newcommand{\Obj}{\mathrm{Obj}\, }
\newcommand{\op}{\mathrm{op}}

\newcommand{\Presh}{\mathrm{Presh}}

\newcommand{\Sets}{\mathrm{Sets}}

\newcommand{\frh}{\mathfrak{h}}

\newcommand{\frU}{\mathfrak{U}}
\newcommand{\inj}{\hookrightarrow}
\newcommand{\surj}{\twoheadrightarrow}
\newcommand{\resp}{resp.\ }
\newcommand{\xto}[1]{\xrightarrow{#1}}
\newcommand{\wt}[1]{\widetilde{#1}}
\newcommand{\wh}[1]{\widehat{#1}}

\newcommand{\advertisement}[1]{}

\renewcommand{\mod}{\,\mathrm{mod}\,}

\newcommand{\cP}{\mathcal{P}}

\newcommand{\Pre}{\mathrm{Pre}}
\newcommand{\Shv}{\mathrm{Shv}}

\newcommand{\Lat}{\mathbf{Lat}}
\newcommand{\Pair}{\mathbf{Pair}}

\newcommand{\quotobj}{quotient object }
\newcommand{\quotobjs}{quotient objects }
\newcommand{\quotobjp}{quotient object}
\newcommand{\quotobjsp}{quotient objects}

\newcommand{\quot}[2]{{#2}\backslash{#1}}

\newcommand{\frV}{\mathfrak{V}}

\newcommand{\Isom}{\mathrm{Isom}}
\newcommand{\Mor}{\mathrm{Mor}}

\newcommand{\Ci}{{(\mathcal{C}_0,\iota_0)}}
\newcommand{\Cip}{{(\mathcal{C}_0,\iota_0)}}
\newcommand{\MSet}
{(M_{(\cC_0,\iota)}\text{-set})_{\mathrm{sm}}}

\newcommand{\ProC}{\mathrm{Pro}\text{-}\cC}
\newcommand{\BG}{\mathrm{BG}}

\title[sites for smooth representations]{Sites whose 
topoi are the smooth representations of locally prodiscrete monoids}
\author[Satoshi KONDO, Seidai YASUDA]
       {Satoshi KONDO$^{1,2}$, Seidai YASUDA$^3$       \\
       $^1${Faculty of Mathematics, National Research University Higher School of Economics, Russia}
\\
       $^2$Kavli Institute of the Physics and Mathematcis of the Universe, University of Tokyo (WPI), Japan
\\
       $^3$Department of Mathematics, Graduate School of Science, Osaka University, Japan
}

\address{Corresponding author: 
Satoshi Kondo \\
National Research University 
Higher School of Economics\\
Faculty of Mathematics\\
6 Usacheva Str.\\
Moscow 119048\\ 
Russia\\
Phone/FAX: +7 (495) 772-95-90\\
Email: satoshi.kondo@gmail.com\\
Kavli Institute for the Physics and Mathematics of the Universe\\
University of Tokyo\\
5-1-5 Kashiwanoha\\
Kashiwa 277-8583\\ Japan\\  
Tel: +81-4-7136-4940\\
Fax: +81-4-7136-4941\\
}
\address{Seidai Yasuda\\
Department of Mathematics\\
Graduate School of Science\\
Osaka University\\
1-1 Machikaneyama Toyonaka Osaka 
560-0043 JAPAN\\
Email:s-yasuda@math.sci.osaka-u.ac.jp}

\date{\today}

\begin{document}

\begin{abstract}
We define a class of sites such that the associated
topos is equivalent to the category of smooth sets (representations)
of some locally prodiscrete monoids (to be defined).
Examples of locally prodiscrete monoids include profinite groups 
and finite adele valued points
of algebraic groups.
This is a generalization of the fact that the topos associated 
with the \'etale site of a scheme is equivalent to the category of 
sets with continuous
action by the \'etale fundamental group.

We then define a subclass of sites such that
the topos is equivalent to the category of 
discrete sets with a continuous action of 
a locally profinite group.
\end{abstract}

\maketitle
\section{Introduction}
We begin by presenting our motivation in Section~\ref{sec:motivation}, 
followed by a rough description of the earlier sections of 
our work in Section~\ref{sec:rough intro}.
We provide a more detailed description, with precise statements of 
our theorems, in Sections~\ref{sec:intro thm 1} and~\ref{sec:intro thm 2}.
In Section~\ref{sec:topos Galois theory}, we comment on other generalizations 
of the Galois theory.  
In Section~\ref{sec:table of results}, we give a summary of all of our results.
Many sections were added after suggestions of the referee and 
the structure of the long paper is not so streamlined.   This section is to serve as 
an index of results in this paper.

There is a considerable overlap with our work and Caramello's work.   
In Section~\ref{sec:Caramello intro},  we address this issue.  
In Section~\ref{sec:intro Galois monoids},
we discuss the types of topological monoids 
that appear in our work.
In Section~\ref{sec:on future paper}, 
we provide a brief description of our future paper.   
Lists of contents for each section are
given in Section~\ref{sec:list of contents}.

\subsection{}
\label{sec:motivation}
Our main aim in this paper is to study the 
representation theory 
of algebraic groups with values in $p$-adic fields or in finite adeles using topos theory by establishing basic abstract theory and developing basic tools.
Some of these results are already known and there may be overlap with
existing work, but it will be necessary for us to record all of our results in
one place.   See Section~\ref{sec:on future paper} for more on
our motivation.

\subsection{}
\label{sec:rough intro}
Our result may be regarded as a generalization of Galois theory as seen in 
the following
statements. 
Let $K$ be a field.   There exists a separable closure $\Ksep$ that is the union of finite separable field extensions of $K$ contained in $\Ksep$, which also equals the union of finite Galois extensions (in $\Ksep$) of $K$.   These finite Galois groups form a projective system for which the projective limit (a profinite group) is defined as the absolute Galois group $G_K$ of $K$.    Let $\cC$ be the category of finite separable field extensions of $K$.   Its morphisms are the ring homomorphisms preserving $K$, and we equip the category
$\cC$ with the atomic topology $J$.
Then, the topos $\Shv(\cC,J)$ of sheaves on the site $(\cC, J)$ 
is equivalent to the category of sets with continuous action of 
the absolute Galois group $G_K$.

In this paper, we give a set of conditions for a site, and we call any site satisfying these 
conditions a {\it $Y$-site}.
We then define a grid for a $Y$-site.  
Generally, a grid is an analogue of 
a system of finite separable extensions contained in a separable closure of a field.
Given a $Y$-site and a grid for the site, one can construct a monoid with specified subgroups, 
which we call the {\it{absolute Galois monoid}}.   
Our theorem (see Theorems~\ref{thm:intro1} and~\ref{thm:intro2} for the precise statements) 
says that if a $Y$-site satisfies 
a certain cardinality condition,
then there exists a grid for this $Y$-site, and therefore the absolute Galois monoid,
and
the category of sheaves on the site is equivalent to the category of {\it smooth sets} (to be defined in this paper) of 
the absolute Galois monoid.

For example, 
as the site (described above) of finite separable extensions of a field satisfies all the conditions of a $Y$-site 
and the finiteness condition is satisfied, our theorem may be applied.
The absolute Galois monoid is the usual absolute Galois group
of the field, and the specified subgroups determine the usual profinite topology on 
the group.    The smooth sets are nothing but those sets with continuous action
of the Galois group, and the equivalence is the usual one.

Our work is a generalization in the following sense.
We assume neither the existence of a final object in $\cC$
nor that the Grothendieck topology $J$ is atomic.
The absolute Galois monoid, which is a generalization of the absolute (profinite) 
Galois group,
may not be a group if $J$ is not atomic.
Even when $J$ is atomic and the monoid is indeed a group,
it may not be profinite.    Under the finiteness assumption, 
the absolute Galois monoid is always locally profinite and if there 
exists a final object, it is profinite.

\subsection{}
\label{sec:intro thm 1}
In this subsection, we outline the contents of 
this article leading to the statement of 
our main theorem.  
We must provide many new conditions and 
new definitions, and our goal here is to explain the motivation and ideas 
behind them.   
Concepts in {\it italics} are to be defined 
in this article, while phrases in quotations~`-' represent ideas
intended to help the readers that may not be mathematically rigorous.

We start this article by defining 
{\it semi-localizing collections}.
(We note that the term {\it collection} is used only 
to avoid set theoretic complications: see Section~\ref{sec:collection}.)
A {\it semi-localizing collection} is a collection of 
morphisms in a category 
satisfying three conditions 
(Definition \ref{defn:semi-localizing}).
As noted in Remark~\ref{rmk:semi-localizing},
these conditions are the first three of the four conditions provided by  
Gabriel and Zisman \cite{GZ} to admit a 
`right calculus of fractions'.
From a semi-localizing collection $\cT$,
we can construct (Lemma~\ref{lem:JcT})
a Grothendieck topology 
$J_\cT$.  We call a topology of such form an $A$-topology
(Definition~\ref{defn:A-topology}).
We arrived at this notion when considering 
a class of Grothendieck topologies such that 
one needs to look only at the coverings of the form 
$\{X \to Y\}$ and not at those of the form 
$\{X_i \to Y\}_{i \in I}$
with cardinality of $I$ greater than 1.
In this topology, the set of coverings is 
a subset of the set of morphisms of the category.
This makes the description of the topology easier 
than the more general topologies.

We will be using the collection 
$\cT(J)=\wh{\cT} \supset \cT$
(see Definition~\ref{defn:Tee hat} for $\wh{\cT}$), 
which is a `saturation' in some sense--again semi-localizing--giving the same 
topology.
The idea is that the collection  $\cT(J)$
is ``the set of all coverings.''   
The atomic topology is an example of an $A$-topology
(see Section~\ref{sec:atomic topology}), which is the case when $\cT$ is the collection
of all morphisms of the given category.
This restriction to $A$-topology corresponds to, in concept, a restriction to ``categories where all objects are connected."
 
We define a {\it Galois covering} 
(Definition~\ref{defn:Galois_cov})
to be a morphism $X \to Y$ in a category 
$\cC$ such that there exists a group $G$
for which 
$\Hom_\cC(Z,Y) \to \Hom_\cC(Z,X)$
is a {\it pseudo $G$-torsor} for each object $Z$
(see Definition~\ref{defn:pseudo G-torsor}).
We say that a site equipped with an 
$A$-topology associated with a semi-localizing 
collection $\cT$ {\it has enough Galois coverings}
(Definition~\ref{defn:enough Galois})
if $\wh{\cT}=\wh{\cT'}$ where 
$\cT'$ is the collection of all Galois coverings.

An {\it $E$-category} 
is a category in which all the morphisms
are epimorphisms (Definition~\ref{defn:E-categories}).   
We then define a {\it $B$-site} 
to be a site $(\cC, J)$ where the underlying category
$\cC$ is an $E$-category; has a topology $J$, 
which is an $A$-topology;
and for which the following condition is satisfied:
For any diagram $Z \xto{g} Y \xto{f} X$ in $\cC$, the composite
$g \circ f$ belongs to $\cT(J)$ if and only if $f$ and $g$ belong to 
$\cT(J)$.
This condition may not seem pleasant, and the notion may be better  described if the category 
has coproducts or is equipped with the notion of 
`$\pi_0$', but we have not assumed so.   Note that 
the usual finite \'etale site of a scheme does 
not satisfy this condition, but if we impose the condition that 
all schemes are connected, it does.  
One outcome of these definitions is the following:
if we assume that there are enough
Galois coverings, the sheafification functor 
from the category of presheaves on a $B$-site
can be described using Galois coverings
(see Section~\ref{sec:sheafify}).

Our principal objects of study are {\it $Y$-sites}
(Definition~\ref{defn:Y-sites}).
A {\it $Y$-site} is defined to be a 
$B$-site satisfying two additional 
conditions (set-theoretic conditions are not 
discussed in this paragraph).   
One is that, given any two objects $X_1, X_2$
in the underlying category,
there exist morphisms
$f_1:Y \to X_1$ and $f_2:Y \to X_2$ both 
in $\cT(J)$.
This condition can be thought of 
as a form of the existence of an initial
object, or the existence of a `universal 
covering' if the site has some topological 
meaning.
The second condition is that $\cT(J)$ has enough Galois coverings.

We regard partially ordered sets (posets) 
as categories in a natural manner.
We introduce {\it quasi-posets} as those categories
that are equivalent to posets (Definition~\ref{defn:quasi-posets}); 
however these are used 
only in the proof of the existence of grids.

A {\it grid} $(\cC_0, \iota_0)$ 
of a $Y$-site is a pair of a poset $\cC_0$ 
and a functor $\iota_0$ from this
poset to the underlying category of the site
satisfying some conditions
(Definition~\ref{defn:grids}).
An {\it edge object} is defined to be an object of $\cC_0$ such that 
every morphism to it is mapped to 
$\cT(J)$ by $\iota_0$
(Definition~\ref{defn:edge objects}).     
Using the endomorphisms of the poset $\cC_0$
and some natural isomorphisms,
we can construct a monoid $M_{(\cC_0, \iota_0)}$ that we call 
{\it the absolute Galois monoid}
(Section~\ref{sec:MC}).
For each edge object $X$, 
we also have a subgroup $\bK_X$
of the monoid (Section~\ref{sec:submonoids}); roughly speaking, it is defined as 
the `stabilizer' of the object.
The term `grid' may not be a common one.  
In the easiest case, a grid is something like 
`the set of all finite field extensions of a field,
that is contained 
in a fixed separable closure'.   Such a grid forms a lattice, 
and, although we wanted to call our grid a lattice, 
we have not done so as technically it may not
be a lattice in general.

This absolute Galois monoid and its subgroups 
allow us to define the notion of {\it smooth sets}.  
A {\it smooth $M_{(\cC_0, \iota_0)}$-set} is defined to be 
a set with the action of the absolute Galois monoid 
satisfying  the following condition:
for an element in the set, the stabilizer contains 
the group of the form $\bK_X$ 
for some edge object $X$ (Definition~\ref{defn:smooth sets}).

We may also equip the absolute Galois monoid with 
a structure of a topological monoid using these subgroups
(Section~\ref{sec:topological monoid structure}).
In this case, by definition, the category of smooth $M_\Cip$-sets
is canonically equivalent to the category of 
discrete sets with continuous action of the
topological monoid $M_\Cip$.

When the topology is the atomic topology,
$M_\Cip$ is a group and all objects are edge objects.

We took the term `smooth' 
from the representation 
theory of locally profinite groups,
which is known for its applications in number 
theory (see Remark~\ref{rmk:smooth}).   
Examples of locally profinite groups 
include profinite groups, 
the finite adele 
valued points of an algebraic group,
and the nonarchimedean 
local field valued points
of an algebraic group.   
One of our original
motivations was to describe the categories 
of smooth representations of these groups 
using sheaf theory.

Another entity that we construct from 
a grid is the 
{\it fiber functor} $\omega_{(\cC_0, \iota_0)}$, which is a functor from the topos of sheaves on 
the $Y$-site to the category of sets that 
factors through the category of smooth sets.
A sheaf is sent via this fiber functor 
to the colimit of sections 
over the edge objects, 
which is actually isomorphic 
to the colimit over the entire grid.
This action of the absolute Galois monoid may be 
thought of as a generalization of 
the action of the finite adeles 
on the limit of elliptic modular curves or Drinfeld modular varieties.  
We will consider these examples in a future paper.

Before coming to our theorem, we need to mention 
one concept: {\it cardinality conditions}
(Section~\ref{sec:cardinality}).  
There are two kinds of cardinality conditions.
Cardinality Condition (1) is 
that the hom sets of the 
underlying category of the site are finite.
There is also another type of cardinality condition.
The first kind 
is used primarily in the form:
the projective limit of nonempty finite sets with surjective 
transitive maps is nonempty and the limit surjects onto 
each finite set.    

One important proposition 
(Proposition~\ref{cor:grid existence}) says that,  
under certain cardinality conditions, 
there exists a grid.   The proof of this can be divided into 
two parts.   The main part is the first half (Proposition~\ref{prop:pregrid existence}), which
says that there exists what we call a {\it pregrid}.
The idea for the proof of the first half is taken from 
the proof of the existence of an algebraic 
closure of a field.  The proof of the 
second half appears to be new.

We now come to our main theorem.
Suppose we are given a $Y$-site.
Assume that a certain cardinality condition holds.
Note that, by the proposition, there then exists a grid and
we can construct the absolute Galois monoid and 
its subgroups.   Hence we have smooth sets
and we also have the associated fiber functor.
Our theorem (see Theorem~\ref{thm:Galois_main}
for the precise statement) 
says: 
\begin{thm}
\label{thm:intro1}
Under these assumptions,
the fiber functor gives an
equivalence between the topos and the category 
of smooth sets.
\end{thm}

\subsection{}
\label{sec:intro thm 2}
Let us restrict ourselves to atomic topologies
and assume cardinality Condition~(1),
that is, the finiteness of the hom sets, holds.
Then  
as an application of our theorem, 
we obtain a 
`reconstruction' theorem
 (Theorem~\ref{thm:reconstruction}) 
as follows.   
\begin{thm}
\label{thm:intro2}
Suppose there is a given $Y$-site 
whose topology is the atomic topology.
Assume that Condition (1) of the cardinality 
conditions holds.   Then there exists 
a locally profinite group $G$ such that 
the topos is equivalent 
to the category of discrete sets with 
continuous $G$-action.
\end{thm}

When the topology is the atomic topology, 
the absolute 
Galois monoid becomes a group.  
If there exists a grid, 
we can equip the group with the structure
of a topological group such that the subgroups
$\{\bK_X\}$ for objects $X$ of the grid is
a fundamental system of neighborhoods 
of the unit.
If cardinality 
Condition (1) holds, then there exists a grid
as noted above.  We can further show that
the topological group is locally profinite.
By definition, the category of smooth sets
is equal to the category of discrete sets 
with continuous action of this topological group.
Hence Theorem~\ref{thm:intro2} follows from
Theorem~\ref{thm:intro1}.
\subsection{Grothendieck's Galois theory and its generalizations}
\label{sec:topos Galois theory}
Grothendieck's interpretation of Galois theory takes 
roughly the following form: 
Suppose a topos with conditions and 
some additional data are given.
Then one can construct a profinite group 
(the Galois group).   The theorem establishes that the topos
is equivalent to the category of representations of the profinite group.

There are many other generalizations of Grothendieck's Galois theory.
Joyal and Tierney \cite{JT} 
put only very mild conditions on a topos and 
proved that it was equivalent to the classifying topos of some localic group.
There is also a more constructive proof by Dubuc for when the topos is pointed atomic \cite{DRep}.
There are many other works (many of them categorical or logical) that  we do not cite here; instead, we refer to Dubuc's informative survey-type article \cite{Dlocalic}.

The main differences in our work 
are that the groups considered are localic rather than topological (see Caramello~\cite[p.652]{Caramello} for the difference between the two kinds of groups) and that topoi are considered rather than sites.

Of course our work is not totally unrelated.   We owe much of the formulation to these preceding works.

\subsection{Summary of our results}
\label{sec:table of results}
At referees' suggestions, we were able to improve this paper.
As we added many results later 
during revision, 
the paper is not so streamlined.
The aim of this section is to give a summary of 
results, using Table~\ref{table:results}.
For the contents of the earlier sections, 
the 
reader is referred to Section~\ref{sec:intro thm 1}
and Theorem~\ref{thm:intro2}.
We review mainly the contents of later sections here.

The main object of study in our work is $Y$-site.
The main question is when and how its associated topos 
is equivalent to the category $BM$ of 
continuous $M$-sets for some topological monoid $M$.
The other objects of study are grids (and the associated 
absolute Galois monoids), 
fiber functors $\omega$ (any functor from the 
associate topos to some $BM$) 
and topological monoids $M$.
Recall that if there exists a grid, then, the grid 
gives rise to a fiber functor
to $BM$ where $M$ is the absolute Galois monoid associated
with the grid.   We 
can consider the behavior (i.e., faithfulness, fullness, essential surjetivity) 
of the associated fiber functor.   If there exists a fiber functor (without 
a grid), we can ask if there exists a grid giving rise to this given functor.
Let us describe the results on these matters.

The first block in Table~\ref{table:results} concerns the 
implications of the setup where we
are given a $Y$-site satisfying some cardinality conditions.    
It is the result of the earlier sections 
that there exists a grid and the fiber functor $\omega$ 
induces an equivalence of categories between 
the associated topos and the category of continuous $M$-sets
where $M$ is the associated absolute Galois monoid.   
In this case, we prove that the topos has enough points
by showing that the fiber functor gives a point of the topos.
We can compare our setup and Caramello's setup in this case.
We construct an ultrahomogeneous object, which is the key
object in the input of Caramello's main theorem.   We 
refer to Section~\ref{sec:Caramello intro} for more on this.    
We can understand the cardinality conditions as 
sufficient conditions for the `vanishing' of higher 
derived limits.   It is these higher derived limits 
that can better describe the necessary conditions 
for the statements of our theorems to hold.
We give the simplest examples where the absolute Galois monoids
are the additive group of integers and the monoid of natural numbers.
We also have the example that motivated us.

In the second block in Table~\ref{table:results}, 
we have results where we start with some topological 
monoid (group) $M$ and ask if there exists a $Y$-site and
a grid such that the associated absolute Galois monoid
is the given group.
Indeed, a monoid $M$ is a locally prodiscrete monoid
if and only if $M$ is the absolute Galois
monoid associated with a grid for some $Y$-site.
We show that in this case the fiber functor 
does give an equivalence.   

In the third block in Table~\ref{table:results},
we are given a $Y$-site and ask if there exists a grid
for this $Y$-site.   We have a certain invariant 
in the 2nd derived limit of some pro-group associated
with the $Y$-site.   Then, we can formulate precisely
the necessary and sufficient condition for the existence
of a grid using certain subset of the derived limit.
We give an example where the subset is a proper subset,
and using it, we give an example of a $Y$-site for which
there does not exist a grid.

The fourth block in Table~\ref{table:results}
concerns the setup where we are given a $Y$-site and a grid.
(There is an associated fiber functor $\omega$.)
We show that the fiber functor $\omega$ is automatically faithful
without any assumption.   We describe precisely 
the set of equivalence classes of (pinned) grids 
using the first derived limits.    We do not have a necessary
condition but only a sufficient condition for the fiber 
functor $\omega$ to be full and essentially surjective.   
We give an example of a $Y$-site where $\omega$ is not 
full.   

The fifth block in Table~\ref{table:results}
is where we are given a $Y$-site and a fiber functor.
The referee called for the existence 
of a grid in this situation.   We answer in the affirmative
when the functor is an equivalence to $BM$
where $M$ is a locally prodiscrete monoid.

The sixth block in Table~\ref{table:results}
is concerned with the `enough Galois' property.
The question we ask is if a $B$-site has a topos
equivalent to $BG$ for $G$ profinite,
then is the $B$-site a $Y$-site (that is, 
`enough Galois' property holds).    
We answer this in the negative by constructing 
certain profinite groups.    The treatment is 
very (profinite) group theoretic.

\begin{table}[htb]
\caption{Summary of results}
\label{table:results}
  \begin{tabular}{|l|}
\hline
Given $Y$-site+cardinality conditions: 
\\
(1) $\exists$ grid+$\omega$ equivalence (Thm.\ \ref{thm:Galois_main})
\\
(2) topos has enough points (\S\ref{sec:enough points})
\\
(3) $\exists$ Caramello's ultrahomogeneous object
(\S \ref{sec:Caramello tech})
\\
(4) `triviality' of $R^1\varprojlim, R^2\varprojlim$ (\S \ref{sec:higher derived limits})
\\
\hline 
Two examples of $Y$-sites+cardinality conditions (\S \ref{sec:examples})
\\
\hline \hline
Given (locally) profinite group: $\exists$ $Y$-site+grid+$\omega$ equivalence 
(\S \ref{sec:atomic topological})
\\
Given locally prodiscrete group: $\exists$ $Y$-site+grid+$\omega$ equivalence 
(\S \ref{sec:locally prodiscrete})
\\
Given locally prodiscrete monoid: $\exists$ $Y$-site+grid+$\omega$ equivalence 
(\S \ref{sec:locally prodiscrete monoids})
\\
\hline
Given $M$ topological monoid: \\
$M$ is an absolute Galois monoid $\Leftrightarrow$ $M$ is locally prodiscrete (\S \ref{sec:locally prodiscrete monoids})
\\
\hline \hline
Given $Y$-site:  
\\
(1) `trivial' $R^2\varprojlim   \Leftrightarrow \exists$ grid  (Prop.\ \ref{cor:grid existence}, 
\S\ref{sec:grid and R2})
\\
(2) Example of `nontrivial' $R^2\varprojlim$, $\nexists$ grid (\S \ref{sec:grid and R2})
\\
\hline \hline
Given $Y$-site+grid(+fiber functor $\omega$):  
 \\
(1) $\omega$ faithful (Thm.\ \ref{thm:Galois_main}, \S \ref{sec:faithful})
\\
(2) `trivial' $R^1\varprojlim   \Leftrightarrow$ uniqueness of grid (\S\ref{sec:grid uniqueness}, 
\S\ref{sec:uniqueness and R1})
\\
(3) `trivial' $R^1\varprojlim   \Rightarrow$ fullness, ess.\ surj.\ of $\omega$ (Thm.\ \ref{thm:Galois_main}, \S\ref{sec:full ess surj})
\\
(4) Example of `nontrivial' $R^1\varprojlim$, $\omega$ not full (\S\ref{sec:fiber not full})
\\
\hline \hline
Given $Y$-site+$\omega$ equivalence to $BM$ ($M$: locally prodiscrete monoid):
\\
$\exists$ grid giving $\omega$ (\S \ref{sec:fiber to grid})
\\
\hline \hline
Example of non-$Y$ $B$-site equivalent to $BG$ for $G$ profinite (\S \ref{sec:on enough Galois})
\\
\hline
     \end{tabular}
\end{table}

\subsection{Relations to Caramello's work}
\label{sec:Caramello intro}
Caramello's generalization (\cite{Caramello}, \cite{CaramelloFraisse}) of Galois theory focuses on sites.   There is a high degree of overlap between this article and Caramello's papers which precede our work.   Let us point out the differences.

Both ours and Caramello's work are concerned with the following aim: to provide criteria for a site such that the associated topos is equivalent to the category of continuous representations of some topological monoid (or some topological group).   There are, however, four main differences.   The first is that she works with atomic topology only while we work with a slightly more general Grothendieck topology (namely, the $A$-topology).   This means that we treat monoids that may not be groups, while she treats only groups.   The second is the method of proof; she uses logic, while our proof is entirely categorical.   The third is that we define and assume `enough Galois' condition while there is nothing similar in her work.   The fourth is the assumptions that are used for our results.   We use cardinality conditions or (enough Galois property and) higher derived limits, while she uses 
the assumptions in the Fra\"{i}ss\'e-Kubi\'s theory.   Let us discuss this last 
issue in the following paragraphs.

Let us recall briefly the way our results and her results are formulated.   We start with a $Y$-site.    We consider grids; these are similar to her $\cC$-homogeneous and $\cC$-universal objects.   Further, she considers $\cC$-ultrahomogeneous objects.  
(See Section~\ref{sec:Caramello tech} where we construct an ultrahomogeneous object 
from our setup).  
Using the enough Galois property, we obtain certain pro-groups (see Section~\ref{sec:higher derived limits} for more details).  Our theorem reads, roughly, under the enough Galois property and some `vanishing' of the higher derived limits of those pro-groups, a grid exists, thereby obtaining a fiber functor to a category $BM$ for some topological monoid $M$, and the fiber functor gives an equivalence.   Her theorem reads, roughly, under some conditions coming from the Fra\"{i}ss\'e-Kubi\'s theory, 
a $\cC$-ultrahomogeneous object exists, thereby obtaining a fiber functor which is 
an equivalence.     

For the existence of a $\cC$-homogeneous and 
$\cC$-universal 
object \cite[Theorem 2.8]{CaramelloFraisse}  
(this is more or less equivalent to the existence of a grid), 
she assumes the dominance conditions of  
Fra\"{i}ss\'e-Kubi\'s, i.e., that the underlying category $\cC$ is 
$\kappa$-bounded and admits a dominating family $\cF$
of morphisms with $|\cF| \le \kappa$.   
Under this condition, one can use transfinite induction 
to prove the existence theorem
(see Theorem 3.7 of Kubi\'{s} \cite{Kubis} for the detail).

The key input to the proof of 
the existence of an ultrahomogeneous object
of \cite[Theorem 2.8]{CaramelloFraisse}
is \cite[Lemma 2.7]{CaramelloFraisse},
which is similar to our 
Lemma~\ref{lem:sufficiently many}.
We assume certain cardinality conditions for our lemma, 
while
she restricts to continuous $\kappa$-chains.


Because we assume the enough Galois property, it is not possible to directly compare her work and ours.    We have never tried to work without the assumption.   For example, we have not studied $B$-sites with Fra\"{i}ss\'e-Kubi\'s type assumptions.

\subsection{On the absolute Galois monoids}
\label{sec:intro Galois monoids}
Let us focus on the topological monoids that we encounter.
From a $Y$-site and a grid, we obtain the associated absolute
Galois group, which is a topological monoid.    
We show that
a topological monoid is isomorphic to the associated
absolute monoid for some grid if and only if 
it is locally prodiscrete.     

This does not seem to be general enough in that 
we do not expect the category $BM$ 
of continuous $M$-sets  for some topological monoid $M$
to be equivalent to $BM'$ for some locally prodiscrete monoid $M'$.
The referee suggested that we look for an alternative definition of grids
to cover general topological monoids,
but we have not found one.

It is true that $BG$ for any topological group $G$
is equivalent to some topos as shown in \cite[p.154, Thm 2]{MM}.
In their theorem, the site seems not to be a $Y$-site
in general.   However, even so, one can construct a grid (i.e., 
a category and a functor satisfying the conditions of a grid)
in that case.    The limitation on our absolute Galois groups 
seems to be caused by our assupmtion that $Y$-site
has enough Galois coverings.
For the setup, we may consider a $B$-site and a grid, and it 
may be 
interesting to search for conditions (other than `enough Galois') 
on a $B$-site for which the statements of our theorems hold more 
generally.

\subsection{On our future paper}
\label{sec:on future paper}
Let us give a brief description of our future paper.
The following paragraphs may help explain our motivation, 
other than Galois theory, for writing this article.

This paper grew out of the example in Section~\ref{sec:motivation example}.   
We develop some abstract theory to include this particular example.  
We regard finiteness conditions as something natural (while less general from the point of view of extension of Galois theory to general topological groups) because 
this particular example satisfies them.

In our future paper, we will further develop abstract theory on $Y$-sites.   
Recall the example of the site of finite separable field extensions of a field $K$.
One property of the category is that for any Galois extension $L$ with Galois group $G$ and a subgroup $H \subset G$, the fixed part $L^H$ is also an object.
We will give a similar construction of such saturation on a $Y$-site.

Note also that in the usual \'etale site, where connectedness is not assumed, finite coproducts exist.   It is also possible to add finite coproducts in $Y$-site.   This will no longer be a $Y$-site but we will see that it is still manageable.    

We will also consider not only sheaves but presheaves with transfers.   For the definition of transfers, it is quite important that one can define the degree of a morphism, which usually equals the cardinality of the Galois group for a Galois covering.   This can be made possible only under the finiteness assumption.

\subsection{}
We remark here on the use of universes.
Usually, an author fixes a universe $\frU$ 
and suppresses 
the appearance by declaring that everything belongs
to the fixed universe.   
However, the reader will occasionally find  in this paper the phrase 
``essentially $\frU$-small,'' which is used because the primary 
example
is the \'etale
site of a scheme in $\frU$, 
the underlying category
of which is essentially $\frU$-small.
There is a standard technique
for changing the universe to a larger 
one so that 
proving statements in an arbitrary fixed 
universe suffices (which enables one to 
suppress the appearance of the universe).  However, 
it was not clear 
to us if the use of such a technique should
eliminate the phrase.\\

\subsection{}
\label{sec:list of contents}
We now provide the list of contents of each
section.  More technical details are given 
at the beginning of each section.

In Section~\ref{sec:topos}, we first recall some general 
definitions and constructions from sheaf theory in order 
to make this paper self-contained.  
Basic notions such as Grothendieck topology, sieve, and
sheaf are recalled.
We then define what it means for a collection 
of morphisms in a category to be 
{\it semi-localizing}.  
It is shown that a Grothendieck topology 
is associated with a semi-localizing collection,
which we call {\it $A$-topology}.
After showing certain general properties of 
$A$-topology, we produce a fairly explicit 
criterion for a presheaf 
to be a sheaf in $A$-topology.

In Section~\ref{sec:Galois coverings},
we define 
{\it Galois covering}
and what it means for a site to have 
{\it enough Galois coverings}.   
We spend few pages on the generality on quotient 
objects in a category.  
This will be useful in our future 
paper, in which we consider sheaves with
values in a category other than the category
of sets.

In Section~\ref{sec:B-sites},
we define $B$-sites (Definition~\ref{defn:B-site}) 
and give some of its properties.
In Section~\ref{sec:sheafify}, we give an explicit description
of the sheafification functor on $B$-sites
when there are enough Galois coverings.
This will be used in the proof of our main
theorem.

The aim of Section~\ref{sec:grids}
is to state our main theorem
(Theorem~\ref{thm:Galois_main}).
We define a $Y$-site as a $B$-site with some additional 
conditions.   We define {\it cardinality conditions},
{\it grid}, the {\it absolute 
Galois monoid}, the {\it smooth sets}, 
and 
the {\it fiber functor} $\omega_{(\cC_0, \iota_0)}$
associated with a grid.
Theorem~\ref{thm:Galois_main} 
establishes that the fiber functor 
induces an equivalence of categories
between the topos and the category of 
smooth representations of the absolute 
Galois monoid under the cardinality 
conditions.

In Section~\ref{sec:grid existence}, 
we prove the existence
of a grid under the {\it cardinality conditions}.
We first prove the existence of the {\it pregrid},
and then of the {\it grid}.
The proof for the {\it pregrid} follows the same type of formulation 
as for the proof of the existence of an algebraically 
closed field of a field; the construction of the {\it grid} follows a novel approach.
We note that when the topology is the atomic
topology, the pregrid is already a grid.

In Section~\ref{sec:proof of theorem}, 
we begin the proof of our main theorem.
We show that the fiber functor is fully faithful.
In Section~\ref{sec:ess_surj}, we show the essential 
surjectivity, finishing the proof of Theorem~\ref{thm:Galois_main}.

In Section~\ref{sec:grid uniqueness}, we show that the grid if existed is essentially unique under some finiteness condition
and is not unique if we drop the finiteness assumption.

In Section~\ref{sec:enough points}, we show that
the fiber functor has a left adjoint, thereby showing that 
we have a point of the topos.
As an application of our main theorem, 
we see that the topos has enough points.

In Section~\ref{sec:atomic topological}, 
we show how to equip the absolute Galois monoid
with the structure of a topological monoid.
We then 
give a precise form of Theorem~\ref{thm:intro2}.

In Section~\ref{sec:locally prodiscrete},
we give examples of Y-sites and grids 
that do not meet the cardinality conditions,
making in such cases the absolute Galois monoids 
locally prodiscrete groups.

We give more examples of $Y$-sites with grids in Section~\ref{sec:examples}.
The simplest example in which the absolute Galois monoid
is the monoid
of non-negative integers is given in 
Section~\ref{sec:monoid example}.
The example in Section~\ref{sec:motivation example}
gave us the motivation to write this article; its details and an application will be given in a 
future paper.

In Section~\ref{sec:Caramello tech}, we show how our $Y$-site and a grid give rise to an ultrahomogeneous 
object 
used in Caramello's work \cite{Caramello}.

In Section~\ref{sec:fiber not full}, we give an example where 
the fiber functor is not full, when we do not assume the cardinaility 
condition.

In Section~\ref{sec:locally prodiscrete monoids},
we start with an admissible topological monoid,
and construct a $Y$-site and a grid 
whose associated absolute Galois monoid is related
to the given monoid.   It is also shown that 
a topological monoid $M$ equals the absolute Galois 
monoid for some grid if and only if $M$ is locally prodiscrete.

In Section~\ref{sec:fiber to grid}, we construct a grid 
out of a fiber functor, which is an equivalence,
to the category of continuous $M$-sets where $M$ is
a locally prodiscrete monoid.

In Section~\ref{sec:on enough Galois}, 
we give examples of $B$-sites which do 
not have enough Galois
coverings, yet the toposes are isomorphic to the category
of $G$-sets for some profinite groups $G$.

In Section~\ref{sec:higher derived limits},
we introduce higher derived limits of pro-groups
and describe our theorems
in terms of them.

{\bf Acknowledgment} 
During this research, the first author was supported as a Twenty-First Century COE Kyoto Mathematics Fellow and was partially supported by the JSPS Grant-in-Aid for Scientific 
Research 17740016 21654002 and by the World Premier International Research Center Initiative (WPI Initiative), MEXT, Japan. The second author was partially supported by the JSPS Grant-in-Aid
for Scientific Research 15H03610, 24540018, 21540013, 16244120.

The first author thanks Kei Hagihara and Florian Heiderich 
for useful conversations.  
The notion of grids was conceived during our stay at Hokkaido University.   
We thank Masanori Asakura and Kei Hagihara who made our visits possible.
We thank Yuichiro Hoshi for the advice on finiding examples of some profinite groups.
We thank the referee for posing interesting problems which imporoved our paper better.

\section{$A$-topology}
\label{sec:topos}
First, we recall the definitions of 
Grothendieck topology and sieve.

A collection of morphisms in a category 
is defined (Definition~\ref{defn:semi-localizing})
to be {\it semi-localizing},
when it satisfies the first three of the four axioms 
for the collection to admit `right calculus of 
fractions' in the sense of Gabriel and 
Zisman~\cite{GZ}.
We arrived at this definition when considering 
the class of Grothendieck topologies that is generated
by coverings of the form $\{X_i \to Y\}_{i \in I}$
where the cardinality of $I$ is one.

\subsection{Presheaves}
\label{sec:collection}

Throughout the paper we fix once for all 
a Grothendieck universe $\frU$.

Recall from \cite[EXPOSE I, 1.0, and D\'efinition 1.1]{SGA4} that 
a set $X$ is called {\em $\frU$-small} 
if $X$ is isomorphic to an element in $\frU$,
and that a category $\cC$ is called an $\frU$-category if for any
objects $X$, $Y$ of $\cC$, the set $\Hom_\cC(X,Y)$ is $\frU$-small.
From now on, unless otherwise stated, a set is assumed to be $\frU$-small 
and a category is assumed to be a $\frU$-category.
We use the terminology ``collection"
to refer to a set that is not
necessarily $\frU$-small.
A category $\cC$ is called {\em $\frU$-small}
if the collection of objects of $\cC$ is a set.

\subsubsection{ }
Let $\cC$ and $\cD$ be categories.
We call a contravariant functor from $\cC$ to $\cD$
a {\em presheaf 
on $\cC$ with values in $\cD$}.
When $\cD$ is the category of $\frU$-small sets 
(\resp $\frU$-small abelian groups,
\resp $\frU$-small rings with units), a presheaf on
$\cC$ with values in $\cD$ is called a presheaf
(\resp an {\em abelian presheaf},
 \resp a {\em presheaf of rings}
) on $\cC$.

In this article, we will only consider presheaves of sets.
In our future article, we will consider some more general categories.

\subsubsection{ }
Let $\cC$ be a category and let $X$ be an object
of $\,\cC$. We let $\frh_\cC(X)=\Hom_\cC(-,X)$ 
denote the presheaf on $\cC$ that associates, 
to each object $Y$ of $\,\cC$,
the set $\Hom_\cC(Y,X)$. The presheaf $\frh_\cC(X)$
on $\cC$ is called the {\em presheaf represented by $X$}.

We denote by $\frh_{\cC} : \cC \to \Presh(\cC)$ 
 the functor that associates, to each object
$X$ of $\cC$, the presheaf $\frh_{\cC}(X)=\Hom_{\cC}(-,X)$
represented by $X$ on $\cC$.
It follows from Yoneda's lemma that the functor
$\frh_{\cC}$ is fully faithful.

\subsubsection{ }
A category $\cC$ is called {\em essentially $\frU$-small} 
if $\cC$ is equivalent to a $\frU$-small category, or equivalently,
if there exists a set $S$ of objects of $\cC$ such that
any object of $\cC$ is isomorphic to an object that 
belongs to $S$.
Let $\cC$ be an essentially $\frU$-small category and
let $\cD$ be a category.
Then the presheaves on $\cC$ with values in $\cD$
form a category, which we denote by $\Presh(\cC,\cD)$. 
When $\cD$ is the category of $\frU$-small sets, we
simply write $\Presh(\cC)$ for the category $\Presh(\cC,\cD)$.

\subsubsection{ }
Let $\cC$ be an essentially $\frU$-small category
and let $F$ be a presheaf on $\cC$.
Let $X$ be an object of $\cC$. We denote by
$y_{F,X}$ the map
\begin{equation}\label{eq:yoneda}
y_{F,X}:\Hom_{\Presh(\cC)}(\frh_\cC(X),F)
\to F(X)
\end{equation}
which sends a morphism 
$\phi:\frh_\cC(X) \to F$ of presheaves on $\cC$
to the image of $\id_X \in \Hom_{\cC}(X,X)$
under the map $\phi(X):\Hom_{\cC}(X,X)
\to F(X)$.
It follows from Yoneda's lemma that the map
$y_{F,X}$ is bijective.

\subsubsection{ }\label{sec:IFX}
Let $\cC$ and $\cD$ be categories and
let $F:\cC \to \cD$ be a covariant functor.
For an object $X$ of $\cD$, we denote by
$I^F_X$ the following category.
The objects of $I^F_X$ are the pairs 
$(Y,f)$ of an object $Y$ of $\cC$ and 
a morphism $f:F(Y) \to X$ in $\cD$.
For two objects $(Y_1,f_1)$ and $(Y_2,f_2)$
of $I^F_X$, the morphisms from $(Y_1,f_1)$
to $(Y_2,f_2)$ in $I^F_X$ are the morphisms
$g:Y_1 \to Y_2$ in $\cC$ satisfying
$f_1 = f_2 \circ F(g)$.

\subsubsection{ }
Let $\cC$ be an essentially $\frU$-small category.
Let us choose a full subcategory 
$\cC' \subset \cC$ such that $\cC'$ is $\frU$-small
and that the inclusion functor $\cC' \inj \cC$
is an equivalence of categories.
We call such a full subcategory $\cC'$ of $\cC$
a $\frU$-small skeleton of $\cC$.
We denote by $\frh_{\cC}|_{\cC'} : \cC' \to
\Presh(\cC)$ the composite of the inclusion
functor $\cC' \inj \cC$ with the functor
$\frh_{\cC}$.

For an object $G$ of $\Presh(\cC)$, let
$I_G$ denote the category $I_G^{\frh_\cC|_{\cC'}}$.
By definition, the objects of $I_G$ are the pairs
$(X,\xi)$ of an object $X$ of $\cC'$
and a morphism $\xi: \frh_\cC(X) \to G$
in $\Presh(\cC)$.

 For two objects $(X_1,\xi_1)$ and $(X_2,\xi_2)$
 of $I_G$, the morphisms from $(X_1,\xi)$ to
 $(X_2,\xi_2)$ in $I_G$ are the morphisms
 $f:X_1 \to X_2$ in $\cC'$ satisfying
 $\xi_1 = \xi_2 \circ \frh_\cC(f)$.

This shows that the category $I_G$ is $\frU$-small.
For an object $(X,\xi)$ of $I_G$, we
let $y_{G,X}(\xi)$ denote the element of $G(X)$
that is the image of $\xi$ under the
bijection $y_{G,X}:
\Hom_{\Presh(\cC)}(\frh_\cC(X),G)
\xto{\cong} G(X)$ in \eqref{eq:yoneda}.

Let $g:G \to H$ be a morphism in $\Presh(\cC)$.
For an object $(X,\xi)$ of $I_G$, let
$g_{X,\xi} \in H(X)$ denote the image of
$y_{G,X}(\xi) \in G(X)$ under the map $G(X) \to H(X)$
given by $g$.
By varying $(X,\xi)$, we obtain an element
$(g_{X,\xi})$ in the limit
$\varprojlim_{(X,\xi) \in \Obj I_G} H(X)$.
It then can be checked easily that the map
\begin{equation}\label{eq:I_G}
\Hom_{\Presh(\cC)}(G,H) 
\to \varprojlim_{(X,\xi) \in \Obj I_G} H(X)
\end{equation}
that sends $g$ to $(g_{X,\xi})$ is bijective.

\subsection{Sieves and Grothendieck topologies}

\subsubsection{ }
Let us recall the notion of sieve (cf.\ 
\cite[EXPOSE I, D\'efinition 4.1]{SGA4}, 
\cite[Arcata, (6.1)]{SGA4h}).
Let $\cC$ be a category and let $X$ be an object of $\cC$.
A sieve on $X$ is a full subcategory $R$ of the overcategory
$\cC_{/X}$ satisfying the following condition: let $f:Y \to X$
be an object of $\cC_{/X}$ and suppose that there exist
an object $g: Z \to X$ of $R$ and a morphism $h:Y \to Z$
in $\cC$ satisfying $f=g \circ h$. Then $f$ is an object of $R$.

For a sieve $R$ on $X$, we denote by $\frh_\cC(R)$ 
the following subpresheaf of $\frh_\cC(X)$: 
for each object $Y$ of $\cC$,
the subset $\frh_\cC(R)(Y) \subset \frh_\cC(X)(Y) = \Hom_\cC(Y,X)$ 
consists of the morphisms $f:Y \to X$ in $\cC$ 
such that $f$ is an object of $R$.

\subsubsection{ }
Let $\cC$ and $\cD$ be categories,
let $X$ be an object of $\cC$, 
and let $Y$ be an object of $\cD$.
Suppose that a covariant functor $F: \cC_{/X}
\to \cD_{/Y}$ is given.
For a sieve $R$ on $Y$, we denote by $F^* R$
the full subcategory of $\cC_{/X}$ whose objects
are those objects $f:Z \to X$ of $\cC_{/X}$ 
such that $F(f)$ is an object of $R$.
It is then easy to check that $F^* R$ 
is a sieve on $X$.

Let $G:\cC \to \cD$ be a covariant functor.
Suppose that $G(X) = Y$ and that $F$ is equal to
the covariant functor $\cC_{/X} \to \cD_{/Y}$ 
induced by $G$. In this case we
denote the sieve $F^* R$ on $X$ by $G^* R$.

Let $f:X \to Z$ be a morphism in $\cC$. 
Suppose that $\cC=\cD$, $Y=Z$, and
$F$ is equal to the covariant functor 
$\cC_{/X} \to \cC_{/Z}$ 
which sends an object $g:W \to X$ of $\cC_{/X}$
to the object $f \circ g$ of $\cC_{/Z}$.
In this case we denote the sieve $F^* R$ 
on $Y$ by $R \times_Z X$ and call it the pullback
of $R$ with respect to the morphism $f$.

\subsubsection{ }
For a morphism $f:Y \to X$ in a category $\cC$, we let $R_f$ denote the
full subcategory of $\cC_{/X}$ whose objects are the morphisms
$g:Z \to X$ in $\cC$ such that $g=f \circ h$ for some
morphism $h:Z \to Y$ in $\cC$. It is then easy to
check that $R_f$ is a sieve on $X$.

More generally, suppose that $X$ is an object of a category $\cC$ and
that a family $(f_i:Y_i \to X)_{i \in I}$ of objects of $\cC_{/X}$ 
indexed by a set $I$ is given.
We then let $R_{(f_i)_{i \in I}}$ denote the full subcategory of 
$\cC_{/X}$ whose objects are the morphisms 
$g : Z \to X$ in $\cC$ such that $g = f_i \circ h_i$
for some $i\in I$ and for some morphism $h_i: Z \to Y_i$
in $\cC$. It is then easy to
check that $R_{(f_i)_{i \in I}}$ is a sieve on $X$.

\subsubsection{ }
\label{sec:topology}
Let us recall the notion of Grothendieck topology 
(cf.\ 
\cite[EXPOSE II, D\'efinition 1.1]{SGA4}, 
\cite[Arcata, (6.2)]{SGA4h}).
Let $\cC$ be a category.
\begin{defn}
\label{defn:topology}
A Grothendieck topology 
$J$ on $\cC$ is an assignment
of a collection $J(X)$ of sieves on $X$ to each object $X$ of $\cC$
satisfying the following conditions:
\begin{enumerate}
\item For any object $X$ of $\cC$, the overcategory
$\cC_{/X}$ is an element of $J(X)$.
\item For any morphism $f: Y \to X$ in $\cC$
and for any element $R$ of $J(X)$, the sieve
$R \times_X Y$ on $Y$ is an element of $J(Y)$.
\item Let $X$ be an object of $\cC$, and let $R$, $R'$ be 
two sieves on $X$. Suppose that $R$ is an element of $J(X)$
and that for any object $f:Y \to X$ of $R$, the sieve
$R' \times_X Y$ on $Y$ is an element of $J(Y)$.
Then $R'$ is an element of $J(X)$.
\end{enumerate}
\end{defn}

Let $J$ be a Grothendieck topology on $\cC$ and
let $X$ be an object of $\cC$.
We say that a morphism $f:F \to \frh_\cC(X)$
of presheaves on $\cC$ is a covering of $X$ with respect to $J$ 
if the image of $f$ is equal to the subpresheaf 
$\frh_\cC(R)$ of $\frh_\cC(X)$ for some sieve 
$R$ on $X$ which belongs to $J(X)$.
We say that a morphism $f:F \to G$ of presheaves on $\cC$ 
is a covering with respect to $J$
if for any object
$X$ of $\cC$ and for any element $\xi \in G(X)$, the
first projection from the fiber product 
$\frh_\cC(X) \times_G F$ of the diagram
$$
\frh_\cC(X) \xto{y_{G,X}^{-1}(\xi)} G \xleftarrow{f} F
$$
to $\frh_\cC(X)$ is a covering of $X$ with respect to $J$.
When $G = \frh_\cC(X)$ for some object $X$ of $\cC$,
it follows from Condition~(2)
in Definition~\ref{defn:topology} that
$f$ is a covering with respect to $J$ if and only if
$f$ is a covering of $X$ with respect to $J$.
Let $(f_i:Y_i \to X)_{i \in I}$ be a family of 
objects of $\cC_{/X}$ indexed by a set $I$.
We say that $(f_i)_{i \in I}$ is a family covering $X$
with respect to $J$ if the sieve $R_{(f_i)_{i \in I}}$
on $X$ belongs to $J(X)$.

\subsection{Semi-localizing collections}
\label{sec:semi-localizing}

\begin{defn}\label{defn:semi-localizing}
Let $\cC$ be a category.
We say that a collection $\cT$ of morphisms in $\cC$ is semi-localizing
if it satisfies the following conditions:
\begin{enumerate}
\item For any object $X$ of $\cC$, the identity morphism $\mathrm{id}_X$ 
belongs to $\cT$.
\item The collection $\cT$ is closed under composition.
\item Let $Y_1 \xto{f_1} X \xleftarrow{f_2} Y_2$ be a diagram in $\cC$.
Suppose that $f_1$ belongs to $\cT$. Then there exist an object $Z$ in
$\cC$ and morphisms $g_1 : Z \to Y_1$ and $g_2 :Z \to Y_2$ such that
$g_2$ belongs to $\cT$ and $f_1 \circ g_1 = f_2 \circ g_2$.
\end{enumerate}
\end{defn}

\begin{rmk}
\label{rmk:semi-localizing}
The three conditions above are taken from \cite[I.2.2]{GZ}.  In their book, Gabriel and Zisman give a list
of four conditions on a collection of morphisms in a category.   They say that a collection admits a right
calculus of fractions when the four conditions are met.   In \cite[III.2.6]{GM}, Gelfand and Manin call
such collections ``localizing."  Our
conditions are the first three of the four conditions.  Therefore we say that  the collection is ``semi-localizing."
 We note also that in Definition 10.3.4 of the
textbook \cite{W}, Condition (3) is called the \O re condition.
\end{rmk}

\subsubsection{}
\begin{defn}
\label{defn:JcT}
Let $\cC$ be a category and let $\cT$ be a collection
of morphisms in $\cC$.
For an object $X$ of $\cC$, we let $J_\cT(X)$ denote the
collection of sieves $R$ on $X$ such that there exists an
object $f :Y \to X$ of $R$ that belongs to $\cT$.
\end{defn}

\begin{lem}\label{lem:JcT}
Let $\cT$ be a semi-localizing collection of morphisms in 
a category $\cC$. 
Then the assignment $J_\cT$ of the collection 
$J_\cT(X)$ to each object $X$ in $\cC$
is a Grothendieck topology on the category $\cC$.
\end{lem}

\begin{proof}
We prove that $J_\cT$ satisfies the three conditions in
Definition~\ref{defn:topology}.

Let $X$ be an object of $\cC$. It follows from Condition (1)
in Definition \ref{defn:semi-localizing} that the identity
morphism $\id_X:X \to X$ in $\cC$ belongs to $\cT$.
This shows that the sieve $\cC_{/X}$ of $X$ belongs to $J_\cT(X)$.
Hence $J_\cT$ satisfies Condition (1) in 
Definition~\ref{defn:topology}.

Let $f:Y \to X$ be a morphism in $\cC$ and let $R$ be a sieve on $X$
that belongs to $J_\cT(X)$.
By the definition of $J_\cT(X)$, there exist an object $Z$ of $\cC$
and a morphism $g:Z \to X$ in $\cC$ such that $g$ belongs to $\cT$
and that $g$ is an object of $R$.
It follows from Condition (3)
in Definition \ref{defn:semi-localizing} that there exist an
object $W$ of $\cC$ and morphisms $g':W \to Y$ and
$f':W \to Z$ in $\cC$
such that $g'$ belongs to $\cT$ and that 
$f \circ g' = g \circ f'$.
Because $f\circ g' = g \circ f'$ is an object of $R$,
the morphism $g'$ is an object of $R \times_X Y$.
Because $g'$ belongs to $\cT$, the sieve $R \times_X Y$ on $Y$
belongs to $J_\cT(Y)$.
This shows that 
$J_\cT$ satisfies Condition~(2) in Definition~\ref{defn:topology}.

Let us turn to the proof of (3).  
Suppose $R$ belongs to $J_\cT(X)$.  Then there exists an object
$f:Y \to X$ of $R$ that belongs to $\cT$.
As $R' \times_X Y$ belongs to $J_\cT(Y)$,
there exists a morphism $g:Z \to Y$ in $\cC$ such that
$g$ belongs to $\cT$ and that the composite $f \circ g$
is an object of $R'$.
It follows from Condition (2)
in Definition \ref{defn:semi-localizing}, that $f \circ g$ 
belongs to $\cT$.
This shows that $R'$ belongs to $J_\cT(X)$.
Hence $J_\cT$ satisfies Condition~(3) 
in Definition~\ref{defn:topology}.
This completes the proof.
\end{proof}

\subsubsection{ }
\begin{defn}
\label{defn:Tee hat}
Let $\cC$ be a category.
For a collection $\cT$ of morphisms in $\cC$,
we let $\wh{\cT}$ denote the set of morphisms $f:Y \to X$ in $\cC$ 
such that there exists a morphism $g:Z \to Y$ satisfying $f \circ g \in \cT$.
\end{defn}

\begin{lem}\label{lem:composite}
Let $\cT$ be a semi-localizing collection of morphisms in $\cC$.
Then the collection $\wh{\cT}$ contains $\cT$ and is semi-localizing.
\end{lem}

\begin{proof}
As the identity morphisms are contained in $\cT$, we see that $\cT \subset \wh{\cT}$ holds.
In particular, Condition (1) is satisfied.

Let $f: Y \to X$ and $g:Z \to Y$ be morphisms
in $\cC$ that belong to $\wh{\cT}$.
There then exist objects $Y'$, $Z'$ of $\cC$ 
and morphisms $f':Y' \to Y$ and $g':Z' \to Z$ in $\cC$ 
such that the composites $f \circ f'$ and $g \circ g'$ are morphisms
in $\cC$ that belong to $\cT$.
Because $\cT$ is semi-localizing, there exist an
object $W$ of $\cC$ and morphisms $g'':W \to Y'$ and
$f'':W \to Z'$ in $\cC$ such that $g''$ belongs to $\cT$
and that $f' \circ g'' = g \circ g' \circ f''$.
As $\cT$ is semi-localizing, there exist an
object $V$ of $\cC$ and a morphism $h:V \to W$
such that the composite $f \circ f' \circ g'' \circ h$
is a morphism in $\cC$ belonging to $\cT$.
As $(f \circ g) \circ (g' \circ f'' \circ h)
= f \circ f' \circ g'' \circ h$ is a morphism in $\cC$
that belongs to $\cT$, it follows from the definition
of $\wh{\cT}$ that we have $f \circ g \in \wh{\cT}$.
This shows that Condition (3) is satisfied.

Let $Y_1 \xto{f_1} X \xleftarrow{f_2} Y_2$ be a diagram in $\cC$ and
suppose that $f_1$ belongs to $\wh{\cT}$. Let us take a morphism
$f_3:Y_3 \to Y_1$ in $\cC$ such that $f_1 \circ f_3$ belongs to $\cT$.
As the collection $\cT$ satisfies Condition (3) 
in Definition \ref{defn:semi-localizing}, there exist an object $Z$ in
$\cC$ and morphisms $g_2 : Z \to Y_2$ and $g_3 :Z \to Y_3$ such that
$g_2$ belongs to $\cT$ and $(f_1 \circ f_3) \circ g_3 = f_2 \circ g_2$.
As $\cT \subset \wh{\cT}$, the morphism $g_2$ belongs to $\wh{\cT}$
and we have $f_1 \circ (f_3 \circ g_3) = f_2 \circ g_2$.
This shows that Condition (3) 
in Definition \ref{defn:semi-localizing} is satisfied for $\wh{\cT}$.
\end{proof}

\subsubsection{ }\label{sec:CT}

\begin{lem}\label{lem:whcT}
Let $\cT$ be a collection of morphisms in a category $\cC$. Then, for any object $X$ of $\cC$, we have $J_\cT(X) = J_{\wh{\cT}}(X)$.
\end{lem}

\begin{proof}
Let $X$ be an object of $\cC$.
As $\cT \subset \wh{\cT}$, we have $J_\cT(X) \subset J_{\wh{\cT}}(X)$.
Hence it suffices to prove $J_{\wh{\cT}}(X) \subset J_\cT(X)$.
Let $R$ be a sieve belonging to $J_{\wh{\cT}}(X)$.
Then there exists an object 
$f:Y \to X$ of $R$ that belongs to $\wh{\cT}$.
It follows from the definition of $\wh{\cT}$ that there
exists a morphism $g:Z \to Y$ in $\cC$ such that the
composite $f \circ g$ belongs to $\cT$.
As $f \circ g$ is an object of $R$, it follows that
the sieve $R$ belongs to $J_\cT(X)$.
This proves that $J_{\wh{\cT}}(X) \subset J_\cT(X)$.
This completes the proof.
\end{proof}

\subsection{$A$-topologies}
\begin{defn}
\label{defn:A-topology}
We say that a Grothendieck topology $J$
on a category $\cC$ is an {\it $A$-topology}
if there exists a semi-localizing
collection $\cT$
of morphisms in $\cC$ such 
that
 $J=J_\cT$.
Such a collection $\cT$ is called a {\it basis of the $A$-topology} $J$.
\end{defn}

\begin{defn}
\label{defn:Tee jay}
For a Grothendieck topology $J$ on a category $\cC$,
we let $\cT(J)$ denote the collection of morphisms $f:Y \to X$
in $\cC$ such that $R_f$ belongs to $J(X)$.
\end{defn}

\begin{prop}\label{prop:cTJ}
Let $J=J_\cT$ be an $A$-topology on a
category $\cC$. 
Then $\cT(J)=\wh{\cT}$, and it is a basis of the $A$-topology.
\end{prop}

\begin{proof}
It follows from the definition of $\cT(J)$
that we have $\cT(J) = \wh{\cT}$.
Hence from Lemma~\ref{lem:composite}, 
we conclude that $\cT(J)$ is semi-localizing.
Using Lemma~\ref{lem:whcT}, we have 
$J = J_{\cT(J)}$.
This proves the claim.
\end{proof}

It follows immediately from the definition that
any  
basis of an $A$-topology 
$J$ is contained in $\cT(J)$.

\subsubsection{Semi-cofiltered, atomic topology}
\label{sec:atomic topology}
\begin{defn}
\label{defn:semi-cofiltered}
We say that a category $\cC$ is {\it 
semi-cofiltered}
if the collection $\Mor(\cC)$ of the morphisms
in $\cC$ is semi-localizing.
\end{defn}
Set $\cT=\Mor(\cC)$.   
Then, Conditions (1)(2) of Definition~\ref{defn:semi-localizing} are satisfied automatically.
Hence,
a category $\cC$ is semi-cofiltered if and only if
$\cT$ satisfies Condition (3).
In \cite[A.2.1.11 (h)]{J}, 
the terminology ``the right \O re condition" is suggested for this Condition (3).

When $\cC$ is semi-cofiltered,
we call, following \cite{BD} and \cite[p.\ 115]{MM}, 
the Grothendieck topology $J_{\Mor(\cC)}$ on $\cC$
the atomic topology on $\cC$.

\subsection{Sheaves for $A$-topology}

\subsubsection{ }
Let $\cC$ be an essentially $\frU$-small category and
let $J$ be a Grothendieck topology on $\cC$.
Let $F$ be a presheaf on $\cC$.
For an object $X$ of $\cC$ and for an element $R$ of $J(X)$,
we let
$$
c_{F,X,R} : F(X) \xto{y_{F,X}^{-1}} 
\Hom_{\Presh(\cC)}(\frh_\cC(X),F)
\to \Hom_{\Presh(\cC)}(\frh_\cC(R),F)
$$
denote the map given by the composition with the inclusion
$\frh_\cC(R) \to \frh_\cC(X)$.

We say that a presheaf $F$ on $\cC$ is {\it $J$-separated}
(\resp a {\it $J$-sheaf}, or simply a {\it sheaf} if the topology is
clear from the context) if the map $c_{F,X,R}$ is injective
(\resp bijective) for every object $X$ of $\cC$ and for
every element $R$ of $J(X)$.
We denote by $\Shv(\cC, J) \subset \Presh(\cC)$ 
the full subcategory of $J$-sheaves on $\cC$.
We will use separated presheaves in our future paper.

\begin{prop}\label{prop:sheaf_criterion1}
Let $\cT$ be a semi-localizing collection of morphisms in 
an essentially $\frU$-small category $\cC$.
Then a presheaf $F$ on $\cC$ is a $J_\cT$-sheaf if and only if
the following condition is satisfied:
\begin{description}
\item[(*)] 
For any object $X$ of $\cC$ and for any morphism $f:Y \to X$ 
in $\cC$ which belongs to $\cT$, the map
$c_{F,X,R_f}$ is bijective.
\end{description}
\end{prop}

\begin{proof}
The ``only if" part is easy as, for any
$f:Y \to X$ in $\cT$, it follows from the definition of
$J_\cT$ that the sieve $R_f$ belongs to $J_\cT$.

We now prove the ``if" part.
Let $F$ be a presheaf on $\cC$ that satisfies the condition~(*).
Let $X$ be an object of $\cC$ and let 
$R$ be an element of $J_\cT(X)$.
We prove that the map $c_{F,X,R}$ is bijective.
It follows from the definition of $J_\cT(X)$ that
there exist an object $Y$ of $\cC$ and a morphism
$f:Y \to X$ in $\cC$ such that $f$ belongs to $\cT$
and $f$ is an object of $R$.
Then $R_f$ is a full subcategory of $R$ and hence
$\frh_\cC(R_f)$ is a subpresheaf of $\frh_\cC(R)$.
Let 
$$
c: \Hom_{\Presh(\cC)}(\frh_\cC(R),F)
\to \Hom_{\Presh(\cC)}(\frh_\cC(R_f),F)
$$
denote the map given by the composition with the inclusion
$\frh_\cC(R_f) \to \frh_\cC(R)$.
We then have $c_{F,X,R_f} = c \circ c_{F,X,R}$.
By assumption, the map $c_{F,X,R_f}$ is bijective.
Hence, it suffices to prove that the map $c$ is injective.

Suppose that the map $c$ is not injective.
There then exist two elements 
$\alpha_1,\alpha_2 \in \Hom_{\Presh(\cC)}(\frh_\cC(R),F)$
such that $\alpha_1 \neq \alpha_2$ and $c(\alpha_1) = c (\alpha_2)$.
For $i=1,2$ and for an object $Z$ of $\cC$,
we let $\alpha_i(Z): 
h_\cC(R)(Z) \to F(Z)$ 
denote the
map induced by $\alpha_i$ on the sections over $Z$.
As $\alpha_1 \neq \alpha_2$, there exists an object $g:Z \to X$ in $R$
such that $\alpha_1(Z)(g) \neq \alpha_2(Z)(g)$.
As $\cT$ is semi-localizing, there exists an object $W$ in $\cC$ and
morphisms $f' : W \to Z$ and $g':W \to Y$ such that $f'$ belongs to $\cT$
and that $g \circ f' = f \circ g'$.

As $g \circ f' = f \circ g'$, the composite $g \circ f'$ is an
object of $R_f$. As $c(\alpha_1) = c(\alpha_2)$, the two elements
$\alpha_1(Z)(g), \alpha_2(Z)(g) \in F(Z)$ are mapped to the same element
of $F(W)$ under the pullback map $F(Z) \to F(W)$ with respect to $f'$.  
This shows that the images of $\alpha_1(Z)(g), \alpha_2(Z)(g) \in F(Z)$ 
under the map $c_{F,Z,R_{f'}}$ coincide.
As $f'$ belongs to $\cT$, the map $c_{F,Z,R_{f'}}$ is bijective.
Hence we have $\alpha_1(Z)(g) = \alpha_2(Z)(g)$, which
leads to a contradiction. This completes the proof.
\end{proof}

\begin{lem}\label{lem:equalizer}
Let $\cC$ be a category and let
$f: Y \to X$ be a morphism in $\cC$.
Let us consider the sieve $R_f$ on $X$.
Let $e_f : \frh_\cC(Y) \to \frh_\cC(R_f)$ denote 
the morphism of presheaves on $\cC$ defined as follows:
for each object $Z$ of $\cC$, the map
$e_f(Z): \frh_\cC(Y)(Z) = \Hom_\cC(Z,Y) 
\to \frh_\cC(R_f)(Z)$
sends $g \in \Hom_\cC(Z,Y)$ to $f \circ g \in \frh_\cC(R_f)(Z)$.
Then for any presheaf $F$ on $\cC$, the map
$\Hom_{\Presh(\cC)}(\frh_\cC(R_f),F) \to 
\Hom_{\Presh(\cC)}(\frh_\cC(Y),F) \xto{y_{F,Y}} F(Y)$
is injective and its image is equal to the equalizer of
the two maps 
$\Hom_{\Presh(\cC)}(\frh_\cC(Y),F) \rightrightarrows
\Hom_{\Presh(\cC)}(\frh_\cC(Y) \times_{\frh_\cC(X)} \frh_\cC(Y),F)$
given by the composition with the first and the second
projections 
$\frh_\cC(Y) \times_{\frh_\cC(X)} \frh_\cC(Y)
\to \frh_\cC(Y)$.
\end{lem}

\begin{proof}
It follows from the definition that the subpresheaf 
$\frh_\cC(R_f)$ of $\frh_\cC(X)$ is equal to the image
of the morphism $\frh_\cC(f):\frh_\cC(Y) \to \frh_\cC(X)$
and the induced epimorphism $\frh_\cC(Y) \to \frh_\cC(R_f)$
is equal to $e_f$.
From this we see that 
the inclusion $\frh_\cC(Y) \times_{\frh_\cC(X)} \frh_\cC(Y)
\to \frh_\cC(Y) \times \frh_\cC(Y)$ is an equivalence
relation in $Y$ in the sense of \cite[EXPOS\'E IV, 3.1]{SGA3},
and that the object $\frh_\cC(R_f)$ together with morphism $e_f$ 
is the quotient object (in the category $\Presh(\cC)$) 
of $\frh_\cC(Y)$ by this equivalence relation.
Hence the claim follows.
\end{proof}

\begin{cor}\label{cor:sheaf_criterion2}
Let $\cC$ be a category and let $\cT$ be a semi-localizing collection
of morphisms in $\cC$. Then a presheaf $F$ on $\cC$ is a
$J_\cT$-sheaf if and only if for any morphism $f:Y \to X$
in $\cC$ belonging to $\cT$, the map $F(f):F(X) \to F(Y)$
is injective and its image is equal to the equalizer of
$$
F(Y) \xleftarrow[\cong]{y_{F,Y}} \Hom_{\Presh(\cC)}(\frh_\cC(Y),F)
\rightrightarrows \Hom_{\Presh(\cC)}(\frh_\cC(Y) \times_{\frh_\cC(X)}
\frh_\cC(Y), F).
$$
\end{cor}

\begin{proof}
This follows from Proposition \ref{prop:sheaf_criterion1}
and Lemma \ref{lem:equalizer}.
\end{proof}

\section{Galois coverings}
\label{sec:Galois coverings}
We define what it means for a morphism in 
a category to be a Galois covering.
Then, we define what it means for a topology to
 have enough Galois coverings.
In Section~\ref{sec:quotient objects},
we collect some abstract theory concerning
quotient objects in a category.   
As we deal with automorphisms of an object, it is necessary to make this notion precise.

\subsection{Galois coverings}
\subsubsection{}
Let $X$ be an object of a category $\cC$.
Let $Y_1$ and $Y_2$ be objects of $\cC$ and
suppose that morphisms $f_1:Y_1 \to X$ and $f_2:Y_2 \to X$
are given.
We say that a morphism 
(\resp an isomorphism) $g:Y_1 \to Y_2$ in $\cC$ is
a morphism (\resp an isomorphism) over $X$ if $f_2 \circ g = f_1$.
In other words, $g$ is a morphism (\resp an isomorphism)
over $X$ if it is a morphism 
(\resp an isomorphism) from $f_1$ to $f_2$ in the overcategory $\cC_{/X}$.
The set of morphisms (\resp an isomorphism) from $Y_1$ to $Y_2$ over $X$ 
is denoted by $\Hom_X(Y_1,Y_2)$ (\resp by $\Isom_X(Y_1,Y_2)$).
For a morphism $f:Y \to X$ in $\cC$, we write $\End_X(Y)$ 
(\resp $\Aut_X(Y)$) for $\Hom_X(Y,Y)$ (\resp $\Isom_X(Y,Y)$).
The set $\End_X(Y)$ forms a monoid with respect to the
composition of morphisms, and $\Aut_X(Y)$ is equal to the group
of invertible elements of $\End_X(Y)$.

\begin{defn}
\label{defn:pseudo G-torsor}
Let $S$ be a set on which a group $G$ acts from the left.
We say that a map $\phi : S \to S'$ of sets is a pseudo $G$-torsor
if the following three conditions are satisfied:
\begin{enumerate}
\item We have $\phi(g s) = \phi(s)$ for any $g \in G$ and for any $s \in S$.
\item The group $G$ acts freely on $S$.
\item The map $\quot{S}{G} \to S'$ induced by $\phi$ is injective.
\end{enumerate}
\end{defn}

\begin{defn}\label{defn:Galois_cov}
Let $\cC$ be a category and let $f:Y \to X$
be a morphism in $\cC$.
We say that $f$ is a Galois
covering in $\cC$ if there exists a group $G$ and a homomorphism
$\rho:G \to \Aut_X(Y)$ of groups 
such that the following condition is satisfied:
for any object $Z$ of $\cC$,
the map $\Hom_{\cC}(Z,Y) \to \Hom_{\cC}(Z,X)$
given by the composition with $f$ is a pseudo $G$-torsor. Here each $g \in G$ acts on $\Hom_{\cC}(Z,Y)$ 
by the composition with $\rho(g)$.
\end{defn}

\subsubsection{}
\label{sec:fiber Galois}
Let $f:Y \to X$ be a Galois covering.   
The fiber product of $f$ and $f$ has the following 
description.
Let $\rho:G \to \Aut_X(Y)$ be a homomorphism as in 
Definition \ref{defn:Galois_cov}. Then it can be
checked easily that the diagram
$$
\begin{CD}
\coprod_{g \in \rho(G)} \frh_\cC(Y)
@>{p_2}>> \frh_\cC(Y) \\
@V{p_1}VV @VV{\frh_\cC(f)}V \\
\frh_\cC(Y) @>{\frh_\cC(f)}>>
\frh_\cC(X)
\end{CD}
$$
is cartesian, which induces an isomorphism
from the coproduct 
$\coprod_{g \in \rho(G)} \frh_\cC(Y)$ to the fiber
product $\frh_\cC(Y) \times_{\frh_\cC(X)} \frh_\cC(Y)$.
Here $p_1$ (\resp $p_2$)
is the morphism $\coprod_{g \in \rho(G)} \frh_\cC(Y)
\to \frh_\cC(Y)$ whose component 
$g \in \rho(G)$ is the morphism $\frh_\cC(g)$
(\resp the identity morphism on $\frh_\cC(Y)$).

\begin{lem}\label{lem:Galois_group}
Let $\cC$ be a category and let
$f: Y \to X$ in $\cC$ be a Galois covering in $\cC$.
Let $\rho:G \to \Aut_X(Y)$ be the group homomorphism
satisfying the condition in Definition \ref{defn:Galois_cov}.
Then, we have $\Aut_X(Y) = \End_X(Y)$, and
$\rho$ is an isomorphism.
\end{lem}

\begin{proof}
Let $\phi:\Hom_{\cC}(Y,Y) \to \Hom_{\cC}(Y,X)$
denote the map given by the composition with $f$.
Then, $\phi$ is a pseudo $G$-torsor.
As $\phi^{-1}(f) = \phi^{-1}(\phi(\id_Y))$ is non-empty,
the group $G$ acts simply transitively on $\phi^{-1}(f)$.
As $G \subset \Aut_{Y}(X) \subset \phi^{-1}(f)$,
it follows that $G = \Aut_{Y}(X) = \phi^{-1}(f)$.
This proves the claim.
\end{proof}

\subsubsection{Enough Galois coverings}
\begin{defn}
\label{defn:enough Galois}
Let $\cC$ be a category and $\cT$ 
be a collection of morphisms in $\cC$.
Let $\cT' \subset \cC$ be the collection of 
Galois coverings.
We say that $\cT$  has enough Galois coverings if 
$\wh{\cT'}=\wh{\cT}$.
We say that an $A$-topology $J$ on $\cC$ 
has enough Galois coverings
if $\cT(J)$ has enough Galois coverings
\end{defn}

\begin{cor}\label{cor:sheaf_criterion3}
Let $\cC$ be a category and let $J$ be 
an $A$-topology on $\cC$.
Suppose that $J$ has enough Galois coverings.
Then a presheaf $F$ on $\cC$ is a
sheaf if and only if for any object $X$ of $\cC$ 
and for any Galois covering $f:Y \to X$ in $\cC$ 
which belongs to $\cT(J)$, the map $F(f):F(X) \to F(Y)$
is injective and its image is equal to the $\Aut_X(Y)$-invariant
part $F(Y)^{\Aut_X(Y)}$ of $F(Y)$.
\end{cor}

\begin{proof}
This follows from Lemma \ref{lem:Galois_group}, 
Corollary \ref{cor:sheaf_criterion2}, and 
the remark in Section~\ref{sec:fiber Galois}.
\end{proof}

\subsection{Quotient objects}
\label{sec:quotient objects}
In this paragraph, we recall the notion of a \quotobj
by an action of a group in a general category 
and prove some of its basic properties.

\begin{defn}\label{defn:quotient}
Let $\cC$ be a category, $Y$ an object in $\cC$,
and $G$ a subgroup of $\Aut_{\cC}(Y)$.
A \quotobj $X$ of $Y$ by $G$ is an object
in $\cC$ equipped with
a morphism $c:Y \to \quot{Y}{G}$ in $\cC$ satisfying the
following universal property: for any object $Z$
in $\cC$ and for any morphism
$f : Y \to Z$ in $\cC$ satisfying $f\circ g=f$ for
all $g \in G$, there exists a unique morphism
$\overline{f} : X \to Z$ such that $f=\overline{f}\circ c$.
In other words, a \quotobj $X$ is an
object in $\cC$ that co-represents the covariant functor
from $\cC$ to the category of sets that associates,
to each object $Z \in \cC$, the $G$-invariant part
$\Hom_{\cC}(Y,Z)^G$ of the set $\Hom_{\cC}(Y,Z)$.
We call the morphism $c:Y \to X$ 
the canonical quotient morphism.
\end{defn}
\subsubsection{}
A \quotobj of $Y$ by  $G$ is unique up to
unique isomorphism in the following sense. Suppose that
both $Y'_1$ and $Y'_2$ are \quotobjs of $Y$ by $G$.
We denote by $c_1:Y \to Y'_1$ and $c_2:Y\to Y'_2$ the
canonical quotient morphisms. Then, there exists a unique
isomorphism $\alpha : Y'_1 \xto{\cong} Y'_2$ satisfying
$\alpha \circ c_1 = c_2$.
This claim follows easily from the universality of
\quotobjsp.
We use the symbol $\quot{Y}{G}$ to denote 
any quotient object of $Y$ by $G$.

\subsubsection{}
There is another equivalent way of defining a \quotobjp.
Let $*_G$ denote the category such that $*_G$ 
has only one object $*$, that the set $\Hom_{*_G}(*,*)$ is
equal to $G$, and the composite of morphisms 
and the identity morphism are given by the group structure of $G$.
Then, the \quotobj of $Y$ by $G$ is nothing but 
a colimit in the category $\cC$ of the diagram $*_G \to \cC$ that sends
$*$ to $Y$ and sends $g:*\to *$ to $g:Y\to Y$
for all $g \in G$.

\begin{lem}\label{lem:quotient_epi}
Let the notation be as in Definition \ref{defn:quotient}.
Then, the canonical quotient morphism $c:Y\to \quot{Y}{G}$ 
is an epimorphism.
\end{lem}

\begin{proof}
Suppose that there exist an object $Z \in \cC$ and 
morphisms $f_1,f_2:\quot{Y}{G} \to Z$ satisfying 
$f_1 \circ c = f_2 \circ c$. We prove that $f_1=f_2$.

We set $f=f_1 \circ c = f_2 \circ c$. It follows from
the definition of $\quot{Y}{G}$ that there exists a unique
morphism $\overline{f}:\quot{Y}{G} \to Z$ such that
$f=\overline{f} \circ c$. By the uniqueness of $\overline{f}$,
we have $f_1=f_2=\overline{f}$.
This proves the claim.
\end{proof}

\begin{lem}\label{lem:quot_fullsub}
Let $\cC$ be a category and 
let $\cC' \subset \cC$ be a full subcategory.
Let $Y$ be an object in $\cC'$,
and $G$ a subgroup of $\Aut_{\cC}(Y)$.
Suppose that a \quotobj $\quot{Y}{G}$
of $Y$ by $G$ in $\cC$ exists and that
$\quot{Y}{G}$ is isomorphic in $\cC$ to an object $Z$
in $\cC'$. Then, $Z$ is a \quotobj
of $Y$ by $G$ in $\cC'$.
\end{lem}

\begin{proof}
The universality of $Z$ can be checked easily.
\end{proof}

\begin{lem}\label{lem:Galois_quotient}
Let $\cC$ be a category and let $J$ be a Grothendieck topology on $\cC$.
Suppose that any representable presheaf on $\cC$ is a $J$-sheaf.
Let $f:Y \to X$ be a Galois covering in $\cC$ such that the sieve 
$R_f$ on $X$ belongs to $J(X)$.
Then the object $X$ of $\cC$ together with the morphism $f$
is a quotient object of $Y$ by $\Aut_X(Y)$.
\end{lem}

\begin{proof}
Let $Z$ be an arbitrary object of $\cC$.
It suffices to show that the map $\Hom(X,Z) \to \Hom(Y,Z)$
given by the composition with $f$ is injective and its image
is equal to the $\Aut_X(Y)$-invariant part $\Hom(Y,Z)$.
As $\frh_\cC(Z)$ is a $J$-sheaf and $R_f$ belongs to $J(X)$,
this follows from Corollary \ref{cor:sheaf_criterion3}.
\end{proof}

\section{$B$-sites}
\label{sec:B-sites}
We define $B$-sites in this section and study 
their properties.
The reader will find in Section~\ref{sec:basic} 
that the basic statements from 
Galois theory also hold true in our setting.   
Section~\ref{sec:cofinality} contains a technical 
proposition and its corollary, which are derived in a manner very different 
from the classical Galois theory.
In the classical case, the underlying category
of the site contains a final object and the proofs 
of these statements 
are much easier.
In Section~\ref{sec:sheafify}, we assume that 
the $B$-site has enough Galois coverings,
and give an explicit description of the
sheafification functor in terms of the Galois
coverings.  
This will be used in Section~\ref{sec:9.1}.

We will use the convention for the terminology ``poset'' used in Section~\ref{sec:poset}.

\subsection{$E$-categories}
\label{sec:E-categories}
\begin{defn}
\label{defn:E-categories}
We say that a category $\cC$ is an $E$-category if every morphism
in $\cC$ is an epimorphism.
\end{defn}

\begin{lem}\label{lem:Galois_descends}
Let $\cC$ be a category.
Let $Y_1 \xto{f_1} X \xleftarrow{f_2} Y_2$ be a diagram in $\cC$.
Suppose that $f_2$ is 
a Galois covering in $\cC$.
Then, for any two morphisms $h_1,h_2 :Y_1 \to Y_2$
that are over $X$, there exists a unique element 
$g \in \Aut_{X}(Y_2)$
satisfying $h_1 = g \circ h_2$. 
\end{lem}
\begin{proof}
Let $\alpha: \Hom_{\cC}(Y_1,Y_2) \to 
\Hom_{\cC}(Y_1,X)$ denote the map given by the composition with $f_2$.
We have $\alpha(h_1) = \alpha(h_2) = f_1$.
As $f_2$ is a Galois covering, it follows from Lemma \ref{lem:Galois_group}
that $\alpha$ is a pseudo $\Aut_{X}(Y_2)$-torsor,
i.e.,
the group $\Aut_{X}(Y_2)$ acts simply transitively
on the set $\alpha^{-1}(f_1)$.
This proves the claim.
\end{proof}

\begin{lem}\label{lem:Galois_epi}
Let $\cC$ be an $E$-category.
Let $f:Y \to X$ be a Galois covering in $\cC$.
Suppose that $f$ is written as the composite
$f=f_1 \circ f_2$ of two morphisms in $\cC$.
Then $f_2$ is a Galois covering in $\cC$.
\end{lem}

\begin{proof}
Let $X'$ denote the target of the morphism $f_2$.
The group $\Aut_{X'}(Y)$ is a subgroup of the group $\Aut_{X}(Y)$.
It suffices to show that, for any commutative
diagram
$$
\begin{CD}
Z @>{h}>> Y \\
@V{h'}VV @VV{f_2}V \\
Y @>{f_2}>> X',
\end{CD}
$$
there exists a unique automorphism
$g \in \Aut_{X'}(Y)$
satisfying $h' = g \circ h$.
As $f$ is a Galois covering, there exists
a unique automorphism $g \in \Aut_{X}(Y)$
satisfying $h' = g \circ h$.
Hence to prove the claim it suffices to prove
that $g \in \Aut_{X'}(Y)$.
We have $f_2 \circ g \circ h
= f_2 \circ h' = f_2 \circ h$.
As $h$ is an epimorphism, we have
$f_2 \circ g = f_2$. Hence we have
$g \in \Aut_{X'}(Y)$, which completes the proof.
\end{proof}

\subsection{$B$-sites}
\label{sec:basic}
We write $(\cC, J)$ to denote a site whose underlying 
category is $\cC$ and whose Grothendieck
topology is $J$.
\begin{defn}
\label{defn:B-site}
A $B$-site $(\cC, J)$ is a site satisfying the following
conditions:
\begin{enumerate}
\item $\cC$ is an $E$-category.
\item $J$ is an $A$-topology
\item For any diagram $Z \xto{g} Y \xto{f} X$ in $\cC$, the composite
$g \circ f$ belongs to $\cT(J)$ if and only if $f$ and $g$ belong to 
$\cT(J)$.
\end{enumerate}
\end{defn}
\noindent 
It follows 
from Proposition~\ref{prop:cTJ} that
if $g \circ f \in \cT(J)$, then $g \in \cT(J)$.
Hence we may replace Condition(3) above by 
the weaker condition:
If $g \circ f \in \cT(J)$, then $f \in \cT(J)$.

Let $(\cC,J)$ be a $B$-site.
We say that a morphism $f:Y \to X$ in $\cC$ 
is a Galois covering in $\cT(J)$ if $f$ belongs to $\cT(J)$, 
and $f$ is a Galois covering in $\cC$.

\newcommand{\EtCon}{\mathrm{EtConn}}
\begin{ex}
\label{ex:connected etale site}
Here we give a basic example of a $B$-site.
Let $S$ be a connected noetherian scheme.
Let us consider the full subcategory $\EtCon_S$ of the
category of schemes over $S$ whose objects are
connected $S$-schemes that are finite \'etale over $S$.
Then, the category $\EtCon_S$ is semi-cofiltered,
and the pair of $\EtCon_S$ and the atomic topology
is $B$-site. 

If we do not assume them to be connected,
Condition (3) is not satisfied.
\end{ex}

\begin{lem}\label{lem:Galois_lifts}
Let $(\cC,J)$ be a $B$-site.
Let $Y_1 \xto{f_1} X \xleftarrow{f_2} Y_2$ be a diagram
in $\cC$.
Suppose that $f_1$ is a Galois covering in $\cT(J)$ and that
there exists a morphism $h:Y_1 \to Y_2$ over $X$.
Then 
\begin{enumerate}
\item For any commutative diagram
$$
\begin{CD}
Z @>{f'_2}>> Y_1 \\
@V{f'_1}VV @VV{f_1}V \\
Y_2 @>{f_2}>> X
\end{CD}
$$
in $\cC$,
there exists an automorphism $g \in \Aut_{X}(Y_1)$ such that
$f'_1 = h \circ g \circ f'_2$.
Moreover, the morphism $h \circ g$ is uniquely determined
by the commutative diagram above.
\item The group $\Aut_{X}(Y_1)$ acts transitively on the
set $\Hom_{X}(Y_1,Y_2)$ and the diagram
$$
\begin{CD}
\coprod_{h \in \Hom_{X}(Y_1,Y_2)} \frh_\cC(Y_1) 
@>{f'_2}>> \frh_\cC(Y_1) \\
@V{f'_1}VV @VV{\frh_\cC(f_1)}V \\
\frh_\cC(Y_2) @>{\frh_\cC(f_2)}>> \frh_\cC(X)
\end{CD}
$$
in $\Presh(\cC)$ is cartesian. Here, $f'_1$
denotes the morphism whose component at $h$ is the identity map
for every $h$, and $f'_2$ denotes the morphism whose component
at $h$ is equal to $h_\cC(h)$.
\end{enumerate}
\end{lem}

\begin{proof}
Let the notation be as in claim (1).
Observe that by Condition~(3) of 
Definition~\ref{defn:B-site} the morphism $h$ belongs to $\cT(J)$.
As $\cT(J)$ is semi-localizing,
there exists an object $Z'$ of $\cC$ 
and morphisms $h': Z' \to Z$ and
$f''_1:Z'\to Y_1$ in $\cC$ such that $h'$ belongs to $\cT(J)$ 
and that the diagram
$$
\begin{CD}
Z' @>{h'}>> Z \\
@V{f''_1}VV @VV{f'_1}V \\
Y_1 @>{h}>> Y_2
\end{CD}
$$
is commutative. 

Set $h'' =f'_2 \circ h' : Z' \to Y_1$. 
As both $h''$ and $f''_1$ are morphisms over $X$
and $f_1$ is a Galois covering,
it follows from Lemma \ref{lem:Galois_descends}
that there exists an element $g \in \Aut_{X}(Y_1)$ 
satisfying $f''_1=g \circ h''$.
Hence, we have 
$f'_1 \circ h' = h \circ f''_1 = 
h \circ g \circ h'' = 
h \circ g \circ f'_2 \circ h'$.
As $\cC$ is an $E$-category,
the morphism $h'$ is an epimorphism. 
Hence we have $f'_1 = h \circ g \circ f'_2$.
The uniqueness of $h \circ g$ follows as
$f'_2$ is an epimorphism.
This proves claim (1).

Let the notation be as in claim (2).
Let $h_2 :Y_1 \to Y_2$ be any morphism in $\cC$ over $X$.
We apply claim (1) to the commutative diagram
$$
\begin{CD}
Y_1 @>{\id_{Y_1}}>> Y_1 \\
@V{h_2}VV @VV{f_1}V \\
Y_2 @>{f_2}>> X.
\end{CD}
$$
There exists an automorphism $g \in \Aut_{X}(Y_1)$
such that $h_2 = h \circ g$.
This proves that the action of $\Aut_{X}(Y_1)$
on $\Hom_{X}(Y_1,Y_2)$ is transitive.
The rest of claim (2) follows as an immediate
consequence of claim (1).
\end{proof}

\begin{cor}
\label{cor:Hom bijective}
Let $(\cC,J)$ be a $B$-site. 
Let $Y_1 \xto{f_1} X \xleftarrow{f_2} Y_2$ be a diagram
in $\cC$.
Suppose that $f_1$ is a Galois covering in $\cT(J)$ and that
there exists a morphism $h:Y_1 \to Y_2$ over $X$.
Then, for any morphism $f':Z \to Y_1$ in $\cC$,
the map $\Hom_{X}(Y_1,Y_2) \to \Hom_{X}(Z,Y_2)$
given by the composite with $f'$ is bijective.
\end{cor}

\begin{proof}
The injectivity follows because $f'$ is an epimorphism.
We prove the surjectivity as follows.
Let $h':Z \to Y_2$ be a morphism over $X$.
It follows from Lemma \ref{lem:Galois_lifts} (1) that
there exists an automorphism $g \in \Aut_{X}(Y_1)$ such that
$h' = h \circ g \circ f'$. This proves the surjectivity.
\end{proof}

\subsubsection{ }
Let $f:Y\to X$ be a morphism in a category $\cC$.
Let $\wt{\alpha} : Y \xto{\cong} Y$ and
$\alpha : X \xto{\cong} X$ be automorphisms in $\cC$.
We say that $\wt{\alpha}$ descends to $\alpha$ via $f$,
or that $\alpha$ ascends to $\wt{\alpha}$, 
if the diagram
\begin{equation}\label{eq:stabilize}
\begin{CD}
Y @>{\wt{\alpha}}>{\cong}> Y \\
@V{f}VV @VV{f}V \\
X @>{\alpha}>{\cong}> X
\end{CD}
\end{equation}
is commutative.
We say that an automorphism $\wt{\alpha} \in \Aut_{\cC}(Y)$ 
descends to $X$ via $f$ if it descends to some element
in $\Aut_{\cC}(X)$.
Suppose that the morphism $f$ is 
an epimorphism. Then, if $\wt{\alpha} \in \Aut_{\cC}(Y)$ 
descends to $X$ via $f$, it descends to a unique element
in $\Aut_{\cC}(X)$ via $f$.

\begin{lem}\label{lem:descends}
Let $\cC$ be a category.
Let $Y \xto{f} X \xto{f'} W$ be a 
diagram in the category $\cC$. Suppose that
$f'$ is a Galois covering in $\cC$.
\begin{enumerate}
\item Any automorphism $\wt{g} \in \Aut_{W}(Y)$ descends to 
a unique element in $g \in \Aut_{W}(X)$ via $f$.
\item The map $\Aut_{W}(Y) \to \Aut_{W}(X)$
which sends $\wt{g}$ to $g$ is a group homomorphism.
\end{enumerate}
\end{lem}

\begin{proof}
We apply Lemma \ref{lem:Galois_descends}
to the two morphisms $f, f \circ \wt{g}: Y \to X$
over $W$. There exists a unique element 
$g \in \Aut_{W}(X)$ satisfying 
$f \circ \wt{g} = g \circ f$.
Hence, $\wt{g}$ descends to $g$ via $f$.
This proves claim (1).

Claim (2) follows from the uniqueness of
$g \in \Aut_{\cC}(X)$ to which $\wt{g}$ 
descends via $f$.
\end{proof}

\begin{lem}\label{lem:lifts}
Let $(\cC,J)$ be a $B$-site. 
Let $Y \xto{f} X \xto{f'} W$ be a 
diagram in $\cC$ such that the
composite $f'\circ f$ is a Galois covering in $\cT(J)$.
\begin{enumerate}
\item Any automorphism in $\Aut_{W}(X)$
ascends to an automorphism in $\Aut_{W}(Y)$
via $f$.
\item $f'$ is a Galois covering if and only if any automorphism
in $\Aut_{W}(Y)$ descends to an automorphism in $\Aut_{W}(X)$ via $f$.
\item Suppose that $f'$ is a Galois covering. 
Then the group homomorphism 
$\Aut_{W}(Y) \to \Aut_{W}(X)$ in Lemma \ref{lem:descends}
induces a short exact sequence
$$
1 \to \Aut_{X}(Y) \to
\Aut_{W}(Y) \to \Aut_{W}(X) \to 1
$$
of groups.
\end{enumerate}
\end{lem}
\begin{proof}
Claim (1) follows from Lemma \ref{lem:Galois_lifts}.

Suppose that any automorphism in $\Aut_{W}(Y)$ 
descends to an automorphism in $\Aut_{W}(X)$.
Let $f'':W' \to W$ be a morphism in $\cC$.
As $\cC$ is an $E$-category, 
the group $\Aut_{W}(X)$ acts freely on the
set $\Hom_{W}(W',X)$. Hence to prove that $f'$ is a Galois covering, 
it suffices to prove that
the group $\Aut_{W}(X)$ acts 
transitively on the set $\Hom_{W}(W',X)$.
Let $h_1,h_2 : W' \to X$ be morphisms over $W$.
As $f$ belongs to $\cT(J)$ and $\cT(J)$ is 
semi-localizing, 
there exist
an object $W''$ of $\cC$ and morphisms 
$h'_1 :W'' \to Y$ and $f'_2:W'' \to W'$ such that
$f'_2$ belongs to $\cT(J)$ and that $f \circ h'_1 = h_1 \circ f'_2$.
It then follows from Lemma \ref{lem:Galois_lifts} (1)
that there exists an automorphism $\wt{g} \in \Aut_{W}(Y)$
satisfying $h_2 \circ f'_2 = f \circ \wt{g} \circ h'_1$.
By assumption, the automorphism $\wt{g}$ descends to
an automorphism $g \in \Aut_W(X)$ via $f$. 
Hence,
$h_2 \circ f'_2 = f \circ \wt{g} \circ h'_1
= g \circ f \circ h'_1 = g \circ h_1 \circ f'_2$.
As $f'_2$ is an epimorphism, we have $h_2 = g \circ h_1$.
Hence, the group $\Aut_{W}(X)$ acts 
transitively on the set $\Hom_{W}(W',X)$.
This proves that $f'$ is a Galois covering
if any automorphism in $\Aut_{W}(Y)$ 
descends to an automorphism in $\Aut_{W}(X)$.
The converse follows immediately from Lemma \ref{lem:descends}.
This proves claim (2).

It follows from claim (1) that the homomorphism
$\Aut_{W}(Y) \to \Aut_{W}(X)$ is surjective.
It is easy to see that the kernel of this homomorphism
is equal to $\Aut_{X}(Y)$, which proves claim (3).
\end{proof}

\subsection{Cofinality}
\label{sec:cofinality}
\subsubsection{}
For a poset $P$, we denote by $\cP^\op$ the dual of the poset $P$,
regarded as a $\frU$-small category. Let $\cC$ be a category and
let $F:\cP^\op \to \cC$ be a covariant functor.
Let $(X,(f_y)_{y \in P})$ be a pair of an object $X$ of $\cC$
and a morphism $f_y:X \to F(y)$ in $\cC$ for each $y \in P$.
We say that the pair $(X,(f_y)_{y \in P})$ is an object
of $\cC$ over $F$ if for any $y_1,y_2 \in P$ with $y_1 \le y_2$, 
we have $f_{y_1} = F(g) \circ f_{y_2}$, where $g$ 
denotes the
unique morphism from $y_2$ to $y_1$ in $\cP^\op$.

\begin{prop}\label{prop:tree}
Let $\cC$ be a category that is semi-cofiltered.
Let $P$ be a non-empty finite poset.
Suppose that $P$ is a rooted tree, which means that
$P$ has the bottom element and for any element
$y \in P$, the subset $\{z \in P \ |\ z \le y \}$
is a totally ordered set.
Let $\cP^\op$ denote the dual of the poset $P$,
regarded as a $\frU$-small category.
Let $F:\cP^\op \to \cC$
be a covariant functor from $\cP^\op$ to $\cC$.
Then there exists an object $(X,(f_y)_{y \in P})$ 
of $\cC$ over $F$.
\end{prop}

\begin{proof}
We proceed by induction on the cardinality of $P$.
If $P$ consists of a single element $y$, then 
$X=F(y)$ and $f_y=\id_X$ satisfies the desired property.
Suppose that $P$ consists of more than two elements.
As $P$ is a finite poset, there exists a maximal
element $z\in P$. Set $P' = P \setminus \{z\}$.
Then $P'$ is again a tree.
Let $F:\cP^{\op} \to \cC$ be a covariant functor.
We denote by $F'$ the restriction of $F$ to ${\cP'}^{\op}$. 
By the induction hypothesis, there exists
an object $(X',(f'_y)_{y \in P'})$ in $\cC$ 
over $F'$.
We let $z' \in P'$ denote the maximal element of the
finite totally ordered set $\{ y \in P'\ |\ y \le z \}$.
Let $f:z \to z'$ denote the unique morphism in $\cP^{\op}$.
Let us consider the diagram 
$$
F(z) \xto{F(f)} F(z') \xleftarrow{f'_{z'}} X'
$$
in $\cC$. (We wrote $X'$ for the source of $f'_{z'}$).
As $\cC$ is semi-cofiltered,
there exist an object $X$ of $\cC$ and morphisms
$f': X \to F(z)$ and $g': X \to X'$ in $\cC$
satisfying $F(f) \circ f' = f'_{z'} \circ g'$.
We set $f_y = f'_y \circ g'$ for $y \in P'$
and $f_z = f'$.
It then logically follows to check that the pair 
$(X,(f_y)_{y \in P})$ is an object in $\cC$ over $F$.
This proves the claim.
\end{proof}

\subsubsection{}
When $\cT$ is a semi-localizing collection of morphisms in $\cC$, recall that we have
$\wh{\cT}=\cT(J)$
where $J$ is the $A$-topology associated with $\cT$.
We let $\cC(\cT(J))$ denote the following subcategory of $\cC$.
The objects of $\cC(\cT(J))$ are the objects of $\cC$.
For two objects $X$, $Y$ of $\cC$, the morphisms from $X$ to $Y$
in $\cC(\cT(J))$ are the morphisms from $X$ to $Y$ in $\cC$
belonging to $\wh{\cT}$.
It follows from Lemma \ref{lem:composite}
that the composition of morphisms in $\cC(\cT(J))$ is well-defined.

\begin{cor}[cofinality]
\label{lem:cofinality}
Let $(\cC,J)$ be a $B$-site. 
Suppose that $\cT(J)$ 
has enough Galois coverings.
Let $P$ be a finite poset. 
Suppose that $P$ is a rooted tree.
Let $\cP^{\op}$ denote the dual of the poset $P$
regarded as a small category.
Let $F:\cP^{\op} \to \cC$ be a covariant functor
such that $F(h)$ belongs to $\cT(J)$ for any morphism 
$h$ in $\cP^{\op}$.
Then there exists an object $(X,(f_y)_{y \in P})$ of $\cC$ 
over $F$ such that
the morphism $f_y:X \to F(y)$ is a Galois covering in 
$\cT(J)$ for each $y \in P$.
\end{cor}

\begin{proof} 
Set $\cT=\cT(J)$.  Recall that $\wh{\cT} = \cT$.
It follows from the definition that
$\cC(\cT(J))$ is semi-cofiltered.

We regard $F$ as a covariant functor from $\cP^\op$ to $\cC(\cT(J))$.
It follows from Proposition \ref{prop:tree}
that there exists an object $(X',(f'_y)_{y \in P})$ of $\cC(\cT(J))$ over $F$.
Let $y_0 \in P$ denote the bottom element.
As $\cT(J)$ has enough Galois coverings, 
there exist an object $X$ of $\cC$ and a morphism $h:X \to X'$ in $\cC$
such that $h$ belongs to $\cT(J)$ 
and that the composite $f'_{y_0} \circ h$ is a Galois covering.
For each $y \in P$, we set $f_y = f'_y \circ h$.
Then $(X,(f_y)_{y \in P})$ is an object in $\cC$ over $F$.
As the morphism $f_{y_0}$ factors through $f_y$,
it follows from Lemma \ref{lem:Galois_epi} that
$f_y$ is a Galois covering in $\cT(J)$.
This proves the claim.
\end{proof}

\subsection{An explicit construction of the sheafification functor}
\label{sec:sheafify}

Let $(\cC,J)$ be a $B$-site.
Suppose that $\cC$ is essentially $\frU$-small.
In the paragraphs below, we give an explicit description
of the sheafification functor $a_J:\Presh(\cC)
\to \Shv(\cC,J)$ 
when there are enough Galois coverings.
Let us assume throughout 
this subsection that there are enough Galois coverings.
\subsubsection{ }
\label{sec:sheafify1}
Let us fix a set $G$ of objects of $\cC$ such that any object
of $\cC$ is isomorphic to an object that belongs to $G$.
Let $F: \cC^\op \to (\Sets)$ be a presheaf on $\cC$.
In the following paragraphs we give a description of the sheaf 
$a_J(F):\cC^\op \to (\Sets)$ associated with $F$.

\subsubsection{ }
\label{sec:Gal/X}
For an object $X$ of $\cC$, we let
$(\Gal/X)'$ denote the following category.
The objects in $(\Gal/X)'$ are the pairs
$(Y,f)$ of an object $Y$ of $\cC$ that belongs to $G$ 
and a morphism $f:Y\to X$ in $\cC$ that is a Galois
covering in $\cT(J)$. For two objects $(Y_1,f_1)$ and $(Y_2,f_2)$
of $(\Gal/X)'$, the set of morphisms from $(Y_1,f_1)$ to $(Y_2,f_2)$ 
in $(\Gal/X)'$ is the set $\Aut_{X}(Y_2) \backslash \Hom_X(Y_1,Y_2)$.
It is clear that the category $(\Gal/X)'$ has a $\frU$-small 
full subcategory whose embedding into $(\Gal/X)'$ is an equivalence
of categories. We fix such a $\frU$-small full subcategory of
$(\Gal/X)'$ and denote it by $\Gal/X$.
It follows from Lemma \ref{lem:Galois_descends} that
for any two objects $(Y_1,f_1)$ and $(Y_2,f_2)$
of $\Gal/X$, there exists at most one morphism
from $(Y_1,f_1)$ to $(Y_2,f_2)$ in $\Gal/X$.
Hence, the composition of two morphisms is well-defined.
It follows from Lemma \ref{lem:cofinality} that the
category $\Gal/X$ is cofiltered.

Let $X$ be an object of $\cC$. 
For an object $(Y,f)$ of $\Gal/X$, we set
$F_{/X}(Y,f) = F(Y)^{\Aut_X(Y)}$.
Let $(Y_1,f_1)$ and $(Y_2,f_2)$ be two objects of
$\Gal/X$ and let $h:Y_1 \to Y_2$ be a morphism over $X$.
It then follows from Lemma \ref{lem:Galois_descends}
that the map $F(h) :F(Y_2) \to F(Y_1)$ sends 
an element in the $\Aut_{X}(Y_2)$-invariant part 
$F_{/X}(Y_2,f_2) \subset F(Y_2)$ to an element in $F_{/X}(Y_1,f_1)$.
The induced map $F_{/X}(h) : F_{/X}(Y_2,f_2)
\to F_{/X}(Y_1,f_1)$ depends only on the class of the class of $h$
in $\Aut_{X}(Y_2) \backslash \Hom_X(Y_1,Y_2)$.
We obtain a contravariant functor $F_{/X}$ from $\Gal/X$ 
to the category of sets.

For an object $(Y,f)$ of $\Gal/X$, Lemma \ref{lem:equalizer} gives
a bijection $F_{/X}(Y,f) \cong \Hom_{\Presh(\cC)}(\frh_\cC(R_f),F)$
that is functorial in $(Y,f)$.
It follows from 
Definition~\ref{defn:semi-localizing} that as 
there are enough Galois coverings,
the set 
$\{ R_f\ | f \in \Obj \Gal/X \}$ is cofinal in the collection $J(X)$.
Hence, it follows from the construction of the sheafifiction functor,
given in Section 3 of \cite[EXPOSE II]{SGA4}, that we have a bijection
\begin{equation}\label{eq:a_J}
a_J(F)(X) \cong \displaystyle 
\varinjlim F_{/X} := \varinjlim_{(Y,f)} F_{/X}(Y,f)
\end{equation}
where the colimit is taken over the objects in
the $\frU$-small category $\Gal/X$.

\subsubsection{ }
Let $f:X \to X'$ be a morphism in $\cC$.
We give a description of the restriction map 
$a_J(F)(f): a_J(F)(X') \to a_J(F)(X)$ via the bijection \eqref{eq:a_J}.
Let $\Gal/f$ denote the full subcategory of $\Gal/X$
whose objects are the pairs $(Y,h)$ in $\Gal/X$ such
that the composite $f \circ h : Y \to X'$ is a
Galois covering in $\cC$.
It follows from Lemma \ref{lem:cofinality} that
the subcategory $\Gal/f$ is cofinal in $\Gal/X$.

To each object $(Y,h)$ in $\Gal/f$, we
associate the object $(Y,f\circ h)$ in $\Gal/{X'}$.
It is easy to check that this defines a
functor $\Gal/f \to \Gal/{X'}$.

\begin{lem}\label{lem:fully_faithful}
The functor $\Gal/f \to \Gal/{X'}$ is fully faithful,
and its image is cofinal in $\Gal/{X'}$.
\end{lem}

\begin{proof}
First we prove that the functor is fully faithful.
Let $f'_1:Y_1 \to X$ and $f'_2:Y_2 \to X$ be Galois
coverings such that both $f \circ f'_1$ and
$f \circ f'_2$ are Galois coverings,
and let $h:Y_1 \to Y_2$ be a morphism over $X'$.
It suffices to prove that there exists
an automorphism $g \in \Aut_{X'}(Y_2)$
such that $g \circ h$ is a morphism over $X$.
We apply Lemma \ref{lem:Galois_lifts} (1) to the
commutative diagram
$$
\begin{CD}
Y_1 @>{h}>> Y_2 \\
@V{f'_1}VV @VV{f \circ f_2}V \\
X @>{f}>> X'.
\end{CD}
$$
There exists $g \in \Aut_{X'}(Y_2)$ such that
$f'_1 = f'_2 \circ g \circ h$. This proves that
the functor is fully faithful.

The cofinality of the image of this functor
follows from Corollary~\ref{lem:cofinality}.
\end{proof}

\subsubsection{}
We define the map 
$f^*: \varinjlim F_{/X'} 
\to \varinjlim F_{/X}$ as follows.
As $\Gal/f$ is cofinal in $\Gal/X$, the set
$\varinjlim F_{/X}$ is equal to the colimit
$$
\varinjlim F_{/X}|_{\Gal/f}:=\varinjlim_{(Y,h) \in \Obj \Gal/f} F_{/X}(Y,h)
$$
of the restriction of $F_{/X}$ to $\Gal/f$.
It follows from Lemma \ref{lem:fully_faithful} that
$\varinjlim F_{/X'}$ 
is equal to the colimit 
$$
\varinjlim F_{/X'}|_{\Gal/f} 
:= \varinjlim_{(Y,h) \in \Obj \Gal/f} F_{/X'}(Y,f \circ h)
$$
of the restriction of $F_{/X'}$ to $\Gal/f$.
The inclusion map $F_{/X'}(Y,f \circ h)=
F(Y)^{\Aut_{X'}(Y)} \inj
F(Y)^{\Aut_{X}(Y)} = 
F_{/X}(Y,f)$ 
for each
object $(Y,h)$ in $\Gal/f$ gives a natural transform
$F_{/X'}|_{\Gal/f} \to 
F_{/X}|_{\Gal/f}$.
Passing to the colimit, we obtain the desired map
$f^* :\varinjlim F_{/X'} \to \varinjlim F_{/X}$.

It then follows from the construction of $a_J(F)$ given in
Section 3 of \cite[EXPOSE II]{SGA4} that 
via the bijections \eqref{eq:a_J} for $X$ and $X'$, 
the map $f^*$ gives the restriction map 
$a_J(f): a_J(F)(X) \to a_J(F)(X')$.

\section{Grid, fiber functor, 
Galois monoid of a $Y$-site}
\label{sec:grids}

The primary aim of this section is to state our main
theorem (Theorem~\ref{thm:Galois_main}).   

We define what we call a $Y$-site, which is a $B$-site
satisfying additional conditions.  
We define a pregrid and a grid of such a site.
When we are considering the \'etale site 
of the spectrum of a field, the pregrid and the grid
are analogues
of the algebraic closure of the field.
A pregrid may also be regarded as an analogue
of the maximal tree associated with a graph.

To a grid, we associate an analogue of the fiber functor 
(in the sense of Galois category theory).
We also have 
an analogue of the absolute Galois group,
which turns out to be a monoid (not necessarily 
a group) with a certain family of specified submonoids.

We prove in Section~\ref{sec:grid existence} 
that under certain cardinality conditions, a grid of $Y$-site exists.

The main theorem states that,
under some cardinality assumptions,
when a grid exists we have the equivalence of 
categories of sheaves on the site and the smooth
sets of the absolute Galois monoid.

In Sections~\ref{sec:poset} and~\ref{sec:quasi-posets},
we develop some conventions regarding the term poset and 
define the term quasi-poset. To clarify terminology, we regard a partially ordered set
naturally as a category, and we call it a poset 
in this article.   A quasi-poset is a category 
that is similar to and almost a poset.
In Section~\ref{sec:torsors},
we prove a lemma on the non-emptiness of 
certain projective limit.
This will be used in Section~\ref{sec:grid existence}, 
where
we construct a pregrid of a $Y$-site
assuming certain cardinality conditions.

\subsection{Convention: Posets}
\label{sec:poset}
By a partially ordered set, or a poset, we mean a 
set equipped with a partial order.  We assume 
throughout that the partial order is antisymmetric.  

Let $P$ be a poset.
Let $\cC_P$ denote the following category.
The objects of $\cC_P$ are the elements of $P$.
For two elements $x$  and $y$ of $P$, there exists
at most one morphism from $x$ to $y$ in $\cC_P$
and there exists a morphism from $x$ to $y$
if and only if $x \le y$. 
The assignment that sends a partially ordered set $P$
to the category $\cC_P$ is a functor from 
the category of partially ordered sets 
to the category of categories.  
This functor is faithful.  

In what follows, we regard a poset as a category
by the functor above.  By abuse of notation, 
we say that a category
is a poset if it lies in the image of the 
functor above.

\subsection{Quasi-posets}
\label{sec:quasi-posets}
\begin{defn}
\label{defn:quasi-posets}
We say that a category $\cC$ is 
a {\em quasi-poset}
if the following two conditions are satisfied:
\begin{enumerate}
\item The category $\cC$ is non-empty and 
essentially $\frU$-small.
\item The category $\cC$ is thin, i.e., for
any two objects $X$, $Y$ of $\cC$, there exists
at most one morphism from $X$ to $Y$ in $\cC$.
\end{enumerate}
\end{defn}

Let $\cC$ be a quasi-poset
satisfying the following
condition.
\begin{enumerate}
\setcounter{enumi}{2}
\item The category $\cC$ is skeletal, i.e.,
any two objects of $\cC$ that are isomorphic are equal.
\end{enumerate}
Then $\cC$ is a poset.

\subsubsection{}
Let $\cC$ be a quasi-poset.
Then there exists a poset 
subcategory (i.e., a subcategory that is a poset)
$\cC_1$ of $\cC$
such that the inclusion functor
$i: \cC_1 \inj \cC$ 
is an equivalence of categories.
The subcategory $\cC_1$ is unique in
the following sense: 
for any other such poset subcategory
$i' : \cC'_1 \inj \cC$, there
exists a unique isomorphism
$\alpha: \cC_1 \cong \cC'_1$ of categories
such that the functor
$i$ is naturally isomorphic to the composite
functor $i' \circ \alpha$.
We call the poset subcategory $\cC_1$
a {\it poset skeleton} of $\cC$.
The following lemma can be checked easily:
\begin{lem}\label{lem:induced_skeleton}
Let $\cC$ and $\cC'$ be quasi-posets 
and let $\cC_1$ and $\cC'_1$
be their poset skeletons.
Then, for any functor $F:\cC \to \cC'$, there
exists a unique functor $F_1:\cC_1 \to \cC'_1$
such that the diagram
$$
\begin{CD}
\cC_1 @>{F_1}>> \cC'_1 \\
@V{i}VV @VV{i'}V \\
\cC @>{F}>> \cC',
\end{CD}
$$
where $i$ and $i'$ denote the
inclusion functors, is commutative 
up to natural isomorphisms.
\qed
\end{lem}

\begin{cor}
Any equivalence of categories 
between two posets
is an isomorphism of categories.
\qed
\end{cor}

It follows from condition (3) that
any category $\cC$ and for any 
poset $\cC'$,
the natural isomorphisms between 
two functors from $\cC$ to $\cC'$ are the identities.
In particular, the $2$-category of
posets (the full subcategory of the category of 
$\frU$-small categories where the objects are posets) 
can be regarded as a $1$-category.

\subsection{Torsors under pro-groups}
\label{sec:torsors}
Let $I$ be a directed partially ordered set and
let $G = (G_i)_{i \in I}$ be a projective
system of groups indexed by the elements of $I$.
For $i,j \in I$ with $i \le j$, we denote by
$\phi_{j,i}:G_j \to G_i$ the transition homomorphism
in the projective system $G$.
A (left) $G$-torsor is a projective system $(S_i)_{i \in I}$
of sets equipped with a left action of
the group $G_i$ on $S_i$ for each $i \in I$
such that $S_i$ is a left $G_i$-torsor for each $i \in I$
and that for any two $i,j$ with $i \le j$ and for
any $g \in G_j$, $s \in S_j$, we have
$f_{j,i}(g \cdot s) = \phi_{j,i}(g) \cdot f_{j,i}(s)$,
where $f_{j,i}:S_j \to S_i$ denotes the
transition map.

We say that the projective system $G$ has 
{\em trivial first non-abelian cohomology} if
for any $G$-torsor $(S_i)_{i \in I}$, the
limit $\varprojlim_{i \in I} S_i$ of sets is non-empty.
\begin{lem}\label{lem:admissible}
Let $I$ be a directed partially ordered set and
let $G=(G_i)_{i \in I}$ be a projective system of groups.
Suppose that at least one of the following conditions is
satisfied:
\begin{enumerate}
\item For every $i \in I$, $G_i$ is a finite group.
\item $I$ is a finite set.
\item $I$ is a countable set, and the transition maps 
are surjective.
\end{enumerate}
Then, $G$ has trivial first non-abelian cohomology.
\end{lem}

\begin{proof}
In the case in which Condition (1) is satisfied, the claim follows, as
any filtered limit of non-empty finite sets is non-empty.
In the case in which Condition (2) is satisfied, the claim can be checked
directly.    Suppose Condition (3) is satisfied.
Observe that then there exists a cofinal partially ordered subset $J \subset I$
which is isomorphic to the opposite
of the partially ordered set of natural numbers 
$\Z_{\ge 0}$.   The claim follows by using $J$ 
and the surjectivity.
\end{proof}

We can rephrase the condition that
$G$ has trivial first non-abelian cohomology
by introducing the notion of 
the {\em first non-abelian cohomology}
$R^1 \varprojlim_{i \in I} G_i$. This is a pointed set
defined as follows. Let $J$ denote the set of pairs
$(i,j) \in I \times I$ with $i \le j$.
Let $Z^1 \varprojlim_{i \in I} (G_i)$
denote the set of elements 
$(g_{i,j})_{(i,j) \in J} \in \prod_{(i,j) \in J} G_i$
satisfying $g_{i,j} \phi_{j,i}(g_{j,k}) = g_{i,k}$
for any $i,j,k \in I$ with $i \le j \le k$.
We say that two elements $(g_{i,j})$ and
$(g'_{i,j})$ in $Z^1 \varprojlim_{i \in I} (G_i)$
are equivalent if there exists an element
$(h_i) \in  \prod_{i \in I} G_i$ satisfying
$g'_{i,j} = h_i^{-1} g_{i,j} \phi_{j,i}(h_j)$
for any $(i,j) \in J$.
We define $R^1 \varprojlim_{i \in I} G_i$
to be the quotient of $Z^1 \varprojlim_{i \in I} (G_i)$ 
under the equivalence relation above.
We regard $R^1 \varprojlim_{i \in I} G_i$ as the
set pointed at the class of $(1)_{(i,j) \in J} 
\in Z^1 \varprojlim_{i \in I} (G_i)$.

\begin{lem}\label{lem:non-empty}
Let $G=(G_i)_{i \in I}$ be a projective system of groups.
Then, $G$ has trivial first non-abelian cohomology if
and only if $R^1 \varprojlim_{i \in I} G_i$ consists
of one point.
\end{lem}

\begin{proof}
Let $J$ denote the set of pairs $(i,j)$ with $i \le j$.
For $(i,j) \in J$, let $\phi_{j,i}:G_j \to G_i$
denote the transition map.

For an element $z=(g_{i,j})_{(i,j) \in J}$ of
$Z^1 \varprojlim_{i\in I} G_i$, we associate
a left $G$-torsor $(S_{z,i})_{i \in I}$ as follows.
For $i \in I$, we set $S_{z,i}=G_i$, and
for $(i,j) \in J$, let $f_{j,i}: S_{z,j}
\to S_{z,i}$ denote the map that sends
$g \in G_j$ to $\phi_{j,i}(g) g_{i,j}^{-1}$.
Then $(S_{z,i})_{i \in I}$ with the system 
$(f_{j,i})_{(i,j) \in J}$
of transition maps is a left $G$-torsor, which we
denote by $S_z$.
One can then easily check that the map that sends $z \in Z^1 \varprojlim_{i\in I} G_i$
to the isomorphism class of $S_z$ induces
a bijection from $R^1 \varprojlim_{i\in I} G_i$
to the set of isomorphism classes of left $G$-torsors.
Thus, the claim follows.
\end{proof}

\subsection{$Y$-sites}

\begin{defn}
We say that a category $\cC$
is $\Lambda$-connected if, for any two objects
$X$, $Y$ in $\cC$, there exists an object
$Z$ in $\cC$ and morphisms 
$Z \to X$ and $Z \to Y$ in $\cC$.
\end{defn}

\begin{defn}
\label{defn:Y-sites}
Let $(\cC, J)$ be a $B$-site.
Recall that $\cT(J)=\wh{\cT(J)}$.
We say that a $B$-site $(\cC, J)$
is a $Y$-site if the following properties 
hold:
\begin{enumerate}
\item
$\cC$ is essentially $\frU$-small.
\item
$C(\cT(J))$ is $\Lambda$-connected.
\item
$\cT(J)$ has enough Galois coverings.
\end{enumerate}
\end{defn}

\subsection{Pregrids}
\begin{defn}\label{defn:pregrid}
A pregrid of a $\Lambda$-connected category
$\cC$ is a pair $(\cC_1, \iota_1)$
of a $\frU$-small category $\cC_1$ 
and a covariant functor
$\iota_1:\cC_1 \to \cC$ satisfying the following
conditions:
\begin{enumerate}
\item The category $\cC_1$ is a poset and is
$\Lambda$-connected.
\item The functor $\iota_1$ is essentially surjective.
\item For any object $X$ of $\cC_1$,
the functor $\cC_{1,/X} \to \cC_{/\iota_1(X)}$ between
the overcategories induced by $\iota_1$
is essentially surjective.
\item For any object $X$ of $\cC_1$, the functor
$\cC_{1,X/} \to \cC_{\iota_1(X)/}$
between the undercategories induced by $\iota_1$
is an equivalence of categories.
\end{enumerate}
\end{defn}

Let $(\cC, J)$ be a $Y$-site.
Let $(\cC_0,\iota_0)$ be a pair of a category $\cC_0$
and a functor $\iota_0: \cC_0 \to \cC$.
\begin{defn}(edge objects)
\label{defn:edge objects}
A morphism $f$ of $\cC_0$ is called of type $J$ if
the morphism $\iota_0(f)$ belongs to $\cT(J)$.
An object $X$ of $\cC_0$ is called an edge object of $\cC_0$
if any morphism $f:Y \to X$ in $\cC_0$ is of type $J$.
\end{defn}

For any object $X$ of $\cC_0$, we let
$\iota_{0,X/}$ denote the functor
$\iota_{0,X/} : \cC_{0,X/} \to \cC_{\iota_0(X)/}$
between undercategories induced by $\iota_0$.
\begin{defn}(grids)
\label{defn:grids}
Let $(\cC,J)$ be a $Y$-site.
We define a grid of $(\cC,J)$ 
to be a pair $(\cC_0, \iota_0)$
of a $\frU$-small category $\cC_0$ and a functor 
$\iota_0: \cC_0 \to \cC$ satisfying the following 
properties:
\begin{enumerate}
\item The category $\cC_0$ is a poset and
is $\Lambda$-connected.
\item For any object $X'$ of $\cC$, there exists an
edge object $X$ of $\cC_0$ such that $\iota_0(X)$ is isomorphic
to $X'$ in $\cC$.
\item For any object $X$ of $\cC_0$ and for any
morphism $f:Y \to \iota_0(X)$ in $\cC$ that
belongs to $\cT(J)$, there exists
a morphism $f':Y' \to X$ in $\cC_0$ and an
isomorphism $\alpha:Y \xto{\cong} \iota_0(Y')$
in $\cC$ satisfying $f= \iota_0(f') \circ \alpha$.
\item For any object $X$ of $\cC_0$, the functor
$\iota_{0,X/}:\cC_{0,X/} \to \cC_{\iota_0(X)/}$
is an equivalence of categories.
\end{enumerate}
\end{defn}

\subsection{The absolute Galois monoid associated with a grid}
Let $(\cC_0, \iota_0)$ be a grid 
of a $Y$-site $(\cC,J)$.
\subsubsection{The absolute Galois monoid $M_\Cip$}
\label{sec:MC}
We denote by
$M_\Cip$ the set of pairs $(\alpha,\gamma_\alpha)$ 
of an endomorphism $\alpha :\cC_0 \to \cC_0$ 
of the category $\cC_0$ and a natural isomorphism
$\gamma_\alpha : \iota_0 \xto{\cong} \iota_0 \circ \alpha$.
For two elements $(\alpha,\gamma_\alpha)$
and $(\beta,\gamma_\beta)$ of $M_\Cip$,
we define the composite 
$(\alpha,\gamma_\alpha)\circ 
(\beta, \gamma_\beta) \in M_\Cip$ 
to be
$$
(\alpha,\gamma_\alpha)\circ 
(\beta, \gamma_\beta) =
(\alpha \circ \beta,
(\beta^* \gamma_\alpha) \circ \gamma_\beta)
$$
where $\beta^* \gamma_\alpha$ denotes the
natural isomorphism $\iota_0 \circ \beta 
\xto{\cong} \iota_0 \circ \alpha \circ \beta$
induced by $\gamma_\alpha$.
The binary operation $-\circ-$ on $M_\Cip$
gives a monoid structure on the set $M_\Cip$,
whose unit element of $M_\Cip$ 
is equal to $(\id_{\cC_0},\id_{\iota_0})$.
We call this monoid the absolute Galois monoid 
associated with the grid $\Cip$.

\subsubsection{The associated submonoids}
\label{sec:submonoids}
Let $X$ be an  
object of $\cC_0$.
We define the submonoid $\bK_X \subset M_\Cip$
to be
$$
\bK_X = \{
(\alpha,\gamma_\alpha) \in M_\Cip \ |\ 
\alpha(X) = X,
\gamma_\alpha(X)= \id_{\iota_0(X)}:
\iota_0(X) \xto{\cong} \iota_0(\alpha(X))
= \iota_0(X) \}.
$$
We use this monoid only when $X$ is an edge object.
We will see later (Lemma~\ref{lem:bKX2}) that 
when $X$ is an edge object, $\bK_X$ is a group.

\subsection{The fiber functor associated with a grid}

\subsubsection{A grid is cofiltered}

\begin{lem}\label{lem:core_cofiltered}
Let $\cC'$ be a category that is 
non-empty, thin, and $\Lambda$-connected.
Then the category $\cC'$ is cofiltered.
\end{lem}

\begin{rmk}
Here we assume that the notion ``cofiltered"
is defined as in I.2.7 of \cite{SGA4}. There are several
equivalent definitions for this notion, and 
for some of these,
the statement of Lemma \ref{lem:core_cofiltered} 
below may be an immediate consequence of the definition.
\end{rmk}

\begin{proof}
As $\cC'$ is non-empty and thin, 
it suffices to
show that $\cC'$ is semi-cofiltered.

Let $Y_1 \xto{f_1} X \xleftarrow{f_2} Y_2$ 
be a diagram in $\cC'$.
As $\cC'$ is $\Lambda$-connected, there exists a
diagram $Y_1 \xleftarrow{g_1} Z \xto{g_2} Y_2$
in $\cC'$.
By noting that $\cC'$ is thin again, we have
$f_1 \circ g_1 = f_2 \circ g_2$, which
proves the claim.
\end{proof}

\begin{cor}\label{cor:core_cofiltered}
Let $(\cC_0,\iota_0)$ be a grid of a $Y$-site $(\cC, J)$
such that $\cC$ is non-empty.
Then $\cC_0$ is cofiltered.
\end{cor}

\begin{proof}
The claim follows immediately from the definition.
\end{proof}

\subsubsection{The fiber functor associated with a grid}
\label{sec:fiber functor}
Let $(\cC_0,\iota_0)$ be a grid of a Y-site $(\cC,J)$
such that $\cC$ is non-empty.
Let $\cCe$ denote the full subcategory of $\cC_0$
whose objects are the edge objects of $\cC_0$.
\begin{lem} \label{lem:lower}
Let $X$ be an edge object of $\cC_0$ and let
$f:Y \to X$ be a morphism in $\cC_0$.
Then, $Y$ is an edge object of $\cC_0$.
\end{lem}

\begin{proof}
Let $g:Z \to Y$ be an arbitrary morphism in $\cC_0$.
As $X$ is an edge object of $\cC_0$,
the morphisms $f$ and $f \circ g$ are of type $J$.
Hence, it follows from Condition (3) of Definition
\ref{defn:B-site} that morphism $g$ is of type $J$.
This shows that $Y$ is an edge object of $\cC_0$.
\end{proof}

\begin{lem} \label{lem:C1_cofinal}
The edge objects are cofinal in $\cC_0$.
\end{lem}

\begin{proof}
Let $X$ be an object of $\cC_0$.
Let us choose an edge object $Y$ of $\cC_0$.
Then, $\cC_0$ is $\Lambda$-connected, and there exists an
object $Z$ of $\cC_0$ and morphisms 
$Z \to X$ and $Z \to Y$.
It follows from Lemma \ref{lem:lower} that $Z$ is
an edge object of $\cC_0$.
We thus obtain a morphism $Z \to X$ from an edge object
of $\cC_0$ to $X$. As $X$ is arbitrary, this shows that
the edge objects are cofinal in $\cC_0$.
\end{proof}

\begin{lem} \label{lem:core_cofiltered2}
The category $\cCe$ is cofiltered.
\end{lem}
\begin{proof}
The category $\cCe$ is non-empty and thin as
it is a full subcategory of $\cC_0$.
Hence, it suffices to show that $\cCe$ is semi-cofiltered.
Let $Y \xto{f} X \xleftarrow{g} X'$ be a diagram in $\cCe$.
It follows from Corollary \ref{cor:core_cofiltered} that
there exists an object $Y'$ of $\cC_0$ and morphisms
$f' : Y' \to X'$ and $g':Y' \to Y$ in $\cC_0$ satisfying
$g \circ f' = f \circ g'$.
As it follows from Lemma \ref{lem:lower} that
$Y'$ is an object of $\cCe$, this shows that
the category $\cCe$ is semi-cofiltered, which
proves the claim.
\end{proof}

For a presheaf $F$ on $\cC$, we define 
$\omega_\Cip$ 
to be the filtered colimit
\[
\omega_\Cip(F) =
\varinjlim_{X \in \Obj(\cCe)}
F(\iota_0(X)).
\]
We note that, as the objects of $\cCe$ are cofinal in
$\cC_0$ by Lemma \ref{lem:lower}, the natural map
\[
\omega_\Cip(F) \to 
\varinjlim_{X \in \Obj(\cC_0)}
F(\iota_0(X))
\]
is bijective. By associating 
$\omega_\Cip(F)$
to each 
presheaf $F$ on $\cC$, we obtain the
functor $\omega_\Cip$
 from $\Presh(\cC)$
to the category of sets. By abuse of notation,
we denote by the same symbol 
$\omega_\Cip$ 
the restriction of 
$\omega_\Cip$
to the full subcategory 
$\Shv(\cC,J)\subset\Presh(\cC)$.

\subsubsection{The action of the absolute Galois monoid}
Let $F$ be a presheaf on $\cC$.
We define the action of the monoid
$M_\Cip$ on the set 
$\omega_\Cip(F)$.
Let $(\alpha,\gamma_\alpha) \in M_\Cip$.
We define the map
$(\alpha,\gamma_\alpha)_* :
\omega_\Cip(F) \to \omega_\Cip(F)$ 
to be the composite
$$
\omega_\Cip(F)
= \varinjlim_{X} F(\iota_0(X))
\xto{(\gamma_\alpha^{-1})^*} 
\varinjlim_{X} F(\iota_0(\alpha(X)))
\xto{j} \varinjlim_{Y \in \Obj(\cC_0)} F(\iota_0(Y))
\cong \omega_\Cip(F),
$$
where $X$ runs over the edge objects of $\cC_0$ and $j$
is the map induced by the inclusion
$$
\{ \alpha(X)\ |\ X \in \Obj(\cCe) \} \subset \Obj(\cC_0).
$$
\begin{defn}
\label{defn:smooth sets}
We say that a left $M_\Cip$-set $S$ is {\em smooth}
if for any $s \in S$ there exists an edge
object $X$ of $\cC_0$ such that
$s=gs$ holds for any $g \in \bK_X$.
We denote by $\MSet$ the category of smooth
$M_\Cip$-sets.
\end{defn}

\begin{rmk}
\label{rmk:smooth}
We remark here on the use of the term 
{\it smooth} above.  A locally profinite group is 
defined to be a locally compact Hausdorff group such 
that the compact open subgroups form a basis for the 
neighborhoods of the
unit.   A {\it smooth representation} 
of such groups is defined to be a representation
in which each vector has an open isotropy subgroup.
We refer to Casselman's article 
\cite{Casselman} for 
generalities.
Smooth representations of locally profinite groups
are of 
interest in the theory of automorphic forms.
Examples of such groups
include $\GL_n(\mathbb{A}_f)$
or $\GL_n(F)$, where $\mathbb{A}_f$
denotes the ring of finite adeles of some global field,
and $F$ is a nonarchimedian local field.
\end{rmk}

\begin{rmk}
We will see later in 
Section~\ref{sec:topological monoid structure}
that
the absolute Galois monoid is naturally 
equipped with the structure of a topological 
monoid such that 
the category of smooth sets is equivalent to 
the category of discrete sets with continuous 
action for this structure.
\end{rmk}

We note that, if $F$ is a presheaf on $\cC$,
then for each edge object $X$ of $\cC_0$,
the image of the map $F(X) \to \omega_\Cip(F)$
given by the universality of colimit
is contained in the $\bK_X$-invariant
part $\omega_\Cip(F)^{\bK_X}$ of $\omega_\Cip(F)$,
with respect to the action of
$M_\Cip$ defined as above.
It follows that $\omega_\Cip(F)$ is a smooth
$M_\Cip$-set. 
It is straightforward to check that the action of 
$M_\Cip$ on $\omega_\Cip(F)$ is functorial on $F$.
Hence, the functor $\omega_\Cip$ factors through
the category of
smooth $M_\Cip$-sets. By abuse of notation,
we denote by the same symbol $\omega_\Cip$ the latter
functor and its restriction 
to the full subcategory $\Shv(\cC, J)\subset\Presh(\cC)$.

\subsection{The statement of the main theorem}
\subsubsection{Conditions on cardinality}
The main result of this article is the following.
\label{sec:cardinality}
For a category $\cD$,
let us consider the following conditions:
\begin{enumerate}
\item For any objects $X$, $Y$ of $\cD$,
the set $\Hom_{\cD}(X,Y)$ is a finite set.
\item There exists a set $S$, whose cardinality 
is at most countable,
of objects of $\cD$ such that to 
any object $X$ of $\cD$, 
there exists a morphism from an object of $\cD$ belonging to $S$. 
\end{enumerate}
These conditions appear in the statement of our 
main theorem below.
Technically, these will appear 
only at the following two points:
\begin{enumerate}
\item In the proof of the existence of 
a pregrid, in which we use Lemma~\ref{lem:admissible}, and
\item In the proof of Lemma~\ref{lem:Gal gp surjective}.
\end{enumerate}

\subsubsection{The statement}
\begin{thm}\label{thm:Galois_main}
Let $(\cC,J)$ be a $Y$-site. 
Let $(\cC_0, \iota_0)$ be a grid.
Then the functor 
\begin{equation}\label{main_equiv}
\omega_\Cip:
\Shv(\cC,J) \to \MSet
\end{equation}
is faithful.

If we moreover assume that,
for each object $X$ of $\cC$, 
the overcategory $\cC(\cT(J))_{/X}$ 
satisfies at least one of the two conditions 
in Section \ref{sec:cardinality},
then the functor $\omega_\Cip$ is
an equivalence of categories.
\end{thm}

\section{Existence of a grid of a $Y$-site}
\label{sec:grid existence}
Assuming a certain cardinality condition, in this section,
we prove the existence of a grid of a $Y$-site.

\subsection{Existence of a pregrid}
\label{sec:existence of a pregrid}

\subsubsection{Existence of a pregrid}

\begin{prop}
\label{prop:pregrid existence}
Let $(\cC,J)$ be a $Y$-site.
Suppose that there exists an object $X_0$
of $\cC(\cT(J))$
such that the overcategory $\cC(\cT(J))_{/X_0}$
satisfies at least one of the two cardinality 
conditions in 
Section~\ref{sec:cardinality}.
Then there exists a pregrid of $\cC(\cT(J))$.
\end{prop}

\subsubsection{Construction step 1}
\label{sec:construction 1}
For brevity, set $\cD=\cC(\cT(J))$.

We use the category $\Gal/X_0$ introduced in Section~\ref{sec:Gal/X} in conjunction with the following terminology.
For an object $(Y,f)$ of $\Gal/X_0$,
we denote by $G(Y,f)$ the Galois group of $f$. 
For an object $(Y,f)$ of $\Gal/X_0$ and for an object $h$
of $\cC_{/X_0}$, we say that $(Y,f)$ dominates $h$
if there exists a morphism from $f$ to $h$ in $\cC_{/X_0}$.

As $\cC$ is essentially $\frU$-small,
we can take a $\frU$-small set $V$ of
objects of $\cD_{/X_0}$ such that any 
object of $\cD_{/X_0}$ admits a morphism from an object 
that belongs to $V$.
When $\cC$ satisfies Condition (2), we may and will
assume that $V$ is at most countable.
Let $S\subset V$ be a subset. 
For each object $(Y,f)$ of $\Gal/X_0$ we set
$$
\wt{U}_S(Y,f) = \prod_{s \in S} \Hom_{\cC_{/X_0}}(f,s)
$$
and $U_S(Y,f) = G(Y,f) \backslash \wt{U}_S(Y,f)$,
where the quotient is taken with respect
to the diagonal action of $G(Y,f)$.
When $S = \emptyset$ is an empty set, we understand that
the set $\wt{U}_\emptyset (Y,f)$ 
consists of a single element.
Let $(Y,f)$ and $(Y',f')$ be two objects of $\Gal/X_0$.
Let $h: Y' \to Y$ be a morphism from $f'$ to $f$ in $\cC_{/X_0}$.
The composition with $h$ gives a map
$U_S(h):U_S(Y,f) \to U_S(Y',f')$. It can be checked easily
that the map $U_S(h)$ depends only on the class of $h$
in $\Hom_{\Gal/X_0}((Y',f'),(Y,f))$.
Hence, by associating $U_S(Y,f)$ to each object
$(Y,f)$ of $\Gal/X_0$, we obtain a presheaf on
$\Gal/X_0$, which we denote by $U_S$.
As the topology has enough
Galois coverings, there exists
an object $(Y,f)$ of $\cD_{/X_0}$
that dominates $s$ for every $s \in S$.
For such an object $(Y,f)$, the set $U_S(Y,f)$ is non-empty.
Let $(Y,f)$ and $(Y',f')$ be two objects of $\Gal/X_0$
such that both $U_S(Y,f)$ and $U_S(Y',f')$ are non-empty.
Let $h$ be a morphism from $f'$ to $f$ in $\cD_{/X_0}$,
It follows from Lemma \ref{lem:lifts} that
the morphism $h$ induces a surjective homomorphism
$G(Y',f') \surj G(Y,f)$ of groups.
It follows from Corollary~\ref{cor:Hom bijective}
that the composition with $h$ gives a bijection
$\Hom_{\cD_{/X_0}}(f,s) \to \Hom_{\cD_{/X_0}}(f',s)$.
This shows that the map 
$U_S(h) : U_S(Y,f) \to U_S(Y',f')$ is bijective.

Let us fix a $\frU$-small subcategory $\Gal'/X_0$ of
$\Gal/X_0$ that is equivalent to $\Gal/X_0$. 
For a finite subset $S \subset V$, we set
$$
U(S) := \varinjlim_{(Y,f)} U_S(Y,f)
$$
where the colimit is taken over the objects in the
category $\Gal'/X_0$.
We note that, for any object $(Y,f)$ of $\Gal'/X_0$,
the map $U_S(Y,f) \to U(S)$ is bijective if $U_S(Y,f)$
is non-empty.
Let $S$ and $S'$ be two finite subsets of $V$
with $S' \subset S$.
The projection $\wt{U}_S(Y,f) \to \wt{U}_{S'}(Y,f)$
for each object $(Y,f)$ of $\Gal/X_0$
gives a morphism $U_S \to U_{S'}$ of presheaves on $\Gal/X_0$.
The collection $(U_S)_{S \subset V}$ forms
a projective system of presheaves on $\Gal/X_0$ 
indexed by the finite subsets of $V$.
Hence, $(U(S))_{S \subset V}$ forms a projective system of
sets indexed by the finite subsets of $V$.

\begin{lem}
\label{lem:US non-empty}
A filtered projective limit 
$U = \varprojlim_{S} U(S) $ is non-empty.
\end{lem}
\begin{proof}
It is easy to see that $U(S)$ is non-empty for any 
finite subset $S \subset V$
and the transition map $U(S) \to U(S')$ is surjective for
any finite subsets $S,S' \subset V$ with $S' \subset S$.
Recall that we have assumed Condition (1) or (2) on $\cC$.
If $\cC$ satisfies Condition (1), then $U(S)$ is a finite set
for every finite subset $S \subset V$.
If $\cC$ satisfies Condition (2), then $V$ is at most countable.
In the first case, one can use the fact that 
the projective limit of non-empty finite sets 
is non-empty to conclude.

Now let us consider the second case.
Observe that if $I$ is a cofiltered partially ordered set 
whose cardinality is countably infinite, then
there exists a cofinal partially ordered subset
which is isomorphic to the opposite of the 
partially ordered set of natural numbers $\Z_{\ge 0}$.
Using the surjectivity of the transition maps,
we see that $U = \varprojlim_{S} U(S) $ is non-empty.
\end{proof}

\subsubsection{Construction step 2}

Let us fix an element $(x_S) \in U$.
For each finite subset $S \subset V$, let us fix
an object $(Y_S,f_S)$ of $\Gal'/X_0$ such that
$U_S(Y_S,f_S)$ is non-empty.
Let $y_S$ denote the element of $U_S(Y_S,f_S)$
that is mapped to $x_S$ via the bijection
$U_S(Y_S,f_S) \to U(S)$.
Let us take a representative 
$\wt{y}_S=(\wt{y}_{S,s})_{s \in S} 
\in \wt{U}_S(Y_S,f_S)$ of $y_S$.
By definition, $\wt{y}_{S,s}$ is a morphism
from $f_S$ to $s$ in $\cD_{/X_0}$ for each $s \in S$.
Let $\cC'_{1,S}=\cC'_{1,S,\wt{y}_S}$ denote
the full subcategory of the undercategory $\cD_{Y_S/}$,
whose objects are the morphisms $f: Y_S \to Z$
in $\cD$ satisfying $f = g \circ \wt{y}_{S,s}$ for
some $s \in S$ and for some morphism $g$ in $\cD$.
As $\cD$ is essentially $\frU$-small and 
all morphisms in $\cD$ are epimorphisms,
it follows that the category $\cC'_{1,S}$ is a quasi-poset.
Let us choose a poset skeleton $\cC_{1,S}$ of $\cC'_{1,S}$.

For each pair $(S_1,S_2)$ of finite subsets of $V$
with $S_1 \subset S_2$,
we construct a fully faithful functor
$\cC_{1,S_1} \to \cC_{1,S_2}$ as  follows.
Let us choose an object $(W,f)$ of $\Gal/X_0$
that dominates both $f_{S_1}$ and $f_{S_2}$.
For $i=1,2$, let us choose a morphism $h_i$ 
from $f$ to $f_{S_i}$ in $\cD_{/X_0}$.
Let $\wt{y}'_{S_i} \in \wt{U}_{S_i}(W,f)$
denote the image of $\wt{y}_{S_i}$ under the map
$\wt{U}_{S_i}(Y_{S_i},f_{S_i}) \to \wt{U}_{S_i}(W,f)$
given by the composition with $h_i$.
As $x_{S_2}$ is mapped to $x_{S_1}$ under the map
$U(S_2) \to U(S_1)$, there exists an element
$\alpha \in G(W,f)$ such that $\wt{y}'_{S_1}$ is equal to the
image of $\alpha\cdot \wt{y}'_{S_2}$ under the projection
$\wt{U}_{S_2}(W,f) \to \wt{U}_{S_1}(W,f)$.
We have a diagram
$$
\cD_{Y_{S_1}/} \xto{-\circ h_1} \cD_{W/} 
\xrightarrow[\cong]{-\circ \alpha}
\cD_{W/} \xleftarrow{-\circ h_2} \cD_{Y_{S_2}/}
$$
of undercategories, which induces a diagram
$$
\cC_{1,S_1,\wt{y}_{S_1}} \xto{(1)} \cC_{1,S_1,\wt{y}'_{S_1}}
\xto{(2)} \cC_{1,S_2,\alpha \wt{y}'_{S_2}}
\xto{(3)} \cC_{1,S_2, \wt{y}'_{S_2}}
\xleftarrow{(4)} \cC_{1,S_2,\wt{y}_{S_2}}.
$$
\begin{lem}
The functors (1), (3), and (4) are 
equivalences of categories and the functor (2) is fully faithful.
\end{lem}
\begin{proof}
It can be checked easily that the functor (2) is fully faithful,
that the functor (3) is an equivalence of
categories, and that the functors (1) and (4) are essentially
surjective. As the morphisms $h_1$ and $h_2$ in $\cD$
are epimorphisms, it follows that 
the functors (1) and (4) are fully faithful, which proves the claim.
\end{proof}

By taking the composite of (1), (2), (3) 
and the inverse of (4), 
we obtain a fully faithful
functor $i_{S_1,S_2} : \cC_{1,S_2} \to \cC_{1,S_2}$ between posets.
When we fix $(W,f)$, $h_1$, and $h_2$,
the element $\alpha \in G(W,f)$ is not uniquely determined
and the functor $-\circ\alpha:\cD_{W/} \to \cD_{W/}$ 
may depend on the choice of $\alpha$.
However, the functor (3) induced by $-\circ\alpha$ 
is independent of the choice of $\alpha$. 
Using this observation, one can check straightforwardly that
the functor $i_{S_1,S_2}$ does not depend on the choice of
$(W,f)$, $h_1$, $h_2$, and $\alpha$.
For three finite subsets $S_1,S_2,S_3 \subset V$ with
$S_1 \subset S_2 \subset S_3$, we have
$i_{S_1,S_3}= i_{S_2,S_3} \circ i_{S_1,S_2}$.
Hence, the pair $((\cC_{1,S})_{S \subset V}, (i_{S',S})_{S'\subset S})$
forms an inductive system in the category of poset categories.
We define $\cC_1$ to be the colimit of this inductive system.
(The notation $\cC_1$ has already 
appeared in Section~\ref{sec:fiber functor}.
The conflict is explained in 
Remark~\ref{rmk:C1 edge objects}.)
By taking a limit of the composite 
$\cC_{1,S} \to \cC'_{1,S} \to \cD$, we obtain a functor
$\iota_1 : \cC_1 \to \cD$, that is uniquely determined up to
natural equivalences.

\subsubsection{Proof of Proposition~\ref{prop:pregrid existence}}
We claim that the pair $(\cC_1,\iota_1)$ 
is a pregrid of $\cD$. 
By definition, the category $\cC_1$ is a poset.
We show that $\cC_1$ is $\Lambda$-connected.
Let $Z_1$ and $Z_2$ be two objects of $\cC_1$.
For $i=1,2$, take a finite subset $S_i \subset V$
such that $Z_i$ belongs to the image of
the functor $\cC_{1,S_i} \to \cC_1$.
We set $S=S_1 \cup S_2$.
For $i=1,2$, there exists an object $z_i$ of
$\cC_{1,S}$ whose image under the 
functor $\cC_{1,S} \to \cC_1$ is equal to $Z_i$.
Then, $z_i$ is, by definition, a morphism from $Y_{S}$ to an object $Z'_i$ of $\cD$.
Let us take an object $f:W \to X_0$ of $\cD_{/X_0}$
that is isomorphic to $f_{S}$ in $\cD_{/X_0}$ and that belongs to $V$.
Let us choose an isomorphism $\alpha$ from $f$ to $f_{S}$.
We set $T = S \cup \{f\}$.
Observe that the morphism $\wt{y}_{T,f}: Y_{T} \to W$ is 
an object of $\cC'_{1,T}$. Hence, the diagram
$$
Z'_1 \xleftarrow{z_1 \circ \alpha}
W \xto{z_2 \circ \alpha}  Z'_2 
$$
in $\cD$ gives a diagram 
$z_1 \circ \alpha \circ \wt{y}_{T,f}
\leftarrow \wt{y}_{T,f} \to z_2 \circ 
\alpha \circ \wt{y}_{T,f}$ in $\cC'_{1,T}$.
The last diagram induces a diagram of the form 
$Z_1 \leftarrow Z' \to Z_2$ for some $Z'$ in 
$\cC_1$.
This proves that the category $\cC_1$ is 
$\Lambda$-connected.
Next, we prove that the functor $\iota_1$ 
is essentially surjective.
Let us take an arbitrary object $X$ of $\cD$.
As $\cD$ is $\Lambda$-connected, there exists a diagram
$$
X_0 \xleftarrow{f_0} Y \xto{f} X
$$
in $\cD$. We may choose this diagram in such a way that
$f_0$ belongs to $V$.
Let $S=\{f_0\}$. Then $\wt{y}_{S,f_0}$, regarded as
a morphism in $\cD$, is an object of $\cC'_{1,S}$.
Let $X'$ denote the object of $\cC_1$ given by
the object $f \circ \wt{y}_{S,f_0}$ of $\cC'_{1,S}$.
It then follows from the definition of $\iota_1$ that
the object $\iota_1(X')$ of $\cD$ is isomorphic to $X$,
which proves that the functor $\iota$ is essentially
surjective.

\begin{lem}
The pair $(\cC_1,\iota_1)$ satisfies
the condition (3) in Definition \ref{defn:pregrid}.
\end{lem}
\begin{proof}
Let $X$ be an object of $\cC_1$ and let 
$f:Y \to \iota_1(X)$ be a morphism in $\cD$.
Let us choose a finite subset $S \subset V$ such that
$X$ belongs to the image of an object $h$
of $\cC_{1,S}$ under the functor
$\cC_{1,S} \to \cC_1$.
The object $h$ of $\cC_{1,S}$ is, by definition,
a morphism from $Y_{S}$ to an object $X'$ of $\cD$,
and there exists an element $s \in S$ such that
the morphism $h$ is the composite of
the morphism $\wt{y}_{S,s}$ and a morphism
$h_1 : Y_s \to X'$, where $Y_s$ denotes the domain of $s$.
Observe that $X'$ and $\iota_1(X)$ are isomorphic in $\cD$.
Let us fix an isomorphism 
$\alpha : \iota_1(X) \xto{\cong} X'$.
As $\cD$ is semi-cofiltered, there exists a
commutative diagram
$$
\begin{CD}
W @>{h'}>> Y \\
@V{f'}VV @VV{\alpha \circ f}V \\
Y_S @>{h}>> X'
\end{CD}
$$
in $\cD$. We may and will assume that the composite
$f_S \circ f'$ is a Galois covering and belongs to $V$.
We set $T= S \cup \{f_S \circ f' \}$.
As $x_T$ is mapped to $x_S$ under the map
$U(T) \to U(S)$, there exists an element $\beta \in G(Y_T,f_T)$
such that $\wt{y}_{S,s} \circ f' \circ \wt{y}_{T,f_S \circ f'}$
is equal to $\wt{y}_{T,s} \circ \beta$ and that the
image of $h$ under the functor $\cC_{1,S} \to \cC'_{1,T}$
is isomorphic to 
$h \circ f' \circ \wt{y}_{T,f_S \circ f'} \circ \beta^{-1}$.
As $f_S \circ f'$ is a Galois covering, there
exists an element $\beta' \in G(W,f_S \circ f')$ satisfying
$\wt{y}_{T,f_S \circ f'} \circ \beta = \beta' \circ
\wt{y}_{T,f_S \circ f'}$.
Hence
$h \circ f' \circ \wt{y}_{T,f_S \circ f'} \circ \beta^{-1}
= \alpha \circ f \circ h' \circ \beta'^{-1} \circ \wt{y}_{T,f_S \circ f'}$.
Hence, the morphism $\alpha \circ f$ can be regarded as a morphism from
$h' \circ \beta'^{-1} \circ \wt{y}_{T,f_S \circ f'}$
to $h \circ f' \circ \wt{y}_{T,f_S \circ f'} \circ \beta^{-1}$
in $\cC'_{1,T}$. This morphism induces an object of the
overcategory $\cC_{1,/X}$ whose image under the functor
$\cC_{1,/X} \to \cD_{/X}$ induced by $\iota_1$ is
isomorphic to $f$. This proves that the pair $(\cC_1,\iota_1)$
satisfies Condition (3) in Definition \ref{defn:pregrid}.
\end{proof}

\begin{lem}
The pair $(\cC_1,\iota_1)$ satisfies
Condition (4) in Definition \ref{defn:pregrid}.
\end{lem}
\begin{proof}
Let $X$ be an object of $\cC_1$.
Let us choose a finite subset $S \subset V$ such that
$X$ belongs to the image of an object $h$
of $\cC_{1,S}$ under the functor
$\cC_{1,S} \to \cC_1$.
For any finite subset $T \subset V$ with $S \subset T$,
let $h_T$ denote the image of $h$ under the
functor $i_{S,T}:\cC_{1,S} \to \cC_{1,T}$.
The object $h_T$ of $\cC_{1,T}$ is, by definition,
a morphism from $Y_T$ to an object $X_T$ of $\cD$.
It can then be checked easily that the inclusion
functor $\cC_{1,T} \to \cD_{Y_T/}$ induces an equivalence
$\cC_{1,T,h_T/} \xto{\cong} \cD_{X_T/}$ of undercategories.
This shows that the functor $\iota_1$ induces an equivalence 
$\cC_{1,X/} \xto{\cong} \cD_{\iota_1(X)/}$ of undercategories.
This completes the proof.
\end{proof}

\subsection{Existence of a grid}
\label{sec:C0}
The goal of this section is to 
prove the following proposition.
Its proof is given at the end
of this subsection.
\begin{prop}
\label{cor:grid existence}
Let $(\cC,J)$ be a $Y$-site.
Suppose that there exists an object $X_0$
of $\cC(\cT(J))$
such that the overcategory $\cC(\cT(J))_{/X_0}$
satisfies at least one of the two cardinality 
conditions in 
Section~\ref{sec:cardinality}.
Then there exists a grid of $(\cC, J)$.
\end{prop}

\subsubsection{Construction of a grid from a pregrid}
We construct
a pair $\Ci$ of a $\frU$-small category
$\cC_0$ and a covariant functor $\iota_0:
\cC_0 \to \cC$,
from a pregrid 
$(\cC_1, \iota_1:\cC_1 \to \cD=\cC(\cT(J))$,
as follows.

For each object $X$ of $\cC_1$,
the undercategory $\cD_{\iota_1(X)/}$ is 
a quasi-poset, as $\cC$ is an $E$-category 
and is
essentially $\frU$-small.
Let us choose a poset skeleton
$\cC_{0,X}$ of $\cD_{\iota_1(X)/}$.
When $f:Y\to X$ is a morphism in $\cC_1$,
the functor $\cD_{\iota_1(X)/}
\to \cD_{\iota_1(Y)}$ /, 
given by the composition
with $f$, is fully faithful.
Hence it induces a fully faithful functor
$\cC_{0,X} \to \cC_{0,Y}$ between posets.
We define $\cC_0$ to be the colimit
$\cC_0 = \varinjlim_{X \in \Obj \cC_1}
\cC_{0,X}$.  
(Consider the diagram over $X \in \Obj \cC_1$ 
in the category of partially ordered
sets, take the colimit as a partially ordered set, 
and regard it as a category using the functor 
in Section~\ref{sec:poset}).

By associating, to each object $Y$ of $\cC_1$,
the image of the initial object of $\cC_{0,Y}$
under the functor $\cC_{0,Y} \to \cC_0$, 
we obtain a functor $\jmath:\cC_1 \to \cC_0$.
One can check easily that the functor $\jmath$
is fully faithful.

By patching the functors 
$\cC_{0,X} \to \cD_{\iota_1(X)/}\to \cC$
for various $X$, we obtain a functor 
$\iota_0:\cC_0 \to \cC$ satisfying
$\varphi \circ \iota_1 = 
\iota_0 \circ \jmath$,
where $\varphi : \cD \to \cC$
denotes the inclusion functor.

\subsubsection{}
\begin{lem}
\label{lem:C0_core}
The pair $(\cC_0, \iota_0)$
constructed above is a grid.
\end{lem}
We check below that the conditions of 
a grid (Definition~\ref{defn:grids}) are satisfied.
\begin{lem}
Conditions (1) is satisfied.
\end{lem}
\begin{proof}
It follows from the construction of
$\cC_0$ that given an object $X$ of $\cC_0$,
there exist an object $X'$ of $\cC_1$
and a morphism $\jmath(X') \to X$.
Note that the category $\cC_1$ 
is $\Lambda$-connected, 
as $(\cC_1, \iota_1)$
is a pregrid.
Hence, $\cC_0$ is $\Lambda$-connected.
This shows that Condition (1) is 
satisfied.
\end{proof}

\begin{lem}
Condition (3) is satisfied.
\end{lem}
\begin{proof}
Let $X$ be an object of $\cC_0$
and let $f: Y \to \iota_0(X)$ 
be a morphism in $\cC(\cT(J))$.
Let us choose an object $X'$ in $\cC_1$ and
a morphism 
$g: \jmath(X') \to X$ in $\cC_0$.
As $\cT(J)$ is semi-localizing,
there exist an object $Y'$ of $\cC$ and
morphisms $f':Y' \to \iota_1(X')$ and
$g':Y' \to Y$ in $\cC$ such that
$f'$ belongs to $\cT(J)$ and that the equality
$f \circ g' = \iota_0(g) \circ f'$ holds.
As $(\cC_1, \iota_1)$
is a pregrid,
there exists an object $Y'_1$ of $\cC_1$,
a morphism $f'_1 : Y'_1 \to X'$ in $\cC_1$,
and an isomorphism 
$\alpha: Y' \xto{\cong} \iota_1(Y'_1)$
in $\cC(\cT(J))$ satisfying 
$f' = \iota_1(f'_1) \circ \alpha$.
As the inclusion functor 
$\cC_{0,Y'_1} \inj \cC_{\iota_1(Y'_1)/}$
and the functor
$\cC_{\iota_1(Y'_1)/} \to \cC_{Y'/}$
given by the composition with $\alpha$
are equivalences of categories,
there exists an object 
$\iota_1(Y'_1) \to Y_1$ of $\cC_{0,Y'_1}$
and an isomorphism $\beta: Y_1 \xto{\cong} Y$
such that the map $g'$ is equal to the
composite $Y' \xto{\alpha} \iota_1(Y'_1)
\to Y_1 \xto{\beta} Y$.
Let $Y_2$ denote the image of the object
$\iota_1(Y'_1) \to Y_1$ of $\cC_{0,Y'_1}$
under the functor $\cC_{0,Y'_1} \to \cC_0$.
Then, there exists a morphism $f_2: Y_2 \to X$
in $\cC_0$, and the isomorphism $\beta$
induces an isomorphism 
$\beta_2: \iota_0(Y_2)
\xto{\cong} Y$ satisfying $\iota_0(f_2)
= f \circ \beta_2$.
This shows that the pair 
$(\cC_0,\iota_0)$
satisfies property (3).
\end{proof}

\begin{lem}
Condition (4) is satisfied.
\end{lem}
\begin{proof}
We prove that the pair 
$(\cC_0,\iota_0)$ satisfies
property (3). 
It remains to be proven that
the pair $(\cC_0,\iota_0)$ satisfies 
property (4).
Let $X$ be an object of $\cC_0$.
Let us choose an object $X'$ in $\cC_1$ and
an object $g$ in $\cC_{0,X'}$
whose image under the functor $\cC_{0,X'} \to \cC_0$
is equal to $X$.
We regard $g$ as a morphism
$g: \jmath(X') \to \iota_0(X)$ in $\cC$.
Then the inclusion functor 
$\cC_{0,X'} \to \cC_{\iota_1(X')/}$
induces equivalences of categories 
$\cC_{0,X',g/} \xto{\cong} 
(\cC_{\iota_0(X')/})_{g/} \cong 
\cC_{\iota_0(X)}$.
For any morphism $f:Y \to X'$ in $\cC_1$,
the functor 
$f^* : \cC_{0,X'} \to \cC_{0,Y}$ induced
by the composition with $f$ gives an equivalence
of categories from $\cC_{0,X',g/}$ to
$\cC_{0,Y,f^*(g)/}$.
Passing to the colimit, we see that the functor
$\cC_{0,X} \to \cC_0$ induces an equivalence
of categories from $\cC_{0,X',g/}$ to
$\cC_{0,X/}$. Hence the functor
$\cC_{0,X/} \to \cC_{\iota_0(X)/}$ induced by
$\iota_0$ is an equivalence of categories.
This shows that that the pair $(\cC_0,\iota_0)$ satisfies
property (4), which completes the proof.
\end{proof}

\begin{lem}
Condition (2) is satisfied.
\end{lem}

\begin{proof}
Via the functor $\jmath$, we regard $\cC_1$ as a full subcategory of $\cC_0$.
Note that the restriction of $\iota_0$ to $\cC_1$ is equal to $\iota_1$.
As $\iota_1$ is essentially surjective, it suffices to show that
any object of $\cC_1$ is an edge object of $\cC_0$.
Let $X$ be an object of $\cC_1$ and let
$f:Y \to X$ be a morphism in $\cC_0$.
It follows from the construction of category $\cC_0$
that there exists an object $Z$ of $\cC_1$ and morphisms
$g:Z \to X$ and $h:Z \to Y$. As $\cC_0$ is thin, we have
$g = f \circ h$. As $g$ is a morphism in $\cC_1$,
the morphism $\iota_0(g)$ belongs to $\cT(J)$.
Hence, $\iota_0(f)$ belongs to $\cT(J)$. As $f$ is arbitrary,
this shows that $X$ is an edge object of $\cC_0$.
This completes the proof.
\end{proof}

This completes the proof of Lemma~\ref{lem:C0_core}.
\qed

\begin{rmk}
\label{rmk:C1 edge objects}
More strongly, one can show that an object of $\cC_0$ is 
an edge object if and only if it belongs
to $\Obj(\cC_1)$. 
The ``if" part follows from the argument of the proof of
the lemma above.
The ``only if" part can be proved as follows.
Let $X$ be an edge object of $\cC_0$.
It follows from the construction of $\cC_0$ that
there exists an object $Y$ of $\cC_1$ and a morphism $f:Y \to X$.
in $\cC_0$. As $X$ is an edge object, $f$ is of type $J$.
As $(\cC_0,\iota_0)$ is a grid of $(\cC,J)$, and
$(\cC_1,\iota_1)$ is a pregrid of $\cC(\cT(J))$,
the functor $\iota_{0,Y/} : \cC_{0,Y/} \to \cC_{\iota_0(Y)/}$ is
an equivalence of categories and its restriction to 
$\cC_{1,Y/}$ induces an equivalence $\cC_{1,Y/}
\cong \cC(\cT(J))_{\iota_0(Y)/}$ of categories.
Hence, if we regard $f$ as an object of $\cC_{0,Y/}$,
then $f$ is isomorphic to an object of $\cC_{1,Y/}$ in
$\cC_{0,Y/}$. As $\cC_0$ is skeletal, this shows that
$f$ is a morphism in $\cC_1$. Hence $X$ is an object of $\cC_1$.
\end{rmk}

\begin{proof}
[Proof of Proposition~\ref{cor:grid existence}]
This follows from 
Proposition~\ref{prop:pregrid existence}
and Lemma~\ref{lem:C0_core}.

\end{proof}

\section{Proof of Theorem \ref{thm:Galois_main}: the fiber functor is fully faithful}
\label{sec:proof of theorem}
\subsection{Properties of a grid}
In this subsection, we will supply some preliminaries
required in our proof of Theorem \ref{thm:Galois_main},
most of which follow from Lemma \ref{lem:C0_core}.

Let $(\cC,J)$ be a $Y$-site and let 
$(\cC_0, \iota_0)$ be a grid of $(\cC,J)$.

\begin{lem}\label{lem:Galois_over}
Let $Y_1 \xto{f_1} X \xleftarrow{f_2} Y_2$
be a diagram in $\cC_0$.
Suppose that $\iota_0(f_1)$ is a Galois covering
and that there exists a morphism
from $\iota_0(Y_1)$ to $\iota_0(Y_2)$ 
over $\iota_0(X)$.
Then there exists a morphism from $Y_1$
to $Y_2$ over $X$.
\end{lem}

\begin{proof}
The category $\cC_0$ is $\Lambda$-connected
(Lemma~\ref{lem:C0_core}).
Hence, there exist an object $Z$ of $\cC_0$ and morphisms
$g_1: Z \to Y_1$ and $g_2 :Z \to Y_2$ making
the diagram
$$
\begin{CD}
Z @>{g_1}>> Y_1 \\
@V{g_2}VV @VV{f_1}V \\
Y_2 @>{f_2}>> X
\end{CD}
$$
commutative. Let us apply the functor $\iota_0$
to this diagram. It then follows from
Lemma \ref{lem:Galois_lifts} (1)  that
there exists a morphism in the 
undercategory $\cC_{\iota_0(Z)/}$
from $\iota_0(g_1)$ to $\iota_0(g_2)$.
Hence, it follows from (4) of Lemma \ref{lem:C0_core}
that there exists a morphism in the undercategory
$\cC_{0,Z/}$ from $g_1$ to $g_2$.
In particular, there exists a morphism
$h:Y_1 \to Y_2$ in $\cC_0$.
As $\cC_0$ is thin, any diagram in $\cC_0$
is commutative. Hence, $h$ is a morphism
over $X$, which proves the claim.
\end{proof}

\begin{cor}\label{cor:Galois_over}
Let $X$ be an object of $\cC_0$ and let
$f_1:Y_1 \to X$ and $f_2 :Y_2 \to X$ be two morphisms
in $\cC_0$ such that 
$\iota_0(f_1)$ and $\iota_0(f_2)$ are
Galois coverings in $\cC$.
Suppose that there exists an isomorphism
$\iota_0(Y_1) \cong \iota_0(Y_2)$ over $\iota_0(X)$ in $\cC$.
Then, we have $Y_1 = Y_2$ and $f_1 = f_2$.
\end{cor}

\begin{proof}
It follows from Lemma \ref{lem:Galois_over} that
there exist a morphism from $Y_1$ to $Y_2$
and a morphism from $Y_2$ to $Y_1$.
As $\cC_0$ is a poset, the claim follows.
\end{proof}

\begin{lem}\label{lem:isom_under}
Let $X$ and $X'$ be objects of $\cC_0$ and
let $\beta: \iota_0(X) \xto{\cong} \iota_0(X')$
be an isomorphism in $\cC$. Then, for any
morphism $f:X \to Y$ in $\cC_0$,
there exists a unique morphism $f':X' \to Y'$
in $\cC_0$ satisfying the following property:
there exists an
isomorphism $\beta': \iota_0(Y) \xto{\cong}
\iota_0(Y')$ in $\cC$ that makes the diagram
$$
\begin{CD}
\iota_0(X) @>{\beta}>{\cong}> \iota_0(X') \\
@V{\iota_0(f)}VV @VV{\iota_0(f')}V \\
\iota_0(Y) @>{\beta'}>{\cong}> \iota_0(Y')
\end{CD}
$$
commutative. (We note that such an isomorphism 
$\beta'$ is unique, as $f$ is an epimorphism.)
\end{lem}

\begin{proof}
From the given data, we obtain a morphism
$\iota_0(f) \circ \beta^{-1}: \iota_0(X') \to \iota_0(Y)$.
Using the fact that the functor $\cC_{0,X'/} \to \cC_{\iota_0(X')/}$
induced by $\iota_0$ is an equivalence of categories
(Lemma \ref{lem:C0_core}(4)),
we obtain morphisms $f'$ and $\beta'$
that make the diagram commutative.
The uniqueness of $f'$ follows, as $\cC_0$ is skeletal.
\end{proof}

\begin{lem}\label{lem:isom_over}
Let $X$ and $X'$ be objects of $\cC_0$ and
let $\beta: \iota_0(X) \xto{\cong} \iota_0(X')$
be an isomorphism in $\cC$. Then, for any
morphism $f:Y \to X$ in $\cC_0$ such that
$\iota_0(f)$ is a Galois covering in $\cT$,
there exists a unique morphism $f':Y' \to X'$
in $\cC_0$ satisfying the following property:
there exists a (not necessarily unique)
isomorphism $\beta': \iota_0(Y) \xto{\cong}
\iota_0(Y')$ in $\cC$ that makes the diagram
$$
\begin{CD}
\iota_0(Y) @>{\beta'}>{\cong}> \iota_0(Y') \\
@V{\iota_0(f)}VV @VV{\iota_0(f')}V \\
\iota_0(X) @>{\beta}>{\cong}> \iota_0(X')
\end{CD}
$$
commutative.
\end{lem}

\begin{proof}
The existence of $f'$ follows from
(3) of Lemma \ref{lem:C0_core}.
We prove the uniqueness.
Suppose that the two morphisms 
$f'_1:Y'_1 \to X'$ and $f'_2:Y'_2 \to X'$
satisfy the property of the lemma.
As both $\iota_0(Y'_1)$ and $\iota_0(Y'_2)$
are isomorphic to $\iota_0(Y)$ over $X$,
there exists an isomorphism from 
$\iota_0(Y'_1)$ to $\iota_0(Y'_2)$ over $\iota_0(X)$.
Hence, it follows from Lemma \ref{lem:Galois_over}
that there exist morphisms $Y'_1 \to Y'_2$
and $Y'_2 \to Y'_1$ over $X$.
As the category $\cC_0$ is a poset,
it follows that $Y'_1=Y'_2$ and
$f'_1 = f'_2$, which proves the claim.
\end{proof}

\begin{lem} \label{lem:yama}
Let $X$ be an object of $\cC_0$ and let 
$\iota_0(X) \xleftarrow{f} Z \to Y$ be a diagram in
$\cC$ such that $f$ belongs to $\cT$.
Then, there exist a diagram
$X \leftarrow Z' \to Y'$ in $\cC_0$
and isomorphisms $\iota_0(Z') \cong Z$
and $\iota_0(Y') \cong Y$ that make
the diagram
\begin{equation}\label{two_squares}
\begin{CD}
\iota_0(X) @<<< \iota_0(Z') @>>> \iota_0(Y') \\
@| @V{\cong}VV @V{\cong}VV \\
\iota_0(X) @<{f}<< Z @>>> Y
\end{CD}
\end{equation}
in $\cC$ commutative.
\end{lem}

\begin{proof}
It follows from (3) of Lemma \ref{lem:C0_core} that
there exist an object $Z'$ of $\cC_0$ and the isomorphism 
$\iota_0(Z') \cong Z$ in $\cC$ that make the left square of
\eqref{two_squares} commutative.
It then follows from (4) of Lemma \ref{lem:C0_core} that
there exist an object $Y'$ of $\cC_0$ and the isomorphism 
$\iota_0(Y') \cong Y$ in $\cC$ that make the right square of
\eqref{two_squares} commutative.
This proves the claim.
\end{proof}

\subsection{Proof of Theorem~\ref{thm:Galois_main}: the functor $\omega_\Cip$ is faithful}
\label{sec:faithful}
\begin{lem}
\label{lem:injectivity}
Let $F$ be a sheaf on $(\cC,J)$.
Then for any edge object $X$ of $\cC_0$,
the map $F(X) \to \omega_\Cip(F)$
given by the universality of colimit
is injective.
\end{lem}

\begin{proof}
Let $f:Y \to X$ be a morphism in $\cC_0$.
Then, $\iota_0(f)$ belongs to $\cT$.
As $F$ is a sheaf on $(\cC,J)$, it follows from
Corollary \ref{cor:sheaf_criterion2} that the map
$F(X) \to F(Y)$ is injective.
Hence, the map $F(X)\to \omega_\Cip(F)$
is injective.
\end{proof}

\begin{proof}[Proof of Theorem~\ref{thm:Galois_main}
: faithfulness]
We next will prove the faithfulness of the functor
\eqref{main_equiv}.
Let $F$ and $F'$ be sheaves on $(\cC,J)$
and let $f,g:F \to F'$ be two morphisms of
sheaves on $(\cC,J)$.
Suppose that $\omega_\Ci(f) = \omega_\Ci(g)$.
We show that $f=g$.
Let us take an arbitrary object $X$ of $\cC$.
It suffices to show that $f(X)=g(X)$.
As $\iota_0$ is essentially surjective,
there exists an object $X'$ of $\cC_0$
and an isomorphism $\iota_0(X')\xto{\alpha}
X$ in $\cC$. Hence, the claim follows
from the commutativity of the diagram
$$
\begin{CD}
F(X) @>{\alpha^*}>> F(\iota_0(X'))
@>>> \omega_\Cip(F) \\
@V{f(X),g(X)}VV @VVV 
@VV{\omega_\Cip(f),\omega_\Cip(g)}V \\
G(X) @>{\alpha^*}>> G(\iota_0(X'))
@>>> \omega_\Cip(G) \\
\end{CD}
$$
and Lemma \ref{lem:injectivity}.
\end{proof}

\subsection{The category $I_X$}
\label{sec:IX}
Let $X$ be an edge object of $\cC_0$.
We denote by $I_X$ the full-subcategory of
the overcategory $\cC_{0,/X}$ whose objects are
the morphisms $f:Y \to X$ in $\cC_0$
such that $\iota_0(f)$ is a Galois covering
in $\cT$. For an object $f:Y\to X$ of $I_X$,
we write $\Gal(f)$ for $\Gal(\iota_0(f))$.
\begin{lem}\label{lem:I_X cofinal}
The category $I_X$ is cofiltered and the
objects of $I_X$ are cofinal in $\cC_{0,/X}$.
\end{lem}
\begin{proof}
As $\cC_0$ is $\Lambda$-connected and thin,
the overcategory $\cC_{0,/X}$ is $\Lambda$-connected 
and thin.
It follows from Lemma \ref{lem:core_cofiltered} that
$\cC_{0,/X}$ is cofiltered.
Hence, it suffices to prove that the objects of
$I_X$ are cofinal in $\cC_{0,/X}$.
Let $f:Y\to X$ be an object of $\cC_{0,/X}$ and
let us regard it as a morphism in $\cC_0$.
As $\cT(J)$ has enough Galois coverings,
there exists a morphism $g':Z' \to \iota_0(Y)$ in $\cC$
such that $g'$ belongs to $\cT(J)$ and the composite
$\iota_0(f) \circ g'$ is a Galois covering in $\cC$.
It follows from property (3) of the grid
$(\cC_0,\iota_0)$ that there exist a morphism
$g:Z \to Y$ in $\cC_0$ and an isomorphism
$\alpha:\iota_0(Z) \xto{\cong} Z'$ in $\cC$ 
satisfying $\iota_0(g) = g' \circ \alpha$.
The morphism $f \circ g$ in $\cC_0$, regarded
as an object of $\cC_{0,/X}$ is an object
of $I_X$, as $\iota_0(f \circ g)=
\iota_0(f) \circ g' \circ \alpha$ is
a Galois covering in $\cC$.
This proves the claim.
\end{proof}

\begin{lem}\label{lem:yama2}
Let $X$ be an edge object of $\cC_0$ and
let $Z$ be an object of $\cC_0$.
Then there exist an object $f:Y \to X$ of $I_X$
and a morphism $Y \to Z$ in $\cC_0$.
\end{lem}

\begin{proof}
As $\cC_0$ is $\Lambda$-connected, there exist
an object $Y'$ of $\cC_0$ and morphisms
$h_1: Y' \to X$ and $h_2: Y' \to Z$ in $\cC_0$.
It follows from Lemma \ref{lem:I_X cofinal}
that there exists a morphism $h:Y \to Y'$
such that the composite $f = h_1 \circ h:Y \to X$ is
an object of $I_X$.
This proves the claim, as $h_2 \circ h$ is a
morphism from $Y$ to $Z$ in $\cC_0$.
\end{proof}

\begin{cor}
\label{cor:C0Y}
Let $X$ be an edge object of $\cC_0$.
Then the functor 
$$
\varinjlim_{(f:Y \to X) \in
\Obj I_X} \cC_{0,Y/} \to \cC_0
$$
given by the forgetful functors
$\cC_{0,Y/} \to \cC_0$ is an isomorphism of categories.
\end{cor}
\begin{proof}
This follows immediately from the previous lemma.
\end{proof}

\subsection{The functor $\theta_{Z,Z',h}$}
\label{sec:theta}
Let $Z, Z'$ be objects of $\cC_0$ and 
let $h:\iota_0(Z) \to \iota_0(Z')$ be a morphism in $\cC$.
Let $- \circ h: \cC_{\iota_0(Z')/} \to 
\cC_{\iota_0(Z)/}$ denote the functor
given by the composition with $h$.
As $\cC$ is an $E$-category,
the two undercategories $\cC_{\iota_0(Z)/}$
and $\cC_{\iota_0(Z')/}$ are thin and the functor
$-\circ h$ is fully faithful.
Let us consider the diagram
$$
\cC_{0,Z'/}
\xto{\iota_{0,Z'/}} \cC_{\iota_0(Z')/}
\xto{-\circ h} \cC_{\iota_0(Z)/}
\xleftarrow{\iota_{0,Z/}} \cC_{0,Z/}
$$
of categories and functors.
As the two functors in this diagram
other than $-\circ h$
are equivalences of categories,
this diagram gives a functor
$\cC_{0,Z'/} \to \cC_{0,Z/}$ which is fully faithful.
We denote this functor by $\theta_{Z,Z',h}$.
As the category $\cC_{0,Z/}$ is skeletal,
the following statement holds (although it is a
simple observation, we state it as a lemma
because we will refer to the statement 
several times):
\begin{lem}
\label{lem:theta}
Let $Z$, $Z'$ and $h$ be as above.
Then, $\theta_{Z,Z',h}$ is the unique
functor from $\cC_{0,Z/}$ to $\cC_{0,Z'/}$
that makes the diagram
\begin{equation}\label{CD:theta}
\begin{CD}
\cC_{0,Z'/} @>{\theta_{Z,Z',h}}>> \cC_{0,Z/} \\
@V{\iota_{0,Z'/}}VV @VV{\iota_{0,Z/}}V \\
\cC_{\iota_0(Z')/} @>{-\circ h}>> \cC_{\iota_0(Z)/}
\end{CD}
\end{equation}
commutative up to natural equivalence.
\qed
\end{lem}
\begin{cor}
Let $Z''$ be another object of $\cC_0$ and 
$h':\iota_0(Z') \to \iota_0(Z'')$
be a morphism in $\cC$. Then, we have
$\theta_{Z,Z'',h'\circ h}
= \theta_{Z,Z',h} \circ \theta_{Z',Z'',h'}$.
\qed
\end{cor}

\begin{cor}
Suppose that $h$ is an isomorphism in $\cC$.
Then, the functor $\theta_{Z,Z',h}$ is an isomorphism 
of categories.
\qed
\end{cor}

Let $f':Z' \to Y'$ be an object of 
$\cC_{0,Z'/}$, and denote
the object $\theta_{Z,Z',h}(f')$ 
of $\cC_{0,Z/}$ by
$f:Z \to Y$. It follows from Lemma~\ref{lem:theta} that
there exists an isomorphism
$\alpha : \iota_0(Y) 
\xto{\cong} \iota_0(Y')$ 
such that the
diagram
$$
\begin{CD}
\iota_0(Z) @>{h}>> \iota_0(Z') \\
@V{\iota_0(f)}VV @VV{\iota_0(f')}V \\
\iota_0(Y) @>{\alpha}>{\cong}> \iota_0(Y')
\end{CD}
$$
is commutative. As $\iota_0(f)$ is an epimorphism,
such an isomorphism $\alpha$ is unique.

\begin{lem}
\label{lem:theta_descends}
Let the notation be as above.
Then the functor $\theta': \cC_{0,Y'/} \to \cC_{0,Y/}$
induced by the functor $\theta_{Z,Z',h}$ is equal
to the functor $\theta_{Y,Y',\alpha}$.
\end{lem}

\begin{proof}
It follows from the definition of the functor $\theta'$
that the diagram \eqref{CD:theta} induces a diagram
$$
\begin{CD}
\cC_{0,Y'/} @>{\theta'}>> \cC_{0,Y/} \\
@V{\iota_{0,Y'/}}VV @VV{\iota_{0,Y/}}V \\
\cC_{\iota_0(Y')/} @>{-\circ \alpha}>> \cC_{\iota_0(Y)/}
\end{CD}
$$
of categories and functors, that is
commutative up to natural isomorphisms.
Hence the claim follows from Lemma \ref{lem:theta}.
\end{proof}

Let $Z$, $Z'$ and $h:\iota_0(Z) \to \iota_0(Z')$ 
be as 
in Lemma \ref{lem:theta}.
Then, the natural equivalence from the composite
$\cC_{0,Z'/} \xto{\theta_{Z,Z',h}}
\cC_{0,Z/} \xto{\iota_{0,Z/}} \cC_{\iota_0(Z)/}$
to the composite $\cC_{0,Z'/} \xto{\iota_{0,Z'/}} \cC_{\iota_0(Z')/}
\xto{-\circ h} \cC_{\iota_0(Z)/}$ is unique, as
the category $\cC_{\iota_0(Z)/}$ is thin.
We denote this natural equivalence by $\xi_{Z,Z',h}$.

\subsection{Some properties of the monoids $M_\Cip$ and $\bK_X$}
\label{sec:bKX_group}

\begin{lem}
\label{lem:MC}
Let $(\alpha,\gamma_\alpha)$ be an element of the monoid $M_\Cip$.
Then, for any object $X$ of $\cC_0$, the functor
$\alpha': \cC_{0,X/} \to \cC_{0,\alpha(X)/}$ induced by $\alpha$
on the undercategories is equal to the isomorphism 
$\theta_{\alpha(X),X,\gamma_\alpha(X)^{-1}}$ of categories.
\end{lem}

\begin{proof}
We set 
$\beta = \gamma_\alpha(X): 
\iota_0(X) \xto{\cong}
\iota_0(\alpha(X))$.
Let $- \circ \beta^{-1}: \cC_{\iota_0(X)/}
\to \cC_{\iota_0(\alpha(X))/}$ 
denote the functor given by
the composition with $\beta^{-1}$.
Then, the natural isomorphism $\gamma_\alpha$ induces a
natural isomorphism from the composite
$\cC_{0,X/} \xto{\iota_{0,X/}} \cC_{\iota_0(X)/} \xto{-\circ\beta^{-1}}
\cC_{\iota_0(\alpha(X))/}$ to the composite
$\cC_{0,X} \xto{\alpha'} \cC_{0,\alpha(X)/}
\xto{\iota_{0,\alpha(X)/}} 
\cC_{\iota_0(\alpha(X))/}$.
Hence, the claim follows from Lemma \ref{lem:theta}.
\end{proof}

\begin{lem}
\label{lem:MC2}
Let $(\alpha,\gamma_\alpha)$ be an element of the monoid $M_\Cip$.
Then the functor $\alpha: \cC_0 \to \cC_0$ is fully faithful.
\end{lem}

\begin{proof}
Let $X_1$ and $X_2$ be objects of $\cC_0$
and suppose that there exists a morphism 
$f: \alpha(X_1) \to \alpha(X_2)$ in $\cC_0$.
As the category $\cC_0$ is thin, it suffices to
show that there exists a morphism from $X_1$ to $X_2$
in $\cC_0$.
As $\cC_0$ is $\Lambda$-connected,
there exist an object $Y$ of $\cC_0$ and
morphisms $h_1:Y \to X_1$ and $h_2:Y\to X_2$.
By applying the functor $\alpha$, we obtain the
diagram 
$\alpha(X_1) \xleftarrow{\alpha(h_1)} Y 
\xto{\alpha(h_2)} \alpha(X_2)$ in $\cC_0$.
As the category $\cC_0$ is thin, we have
$\alpha(h_2) = f \circ \alpha(h_1)$ and therefore $f$ can be
regarded as a morphism from $\alpha(h_1)$ to $\alpha(h_2)$
in $\cC_{0,\alpha(Y)/}$.
It follows from Lemma \ref{lem:MC2} that
the functor $\alpha':\cC_{0,Y/} \to \cC_{0,\alpha(Y)/}$
induced by $\alpha$ is equal to 
$\theta_{\alpha(Y),Y,\gamma_\alpha(Y)^{-1}}$.
Note that $\alpha'$ sends the object $h_i$ of $\cC_{0,Y/}$
to the object $\alpha(h_i)$ of $\cC_{0,\alpha(Y)/}$.
As $\theta_{\alpha(Y),Y,\gamma_\alpha(Y)^{-1}}$ is
an isomorphism of categories, it follows that
there exists a morphism from $h_1$ to $h_2$ that
is sent to the morphism $f$ by $\theta_{\alpha(Y),Y,\gamma_\alpha(Y)^{-1}}$.
This proves that $\alpha$ is fully faithful.
\end{proof}

\begin{lem}
\label{lem:bKX1}
Let $X$ be an edge object of $\cC_0$ and let
$(\alpha,\gamma_\alpha)$ be an element of $\bK_X$.
Let $f : Y \to X$ be a morphism in $\cC_0$ such that
$\iota_0(f)$ is a Galois covering in $\cT$.
Then we have $\alpha(Y)=Y$, $\alpha(f)=f$, and
$\gamma_\alpha(Y) : \iota_0(Y) \xto{\cong} \iota_0(Y)$ 
is an element of the Galois group of $\iota_0(f)$.
\end{lem}

\begin{proof}
As $\alpha(X)= X$ and 
$\gamma_\alpha(X)$ is the identity,
we have the commutative diagram
$$
\begin{CD}
\iota_0(Y) @>{\gamma_\alpha(Y)}>{\cong}> 
\iota_0(\alpha(Y)) \\
@V{\iota_0(f)}VV @VV{\iota_0(\alpha(f))}V \\
\iota_0(X) @= \iota_0(X).
\end{CD}
$$
In particular, $\iota_0(\alpha(f))$ 
is a Galois covering in $\cT$.
Hence by Corollary \ref{cor:Galois_over}, 
we have
$\alpha(Y) = Y$ and 
$\alpha(f) = f$, from which the claim follows.
\end{proof}

\begin{lem}
\label{lem:bKX2}
Let $X$ be an edge object of $\cC_0$.
Then the submonoid $\bK_X \subset M_\Cip$
introduced in Section \ref{sec:submonoids} is a group.
\end{lem}

\begin{proof}
Let $(\alpha,\gamma_\alpha)$ be an
element of $\bK_X$. It suffices to show that
the functor $\alpha:\cC_0 \to \cC_0$ is
an isomorphism of categories.
As the category $\cC_0$ is skeletal, 
it then suffices to prove that
the functor $\alpha$ gives an equivalence of categories.
It follows from Lemma \ref{lem:MC2} that the functor
$\alpha$ is fully faithful.
We prove that $\alpha$ is essentially surjective.
Let us take an arbitrary object $Z$ of $\cC_0$.
It follows from Lemma \ref{lem:yama2} that
there exist a morphism $f:Y \to X$ in $\cC_0$ 
and a morphism $h:Y \to Z$ in $\cC_0$
such that $\iota_0(f)$ is a Galois
covering in $\cT$.
By Lemma \ref{lem:bKX1}, we have $\alpha(Y)=Y$.
We set $\beta = \gamma_\alpha(Y)^{-1}$, which is an
automorphism of the object $\iota_0(Y)$.
As $\theta_{Y,Y,\beta}$ 
is an automorphism of the
category $\cC_{0,Y/}$, 
there exists an object
$h': Y \to Z'$ of $\cC_{0,Y/}$ 
that is sent to
$h$ by the automorphism 
$\theta_{Y,Y,\beta}$.
It then follows from Lemma \ref{lem:MC2} that
we have $\alpha(Z')=Z$, which proves that the functor
$\alpha$ is essentially surjective.
This completes the proof.
\end{proof}

\subsection{The isomorphisms $\psi_X$ and $\phi_X$}
\subsubsection{}
\label{sec:Gal gp surjective}
Let $X$ be an edge object of $\cC_0$.
Let $f_1,f_2$ be two objects of $I_{X}$
and suppose that there exists a morphism
$g$ from $f_2$ to $f_1$ in $I_{X}$.
Note that $\iota_0(f_1)$ and $\iota_0(f_2)$
are Galois coverings in $\cC$.
It follows from Lemma \ref{lem:descends}
(1)  
that for any $\alpha \in \Gal(f_2)$
there exists a unique $\alpha' \in \Gal(f_1)$
satisfying $\iota_0(g) \circ \alpha
= \alpha' \circ \iota_0(g)$.
By sending $\alpha$ to $\alpha'$ we obtain
a map $\Gal(f_2) \to \Gal(f_1)$.
It follows from Lemma \ref{lem:descends} (2)
that this map is a homomorphism of groups.
It follows from Lemma \ref{lem:lifts} 
that this homomorphism is surjective.
We set
$$
H_{X} = \varprojlim_{f \in \Obj I_{X}} 
\Gal(f).
$$

\begin{lem}
\label{lem:Gal gp surjective}
Let the notation be as above.
Suppose that one of the conditions on
cardinality in Section~\ref{sec:cardinality} is satisfied
for the category $\cC(\cT(J))_{/X}$.
Then, the canonical map
\[ \rho:H_X \to \Gal(f)\]
where $f \in I_X$, is surjective.
\end{lem}
\begin{proof}
It follows from Lemma \ref{lem:lifts} that
the transition maps in the limit in the definition of $H_X$
are surjective.
The assumption on the category $\cC(\cT(J))_{/X}$ 
and Lemma \ref{lem:Galois_over} imply that we have either 
that $\Gal(f')$ is a finite group for any $f' \in \Obj I_X$ 
or that the set $I_X$ is at most countable.
Hence the natural projection $H_X \to \Gal(f)$ is surjective
for any $f \in I_X$.
\end{proof}

\subsubsection{}
\label{sec:defn phiX}
It follows from Lemma \ref{lem:bKX2} that
$\bK_X$ is a group.
Let us construct a homomorphism
$\psi_X : H_{X} \to \bK_X$ of groups.
Let $\beta = (\beta_f)_{f \in \Obj I_{X}} \in H_{X}$.
For each object $f:Y \to X$ of $I_X$,
we let $\alpha(\beta)_Y$ denote the automorphism
$\theta_{Y,Y,\beta_f^{-1}}
: \cC_{0,Y/} \xto{\cong}
\cC_{0,Y/}$ of categories.
Let $f':Y' \to X$ be another object of $I_X$ and
let $h: Y' \to Y$ be a morphism in $I_X$.
It then follows from Lemma \ref{lem:theta_descends} that
the diagram
$$
\begin{CD}
\cC_{0,Y/} @>{\alpha(\beta)_Y}>{\cong}>
\cC_{0,Y/} \\
@V{-\circ h}VV @VV{-\circ h}V \\
\cC_{0,Y'/} @>{\alpha(\beta)_{Y'}}>{\cong}>
\cC_{0,Y'/}
\end{CD}
$$
is commutative.
By taking the colimit with respect to $f$
and by using Corollary \ref{cor:C0Y}, we obtain an
isomorphism $\cC_0 \xto{\cong} \cC_0$ of categories 
which we denote by $\alpha(\beta)$.

\subsubsection{}
For each object $f:Y \to X$ of $I_X$,
we denote by $E_{1,Y}$ the composite
$E_{1,Y}: \cC_{0,Y/} \xto{\iota_{0,Y/}} 
\cC_{\iota_0(Y)/}
\xto{-\circ \beta_f^{-1}} \cC_{\iota_0(Y)/}$ 
and by $E_{2,Y}$ the composite
$E_{2,Y}: \cC_{0,Y/} \xto{\alpha(\beta)_{Y}} 
\cC_{0,Y/} \xto{\iota_{0,Y/}} \cC_{\iota_0(Y)/}$.
We set
$\gamma_{\alpha(\beta),Y} =
\xi_{Y,Y,\beta_f^{-1}}^{-1}$,
which is a natural isomorphism from the composite
$E_{1,Y}$ to the composite $E_{2,Y}$.

\subsubsection{}
Let $f':Y' \to X$ be another object of $I_X$ and
let $h: Y' \to Y$ be a morphism in $I_X$.
Let us consider the functors
$-\circ h: \cC_{0,Y/} \to \cC_{0,Y'/}$ 
and $-\circ\iota_0(h):\cC_{\iota_0(Y)/} \to \cC_{\iota_0(Y')/}$ 
given by the compositions with $h$ and
$\iota(h)$, respectively.

For $i=1,2$, the functor 
$E_{i,Y'}\circ(-\circ h) : \cC_{0,Y/}
\to \cC_{\iota_0(Y')}$ is equal to the functor
$(-\circ \iota_0(h)) \circ E_{i,Y}: \cC_{0,Y/}
\to \cC_{\iota_0(Y')/}$.
It follows from the proof of 
Lemma \ref{lem:theta_descends} that
the natural isomorphism
$\gamma_{\alpha(\beta),Y'}\circ (-\circ h) :
E_{1,Y'}\circ(-\circ h) 
\to E_{2,Y'} \circ(-\circ h)$ 
is equal to the natural isomorphism
$\gamma_{\alpha(\beta),Y'}\circ (-\circ h) :
E_{1,Y'}\circ(-\circ h) 
\to E_{2,Y'} \circ (-\circ h)$.
By taking the colimit with respect to $f$,
we obtain a natural isomorphism
$\gamma_{\alpha(\beta)} : 
\iota_0 \xto{\cong} \iota_0 
\circ \alpha(\beta)$.
Thus, we obtain an element
$(\alpha(\beta), \gamma_{\alpha(\beta)})$ in $M_\Cip$.

It is easy to check that the element
$(\alpha(\beta), \gamma_{\alpha(\beta)})$
belongs to $\bK_X$ and that
the map $\psi_X : H_X \to \bK_X$ 
that sends $\beta \in H_X$
to $(\alpha(\beta), \gamma_{\alpha(\beta)})$ is
a homomorphism of groups.

\subsubsection{}
\label{sec:phi_X}
Let us construct a homomorphism
$\phi_X : \bK_X \to H_X$ as follows.
Let $(\alpha,\gamma_\alpha) \in \bK_X$.
For any object $f:Y \to X$ of $I_X$, it follows from
Lemma \ref{lem:bKX1} that we have
$\alpha(Y) = Y$ and 
$\gamma_\alpha(Y) \in \Gal(f)$.
Let $f':Y' \to X$ be another object of $I_X$ and
let $h:Y' \to Y$ be a morphism in $I_X$.
As $\gamma_\alpha(Y)$ is functorial in $Y$,
we have $\iota_0(h) \circ \gamma_\alpha(Y') 
= \gamma_\alpha(Y) \circ \iota_0(h)$.
This shows that 
$(\gamma_\alpha(Y))_{f:Y \to X}$
is an element of $H_X$.
We define 
$\phi_X : \bK_X \to H_X$ to be the map
that sends $(\alpha,\gamma_\alpha) \in \bK_X$
to the element $(\gamma_\alpha(Y))_{f:Y\to X}$ of $H_X$.
One can check easily that $\phi_X$ is a homomorphism
of groups.

\begin{lem}
\label{lem:psi_isom}
The homomorphism $\psi_X : H_X \to \bK_X$ 
is an isomorphism.
\end{lem}
\begin{proof}
It is clear from the construction of $\psi_X$ that
the composite $\phi_X \circ \psi_X$ is equal to
the identity. Hence, to prove that $\psi_X$ is bijective
it suffices to prove that $\phi_X$ is injective.

Let $(\alpha,\gamma_{\alpha}) \in \bK_X$ and suppose that 
$\phi_X((\alpha,\gamma_{\alpha}))=1$.
Then for any object $f:Y \to X$ of $I_X$,
we have $\alpha(Y) = Y$ and 
$\gamma_{\alpha}(Y) = \id_{\iota_0(Y)}$.
It follows from Lemma \ref{lem:MC} that
the automorphism $\cC_{0,Y/}
\to \cC_{0,Y/}$ induced by $\alpha$ 
is equal to the identity.
Hence, it follows from Corollary \ref{cor:C0Y} that the automorphism
$\alpha$ is equal to the identity.
As the category $\cC_{\iota_0(Y)/}$ is thin,
any natural auto-equivalence $\gamma$ of
the functor $\cC_{0,Y/} \to \cC_{\iota_0(Y)/}$ 
induced by $\iota_0$ is equal to the identity.
This shows that the natural auto-equivalence of the
functor $\iota_{0,/Y} 
:\cC_{0,Y/} \to \cC_{\iota_0(Y)/}$ induced
by the natural auto-equivalence $\gamma_\alpha$
is equal to the identity for any object $f:Y \to X$
of $I_X$.
It therefore follows from Corollary \ref{cor:C0Y} that
the natural auto-equivalence $\gamma_\alpha$
is equal to the identity, which proves that
the homomorphism $\phi_X$ is injective.
This proves the claim.
\end{proof}

\begin{cor}
\label{cor:KX-invariant}
Let $F$ be a sheaf on $(\cC,J)$.
Then for any edge object $X$ of $\cC_0$,
the map $F(\iota_0(X)) \to \omega_\Cip(F)$
induces a bijection $F(\iota_0(X)) \xto{\cong} 
\omega_\Cip(F)^{\bK_X}$.
\end{cor}

\begin{proof}
For any object $f:Y \to X$ of $I_X$,
the pullback map $F(\iota_0(X)) \to F(\iota_0(Y))$
induces an isomorphism $F(\iota_0(X)) \to F(\iota_0(Y))^{\Gal(f)}$.
Passing to the inductive limit with respect to $f$,
we see that the map $F(\iota_0(X)) \to \omega_\Cip(F)$
induces an isomorphism $F(\iota_0(X)) \xto{\cong}
\omega_\Cip(F)^{H_X}$.
Hence, the claim follows from
Lemma \ref{lem:psi_isom}.
\end{proof}

\subsection{The monoid $M_\Cip$ has sufficiently many elements}
From now until the end of 
Section~\ref{sec:enough points},
we will assume that, for any object $X$ of $\cC$,
the category $\cC(\cT(J))_{/X}$ satisfies 
at least one of the two conditions 
in Section \ref{sec:cardinality}.

\begin{lem}
\label{lem:sufficiently many}
Let $X$ be an edge object of $\cC_0$,
$X'$ an object of $\cC_0$,
and $\beta: \iota_0(X) \xto{\cong} \iota_0(X')$ 
be an isomorphism in $\cC$.
Then, there exists an element $(\alpha,\gamma_\alpha)$
of $M_\Cip$ that satisfies $\alpha(X) = X'$
and $\gamma_\alpha(X) = \beta$.
\end{lem}

\begin{proof}
For any object $f:Y \to X$ of $I_X$,
let $S_{f}$ denote the set of
pairs $(f',\beta')$ of a morphism $f': Y' \to X'$
in $\cC_0$ and an isomorphism
$\beta': \iota_0(Y) \xto{\cong} \iota_0(Y')$
in $\cC$ that make the diagram
$$
\begin{CD}
\iota_0(Y) @>{\beta'}>{\cong}> \iota_0(Y') \\
@V{\iota_0(f)}VV @VV{\iota_0(f')}V \\
\iota_0(X) @>{\beta}>{\cong}> \iota_0(X') \\
\end{CD}
$$
commutative. For $(f',\beta') \in S_f$
and for $\sigma \in \Gal(f)$, we set
$\sigma \cdot (f',\beta) = (f',\beta \circ \sigma^{-1})$.
This gives an action from the left of the group
$\Gal(f)$ on the set $S_f$.
It follows from 
Lemma \ref{lem:isom_over} that
the set $S_f$ is non-empty.
It follows from Corollary \ref{cor:Galois_over} that
for any two elements $(f', \beta'), (f'',\beta'')$
of $S_f$, we have $f' = f''$.
This shows that the set $S_f$ is a left $\Gal(f)$-torsor.
Let $g: Z \to X$ be another object of $I_X$ and let
$h:Z \to Y$ be a morphism in $I_X$.
Let $(g':Z' \to X',\beta'')$ be an element of $S_g$. 
Let us consider the isomorphism 
$\theta_{Z',Z,\beta''^{-1}}$ of categories, and let us regard $f$ as a morphism from $h$ to $g$
in $\cC_{0,Z/}$ and set 
$f'=\theta_{Z',Z,\beta''^{-1}}(f)$
and $\beta'=\xi_{Z',Z,\beta''^{-1}}$.
Then $(f',\beta')$ is an element of $S_f$.
Let us write $f':Y' \to X'$.
It then follows from Lemma \ref{lem:isom_under} that
the isomorphism $\beta':\iota_0(Y) \to \iota_0(Y')$ in
$\cC$ has the following characterization:
$\beta'$ is the unique isomorphism 
$\iota_0(Y) \xto{\cong} \iota_0(Y')$ 
such that there exists
a morphism $h':Z' \to Y'$ that makes the diagram
$$
\begin{CD}
\iota_0(Z) @>{\beta''}>{\cong}> \iota_0(Z') \\
@V{\iota_0(h)}VV @VV{\iota_0(h')}V \\
\iota_0(Y) @>{\beta'}>{\cong}> \iota_0(Y')
\end{CD}
$$
commutative.
By sending $(f'',\beta'')$ to $(f',\beta')$, we obtain
a map $S_{f''} \to S_f$, which we denote by $S(h)$.
The characterization of the isomorphism $\beta'$
given above shows that we have 
$S(h)(\sigma \cdot (f'',\beta''))
= h_*(\sigma) \cdot S(h)((f'',\beta''))$
for any $\sigma \in \Gal(g)$, where
$h_*: \Gal(g) \to \Gal(f)$
denotes the homomorphism in (2) of Lemma \ref{lem:descends}
, and that
we have $S(h \circ h')
=S(h) \circ S(h)$ for any composable morphisms
$h$, $h'$ in $I_X$.

Let $\wt{H}_X$ denote the projective system
$(\Gal(f))_{f \in \Obj\, I_X}$ of groups.
Then, $(S_f)_{f \in \Obj\, I_X}$ is
a left $\wt{H}_X$-torsor.
As $\cC_{/X}$ satisfies at least one of the conditions
in Section \ref{sec:existence of a pregrid}, the partially ordered set
corresponding to the poset category $I_X$
satisfies at least one of the conditions in 
Lemma \ref{lem:admissible}.
Hence, the limit $\varprojlim_{f \in \Obj\, I_X} S_f$
is non-empty.
Let us choose an element $(h_f,\beta_f)_{f \in \Obj\, I_X}$ 
of $\varprojlim_{f \in \Obj\, I_X} S_f$.
For an object $f:Y \to X$ of $I_X$, let
$Y'$ denote the domain of $h_f$. Let us consider
the isomorphism $\theta_{Y',Y,\beta_f^{-1}}:
\cC_{0,Y/} \xto{\cong} \cC_{0,Y'/}$ 
of categories. Let $g:Z \to X$ be another object of $I_X$
and let $h:Z \to Y$ be a morphism from $g$ to $f$
in $I_X$. Let $Z'$ denote the domain of $h_f$.
It follows from the definition of the transition map
of the projective system $(S_f)$ that
there exists a morphism $h':Z' \to Y'$ in $\cC_0$ 
that makes the diagram
$$
\begin{CD}
\iota_0(Z) @>{\beta_g}>{\cong}> \iota_0(Z') \\
@V{\iota_0(h)}VV @VV{\iota_0(h')}V \\
\iota_0(Y) @>{\beta_f}>{\cong}> \iota_0(Y')
\end{CD}
$$
commutative. It follows from Lemma \ref{lem:isom_under} that
$\theta_{Z',Z,\beta_g^{-1}}$ sends the object
$h$ of $\cC_{0,Z/}$ to the object $h'$ of $\cC_{0,Z/}$.
Hence, it follows from Lemma \ref{lem:theta_descends}
that the functor $\cC_{0,Y/} \to \cC_{0,Y'/}$
induced by the isomorphism 
$\theta_{Z',Z,\beta_g^{-1}}$ of categories
is equal to the isomorphism 
$\theta_{Y',Y,\beta_f^{-1}}$ of categories.
Thus, by taking the colimit with respect to $f$, we
obtain functors 
$$
\cC_0 = \varinjlim_{f \in \Obj\, I_X} \cC_{0,Y/}
\to \varinjlim_{f \in \Obj\, I_X} \cC_{0,Y'/} \to \cC_0,
$$
whose composite we denote by $\alpha: \cC_0 \to \cC_0$.
We have the natural isomorphism 
$\xi_{Y',Y,\beta_f^{-1}}^{-1}$
for each object $f:Y \to X$ of $I_X$.
By taking the colimit with respect to $f$, we
obtain a natural isomorphism $\gamma_\alpha$
from $\iota_0$ to $\iota_0\circ\alpha$.
We thus obtain an element $(\alpha,\gamma_\alpha) \in M_\Cip$.
One can check easily that the element $(\alpha,\gamma_\alpha)$
satisfies the desired properties.
\end{proof}

\subsection{Proof of Theorem~\ref{thm:Galois_main}: 
the functor $\omega_\Cip$ is full}
\label{sec:full}
Let us prove that the functor $\omega_\Cip$ is full.
Let $F_1$ and $F_2$ be sheaves on $(\cC,J)$, and
let $t:\omega_\Cip(F_1) \to \omega_\Cip(F_2)$ be
a morphism of left $M_\Cip$-sets.
For each object $X$ of $\cC$, let us choose an edge object
$E_X$ of $\cC_0$ and an isomorphism 
$\beta_X : \iota_0(E_X) \xto{\cong} X$ in $\cC$.
By Corollary \ref{cor:KX-invariant} we have
isomorphisms $F_1(\iota_0(E_X)) 
\cong \omega_\Cip(F_1)^{\bK_X}$
and $F_2(\iota_0(E_X)) \cong \omega_\Cip(F_2)^{\bK_X}$
We define the map $t_X: F_1(X) \to F_2(X)$ to be the
unique map that makes the diagram
$$
\begin{CD}
F_1(X) @>{\beta_X^*}>{\cong}> F_1(\iota_0(E_X))
@>{\cong}>> \omega(F_1)^{\bK_X} \\
@V{t_X}VV @. @VV{t}V \\
F_2(X) @>{\beta_X^*}>{\cong}> F_2(\iota_0(E_X))
@>{\cong}>> \omega(F_2)^{\bK_X}
\end{CD}
$$
commutative.
Let $f:X \to Y$ be a morphism in $\cC$.
It follows from (4) of Lemma \ref{lem:C0_core} that there
exists a unique morphism $f':E_X \to Y'$ in $\cC_0$
and an isomorphism $\beta' : \iota_0(Y') \cong Y$ in $\cC$ 
such that the diagram
$$
\begin{CD}
\iota_0(E_X) @>{\beta_X}>{\cong}> X \\
@V{\iota_0(f')}VV @VV{f}V \\
\iota_0(Y') @>{\beta'}>{\cong}> Y
\end{CD}
$$
is commutative. We set 
$\beta = {\beta'}^{-1}\circ \beta_Y: \iota_0(E_Y)
\xto{\cong} \iota_0(Y')$.
It follows from Lemma \ref{lem:sufficiently many} that
there exists an element $(\alpha,\gamma_\alpha)$ of
$M_\Cip$ satisfying $\alpha(E_Y) = Y'$ and
$\gamma_\alpha(E_Y)=\beta$.
It then follows from the definition of the action of
$(\alpha,\gamma_\alpha)$ on $\omega_\Cip(F_i)$ that
the diagram
$$
\begin{CD}
F_i(Y) @>{\beta_Y^*}>> F_i(\iota_0(E_Y)) @>{\cong}>> 
\omega_\Cip(F_i)^{\bK_{E_Y}}
@>{\subset}>> \omega_\Cip(F_i) \\
@V{f^*}VV @VV{\iota_0(f')^* \circ \beta^*}V @.
@VV{(\alpha,\gamma_\alpha)\cdot -}V \\
F_i(X) @>{\beta_X^*}>> F_i(\iota_0(E_X)) @>{\cong}>>
\omega_\Cip(F_i)^{\bK_{E_X}}
@>{\subset}>> \omega_\Cip(F_i)
\end{CD}
$$
is commutative for $i=1,2$. Hence, the diagram
$$
\begin{CD}
F_1(Y) @>{t_Y}>> F_2(Y) \\
@V{f^*}VV @VV{f^*}V \\
F_1(X) @>{t_X}>> F_2(X)
\end{CD}
$$
is commutative.
Therefore, the collection of maps
$(t_X : F_1(X) \to F_2(X))_{X \in \Obj \cC}$
gives a morphism $t': F_1 \to F_2$ of sheaves on
$(\cC,J)$ such that the map $\omega_\Cip (F_1)
\to \omega_\Cip (F_2)$ induced by $t'$ is equal to $t$.
This proves that the functor $\omega_\Cip$
is full.

\section{Proof of Theorem~\ref{thm:Galois_main}: 
the fiber functor $\omega_\Cip$ is essentially surjective}
\label{sec:ess_surj}
Until the end of 
Section~\ref{sec:enough points},
we assume that, for any object $X$ of $\cC$,
the category $\cC(\cT(J))_{/X}$ satisfies 
at least one of the two conditions 
in Section \ref{sec:cardinality}.

\subsection{Lemmas on edge objects}
\begin{lem} \label{lem:7-9-1}
Let $(\alpha,\gamma_\alpha)$ be an element of $M_\Cip$.
Then for any morphism $f:Y \to X$ in $\cC_0$ of type $J$,
the morphism $\alpha(f)$ is of type $J$.
\end{lem}

\begin{proof}
In the commutative diagram
$$
\begin{CD}
\iota_0(Y) @>{\iota_0(f)}>> \iota_0(X) \\
@V_{\gamma_\alpha(Y)}VV @VV{\gamma_\alpha(X)}V \\
\iota_0(\alpha(Y)) @>>> \iota_0(\alpha(X)),
\end{CD}
$$
the upper horizontal arrow belongs to $\cT(J)$,
and the vertical arrows are isomorphisms.
Hence, the lower horizontal arrow belongs to $\cT(J)$,
which proves the claim.
\end{proof}

\begin{lem} \label{lem:7-9-2}
Let $(\alpha,\gamma_\alpha)$ be an element of $M_\Cip$.
Let $X$ be an edge object of $\cC_0$ and set $X' = \alpha(X)$.
Then, for an object $Y'$ of $\cC_0$, the following two conditions
are equivalent.
\begin{enumerate}
\item There exists an object $Y$ of $\cC_0$ satisfying $Y'=\alpha(Y)$.
\item There exists an object $Z'$ of $\cC_0$ and a diagram
$X' \xleftarrow{f'} Z' \to Y'$ such that $f'$ is of type $J$.
\end{enumerate}
\end{lem}

\begin{proof}
First we prove that condition (1) implies condition (2).
Suppose that condition (1) is satisfied.
Let us choose an object $Y$ of $\cC_0$ satisfying $Y'=\alpha(Y)$.
It follows from Lemma \ref{lem:yama2} that
there exists an object $Z$ of $\cC_0$ and a diagram
$X \xleftarrow{f} Z \to Y$ in $\cC_0$ such that $f$ is of type $J$.
We set $Z' = \alpha(Z)$ and $f'=\alpha(f)$.
By applying $\alpha$ to the diagram above, we obtain
a diagram $X' \xleftarrow{f'} Z' \to Y'$.
It follows from Lemma \ref{lem:7-9-1} that
the morphism $f'$ is of type $J$.
Hence, condition (2) is satisfied.

Next, we prove that condition (2) implies condition (1).
Suppose that condition (2) is satisfied.
Let us choose an object $Z'$ of $\cC_0$ and a diagram
$X' \xleftarrow{f'} Z' \xto{g'} Y'$ such that $f'$ is of type $J$.
By replacing $f'$ with its composite with a suitable morphism
$Z'' \to Z'$ in $\cC_0$, we may assume that $\iota_0(f')$ is
a Galois covering in $\cC$.
Let us consider the isomorphism
$\iota_0(X) \xto{\gamma_\alpha(X)} \iota_0(X')$ in $\cC$
and the morphism $f': Z' \to X'$ in $\cC_0$.
It follows from Lemma \ref{lem:isom_over} that
there exists a morphism $f:Z \to X$ in $\cC_0$
and an isomorphism $\gamma' : \iota_0 (Z) \cong \iota_0(Z')$
that make the diagram
$$
\begin{CD}
\iota_0(X) @<{\iota_0(f)}<< \iota_0(Z) \\
@V{\gamma_\alpha(X)}VV @VV{\gamma'}V \\
\iota_0(X') @<{\iota_0(f')}<< \iota_0(Z')
\end{CD}
$$
commutative. As $\iota_0(f') \circ \gamma'
= \gamma_\alpha(X) \circ \iota_0(f)
= \iota_0(\alpha(f)) \circ \gamma_\alpha(Z)$, we have
a commutative diagram
$$
\begin{CD}
\iota_0(\alpha(Z)) @>{\gamma' \circ \gamma_\alpha(Z)^{-1}}>> 
\iota_0(Z') \\
@V{\iota_0(\alpha(f))}VV @VV{\iota_0(f')}V \\
\iota_0(X') @= \iota_0(X')
\end{CD}
$$
in $\cC$. As we have assumed that $\iota_0(f')$ is a Galois
covering, $\iota_0(\alpha(f))$ is a Galois covering.
Hence, it follows from Corollary \ref{cor:Galois_over} that
$Z' = \alpha(Z)$ and $f' =\alpha(f)$.
Let us consider the isomorphism $\gamma_\alpha(Y): \iota_0(Y)
\xto{\cong} \iota_0(Y')$ and the morphism $g':Y' \to Z'$ in $\cC_0$.
It follows from Lemma \ref{lem:isom_under} that
there exist a morphism $g:Z \to Y$ in $\cC_0$
and an isomorphism $\gamma'' : \iota_0 (Y) \cong \iota_0(Y')$
that make the diagram
$$
\begin{CD}
\iota_0(Z) @>{\iota_0(g)}>> \iota_0(Y) \\
@V{\gamma_\alpha(Z)}VV @VV{\gamma''}V \\
\iota_0(Z') @>{\iota_0(g')}>> \iota_0(Y')
\end{CD}
$$
commutative. As $\gamma''^{-1} \circ \iota_0(g') 
= \iota_0(g) \circ \gamma_\alpha(Z)^{-1} 
= \gamma_\alpha(Y)^{-1} \circ \iota_0(\alpha(g))$, we have
a commutative diagram
$$
\begin{CD}
\iota_0(Z') @= \iota_0(Z') \\
@V{\iota_0(\alpha(g))}VV @VV{\iota_0(g')}V \\
\iota_0(\alpha(Y)) @>{\gamma'' \circ \gamma_\alpha(Y)}>> 
\iota_0(Y')
\end{CD}
$$
in $\cC$. Because the functor 
$\iota_{0,Z'/} :\cC_{0,Z'/} \to \cC_{\iota_0(Z')/}$ is
an equivalence of categories and $\cC_0$ is skeletal,
we have $Y'=\alpha(Y)$ and $g'=\alpha(g)$.
This in particular shows that condition (1) is satisfied.
\end{proof}

\begin{lem} \label{lem:7-9-3}
Let $X$ be an edge object of $\cC_0$ and let
$(\alpha,\gamma_\alpha)$ be an element of $\bK_X$.
Then for any morphism $f:X \to Y$ in $\cC_0$,
we have $\alpha(Y) = Y$ and $\gamma_\alpha(Y) = \id_{\iota_0(Y)}$.
\end{lem}

\begin{proof}
We have a commutative diagram
$$
\begin{CD}
\iota_0(X) @= \iota_0(X) \\
@V{\iota_0(f)}VV @VV{\iota_0(\alpha(f))}V \\
\iota_0(Y) @>{\gamma_\alpha(Y)}>> \iota_0(\alpha(Y))
\end{CD}
$$
in $\cC$. As the functor 
$\iota_{0,X/} :\cC_{0,X/} \to \cC_{\iota_0(X)/}$ is
an equivalence of categories and $\cC_0$ is skeletal,
we have $\alpha(Y)=Y$ and $f =\alpha(f)$.
Hence, $\iota_0(f) = \gamma_\alpha(Y) \circ \iota_0(f)$.
As $\iota_0(f)$ is an epimorphism, we have
$\gamma_\alpha(Y) = \id_{\iota_0(Y)}$.
This proves the claim.
\end{proof}

\begin{lem} \label{lem:7-9-4}
Let $(\alpha,\gamma_{\alpha})$ and 
$(\alpha', \gamma_{\alpha'})$
be two elements of
$M_\Cip$. Suppose that for any object $X$ of $\cC_0$, there exists
an object $X'$ of $\cC_0$ satisfying 
$\alpha'(X) = \alpha(X')$.
Then, there exists a unique element 
$(\alpha'',\gamma_{\alpha''})$ of $M_\Cip$
satisfying $(\alpha',\gamma_{\alpha'})
= (\alpha,\gamma_{\alpha})\circ (\alpha'',\gamma_{\alpha''})$
\end{lem}

\begin{proof}
Let $X$ be an arbitrary object of $\cC_0$.
By assumption, there exists an object of $\cC_0$ whose image under 
the functor $\alpha$ is equal to $\alpha'(X)$.
We denote this object by $\alpha''(X)$.
It follows from Lemma \ref{lem:MC2} that the object
$\alpha''(X)$ is uniquely determined by this property
and that, if $X' \to X$ is a morphism in $\cC_0$, 
there exists a morphism from $\alpha''(X')$ to 
$\alpha''(X)$ in $\cC_0$.
Hence, by associating $\alpha''(X)$ to each object $X$
of $\cC_0$, we obtain a functor $\alpha'':\cC_0 \to \cC_0$.
By construction, we have $\alpha' = \alpha \circ \alpha''$.
For any object $X$ of $\cC_0$, we have the diagram
$$
\begin{CD}
\iota_0(X) @>{\gamma_{\alpha'}(X)}>{\cong}> 
\iota_0(\alpha'(X)) \\
@. @| \\
\iota_0(\alpha''(X))
@>{\gamma_{\alpha}(\alpha''(X))}>{\cong}> 
\iota_0(\alpha \circ \alpha''(X))
\end{CD}
$$
in $\cC$. Hence there exists a unique isomorphism
$\gamma_{\alpha''}(X): \iota_0(X) \xto{\cong}
\iota_0(\alpha''(X))$ such that
the equality $\gamma_{\alpha}(\alpha''(X))
\circ  \gamma_{\alpha''}(X) = 
\gamma_{\alpha'}(X)$ holds. 
As $\gamma_\alpha$ and $\gamma_{\alpha'}$ are
isomorphisms of functors, it follows that
the isomorphism $\gamma_{\alpha''}(X)$ is functorial in
$X$ in the following sense: for any morphism $f:X' \to X$ in $\cC$,
the diagram
$$
\begin{CD}
\iota_0(X') @>{\iota_0(f)}>> \iota_0(X) \\
@V{\gamma_{\alpha''}(X')}VV @VV{\gamma_{\alpha''}(X)}V \\
\iota_0(\alpha''(X')) @>{\iota_0(\alpha''(f))}>> \iota_0(\alpha''(X)) 
\end{CD}
$$
is commutative. Hence, the isomorphisms 
$\gamma_{\alpha''}(X)$ for the objects $X$ of $\cC_0$
give an isomorphism $\gamma_{\alpha''}: \iota_0 \xto{\cong}
\iota_0 \circ \alpha''$ of functors and the pair
$(\alpha'',\gamma_{\alpha''})$ is an element of $M_\Cip$
satisfying $(\alpha',\gamma_{\alpha'}) =
(\alpha,\gamma_{\alpha})\circ (\alpha'', \gamma_{\alpha''})$.
The uniqueness of $(\alpha'',\gamma_{\alpha''})$ follows from
the uniqueness of $\alpha''(X)$ for each object $X$ of $\cC_0$.
This completes the proof.
\end{proof}

\begin{lem} \label{lem:7-9-5}
Let $X$ be an edge object of $\cC_0$.
Let $(\alpha,\gamma_{\alpha})$ and
$(\alpha',\gamma_{\alpha'})$ be two elements of
$M_\Cip$. Suppose that $\alpha(X) = \alpha'(X)$
and $\gamma_{\alpha}(X) = \gamma_{\alpha'}(X)$.
Then, there exists an element $(\alpha'',\gamma_{\alpha''})$
of $\bK_X$ satisfying $(\alpha',\gamma_{\alpha'})
=(\alpha,\gamma_{\alpha})\circ (\alpha'',\gamma_{\alpha''})$.
\end{lem}

\begin{proof}
Let $Y$ be an arbitrary object of $\cC_0$.
It follows from Lemma \ref{lem:yama2} that
there exists a diagram
$$
X \xleftarrow{f} Z \to Y
$$
in $\cC_0$ such that $f$ is of type $J$.
By applying $\alpha'$, we have a diagram
$$
\alpha'(X) \xleftarrow{\alpha'(f)} \alpha'(Z)
\to \alpha'(Y).
$$
It follows from Lemma \ref{lem:7-9-1} that
$\alpha'(f)$ is of type $J$.
Hence, it follows from Lemma \ref{lem:7-9-2} that
there exists an object of $\cC_0$ whose image under 
the functor $\alpha$ is equal to $\alpha'(Y)$.
Thus, it follows from Lemma \ref{lem:7-9-4} that
there exists a unique element 
$(\alpha'',\gamma_{\alpha''})$
of $M_\Cip$ satisfying $(\alpha',\gamma_{\alpha'})
=(\alpha,\gamma_{\alpha})\circ (\alpha'',\gamma_{\alpha''})$.

To prove the claim, it remains to be shown that
$(\alpha'', \gamma_{\alpha''})$ is an element of $\bK_X$.
As $\alpha' = \alpha \circ \alpha''$
and $\alpha(X) = \alpha'(X)$, we have 
$\alpha(\alpha''(X)) = \alpha(X)$.
Hence, it follows from Lemma \ref{lem:MC2} that we have
$\alpha''(X) = X$.
By applying the equality
$\gamma_{\alpha}(\alpha''(Y))
\circ  \gamma_{\alpha''}(Y) = 
\gamma_{\alpha'}(Y)$ to $Y=X$, we have
$\gamma_{\alpha}(X)
\circ  \gamma_{\alpha''}(X) = 
\gamma_{\alpha'}(X) = \gamma_{\alpha}(X)$.
Hence we have $\gamma_{\alpha''}(X) = \id_{\iota_0(X)}$,
which proves that $(\alpha'', \gamma_{\alpha''})$ 
is an element of $\bK_{X}$.
This completes the proof.
\end{proof}

\begin{rmk}
In the proof of the previous lemma, 
a stronger statement is proved:
There exists a unique element 
$(\alpha'',\gamma_{\alpha''})$
of $M_\Cip$ satisfying $(\alpha',\gamma_{\alpha'})
=(\alpha,\gamma_{\alpha})\circ (\alpha'',\gamma_{\alpha''})$.
Moreover, the element $(\alpha'',\gamma_{\alpha''})$ belongs to
$\bK_X$. 
As we do not need this statement,
the details are suppressed.
\end{rmk}

\begin{lem} \label{lem:7-9-6}
Let $(\alpha,\gamma_\alpha)$ be an element of $M_\Cip$
and 
$X$ be an edge object of $\cC_0$.
Suppose that $\alpha(X)$ is an edge object of $\cC_0$.
Then $(\alpha, \gamma_\alpha)$ is an invertible element
of $M_\Cip$.
\end{lem}

\begin{proof}
By Lemma \ref{lem:7-9-4}, we are reduced to proving that
the functor $\alpha : \cC_0 \to \cC_0$ is an isomorphism of categories.
It follows from Lemma \ref{lem:MC2} that $\alpha$ is
fully faithful. As the category $\cC_0$ is skeletal,
it suffices to show that $\alpha$ is essentially surjective.
Let $Y$ be an arbitrary object of $\cC_0$.
As $\alpha(X)$ is an edge object of $\cC_0$, it follows from
Lemma \ref{lem:yama2} that there exists a diagram
$$
\alpha(X) \xleftarrow{f} Z \to Y 
$$
in $\cC_0$ such that $f$ is of type $J$.
Hence, it follows from Lemma \ref{lem:7-9-2} that
there exists an object $Y'$ of $\cC_0$ satisfying 
$Y = \alpha(Y')$. This shows that $\alpha$ is essentially
surjective, which proves the claim.
\end{proof}

\begin{lem} \label{lem:7-9-7}
Let $Y$ be an edge object of $\cC_0$ and let
$f:Y \to X$ be a morphism in $\cC_0$ of type $J$.
Then $X$ is an edge object of $\cC_0$.
\end{lem}

\begin{proof}
Let $g:X' \to X$ be an arbitrary morphism in $\cC_0$.
As $\cC_0$ is semi-cofiltered (see proof of 
Lemma \ref{lem:core_cofiltered}), 
there exist an object $Y'$ of $\cC_0$ and morphisms
$f':Y' \to X'$ and $g':Y' \to Y$
satisfying $f \circ g' = g \circ f'$.
As $Y$ is an edge object, $g'$ is of type $J$.
Hence, $f \circ g'$ is of type $J$.
It follows from Proposition \ref{prop:cTJ} that 
$g$ is of type $J$.
This proves that $X$ is an edge object of $\cC_0$.
\end{proof}

\subsection{Proof of Theorem~\ref{thm:Galois_main}: the fiber functor is essentially surjective}
\subsubsection{Proof: Step 1}
\label{sec:Step 1}
Let $T$ be an arbitrary smooth left $M_\Cip$-set.
For each object $X$ of $\cC$, let us choose
an edge object $E_X$ of $\cC_0$ and an isomorphism
$\beta_X: \iota_0(E_X) \xto{\cong} X$ in $\cC$.
We set $F_T(X) = T^{\bK_{E_X}}$.
One can check, by modifying the argument 
of the paragraph below, that $F_T(X)$ is independent of the
choice of the pair $(E_X, \beta_X)$ up to canonical isomorphisms.
However we will not use this independence 
in the proof of Theorem \ref{thm:Galois_main} given below.

Let $f:X \to Y$ be a morphism in $\cC$.
In this paragraph, we define a map $f^* : F_T(Y) \to F_T(X)$.
We use the following notation 
introduced in Section \ref{sec:full}.
It follows from (4) of Lemma \ref{lem:C0_core} that there
exist a unique morphism $f':E_X \to Y'$ in $\cC_0$
and an isomorphism $\beta' : \iota_0(Y') \cong Y$ in $\cC$ 
such that the diagram
$$
\begin{CD}
\iota_0(E_X) @>{\beta_X}>{\cong}> X \\
@V{\iota_0(f')}VV @VV{f}V \\
\iota_0(Y') @>{\beta'}>{\cong}> Y
\end{CD}
$$
is commutative. We set 
$\beta = {\beta'}^{-1}\circ \beta_Y: \iota_0(E_Y)
\xto{\cong} \iota_0(Y')$.
It follows from Lemma \ref{lem:sufficiently many} that
there exists an element $(\alpha,\gamma_\alpha)$ of
$M_\Cip$ satisfying $\alpha(E_Y) = Y'$ and
$\gamma_\alpha(E_Y)=\beta$.
We now show that, for any element 
$(\alpha',\gamma_{\alpha'})$ of $\bK_{E_X}$,
there exists a unique element 
$(\alpha'',\gamma_{\alpha''})$ of $\bK_{E_Y}$ 
satisfying
$(\alpha',\gamma_{\alpha'})
\circ (\alpha,\gamma_{\alpha}) =
(\alpha,\gamma_{\alpha})\circ (\alpha'', \gamma_{\alpha''})$.
Let $(\alpha',\gamma_{\alpha'})$ be an element
of $\bK_{E_X}$. It follows from Lemma \ref{lem:7-9-3} that
we have $\alpha'(Y') = Y'$ and $\gamma_{\alpha'}(Y')
= \id_{\iota_0(Y')}$.
Hence, it follows from Lemma \ref{lem:7-9-5} that
there exists an element 
$(\alpha'',\gamma_{\alpha''})$ of $\bK_{E_Y}$ 
satisfying
$(\alpha',\gamma_{\alpha'})
\circ (\alpha,\gamma_{\alpha}) =
(\alpha,\gamma_{\alpha})\circ 
(\alpha'', \gamma_{\alpha''})$.
Hence, the map $T \to T$
given by the multiplication by $(\alpha,\gamma_\alpha)$
induces a map $F_T(Y) = T^{\bK_{E_Y}}
\to T^{\bK_{E_X}} = F_T(X)$ which we
denote by $f^*$.

\subsubsection{Proof: Step 2}
\label{sec:Step 2}
We now show that the map $f^*$ is independent 
of the choice of the element $(\alpha,\gamma_{\alpha})$.
Suppose that $(\alpha_1,\gamma_{\alpha_1})$ is another
choice of an element $M_\Cip$ satisfying $\alpha_1(E_Y) = Y'$ and
$\gamma_{\alpha_1}(E_Y)=\beta$.
Then, it follows from Lemma \ref{lem:7-9-5} that
there exists an element 
$(\alpha'',\gamma_{\alpha''})$ of $\bK_{E_Y}$ 
satisfying
$(\alpha_1,\gamma_{\alpha_1}) =
(\alpha,\gamma_{\alpha})\circ (\alpha'', \gamma_{\alpha''})$.
This implies that the map $f^*$ is independent of
the choice of the element $(\alpha,\gamma_{\alpha})$.

\subsubsection{Proof: Step 3}

We now show that, for any morphism $g:Y \to Z$ in $\cC$, we have
$(g\circ f)^* = f^* \circ g^*$.
It follows from (4) of Lemma \ref{lem:C0_core} that there
exist a unique morphism $g':E_Y \to Z'$ in $\cC_0$
and an isomorphism $\beta'_1 : \iota_0(Z') \cong Z$ in $\cC$ 
such that the diagram
$$
\begin{CD}
\iota_0(E_Y) @>{\beta_Y}>{\cong}> Y \\
@V{\iota_0(g')}VV @VV{g}V \\
\iota_0(Z') @>{\beta'_1}>{\cong}> Z
\end{CD}
$$
is commutative. We set 
$\beta_1 = {\beta'_1}^{-1}\circ \beta_Z: \iota_0(E_Z)
\xto{\cong} \iota_0(Z')$.
It follows from Lemma \ref{lem:sufficiently many} that
there exists an element $(\alpha_1,\gamma_{\alpha_1})$ of
$M_\Cip$ satisfying $\alpha_1(E_Z) = Z'$ and
$\gamma_\alpha(E_Y)=\beta_1$.
We set $Z'' =\alpha(Z)$ and $h= \alpha(g') \circ f'$.
Then, $(h,\beta'_2)$ is the unique pair of 
a morphism $h: E_X \to Z''$ and an isomorphism
$\beta'_2 : \iota_0(Z'') \cong Z$ such that the diagram
$$
\begin{CD}
\iota_0(E_X) @>{\beta_X}>{\cong}> X \\
@V{\iota_0(h)}VV @VV{g\circ f}V \\
\iota_0(Z'') @>{\beta'_1}>{\cong}> Z
\end{CD}
$$
is commutative. We set 
$\beta_2 = {\beta'_2}^{-1}\circ \beta_Z: \iota_0(E_Z)
\xto{\cong} \iota_0(Z'')$.
Then, the element $(\alpha_2,\gamma_{\alpha_2})
= (\alpha,\gamma_{\alpha})\circ 
(\alpha_1,\gamma_{\alpha_1})$
of $M_\Cip$ satisfies $\alpha_2(E_Z) = Z''$ and
$\gamma_\alpha(E_Z)=\beta_2$.
It follows from the definition that
the map $(g \circ f)^*$ is given by multiplication 
by the element $(\alpha_2, \gamma_{\alpha_2})$.
As the maps $f^*$, $g^*$ are
given by multiplication by the elements
$(\alpha, \gamma_{\alpha})$ and
$(\alpha_1, \gamma_{\alpha_1})$,
respectively, we have the desired equality
$f^* \circ g^* = (g \circ f)^*$.
Thus, we obtain a presheaf $F_T$ on $\cC$.
As the action of $M_\Cip$ on $T$ is smooth,
we have $\omega_\Cip(F_T) = T$.

\subsubsection{Proof: Step 4}

Suppose that $f:X \to Y$ is a Galois covering in $\cT$.
Let $f':E_X \to Y'$, 
$\beta':\iota_0(Y') \xto{\cong} Y$, and 
$\beta: \iota_0(E_Y) \xto{\cong} \iota_0(Y')$ 
be as in Section~\ref{sec:Step 1} above. 
It then follows from Lemma \ref{lem:7-9-7} that
$Y'$ is an edge object of $\cC_0$.
Hence, it follows from Lemma \ref{lem:7-9-6} that
the element $(\alpha,\gamma_\alpha) \in M_\Cip$ is invertible.
As $f$ is a Galois covering in $\cT$,
the morphism $\iota_0(f')$ is a Galois
covering in $\cT$.
It follows from Lemma \ref{lem:7-9-3} that
$\bK_{E_X}$ is a subgroup of $\bK_{Y'}$.
Let $i : \bK_{E_X} \to \bK_{Y'}$ denote the inclusion.
Let $\rho : \bK_{Y'} \to \Gal(f')$
denote the composite of $\phi_{Y'}: \bK_{Y'}
\to H_{Y'}$ with the projection map 
$H_{Y'} \to \Gal(f')$.
Let us consider the sequence
\begin{equation} \label{seq:7-9}
1 \to \bK_{E_X} \xto{i} \bK_{Y'} 
\xto{\rho} \Gal(f') \to 1.
\end{equation}
It is thus obvious that the map $i$ is injective.
It follows from the definition of $\bK_{E_X}$ that the
kernel of $\rho$ is equal to the image of $i$.
We have shown in Lemma~\ref{lem:Gal gp surjective}
that $\rho$ is surjective.
when $\cC_{/Y}$ satisfies
one of the conditions in Section \ref{sec:existence of a pregrid}.
Hence, sequence \eqref{seq:7-9} is exact.
As
$(\alpha,\gamma_\alpha) \in M_\Cip$ is invertible,
this implies that the map $f^*: F_T(Y) \to F_T(X)$
induces a bijection $F_T(Y) \to F_T(X)^{\Gal(f)}$.
This shows that $F_T$ is a sheaf on $(\cC,J)$.

Let $X$ be an edge object of $\cC_0$.
Applying Lemma \ref{lem:sufficiently many} to
the isomorphism $\beta_{\iota_0(X)}: \iota_0(E_{\iota_0(X)})
\xto{\cong} \iota_0(X)$, we can choose an element
$(\alpha_X, \gamma_{\alpha_X})$ of $M_\Cip$
satisfying $\alpha_X(E_{\iota_0(X)}) = X$
and $\gamma_{\alpha_X} = \beta_{\iota_0(X)}$.
It follows from Lemma \ref{lem:7-9-6} that 
$(\alpha_X,\gamma_{\alpha_X})$ is an invertible element
of $M_\Cip$.
Hence the action of $(\alpha_X,\gamma_{\alpha_X})$ on $T$ induces
a bijection $F_T(X) = T^{\bK_{E_{\iota_0(X)}}}
\xto{\cong} T^{\bK_X}$, which we denote by $\epsilon_X$.
Let $f: X \to Y$ be a morphism in $\cCe$.
By Lemma \ref{lem:7-9-3} we have the inclusion map
$T^{\bK_Y} \subset T^{\bK_X}$.
We set $(\alpha,\gamma_\alpha)
= (\alpha_X,\gamma_{\alpha_X})^{-1} \circ
(\alpha_Y, \gamma_{\alpha_Y})$.
We set $f' = \alpha_X^{-1}(f)$. It is a morphism in $\cC_0$
whose domain is equal to $E_{\iota_0(X)}$
Let $Y'$ denote the codomain of the morphism $\alpha_X^{-1}(f)$
and set $\beta = \gamma_{\alpha_X}(Y'): \iota_0(Y')
\xto{\cong} \iota_0(Y)$.
We then have $\iota_0(f) \circ \beta_{\iota_0(X)}
= \beta' \circ \iota_0(f')$.
Hence, it follows from the argument in Section \ref{sec:Step 2} that
the map $f^*: F_T(Y) \to F_T(X)$ is given by multiplication
by $(\alpha,\gamma_\alpha)$.
This shows that the diagram
$$
\begin{CD}
F_T(Y) @>{f^*}>> F_T(X) \\
@V{\epsilon_Y}V{\cong}V @V{\cong}V{\epsilon_X}V \\
T^{\bK_Y} @>{\subset}>> T^{\bK_X}
\end{CD}
$$
is commutative.
Hence, the bijections $\epsilon_X$ give
a bijection $\omega_\Cip(F_T) \cong T$.
It is straightforward to check that this bijection is
an isomorphism of $M_\Cip$-sets.
Therefore, the functor $\omega_\Cip$ is essentially
surjective. This completes the proof of
Theorem \ref{thm:Galois_main}.

\section{On the uniqueness of grids}
\label{sec:grid uniqueness}
This section was added later at the request of the referee.
The referee asked to what extent the grids are unique.
We show below (Proposition~\ref{prop:grid uniqueness})
that two grids are isomorphic (in a sense made precise
below) under a certain finiteness condition.
When we remove the finiteness assumption,
there is a $Y$-site with more than one grids that are not 
isomorphic.   The key ingredient in the construction is the non-vanishing of the first derived limit.   We present the example in 
Section~\ref{sec:nonunique grid}.

We give a description of the abundance of  
grids in terms of derived limits in Section~\ref{sec:uniqueness and R1}.

\subsection{}
The aim of this paragraph is to prove the following
proposition, which gives a uniqueness of a grid up to certain
isomorphisms of a $Y$-site under the same cardinality condition
as in Proposition \ref{cor:grid existence}.

\begin{prop}
\label{prop:grid uniqueness}
Let $(\cC,J)$ be a $Y$-site.
Let $(\cC_{0,i},\iota_{0,i})$, $i=1,2$,
be two grids of $(\cC,J)$.
Suppose that there exists an object
$X_0$ of $\cC(\cT(J))$ such that the overcategory
$\cC(\cT(J))_{/X_0}$ satisfies at least one of the
two cardinality conditions in Section \ref{sec:cardinality}.
Then there exists a pair $(\alpha,\gamma_\alpha)$
of an isomorphism
$\alpha: \cC_{0,1} \xto{\cong} \cC_{0,2}$
of categories and a natural isomorphism
$\gamma_\alpha: \iota_{0,1} \xto{\cong} \iota_{0,2} \circ \alpha$.
\end{prop}

\begin{proof}
For $i=1,2$, let us choose an edge object 
$X_{0,i}$ of $\cC_{0,i}$ and an isomorphism
$\beta_i:\iota_{0,i}(X_{0,i}) \xto{\cong} X_0$.
Let $\cC'$ denote the
full subcategory of $\cC(\cT(J))_{/X_0}$ whose objects
are the morphisms in $\cC(\cT(J))_{/X_0}$
which are Galois coverings in $\cC$.
For $i=1,2$, let $\cC'_{0,i}$ denote the full subcategory
of the overcategory $(\cC_{0,i})_{/X_{0,i}}$
whose objects are the morphisms $f:X \to X_{0,i}$
such that the composite $\beta_i \circ \iota_{0,i}(f):
\iota_{0,i}(X) \to X_0$ is an object of $\cC'$
(thus $\cC'_{0,i} = I_{X_{0,i}}$ in the notation in Section \ref{sec:IX}).
Since $\cT(J)$ has enough Galois coverings,
Corollary \ref{cor:core_cofiltered} implies that
$\cC_{0,i}$ is cofiltered.
It follows from Lemma \ref{lem:lower} that
for any object $f:X \to X_{0,i}$ of $\cC'_{0,i}$,
the object $X$ of $\cC_{0,i}$ is an edge object.
Let $f: X \to X_{0,1}$ be an object of $\cC'_{0,1}$.
The condition (3) in the definition of grid
(Definition \ref{defn:grids}) implies that
there exists a pair $(f',\gamma')$ 
of an object $f': X' \to X_{0,2}$ of
$\cC'_{0,2}$ and an isomorphism
$\gamma' : \iota_{0,1}(X) \xto{\cong} \iota_{0,2}(X')$
satisfying 
$\beta_1 \circ \iota_{0,1}(f) = 
\beta_2 \circ \iota_{0,2}(f') \circ \gamma'$.
Let $S_f$ denote the set of such pairs $(f',\gamma')$
(this is a set since $\cC_{0,2}$ is $\frU$-small).
Corollary \ref{cor:Galois_over} implies that
the first component $f'$ of an element $(f',\gamma') \in S_f$
is uniquely determined from $f$ and is independent of $\gamma'$.
Hence $S_f$ forms a left $\Gal(\iota_{0,1}(f))$-torsor.
Let us write $f' = \alpha'(f)$.
Let $g: Y \to X_{0,1}$ be another object of
$\cC'_{0,1}$ and suppose that there exists a morphism
$h$ from $g$ to $f$ in $\cC'_{0,1}$.
It follows from Lemma \ref{lem:Galois_over} that
there exists a morphism $\alpha'(h)$ from $\alpha'(g)$ to $\alpha'(f)$
in $\cC'_{0,2}$.
For two composable morphisms $h,h'$ in 
$\cC'_{0,1}$, we have $\alpha'(h \circ h')
= \alpha'(h) \circ \alpha'(h')$.
Hence by sending $h$ to $\alpha'(h)$ 
we obtain a functor $\alpha' : \cC'_{0,1}
\to \cC'_{0,2}$. By construction, this functor has
the following characterization:
for any object $f$ of $\cC_{0,1}$, there exists
an isomorphism from $\beta_1 \circ \iota_{0,1}(f)$
to $\beta_2 \circ \iota_{0,2}(\alpha'(f))$ in $\cC'$.
By exchanging the roles of $(\cC_{0,1},\iota_{0,1})$ 
and $(\cC_{0,2},\iota_{0,2})$, we obtain a functor
$\alpha'' : \cC'_{0,2} \to \cC'_{0,1}$.
The characterization above implies that
the composites $\alpha' \circ \alpha''$
and $\alpha'' \circ \alpha'$ are identity functors.
In particular the functor $\alpha'$ is an isomorphism
of categories.

Let $f;X \to X_{0,1}$ and $g: Y \to X_{0,1}$ be objects of
$\cC'_{0,1}$ and suppose that there exists a morphism
$h$ from $g$ to $f$ in $\cC'_{0,1}$.
Let $(\alpha'(g),\delta') \in S_g$.
Then it follows from Lemma \ref{lem:isom_under} that
there exists a unique element $(\alpha'(f),\gamma') \in S_f$
satisfying $\iota_{0,2}(\alpha'(h)) \circ \delta'
= \gamma' \circ \iota_{0,1}(h)$.
By sending $(\alpha'(g),\delta')$ to 
$(\alpha'(f),\gamma')$, we obtain a map $h_* : S_g \to S_f$.
It is straightforward to check that, 
for two composable morphisms $h,h'$ in $\cC'_{0,1}$, 
we have $(h\circ h')_* = h_* \circ h'_*$. Hence
we have a filtered projective system
$(S_f)$ of nonempty sets indexed by the objects $f$ of $\cC'_{0,1}$.
The cardinality condition for $\cC(\cT(J))_{/X_0}$ implies that
we have either that each $S_f$ is a finite set or that
$(S_f)$ has a cofinal subsystem indexed by an at-most-countable
filtered set.
This shows that there exists an element 
$((\alpha'(f), \gamma'_f))_{f \in \Obj(\cC'_{0,1})}$ of the
projective limit $\varprojlim_f S_f$.
Let us fix such an element $((\alpha'(f), \gamma'_f))_f$.

Let $f: X \to X_{0,1}$ be an object of $\cC'_{0,1}$.
Since $(\cC_{0,1},\iota_{0,1})$ and
$(\cC_{0,2},\iota_{0,2})$ are grids,
$\iota_{0,1}$ and $\iota_{0,2}$ induce equivalences
$\psi_{f,1} : (\cC_{0,1})_{X/} \to \cC_{\iota_{0,1}(\alpha'(X))/}$
and
$\psi_{f,2} : (\cC_{0,2})_{\alpha'(X)/} \to \cC_{\iota_{0,2}(X)/}$
of categories.
The composition with $(\gamma'_f)^{-1}$ gives an equivalence
$\phi_f : \cC_{\iota_{0,1}(X)/} \to
\cC_{\iota_{0,2}(\alpha'(X))/}$ of categories. By taking composite of
$\psi_{f,1}$, $\phi_f$, and the quasi-inverse of $\psi_{f,2}$
(that is unique since $(\cC_{0,2})_{\alpha'(X)/}$ is a poset),
we obtain an equivalence $\theta_{\alpha',f} : (\cC_{0,1})_{X/}
\to (\cC_{0,2})_{\alpha'(X)/}$. 
By construction, the functor $\theta_{\alpha',f}$ has
the following characterization:
for any object $h:X \to Z$ of $(\cC_{0,1})_{X/}$, 
there exists an isomorphism 
$\delta: \iota_{0,1}(Z) \xto{\cong} \iota_{0,2}(Z')$ in
$\cC$, where $Z'$ is the codomain of $\theta_{\alpha',f}(h)$
regarded as a morphism in $\cC_{0,2}$, such that
$\delta \circ \iota_{0,1}(h)
= \iota_{0,2}(\theta_{\alpha',f}(h)) \circ \gamma'_f$.
Since $\cC$ is an $E$-category, such an isomorphism $\delta$
is unique. Hence by associating $\delta$ to each object of
$(\cC_{0,1})_{X/}$, we obtain a natural isomorphism
$\gamma'_{\alpha',f}: \jmath_{\iota_{0,1}(X)} 
\circ \psi_{f,1} \xto{\cong} 
\jmath_{\iota_{0,2}(\alpha'(X))} \circ \psi_{f,2}
\circ \theta_{\alpha',f}$, where for an object 
$Y$ of $\cC$ we let $\jmath_Y : \cC_{Y/} \to \cC$
denote the inclusion functor that sends
$Y \to Z$ to $Z$.
It follows from the charactrization
of the functor $\theta_{\alpha',f}$ above that
$\gamma'_{\alpha',f}$ is the unique natural isomorphism
from $\jmath_{\iota_{0,1}(X)} 
\circ \psi_{f,1}$ to
$\jmath_{\iota_{0,2}(\alpha'(X))} \circ \psi_{f,2}
\circ \theta_{\alpha',f}$
satisfying $\gamma'_{\alpha',f}(\id_X) = \gamma'_f$.

By exchanging the roles of $(\cC_{0,1},\iota_{0,1})$ 
and $(\cC_{0,2},\iota_{0,2})$, we obtain a functor
$\theta_{\alpha'',\alpha'(f)} : (\cC_{0,2})_{\alpha'(X)/} \to 
(\cC_{0,1})_{X/}$.
The characterization of the functor $\theta_{\alpha',f}$ above
implies that the composites 
$\theta_{\alpha',f} \circ \theta_{\alpha'',\alpha'(f)}$
and $\theta_{\alpha'',\alpha'(f)} \circ \theta_{\alpha',f}$
are the identity functors.
In particular $\theta_{\alpha',f}$ is an isomorphism
of categories.
Let $g: Y \to X_{0,1}$ be another object of $\cC'_{0,1}$
and suppose that there exists a morphism 
$h$ from $g$ to $f$ in $\cC'_{0,1}$.
It is then easy to see that the diagram
$$
\begin{CD}
(\cC_{0,1})_{X/}
@>{\theta_{\alpha',f}}>>
(\cC_{0,2})_{\alpha'(X)/} \\
@V{(1)}VV @VV{(2)}V \\
(\cC_{0,1})_{Y/} 
@>{\theta_{\alpha',g}}>> 
(\cC_{0,2})_{\alpha'(Y)/},
\end{CD}
$$
where the vertical arrows (1), (2) are functors 
given by the composites with $h$ and $\alpha'(h)$,
is commutative in the strict sense, 
not only up to natural isomorphisms.
Hence it follows from Lemma \ref{cor:C0Y}, the
isomorphism $\theta_{\alpha',f'}$ for each $f$
gives an isomorphism $\alpha : \cC_{0,1} \xto{\cong} \cC_{0,2}$
of categories.

Since $\gamma'_g$ is mapped to $\gamma'_f$ under the
map $S_g \to S_f$, it follows that the restriction
of the natural isomorphism $\gamma'_{\alpha',g}$ 
to $(\cC_{0,1})_{X/}$, where we regard
$(\cC_{0,1})_{X/}$ as a subcategory of $(\cC_{0,1})_{Y/}$
via the functor (1) in the diagram above,
is equal to $\gamma'_{\alpha',f}$.
Thus the natural isomorphisms $\gamma'_{\alpha',f}$
give rise to a natural isomorphism
$\gamma_\alpha : \iota_{0,1} \xto{\cong}
\iota_{0,2} \circ \alpha$.
This completes the proof.
\end{proof}

\subsection{On the non-uniqueness of grids}
\label{sec:nonunique grid}

Let $(\cC,J)$ be a $Y$-site and $X_0$ be an object of $(\cC,J)$.
We do not assume that $\cC(\cT(J))_{/X_0}$ satisfies
the cardinality condition in Section \ref{sec:cardinality}.
Let $(\cC_{0,1},\iota_{0,1})$, $(\cC_{0,2},\iota_{0,2})$
be two grids of $\cC$. 
For $i \in \{1,2\}$, choose an edge object
$X_{0,i}$ of $\cC_{0,i}$ and an isomorphism
$\beta_i : \iota_{0,i}(X_{0,i}) \xto{\cong} X_0$ in $\cC$.
Then the argument in the proof of Proposition \ref{prop:grid uniqueness} 
shows that there exists a unique isomorphism
$\alpha : \cC_{0,1} \to \cC_{0,2}$ of categories
satisfying the following condition:
$\alpha(X_{0,1}) = X_{0,2}$ and for any
morphism $f:X \to X_{0,1}$
in $\cC_{0,1}$ such that $\iota_{0,1}(f)$
is a Galois covering in $\cC$, there exist
an isomorphism $\gamma_X : \iota_{0,1}(X) \xto{\cong} 
\iota_{0,2}(\alpha(X))$ that makes the diagram
$$
\begin{CD}
\iota_{0,1}(X)
@>{\gamma_X}>> \iota_{0,2}(\alpha(X)) \\
@V{\iota_{0,1}(f)}VV
@VV{\iota_{0,2}(\alpha(f))}V \\
\iota_{0,1}(X_{0,1})
@>{\beta_2^{-1} \circ \beta_1}>>
\iota_{0,2}(X_{0,2})
\end{CD}
$$
commutative.

However there is a case where there cannot exist a
natural isomorphism from 
$\iota_{0,1}$ to $\iota_{0,2} \circ \alpha$.
We present such an example in this section.
The construction is based on a pro-group
such that the first derived limit is nontrivial.

\subsubsection{}
We will be using the following pro-group.
Let $I$ be a cofiltered poset with a final object $i_0$
and let $(G_i)_{i \in I}$ be a projective system 
$(G_i)_{i \in I}$ of discrete groups such that $G_{i_0} = \{1\}$,
that the transition homomorphism $f_{i,j} : G_i \to G_j$ is surjective
for any morphism $i \to j$ in $I$,
and that $R^1 \varprojlim_i G_i$ is non-empty.

The existence of such a pro-group is due to Todorcevic 
and is given in \cite[Ch.\ X, p 350, Lemma 4.4]{FS}.

\subsubsection{}
Let $\cC$ denote the following $\frU$-small category.
The set of objects of $\cC$ is equal to the set of objects of $I$.
When we regard $i \in I$ as an object of $\cC$, we write
$X_i$ for $i$. For $i,j \in I$, we set
$$
\Hom_{\cC}(X_i,X_j) = 
\begin{cases}
G_j, & \text{ if there exists a morphism }i \to j,\\
\emptyset, & \text{ otherwise}.
\end{cases}
$$
It follows from the cofilteredness of $I$ 
and the surjectivity of $f_{i,j}$ that
the category $\cC$ is semi-cofiltered.
Let $J$ be the atomic topology of $\cC$.
Then $(\cC,J)$ is a $Y$-site.

Let $\cC_{0,1} = \cC_{0,2} = I$ and let
$\iota_{0,1} : I \to \cC$ denote the functor
given as follows: $\iota_{0,1}(i) = X_i$ for $i \in I$
and $\iota_{0,1}(i \to j) = 1 \in G_j$ for any 
morphism $i \to j$ in $I$.
Then the pair $(\cC_{0,1},\iota_{0,1})$ is a grid of $(\cC,J)$.

By definition of the set $R^1 \varprojlim_i G_i$,
there exists a family $(g_{i,j})$ of
elements $g_{i,j} \in G_j$ indexed by the morphisms
$i \to j$ in $I$ such that for two composable morpihsm
$i \to j$ and $j \to k$ in $I$ we have
$g_{i,k} = g_{j,k} f_{j,k} (g_{i,j})$
and that there does not exist a family $(h_i)$
of elements $h_i \in G_i$ such that
$g_{i,j} = h_j^{-1} f_{i,j}(h_i)$ for any
morphism $i \to j$ in $I$.
Let
$\iota_{0,2} : I \to \cC$ denote the functor
given as follows: $\iota_{0,2}(i) = X_i$ for $i \in I$
and $\iota_{0,2}(i \to j) = g_{i,j} \in G_j$ for any 
morphism $i \to j$ in $I$.
Then the pair $(\cC_{0,2},\iota_{0,2})$ is a grid of 
$(\cC,J)$.
Let $X_0 = i_0$ and $X_{0,1}=X_{0,2} = X_{i_0}$.
For $i \in \{1,2\}$, let $\beta_{i} : \iota_{0,i} (X_{0,i}) \to X_0$
denote the unique morphism in $\cC$.
In this case the isomorphism $\alpha : \cC_{0,1} \to \cC_{0,2}$
of categories mentioned above is equal to the identity functor.
However the non-existence of $(h_i)$ described above implies
there does not exist a natural equivalence from
$\iota_{0,1}$ to $\iota_{0,2}\circ \alpha$.

\section{The topos has enough points}
\label{sec:enough points}
In this section, we show that the topos associated with a 
$Y$-site under cardinality conditions has enough points.
We show that the fiber functor 
$\omega_\Cip$
has a left 
adjoint.   It follows that the fiber functor is a point of the 
topos $\Shv(\cC, J)$.   
We then  obtain as a corollary to Theorem~\ref{thm:Galois_main}
that the topos has enough points.

Within this  
section, we assume that, for any object $X$ of $\cC$,
the category $\cC(\cT(J))_{/X}$ satisfies 
at least one of the two conditions 
in Section \ref{sec:cardinality}.

\subsection{ }
\label{sec:9.1}
\begin{lem} \label{lem:7-10-1}
The functor $\omega_\Cip:\Presh(\cC) \to (\Sets)$ 
commutes with finite limits and arbitrary colimits.
\end{lem}

\begin{proof}
Recall that we have defined, for any presheaf $F$ on $\cC$,
the set $\omega_\Cip(F)$ to be a filtered colimit of sections of $F$.
Observe that, in the category $\Presh(\cC)$, limits and colimits
can be taken in a section-wise manner. Hence, claim (1) follows from the fact that
filtered colimits of sets commute with finite limits and arbitrary colimits
in the following sense: for any filtered poset $I$, for any finite poset $J$
(\resp for any poset $J'$), and for any functor $S$ from 
$I \times J^\op$ (\resp $I \times J$) to the category of sets,
the natural map $\varinjlim_{i \in I} \varprojlim_{j\in J} S(i,j)
\to \varprojlim_{j \in J} \varinjlim_{i \in I} S(i,j)$ 
(\resp $\varinjlim_{i \in I} \varinjlim_{j' \in J'} S(i,j')
\to \varinjlim_{j' \in J'} \varinjlim_{i \in I} S(i,j')$)
is a bijection. This proves the claim.
\end{proof}

\begin{lem} \label{lem:7-10-2}
Let $F$ be a presheaf on $\cC$ and $a_J(F)$ its associated sheaf
on $(\cC,J)$. Then the adjunction morphism $F \to a_J(F)$ of presheaves
induces a bijection $\omega_\Cip(F) \xto{\cong} \omega_\Cip(a_J(F))$.
\end{lem}

\begin{proof}
Let $\cCe$ denote the full subcategory of $\cC_0$
whose objects are the edge objects of $\cC_0$.
Let $X$ be an edge object of $\cC_0$.
As $\cC_0$ is a poset, the functor $\cC_{0,/X} \to \cC_0$
that associates to each object $f:Y \to X$ of $\cC_{0,/X}$,
the object $Y$ of $\cC_0$ is fully faithful.
It follows from Lemma \ref{lem:lower} that this functor induces
a fully faithful functor $\cC_{0,/X} \to \cCe$.
Via this functor we regard $\cC_{0,/X}$ as a full subcategory
of $\cCe$.
Lemma \ref{lem:core_cofiltered2} 
shows that $\cCe$ is $\Lambda$-connected.
Hence, the objects of $\cC_{0,/X}$ are cofinal in $\cCe$.
It then follows from Lemma \ref{lem:I_X cofinal} that the objects of $I_X$
are cofinal in $\cCe$.
Thus, the natural map
\begin{equation} \label{eq:7-10-1}
\varinjlim_{(f:Y \to X) \in \Obj I_X} F(\iota_0(Y))
\to \omega_\Cip(F)
\end{equation}
is bijective.

\begin{lem}
The functor $\iota_0$ induces a functor $I_X \to \Gal/\iota_0(X)$
which we denote by $j_X$.
The functor $j_X$ is an equivalence of categories.
\end{lem}
\begin{proof}
It follows from Condition (3) of Definition \ref{defn:grids}
that the functor $j_X$ is essentially surjective.
Let $f_1:Y_1 \to X$ and $f_2:Y_2 \to X$ be two objects
of $I_X$. Suppose that there exists a morphism from $j_X(f_1)$
to $j_X(f_2)$ in $\Gal/\iota_0(X)$. Then, there exists
a morphism $g: \iota_0(Y_1) \to \iota_0(Y_2)$ in $\cC$
satisfying $\iota_0(f_1) = \iota_0(f_2) \circ g$.
It follows from Condition (4) of Definition \ref{defn:grids}
that there exist an object $Y'_2$ of $\cC_0$, morphisms
$g' : Y_1 \to Y'_2$ and $f'_2:Y'_2 \to X$, and an isomorphism 
$\beta : \iota_0(Y_2) \xto{\cong}
\iota_0(Y'_2)$ satisfying $\iota_0(g') =\beta \circ g$
and $\iota_0(f_2) = \iota_0(f'_2) \circ \beta$.
As $\iota_0(f_2)$ is a Galois covering in $\cC$,
it follows that $\iota_0(f'_2)$ is a Galois covering in $\cC$.
Hence, it follows from Corollary \ref{cor:Galois_over}
that we have $Y_2 = Y'_2$ and $f_2 = f'_2$.
This shows that $\beta = \id_{\iota_0(Y_2)}$
and $g = \iota_0(g')$.
Hence, $g'$ gives a morphism from $f_1$ to $f_2$ in $I_X$.
As both $I_X$ and $\Gal/\iota_0(X)$ are thin, this shows
that the functor $j_X$ is fully faithful.
This completes the proof of the claim that
$j_X$ is an equivalence of categories.
\end{proof}

Therefore, the bijection \eqref{eq:a_J} gives a bijection
$$
a_J(F)(\iota_0(X)) 
\cong \varinjlim_{(f:Y \to X) \in \Obj I_X} 
F(\iota_0(Y))^{\Gal(f)}.
$$
As \eqref{eq:7-10-1} is bijective, we have a bijection
$$
a_J(F)(\iota_0(X)) 
\cong \omega_\Cip(F)^{H_X}
$$
where $H_X$ acts on $\omega_\Cip$ via the homomorphism $\psi_X$.
Hence, it follows from Lemma \ref{lem:psi_isom} that
we have a bijection
$$
a_J(F)(\iota_0(X)) 
\cong \omega_\Cip(F)^{\bK_X}.
$$
By composing the inverse of this bijection with
the bijection in Corollary \ref{cor:KX-invariant}, we obtain
a bijection
$$
\delta_{F,X} : \omega_\Cip(F)^{\bK_X} \xto{\cong} 
\omega_\Cip(a_J(F))^{\bK_X}.
$$
It is then straightforward to check that the diagram
$$
\begin{CD}
\omega_\Cip(F)^{\bK_X} @>{\delta_{F,X}}>>
\omega_\Cip(a_J(F))^{\bK_X} \\
@VVV @VVV \\
\omega_\Cip(F) @>>> \omega_\Cip(a_J(F)),
\end{CD}
$$
where the vertical arrows are inclusions and the
lower horizontal arrow is a map induced by the adjunction
morphism $F \to a_J(F)$, is commutative.
As $\omega_\Cip(F)$ and $\omega_\Cip(a_J(F))$ are
smooth $M_\Cip$-sets, this shows that the map
$\omega_\Cip(F) \to \omega_\Cip(a_J(F))$ is bijective,
which proves the claim.
\end{proof}

\subsection{ }
Let us consider the functor $\omega_\Cip$ restricted to 
the full subcategory $\Shv(\cC,J)$ of sheaves 
in $\Presh(\cC)$ denoting it by $F^*: \Shv(\cC,J) \to (\Sets)$.
Lemma \ref{lem:7-10-1} shows that $F^*$ commutes 
with fiber products.
Let us show that the functor $F^*$ has a right adjoint.
For a set $Y$, we construct a presheaf $F_*(Y)$ on $\cC$
by setting $F_*(Y)(X) = \Map(\omega_\Cip(\frh_\cC(X)),Y)$ 
for each object $X$ of $\cC$.
\begin{lem} \label{lem:7-10-3}
The presheaf $F_*(Y)$ is a sheaf on $(\cC,J)$. 
\end{lem}

\begin{proof}
Let $f:X' \to X$ in $\cC$ be an arbitrary 
Galois covering that belongs to $\cT$.
Let us consider the map $\omega_\Cip(\frh_\cC(f)): 
\omega_\Cip(\frh_\cC(X')) \to \omega_\Cip(\frh_\cC(X))$.
By definition, the map $\omega_\Cip(\frh_\cC(f))$ 
is equal to the map
$$
\varinjlim_{Z \in \Obj \cCe}
\Hom_\cC(\iota_0(Z),X') \to
\varinjlim_{Z \in \Obj \cCe}
\Hom_\cC(\iota_0(Z),X)
$$
induced by the composition with $f$.
The map $\omega_\Cip(\frh_\cC(f))$ 
is a pseudo $\Gal(f)$-torsor as it is
a filtered colimit of pseudo $\Gal(f)$-torsors.
It follows from Condition (3) of Definition \ref{defn:semi-localizing}
and Condition (3) of Definition \ref{defn:grids} that the map
$\omega_\Cip(\frh_\cC(f))$ is surjective.
This shows that the set $\omega_\Cip(\frh_\cC(X))$ together with the
map $\omega_\Cip(\frh_\cC(f))$ is a quotient object of 
$\omega_\Cip(\frh_\cC(X'))$ by
$\Gal(f)$ in the category of sets.
Hence,
the pullback map
$F_*(Y)(X) \to F_*(Y)(X')$ induces a bijection
$F_*(Y)(X) \xto{\cong} F_*(Y)(X')^{\Gal(f)}$.
This completes the proof of 
the claim that $F_*(Y)$ is a sheaf on $(\cC,J)$. 
\end{proof}

Let $Y$ be a set.  Using Lemma~\ref{lem:7-10-3}
above, we can apply $F^*$ to $F_*(Y)$.
\begin{lem} \label{lem:7-10-4}
Let $Y$ be a set. Then, we have a bijection
$$
F^* (F_*(Y)) \cong \varinjlim_{Z \in \Obj\cCe}
\Map(M_\Cip/\bK_Z, Y)
$$
that is functorial in $Y$.
\end{lem}

\begin{proof}
By definition, we have
$$
F^*(F_*(Y)) = \varinjlim_{Z \in \Obj\cCe}
\Map(\omega_\Cip(\frh_\cC(\iota_0(Z))), Y).
$$
For any sheaf $G$ on $(\cC,J)$, we have the isomorphisms
\begin{align*}
& \Hom_{M_\Cip}(\omega_\Cip(\frh_\cC(\iota_0(Z))),
\omega_\Cip(G)) \\
\xleftarrow[(1)]{\cong} 
& \Hom_{M_\Cip}(\omega_\Cip (a_J(\frh_\cC(\iota_0(Z)))),
\omega_\Cip(G)) \\
\xrightarrow[(2)]{\cong}
& \Hom_{\Shv(\cC,J)}(a_J(\frh_\cC(\iota_0(Z))),G) \\
\cong 
& \Hom_{\Presh(\cC)}(\frh_\cC(\iota_0(Z)),G) \\
\cong 
& G(\iota_0(Z))
\cong
\omega_\Cip(G)^{\bK_Z} \\
\cong &
\Hom_{M_\Cip}(M_\Cip/\bK_Z, \omega_\Cip(G)),
\end{align*}
where (1) and (2) are isomorphisms given by
Lemma \ref{lem:7-10-2} and Theorem \ref{thm:Galois_main}, respectively.
Hence, by Yoneda's lemma, we have an isomorphism
$$
\omega_\Cip(\frh_\cC(\iota_0(Z)))
\cong M_\Cip/\bK_Z.
$$
of smooth $M_\Cip$-sets. From this we obtain an isomorphism
$$
F^* (F_*(Y)) = \varinjlim_{Z \in \Obj\cCe}
\Map(M_\Cip/\bK_Z, Y).
$$
It is straightforward to check that the last
isomorphism is functorial in $Y$.
\end{proof}

\begin{lem} \label{lem:7-10-5}
Let $H$ be a sheaf on $(\cC,J)$ and let $Y$ be a set.
Then we have a bijection
$$
\Map(F^*(H),Y) \xto{\cong} \Hom_{\Shv(\cC,J)}(H,F_*(Y))
$$
that is functorial in $H$ and $Y$.
\end{lem}

\begin{proof}
It follows from Theorem \ref{thm:Galois_main}
that $\Hom_{\Shv(\cC,J)}(H,F_*(Y))$ is isomorphic to
$\Hom_{M_\Cip}(F^*(H), \omega_\Cip(F_*(Y)))$.
Let us consider the map
$$
j:\Map(F^*(H),Y) \to \Map(M_\Cip \times F^*(H), Y)
$$
given by the composition with the action
$M_\Cip \times F^*(H) \to F^*(H)$ of $M_\Cip$ on $F^*(H)$.
Let us regard the target
$\Map(M_\Cip \times F^*(H), Y)$ of this map as
the set $\Map(F^*(H), \Map(M_\Cip, Y))$.
Then the image of the map $j$ is contained in the
subset $\Hom_{M_\Cip}(F^*(H),\Map(M_\Cip, Y))$
of $\Map(F^*(H), \Map(M_\Cip, Y))$.
One can check easily that the map
$\Map(F^*(H),Y) \to \Hom_{M_\Cip}(F^*(H),\Map(M_\Cip, Y))$
induced by $j$ is bijective.
We have seen in the last paragraph that
$\omega_\Cip(F_*(Y))$ is equal to the smooth part of the
set $\Map(M_\Cip, Y)$.
As $F^*(H)$ is a smooth $M_\Cip$-set, we have
$$
\Hom_{M_\Cip}(F^*(H), \Map(M_\Cip, Y)) = 
\Hom_{M_\Cip}(F^*(H), \omega_\Cip(F_*(Y))).
$$
Thus we have bijections
$$
\Map(F^*(H),Y)
\cong \Hom_{M_\Cip}(F^*(H), \omega_\Cip(F_*(Y)))
\cong \Hom_{\Shv(\cC,J)}(H,F_*(Y)).
$$
It is straightforward to check that these bijections
are functorial with respect to $H$ and $Y$.
\end{proof}

\subsection{ }

\begin{thm} 
\label{thm:7-10-6}
The pair $(F_*,F^*)$ of functors gives 
a point of the topos $\Shv(\cC,J)$.
\end{thm}

\begin{proof}
Lemma~\ref{lem:7-10-5} implies 
that the functor $F^*$ is a right adjoint
to the functor $F_*$.
Thus, the pair $(F_*,F^*)$ of functors gives 
a point of the topos $\Shv(\cC,J)$.
\end{proof}

\begin{cor}
\label{cor:enough points}
The topos $\Shv(\cC,J)$ has enough points.
\end{cor}
\begin{proof}
Let $f:F_1 \to F_2$ be a morphism of sheaves.
Using Theorem~\ref{thm:Galois_main}, we know that
$f$ is an isomorphism if and only if 
$\omega_{(\cC_0, \iota_0)}(f)$
is an isomorphism of smooth $M_{(\cC_0, \iota_0)}$-sets.
This is an isomorphism if it is an isomorphism
of (the underlying) sets.  This implies the claim.
\end{proof}

\section{On locally profinite groups}
\label{sec:atomic topological}
Suppose we are given a $Y$-site and a grid.
Then, our absolute Galois monoid $M_\Cip$ 
comes with the set of subgroups 
indexed by the edge objects in the grid.
We show in 
Section~\ref{sec:topological monoid structure}
that $M_\Cip$ is naturally equipped with 
the structure of a topological monoid such that 
the category of smooth $M_\Cip$-sets is 
canonically equivalent to the category 
of discrete sets with continuous action 
of the topological monoid $M_\Cip$.

In Section~\ref{sec:atomic discussions}, 
we give one of our main theorems, which states that the topos associated with
a $Y$-site with an atomic topology
that satisfies cardinality condition (1)
is equivalent to the category of discrete sets
with continuous action of some locally profinite
group.   
This may be regarded as a reconstruction theorem
for locally profinite topological groups.   

\subsection{The absolute 
Galois monoid as a topological monoid}
\label{sec:topological monoid structure}
Let $(\cC, J)$ be a $Y$-site.
Suppose we are given a grid 
$(\cC_0, \iota_0)$
for this $Y$-site.
Let us equip the associated
Galois monoid $M_\Cip$
with the structure of a topological 
monoid as follows.

For each element $m \in M_\Cip$,
consider the set
\[
\frV_m=\{m \bK_X \,|\,
X \text{ an edge object}\
\}
\]
of subsets of $M_\Cip$.
Let $\frV=(\frV_m)_{m \in M_\Cip}$.

\begin{lem}
The set $\frV$ is a fundamental system of neighborhood for 
some topology on (the underlying set of) $M_\Cip$.
\end{lem}
\begin{proof}
We check below only some of the axioms for the set of subsets
to be a fundamental system of neighborhoods.
The rest is left to the reader.

Let $m, m_1, m_2 \in M_\Cip$
and $X_1, X_2$ be edge object.
Let us show that, 
if $m \in m_1 \bK_{X_1} \cap m_2 \bK_{X_2}$,
there exists an edge object 
$Y$ such that 
$Y \subset m_1 \bK_{X_1} \cap m_2 \bK_{X_2}$.
We can write $m=m_1 k_1=m_2 k_2$ for 
some $k_1 \in \bK_1$ and $k_2 \in \bK_2$.
We want to show that there exists 
an edge object $Y$ such that 
$k_1 \bK_Y \subset \bK_{X_1}$ 
and $k_2 \bK_Y \subset \bK_{X_2}$.
It suffices to find $Y$ such that 
$\bK_Y \subset \bK_{X_1} \cap \bK_{X_2}$.

Using the $\Lambda$-connectedness
of the grid, we see that 
for any $X, X' \in \cC_0$, there exists 
an object $Y \in \cC_0$ such that 
there are morphisms 
$Y \to X$ and $Y \to X'$.
Hence we have $\bK_Y \subset
\bK_X \cap \bK_{X'}$.
By Lemma~\ref{lem:lower}, $Y$ is an edge
object, and therefore the claim follows.
\end{proof}

\begin{lem}
The product map $M_\Cip \times M_\Cip \to M_\Cip$
is continuous when $M_\Cip$ is equipped with
the topology as above.
\end{lem}
\begin{proof}
We will prove the following claim:
for any 
$m=(\alpha, \gamma_\alpha)
\in M_\Cip$ and any edge object $X$,
there exists an edge object $Y$ such that
$\bK_Y m \subset m \bK_X$.  
From Lemma~\ref{lem:7-9-5}, it follows that
the set $m\bK_X$
equals
\[
\{ (\beta, \gamma_\beta) \,|\,
\beta(X)=\alpha(X), 
\gamma_\beta(X)=\gamma_\alpha(X)
\}.
\]
From Lemma~\ref{lem:C1_cofinal}, 
it follows that 
there exists an edge object $Y$
such that there is a morphism 
$Y \to \alpha(X)$.
Now, take $(\delta, \gamma_\delta) \in \bK_Y$; from Lemma~\ref{lem:7-9-3}, it
follows that
$\delta(\alpha(X)=\alpha(X)$
and $\gamma_\delta(\alpha(X))=\id_{\alpha(X)}$.
Let $(\beta, \gamma_\beta)=(\delta, \gamma_\delta) \circ m$.
We then have
$\beta(X)=\alpha(X)$ and $\gamma_\beta(X)=\gamma_\alpha(X)$.
Hence, $(\beta, \gamma_\beta)\in m \bK_X$,
and the claim follows.
\end{proof}
\begin{cor}
The absolute Galois monoid is a topological monoid
for the topology constructed as above.
\end{cor}
\begin{proof}
This follows immediately from the previous lemma.
\end{proof}

\begin{rmk}
\label{rmk:equivalence smooth continuous}
The category of smooth $M_\Cip$-sets defined in 
Section~\ref{defn:smooth sets} is canonically
equivalent to the category of discrete sets 
with continuous action of the topological 
monoid $M_\Cip$.  
This follows from the definitions; see also \cite[p.151]{MM}.
\end{rmk}

\subsection{Locally profinite groups}
\label{sec:atomic discussions}
As an application of our main theorem, 
we obtain a `reconstruction' theorem
as follows.
\begin{thm}
\label{thm:reconstruction}
Let $(\cC, J)$ be a $Y$-site.
Suppose that the topology is atomic
and suppose that Condition (1) of the 
cardinality conditions holds true.
Then, there exists a locally 
profinite group $G$ 
such that the topos $\Shv(\cC,J)$
is equivalent to the category of 
discrete sets with continuous action
of $G$.

If, moreover, there exists a final object
in $\cC$, then the locally profinite group
is profinite.
\end{thm}
\begin{proof}
As the cardinality Condition (1) holds true,
by Proposition~\ref{cor:grid existence}, 
there exists 
a grid $(\cC_0, \iota_0)$ 
for this $Y$-site.
From Theorem~\ref{thm:Galois_main},
it follows that 
the topos is equivalent 
to the category of smooth $M_\Cip$-sets.
Now as the topology is atomic
(Section~\ref{sec:atomic topology}),
by definition all objects of $\cC$
are edge objects.
It follows from Lemma~\ref{lem:7-9-6}
that the associated absolute 
Galois monoid $M_\Cip$ is a group.

We can use the procedure in Section~\ref{sec:topological monoid structure} to equip $M_\Cip$ with 
the structure of a topological group.
We noted in Remark~\ref{rmk:equivalence smooth continuous} 
that the category of smooth $M_\Cip$-sets
is canonically equivalent to the category of 
discrete sets with continuous action of $M_\Cip$
for this topological group structure.
The claim then follows from the following 
Lemma~\ref{prop:Galois is locally profinite}.
\end{proof}
\begin{lem}
\label{prop:Galois is locally profinite} 
Let the setup be as above. Then, equipped with the structure 
of a topological group as in 
Section~\ref{sec:topological monoid structure}, the absolute Galois monoid $M_\Cip$ is locally profinite.
Moreover, if there exists a final object in $\cC$,
then $M_\Cip$ is profinite.
\end{lem}
\begin{proof}
Let $X \in \cC_0$.
Note that under the cardinality condition,
$H_X$ is by definition a profinite group.
Let us equip $\bK_X$ with the structure of 
a profinite group via the isomorphism
$\psi_X: H_X \to \bK_X$ 
(see Section~\ref{sec:defn phiX}).
To prove the proposition, it suffices to 
show that the inclusion
$\bK_X \subset M_\Cip$ is a continuous and open
map of topological groups.

First, for any morphism $Y \to X$, one can check that the 
induced inclusion $\bK_X \to \bK_Y$ 
is a continuous open map.  
Second, 
given an open subgroup $\bK' \subset \bK_X$,
one can find a morphism $Y \to X$
such that $\bK_Y \subset \bK'$.
These two statements can be used 
to prove the claim.

Now suppose that there exists a final object.
As $\iota_0$ is essentially surjective,
there is an object $X \in \cC_0$ that
is sent to a final object in $\cC$.
Let $(\alpha, \gamma_\alpha) \in M_\Cip$.
Then, as $\cC_0$ is thin,
we have $\alpha(X)=X$ and $\gamma_\alpha(X)=\id_X$.
This means that the inclusion
$\bK_X \subset M_\Cip$ is an equality.
We saw above 
that the inclusion is a homeomorphism onto its image.
As $\bK_X$ is profinite, the claim follows.
\end{proof}

\section{$Y$-sites and grids for locally prodiscrete groups}
\label{sec:locally prodiscrete}
We define locally prodiscrete groups and 
a class of topological groups in Section~\ref{sec:certain top groups}
which includes locally prodiscrete groups.
We then construct a $Y$-site and a grid from a group $G$ in this class.
It shown that the absolute Galois monoid is $\wh{G}$ 
(to be defined below), and the fiber functor induces 
an equivalence.

In \cite[p.150, Section 9]{MM}, given a topological group $G$,
MacLane and Moerdijk construct a site such that the associated
topos is equivalent to the category of discrete sets with 
continuous $G$-action.     We use essentially the same site 
and view their fiber functor as a guide to the construction 
of our grid.    As the equivalence is already proved in \cite{MM}, 
the 
emphasis of this section 
is on the computation of the absolute Galois monoid.

In Section~\ref{sec:locally prodiscrete monoids},
we have theorems (Theorems~\ref{thm:CM_Y-site},~\ref{thm:CM_grid}, \ref{thm:locally prodiscrete monoid})
generalizing Lemma~\ref{lem:Ghat isom to M}
to the case of locally prodiscrete monoids.
A different proof (not using \cite{MM})
is provided there.

\subsection{The construction of 
a $Y$-site and a grid}

First, we give the definition of a locally prodiscrete
group.
\begin{defn}
\label{def:prodiscrete groups}
By a {\it prodiscrete group},
we mean a topological group
that is a filtered 
limit of discrete groups in the 
category of topological groups.
A {\it locally prodiscrete group} 
is a topological group
such that there exists an open subgroup which is 
a prodiscrete group.  
\end{defn}

\subsubsection{A certain class of topological groups}
\label{sec:certain top groups}
We consider the following class of topological
groups, which is more general than 
the class of locally prodiscrete groups.

Let $G$ be a topological group.   Consider the 
set of open subgroups
$\frV=\{
H \subset G
\}$,
where $H$ satisfies the following property:
For any open subgroup $U \subset G$,
there exists an open subgroup $K$
such that $K \subset U \cap H$
and $K$ is a normal subgroup of $H$.
We consider those topological groups
such that the set $\frV$ is non-empty.

\subsubsection{}
\label{sec:site from G}
We can construct a $Y$-site and a grid,
starting from a topological group as 
in Section~\ref{sec:certain top groups}.

Let us construct the site.
Let $G$ be a topological group as in 
Section~\ref{sec:certain top groups}.
Let $\cC$ be the category of discrete left $G$-sets
consisting of a single $G$-orbit, which is isomorphic to 
the $G$-set of the form $G/H$ for some $H \in \frV$.
Then the category $\cC$ is semi-cofiltered.   If we 
equip $\cC$ with the atomic topology $J$, 
we obtain a $Y$-site.

We can construct a grid of the $Y$-site 
in the following manner.
Let $P_G$ denote the set of open subgroups which 
belongs to $\frV$.
We regard $P_G$ as a partially ordered set with respect 
to the inclusions.
We denote by $\cC_0$ the poset (viewed as a category)
$\cC_{P_G}$ associated with 
the partially ordered set $P_G$.
The group $G$ acts from the left on the set 
$P_G$ by conjugation, i.e., 
$g \cdot \bK:=g \bK g^{-1}$.
By associating $G/\bK$ to each element 
$\bK$ of $P_G$,
we obtain a functor $\iota_0 : \cC_0 \to \cC$.
Then the pair $(\cC_0,\iota_0)$ is a grid for
the $Y$-site $(\cC,J)$.

\subsubsection{}
Given a topological group $G$ 
as in Section~\ref{sec:certain top groups},
we construct a locally prodiscrete group $\wh{G}$
as follows.    

We set $\widehat{G}=\varprojlim_{H \in \frV} G/H$
to be the limit of discrete sets $G/H$ 
in the category of topological spaces.
We can equip $\widehat{G}$ with the 
structure of a topological group as follows.
For two elements 
$g_1=(g_{1,H})_{H\in \frV}, 
g_2=(g_{2,H})_{H\in \frV} 
\in \widehat{G}$,
we define the product $g_1g_2 \in \widehat{G}$
as follows.   We set the $H$-component
of $g_1g_2$ to be 
$\tilde{g}_{1,H'}g_{2,H} H \in G/H$.
Here, we first take a lift 
$\tilde{g}_{2,H}\in G$ of the 
element $g_{2,H} \in G/H$,
set $H'=\tilde{g}_{2,H} H \tilde{g}_{2,H}^{-1}$,
and take a lift 
$\tilde{g}_{1,H'} \in G$
of $g_{H'} \in G/H'$.
The resulting element does not depend on 
the choice of the lifts.

Let us construct the inverse $s=(s_H)_{H \in \frV}$
of $g=(g_H)_{H \in \frV}$ as follows.
For  $H \in \frV$, we take a lift $\wt{g}_H \in G$ 
of $g_H \in G/H$.   
Take $K \in \frV$.  By the definition of $\frV$ 
(Section~\ref{sec:certain top groups}),
there exists a normal subgroup $K'$
of $K$ such that 
$K' \subset \wt{g}_K^{-1} K \wt{g}_K \cap K$.
Then, we set $s_K=\wt{g}_{K'}^{-1}K \in G/K$.

The topological group $\widehat{G}$
is locally prodiscrete.   Take an open subgroup $H \subset G$
belonging to $\frV$.  Consider its image in $\wh{G}$, which is open.
As $H$ belongs to $\frV$, the image is a prodiscrete group.

There is a canonical morphism of topological 
groups $G \to \widehat{G}$.  We note that this morphism
induces an equivalence of categories from the category of 
discrete $\wh{G}$-sets
to the category of discrete $G$-sets.

\begin{lem}
\label{lem:G isom to Ghat}
Let $G$ be a topological group.
Then $G$ is locally prodiscrete 
if and only if
$G$ satisfies the condition in
Section~\ref{sec:certain top groups} and 
the morphism $G \to \wh{G}$
is an isomorphism of topological groups.
\end{lem}
\begin{proof}
The ``if'' part follows from the argument preceding this lemma.
Let us prove the ``only if'' direction.
If $G$ is locally prodiscrete,
then there exists an open subgroup $H$
which is prodiscrete.
Write $H=\varprojlim_i H_i$
with $H_i$ discrete.
Without loss of generality, we may and will assume each
$H \to H_i$ is surjective.
Let $U_i=\Ker (H \to H_i)$.
Then $U_i$ is a normal subgroup of $H$
and $\{U_i\}_i$ forms a fundamental system
of neighborhoods of 1 in $G$.
It follows that each $U_i$ belongs to 
$\mathfrak{V}$; 
we see that $\{U_i\}$ is cofinal in $\mathfrak{V}$.
Hence $G$ satisfies the condition in Section~\ref{sec:certain top groups},
and we see that $G \to \wh{G}$ 
is an isomorphism.
\end{proof}

\subsubsection{}
We compute the absolute Galois monoid for the grid above.
\begin{lem}
\label{lem:Ghat isom to M}
Let $G$ be a topological group
as in Section~\ref{sec:certain top groups}.
The absolute Galois monoid $M_\Cip$
associated with the grid constructed as in Section~\ref{sec:site from G}
is isomorphic to $\widehat{G}$.
\end{lem}
\begin{proof}
We construct an isomorphism 
$\widehat{G} \to M_\Cip$ as 
follows.
Let $g=(g_H)_{H\in \frV} \in \widehat{G}$.
For an object $H$ of $P_G$,
we set $\alpha_g H=\tilde{g}_H H \tilde{g}_H^{-1}$
for some lift $\tilde{g}_H$ of $g_H \in G/H$.
We have 
$\iota_0(H)=G/H$
and $\iota_0(\alpha_g(H))=G/(gHg^{-1})$.
We can construct a map
\[
\iota_0(H) \to \iota_0(\alpha_g(H))
\]
by sending the coset 
$h H$
to the coset 
$h g^{-1}\cdot g H g^{-1}$.
Then, one can check that these
form a natural isomorphism $\gamma_{\alpha_g}$.
Thus we have a map that sends $g \in \widehat{G}$
to $(\alpha_g, \gamma_{\alpha_g}) \in M_\Cip$.

The proof that this map is an isomorphism is left to the readers.
\end{proof}

\subsection{}
For the associated fiber functor and the topos,
we have the following claim.
\begin{prop}
\label{prop:locally prodiscrete}
Let $G$ be a locally prodiscrete group.
Then there exists a 
$Y$-site and a grid such that 
the fiber functor associated with the grid 
induces an equivalence
of the sheaves and the category of 
discrete sets with continuous $G$-action.
\end{prop}
\begin{proof}
This proposition essentially follows from 
\cite[p.154, Theorem 2]{MM}.
The site constructed above 
is essentially that of loc.\ cit.\ and one can 
also check that the fiber functor is essentially
the functor considered there.
Their theorem says that 
the functor induces an equivalence.
Using the equivalence of the categories of discrete $G$-sets 
and $\wh{G}$-sets,
we obtain the proposition.
\end{proof}

\section{Examples}
\label{sec:examples}
We give two examples.
The examples in Section~\ref{sec:monoid example} 
are the simplest for which the Galois groups
are the abelian group of integers and 
the monoid of natural numbers.
The example in Section~\ref{sec:motivation example}
served as the motivation to write this paper.

\subsection{The simplest examples}
\label{sec:monoid example}
Let $\cC$ be the following category.
The objects of $\cC$ are the sets $[0]$,
$[1]$, $[2]$, $\ldots$ where $[n]$ denote the
set $\{0,1,\ldots,n-1\}$ for $n \ge 0$. 
For two integers $m,n \ge 0$, the morphisms 
from $[m]$ to $[n]$ in $\cC$ are the maps
$f:[n] \to [m]$ satisfying $f(i+1)=f(i)+1$
for $i=0,\ldots,n$ (this is not a typo: the morphisms
go in the ``opposite'' direction).
The category $\cC$ is $\frU$-small, $\Lambda$-connected, semi-cofiltered,
and is an $E$-category.
One can check that any morphism in $\cC$ is a monomorphism.
Hence, any morphism in $\cC$ is a Galois covering whose Galois group
is isomorphic to $\{1\}$.

Let $\cT=\Mor(\cC)$ and let $\cT_{+}$ denote the set of morphisms $f$ in $\cC$ satisfying $f(0)=0$.
Then both $\cT$ and $\cT_{+}$ are semi-localizing collections of
morphisms in $\cC$.
The pairs $(\cC,J_\cT)$ and $(\cC,J_{\cT_+})$
are $B$-sites
that have enough Galois coverings.
We note that $\cT = \cT(J_\cT)$ and $\cT_+ = \cT(J_{\cT_+})$.
In particular the notation $\cC(\cT_+)$ makes sense.
As the set $\Hom_\cC([m],[n])$ is a finite set for
any $m,n \ge 0$, it follows from Lemma \ref{lem:C0_core} that
both $(\cC, J_\cT)$ and 
$(\cC, J_{\cT_+})$ admit grids.
We can explicitly construct grids as follows.

Let $\cC_0$ be the following category: the objects of $\cC_0$
are the finite sets $S$ of the form $S = \{a,a+1,\ldots,b\}$ for some integers
$a,b \in \Z$ with $a \le b$. The morphisms in $\cC_0$ are the opposite of 
the inclusions, i.e., the category $\cC_0$ is thin, and for any two objects 
$S_1$, $S_2$ of $\cC_0$, there exists a morphism from $S_1$ to $S_2$ in $\cC_0$ 
if and only if $S_1 \supset S_2$.
Let $\cC'_{+,0}$ denote the full subcategory of $\cC_0$ whose
objects are $[0]$, $[1]$, $[2]$, $\ldots$.
Let $\iota: \cC_0 \to \cC$ denote the functor which sends
$\{a,a+1,\ldots,b\}$ to $[b-a]$.
Then one can check easily that the pair $(\cC_0,\iota)$
is a grid for $\cC$.  
The pair $(\cC'_{+,0},\iota|_{\cC'_{+,0}})$
is a pregrid for $(\cC, J_{\cT_+})$.
Let $\cC_{+,0}$ denote the full subcategory of $\cC_0$
whose objects are the sets of the form 
$\{a,a+1,\ldots,b\}$ for some $a,b \in \Z$ with $0 \le a \le b$.
Then the pair $(\cC_{+,0},\iota|_{\cC_{+,0}})$ is
a grid of $(\cC, J_{\cT_+})$.

Let $\alpha : \cC_0 \to \cC_0$ denote the isomorphism of
categories that sends $\{a,a+1,\ldots,b\}$ to
$\{a+1,a+2,\ldots,b+1\}$. The isomorphism $\alpha$ induces the
functor $\cC_{+,0} \to \cC_{+,0}$, which we denote by $\alpha_{+}$.
We have $\iota = \iota \circ \alpha$ and $\iota|_{\cC_{+,0}} \circ \alpha_+
= \iota|_{\cC_{+,0}}$.
Hence, the pairs $(\alpha,\id)$ and $(\alpha_+,\id)$
are elements of the monoids
$M_{(\cC_0,\iota)}$ and $M_{(\cC_{+,0},\iota|_{\cC_{+,0}})}$,
respectively.
We then have isomorphisms $\Z \cong M_{(\cC_0,\iota)}$
and $\Z_{\ge 0} \cong M_{(\cC_{+,0},\iota|_{\cC_{+,0}})}$,
which sends $1$ to $(\alpha,\id)$ and $(\alpha_+,\id)$,
respectively.

\subsubsection{}
We give another example that is essentially the same 
as the one above.   Below is the less ad hoc, `coordinate-free' 
version. Note that 
the underlying category is essentially
$\frU$-small but not $\frU$-small in general, 
while in the previous example it was $\frU$-small.

Let $\cC$ denote the following category.
The objects are finite well-ordered sets.
The set of morphisms is the set of maps 
of sets that sends the successor (if it exists)
to the successor.
We set $\cT=\Mor(\cC)$ and 
$\cT_+$ to be those morphisms that send the 
least element to the least element.

We regard the totally ordered sets $[n]$ 
as objects of $\cC$ in a natural manner.
Note that each object of $\cC$ 
is isomorphic to the 
well-ordered set $[n]$ for some $n$.
The pairs $(\cC_0, \iota_0)$
and $(\cC_{+,0}', \iota|_{\cC'_{+,0}})$
defined above make sense in this setup 
and form the grids for 
$(\cC,J_\cT)$ and $(\cC, J_{\cT_+})$,
respectively.
The absolute Galois monoids are therefore 
respectively $\Z$ and $\Z_{\ge 0}$, as above.

\subsection{Our starting example}
\label{sec:motivation example}
The following example was the starting point
of this project.   The Galois group 
is the finite adele valued points of the
general linear group.   
We will come back to this in a future paper. 

Let $d \ge 1$ be an integer.
We define the category $\cCo{d}$ as follows.
An object in $\cCo{d}$ is a finite abelian group
that is generated by at most $d$ elements.
For two objects, $N$ and $N'$ in $\cCo{d}$,
the set $\Hom_{\cCo{d}}(N,N')$ of morphisms from
$N$ to $N'$ is the set of isomorphism classes of diagrams
\[
N' \twoheadleftarrow N''  \hookrightarrow N
\]
in the category of abelian groups,
where the left arrow is surjective and
the right arrow is injective. 
Here, two diagrams
$N' \twoheadleftarrow N''  \hookrightarrow N$
and $N' \twoheadleftarrow N'''  \hookrightarrow N$
are considered to be isomorphic if there exists
an isomorphism $N''\xto{\cong} N'''$ of 
abelian groups such that the diagram
$$
\begin{array}{ccccc}
N' & \twoheadleftarrow & N'' & \inj & N \\
{\Large \parallel} & & 
{\Large \downarrow} \cong & & 
{\Large \parallel} \\
N' & \twoheadleftarrow & N''' & \inj & N \\
\end{array}
$$
is commutative.
The composition of two morphisms 
$N' \twoheadleftarrow M \hookrightarrow N$
and 
$N'' \twoheadleftarrow M' \hookrightarrow N'$
is seen in the following diagram:
\[
\begin{array}{ccccc}
 &&&& N \\
 &&&& \uparrow \\
 &&N' &\twoheadleftarrow & M \\
 &&\uparrow & \ \begin{small}\square\end{small} & \uparrow \\
 N'' & \twoheadleftarrow & M' & \twoheadleftarrow & M\times_{N'} M'
 \end{array}
 \]
where the small box means that the square is cartesian.
This definition of
morphisms is from Quillen (\cite{Qu}),
except that here we take morphisms in the opposite
direction.

Let $\cT=\Mor(\cCo{d})$ and let $\cT_{+}$ denote 
the set of morphisms in $\cC$ represented by diagrams
$N' \twoheadleftarrow N'' \xto{i} N$ of abelian groups with
$i$ bijective.
Then, both $\cT$ and $\cT_{+}$ are semi-localizing collections of
morphisms in $\cCo{d}$.
In a future paper, we shall show that
the pairs $(\cC,J_\cT)$ and $(\cC,J_{\cT_+})$
are $B$-sites
that have enough Galois coverings.
We note that $\cT = \cT(J_\cT)$ and $\cT_+ = \cT(J_{\cT_+})$.
In particular, the notation $\cCo{d}(\cT_+)$ makes sense.
As the set $\Hom_\cCo{d}(M,N)$ is a finite set for
any object $M$, $N$ of $\cCo{d}$, it follows from Lemma \ref{lem:C0_core} that
both $\cCo{d}$ and $\cCo{d}(\cT_+)$ admit grids.
We can explicitly construct grids for $(\cCo{d}, J_{\cT})$ and for $(\cCo{d}, J_{\cT_+})$
as follows.

Let $\Lat^d$ denote the set of $\Z$-submodules
of $\Q^{\oplus d}$ that are free and of rank $d$.
We regard $\Lat^d$ as a partially ordered set
with respect to the inclusions.
We let $\Pair^d$ denote the following poset.
The elements of $\Pair^d$ are the pairs 
$(L_1,L_2)$ of elements in $\Lat^d$ with $L_1 \le L_2$.
For two elements $(L_1,L_2)$ and $(L'_1,L'_2)$
in $\Pair^d$, we have $(L_1,L_2) \le (L'_1,L'_2)$
if and only if $L'_1 \le L_1 \le L_2 \le L'_2$.
Let $\cCo{d}_0$ denote the poset category corresponding
to the order dual of $\Pair^d$.
Let $\cC^{',d}_{+,0}$ denote the full subcategory of $\cCo{d}_0$ whose
objects are the pairs $(L_1, L_2)$ with $L_2 = \Z^{\oplus d}$
(we regard $\Z^{\oplus d} \subset \Q^{\oplus d}$ as the standard lattice).

Let $\iota: \cCo{d}_0 \to \cCo{d}$ denote the functor that sends
an object $(L_1,L_2)$ of $\cCo{d}_0$ to $L_2/L_1$ and  
sends a morphism from $(L_1,L_2)$ to $(L'_1,L'_2)$ in $\cCo{d}_0$
to a morphism in $\cCo{d}$ represented by the diagram
$L'_2/L'_1 \twoheadleftarrow L'_2/L_1 \inj L_2/L_1$.
In our future paper, we shall show that the pair $(\cCo{d}_0,\iota)$
is a grid for $(\cCo{d},J_\cT)$ and that the pair 
$(\cC^{',d}_{+,0},\iota|_{\cC^{',d}_{+,0}})$
is a pregrid for $(\cCo{d}, J_{\cT_+})$.
Let $\cCo{d}_{+,0}$ denote the full subcategory of $\cCo{d}_0$
whose objects are the pairs $(L_1,L_2)$ satisfying
$L_2 \subset \Z^{\oplus d}$.
Then the pair $(\cCo{d}_{+,0},\iota|_{\cCo{d}_{+,0}})$ 
is a grid for $(\cCo{d}, J_{\cT_+})$.

Let $\wh{\Z}= \varprojlim_{n \ge 1} \Z/n\Z$ be the profinite completion
of $\Z$ and let $\A^\infty = \wh{\Z} \otimes_\Z \Q$ denote the ring
of finite adeles over $\Q$.
Let us consider the group $\GL_d(\A^\infty)$.
It is a locally compact, totally disconnected topological group.
We set $\Mat^{-} = \{ g \in \GL_d(\A^\infty)\ 
|\ g^{-1} \in \Mat_d(\wh{\Z}) \}$.
We have inclusions $\GL_d(\wh{\Z}) \subset \Mat^{-}
\subset \GL_d(\A^\infty)$ of monoids.
Let $M$ be $\GL_d(\A^\infty)$ or $\Mat^{-}$.
We say that a left $M$-set $S$ is smooth if, for any
$s \in S$, the $\GL_d(\wh{\Z})$-orbit of $s$ in $S$
is a finite set. If $M=\GL_d(\A^\infty)$ and $S$
is a left $\Z[M]$-module, this coincides
with the usual notion of smoothness given, e.g.\ in 
\cite[I,4.1]{V}.
In a future paper, when $M = \GL_d(\A^\infty)$ (\resp when $M = \Mat^-$), 
we shall construct an isomorphism
$M \cong M_{(\cC_0,\iota)}$ 
(\resp $M \cong M_{(\cC'_{+,0},\iota|_{\cC'_{+,0}})}$)
of monoids and show that this isomorphism induces
a one-to-one correspondence between smooth 
$M$-modules and smooth $M_{(\cC_0,\iota)}$-sets
(\resp smooth $M_{(\cC'_{+,0},\iota|_{\cC'_{+,0}})}$-sets).
Therefore, Theorem \ref{thm:Galois_main} gives an equivalence
from the category $\Shv(\cCo{d},J_\cT)$
(\resp $\Shv(\cCo{d},J_{\cT_+})$) to the category
of smooth left $M$-sets.

\section{Connection with Caramello's ultrahomogeneous objects}
\label{sec:Caramello tech}
As discussed in Section~\ref{sec:Caramello intro}, there is a high degree of overlap between ours and Caramello's work \cite{Caramello}, \cite{CaramelloFraisse}.    In this section, we illustrate this relation by constructing from our setup an object in $\ProC$ that is a key part of the input data in Caramello's theorem.

We remark here that the underlying category $\cC$ of a site $(\cC, J)$ that we consider is opposite to that in \cite{Caramello}.   

\subsection{}
Let $(\cC, J)$ be a $Y$-site.   We assume the following finiteness condition (Section~\ref{sec:cardinality}(1)): for any two objects $X, Y \in \cC$, the Hom set $\Hom_\cC(X,Y)$ is finite.
We assume also that $J$ is the atomic topology.   In this case, $\cT(J)$ consists of all of the morphisms of $\cC$.
By Proposition~\ref{cor:grid existence}, 
there exists a grid $(\cC_0, \iota_0)$ of $(\cC, J)$.
Set $u$ to be the object in $\ProC$ defined by $\{\iota_0(X)\}_{X\in \cC_0}$.
We will show below that this object satisfies Caramello's criteria (note that we work with the opposite category).

The main theorem in \cite{Caramello}
is Theorem 3.5, but we look instead at the author's Proposition 4.3,
which is very close to the former theorem, 
and closer to our setup.


Caramello considers two conditions for the underlying category $\cC$:
the amalgamation property (AP) (p.660), and
the joint embedding property (JEP).
These are satisfied in our setup 
by the definition of $Y$-site.
The author also considers an object $u$ 
in the ind-category 
that is $\cC$-ultrahomogeneous.
We construct such an object below.

\begin{prop}
\label{prop:ultrahomogeneous}
The object $u \in \ProC$ is $\cC$-ultrahomogeneous in the sense of opposite of \cite[p.661]{Caramello}
\end{prop}
\
We regard an object in $\cC$ as an object in $\ProC$ in a natural manner (the constant).
Let $a \in \cC$.   By definition,
\[
\Hom_{\ProC}(u, a)= 
\varinjlim_{X \in \cC_0}
\Hom_\cC(\iota_0(X), a).
\]
Suppose we are given $u \xto{f} a$ and $u\xto{g} a$.
We wish to find a morphism $u \xto{\alpha_{f,g}} u$
such that $f=g \circ \alpha_{f,g}$.

Let $g' \in \Hom_{\cC}(\iota_0(X'), a)$
be a representative of $g$.
We set
\[
S_{g'}=
\{h \in \Hom_{\ProC}(u,\iota_0(X')) \,|\, g' \circ h=f
\}.
\]

\begin{lem}
\label{lem:lift non-empty}
For any representative $g'$, the set $S_{g'}$ is non-empty.
\end{lem}
\begin{proof}
Take a representative $f': \iota_0(Y_0) \to a$ of $f:u\to a$.
As the topology is atomic, $g \in \cT(J)$.
Then, by Definition~\ref{defn:Y-sites}(1)
there exists a commutative diagram:
\[
\begin{CD}
Z @>{\alpha}>> \iota_0(X')
\\
@VVV    @VV{g'}V
\\
\iota_0(Y_0) @>>{f'}> a.
\end{CD}
\]
By Definition~\ref{defn:grids}(3),
there exists $f_0: Z_0 \to Y_0 \in \cC_0$ and
a commutative diagram extending the preceding diagram:
\[
\begin{CD}
\iota_0(Z_0) @>{\cong, \beta}>> Z @>{\alpha}>> \iota_0(X')
\\
@V{\iota_0(f_0)}VV                       @VVV        @VV{g'}V
\\
\iota_0(Y_0) @>{=}>>   \iota_0(Y_0)  @>>{f'}> a.
\end{CD}
\]
Note that $f' \circ \iota_0(f_0)$ represents $f: u\to a$.
Let $\gamma: u\to \iota_0(X')$ be the morphism
represented by $\alpha\circ \beta$.
Then, by definition, we have $f=g' \circ \gamma$; hence, the claim
follows.
\end{proof}
\begin{lem}
The set $S_{g'}$ is finite.
\end{lem}
\begin{proof}
As $(\cC,J)$ is a $Y$-site, there are enough coverings
(Definition~\ref{defn:Y-sites}(3), 
Definition~\ref{defn:enough Galois}). 
Hence, there exists a morphism 
$\varphi: S \to \iota_0(X')$,
which is a Galois covering of 
Galois group $G$
such that
the composite $g' \circ \varphi$
is also a Galois covering of, e.g., group $H$.
From the finiteness assumption, 
it follows that 
the groups $G$ and $H$ are finite.
Let $T_0 \in \cC_0$ 
and consider the maps
\[
\Hom_\cC(T_0, S) \to
\Hom_\cC(T_0, \iota_0(X')) \xto{u_{T_0}}
\Hom_\cC(T_0, a)
\]
obtained by the compositions with $\varphi$ and with $g'$.
The fiber $u_{T_0}^{-1}(g')$ of the map $g' \in \Hom_\cC(T_0, a)$ is then 
isomorphic to $G/H$.
Now take a morphism $\delta: T_0' \to T_0 \in \cC_0$.
In this case, there is a map $u_{T_0}^{-1}(g') \to u_{T_0'}^{-1}(g')$
which respects the $G$-structure.
This proves that the cardinality of $S_{g'}$ is less than the cardinality of $G$. 
\end{proof}

\begin{lem}
The sets $S_{g'}$ form a filtered projective system.
\end{lem}
\begin{proof}
Let 
$u \to \iota_0(X') \in S_{g'}$.
Suppose that we are given a morphism 
$\psi: Q \to X'$.
We want to find a morphism
$u \to \iota_0(Q)$ 
that, 
when composed with $\iota_0(\psi)$,
gives $f: u \to a$.
This can be constructed as was done in the proof of 
Lemma~\ref{lem:lift non-empty}.
\end{proof}

\begin{proof}[Proof of Proposition~\ref{prop:ultrahomogeneous}]
It follows from the previous lemmas that 
$\varprojlim_{g'} S_{g'}$ 
is nonempty, as it is the projective limit of 
nonempty finite sets.
Take any element in 
$\varprojlim_{g'} S_{g'}\subset \Hom_{\ProC}(u,u)$.
\end{proof}

\section{An example where the fiber functor is not full}
\label{sec:fiber not full}
Our main theorem (Theorem~\ref{thm:Galois_main}) says that
under certain cardinality conditions the fiber functor 
$\omega_{(\cC_0, \iota_0)}$
associated to a grid $(\cC_0, \iota_0)$
is an equivalence of categories.
We have also seen the cases (Proposition~\ref{prop:locally prodiscrete})
where the 
conditions on cardinality 
are not satisfied yet the fiber functor gives an 
equivalence.   This is the case where the absolute Galois monoid is a locally prodiscrete group, and the equivalence was proved by a different 
method than our theorems.   
In this section, 
we give examples where we are given a $Y$-site and a grid yet the fiber 
functor does not induce an equivalence.  

The key input in the construction is a pro-group $(G_i)$ 
with surjective transition maps such that the limit in the 
category of topological groups (each $G_i$ is a discrete 
topological group) is trivial  (cf. \cite[p.223, Section 1.3]{M}).

We refer to Section~\ref{sec:full ess surj} 
where we give a sufficient condition 
in terms of derived limits 
for the fiber functor 
to be full and essentially surjective.
In this section, we treat the case where the surjectivity of \eqref{eq:canproj_surj}
fails.
\subsection{}
A pro-group is by definition a functor
$\alpha: J \to (Grp)$
from a cofiltered category to the category of groups.
For an object $j$ of $J$, we sometimes write $G_j$ for $\alpha(j)$.
We assume that the transition maps are surjective.
\subsection{the $Y$-site and the grid}
Let $\alpha$ be a pro-group.
We consider the following category $\cC$.
The set of objects of $\cC$ is the same as the set of 
objects of $J$.
We define the morphisms as follows:
\[
\Hom_\cC(j, j')
=
\left\{
\begin{array}{ll}
\alpha(j)  & \text{if } \Hom_J(j,j')\neq \emptyset,
\\
\emptyset &\text{otherwise.}
\end{array}
\right.
\]
The composition is defined using the group structure of each $\alpha(j)$.
Since $J$ is cofiltered, the category $\cC$ is semi-cofiltered 
(Definition~\ref{defn:semi-cofiltered}).
Let $J$ be the atomic topology on $\cC$.
Using that each $\alpha(j)$ is a group, 
one can check that 
the category $\cC$ is an $E$-category.
It follows that $(\cC, J)$ is a $B$-site.

Let $f: j \to j'$ be a morphism in $\cC$,
i.e., $f$ is an element of $G_{j'}$.
One can check that 
$\Aut_f(j)=\Ker(G_{j} \to G_{j'})$.
Thus, we see that every morphism in $\cC$ is 
a Galois covering.
By definition, $(\cC, J)$ is then a $Y$-site.

We construct a grid for this $Y$-site as follows.
Let $\cC_0$ be the category $J$.
We define a functor 
$\iota_0: \cC_0 \to \cC$.
Let $j \to j'$ be a morphism in $\cC$.
Then we define $\iota_0(j\to j')$ 
to be the unit element in $G_{j'}=\Hom_\cC(j, j')$.
Since the topology is atomic, 
to show that $(\cC_0, \iota_0)$
is a grid, it suffices to check that 
it is a pregrid.    This follows by checking the conditions
directly.

We may without loss of generality assume that 
$J$ has a final object corresponding to the trivial 
group.
This in turn gives a final object $j_0$ of $\cC$.
Then 
the absolute Galois group 
is isomorphic to $\bK_{j_0}$,
which in turn is isomorphic to 
$H_X= \varprojlim_{f \in I_X} \Gal(f)$
by Lemma~\ref{lem:psi_isom} (notation in Section~\ref{sec:Gal gp surjective}).
This means that the absolute Galois monoid is
isomorphic to the limit $\varprojlim_{j \in J} G_j$.

\subsection{}
The aim of this section is to prove the following proposition.
\begin{prop}
Let $\alpha$ be a pro-group with surjective 
transition maps such that
there exists some $j \in J$ for which $G_j$ is nontrivial and 
that 
$\varprojlim_{j\in J} G_j$
is trivial.   
Let $(\cC, J)$ be the Y-site and 
$(\cC_0, \iota_0)$ be the grid constructed in the
previous section.
Then the associated fiber functor
$\omega=\omega_{(\cC_0, \iota_0)}$ 
is not full.
\end{prop}
\begin{proof}
By the argument above, the absolute Galois monoid is trivial.
This implies that 
the category of smooth sets equals the category of sets.

Let $j \in \cC$ be an object.   We claim that the 
presheaf represented by $j$ is a sheaf.
It suffices to check that
for any morphism $f: j_1 \to j_2$
(which is Galois as observed above),
$h_j(j_1)^{\Gal(f)}=h_j(j_2)$.
This holds true because if one side is not 
empty, then both sides equal $G_j$.
We also see that $\omega(h_j)=G_j$.

By Yoneda's lemma (for sheaves),
 we have
\[
\Hom_{\Shv(\cC)}(h_j, h_j) \cong \Hom_\cC(j,j) =G_j .
\]
On the other hand, we have
\[
\Hom_{\text{smooth} M\text{-sets}}(\omega(h_j), \omega(h_j))
=
\Hom_{\text{Sets}}(\omega(h_j), \omega(h_j))
=
\Hom_{\text{Sets}}(G_j, G_j).
\]
These are not equal when $G_j$ is not trivial, hence the functor $\omega$ is not full as claimed.
\end{proof}

\section{$Y$-sites and grids for locally prodiscrete monoids}
\label{sec:locally prodiscrete monoids}
We define a class of topological monoids $M$ (admissible topological monoids) 
which includes the class of groups of Section~\ref{sec:certain top groups}
(in particular, prodiscrete groups of Definition~\ref{def:prodiscrete groups}).
We construct a $Y$-site and a grid such that the associated absolute Galois monoid is isomorphic to $\wh{M}$.    This generalizes Lemma~\ref{lem:Ghat isom to M}.
The fiber functor is automatically faithful (Theorem~\ref{thm:Galois_main}).
We show that the fiber functor is an equivalence.

We record this result in a separate section than Section~\ref{sec:locally prodiscrete} because the result in this section was obtained during revision.
We will use the argument used in this section in Section~\ref{sec:fiber to grid}.

\subsection{Admissible topological monoids}

\subsubsection{}
For a topological monoid $M$, we use the following terminology:
an open submonoid of $M$ which is a group 
is called an open subgroup of $M$.
We let $\mathfrak{V}(M)$ denote the set of
open subgroups $H \subset M$ satisfying the
following property: for any open submonoid 
$U \subset M$, there exists
an open subgroup $K\subset M$ such that $K \subset U \cap H$ 
and $K$ is a normal subgroup of $H$.

\begin{defn}
Let $M$ be a topological monoid.
\begin{enumerate}
\item We say that $M$ is left-integral
if for any element $m \in M$, the subset
$mM$ of $M$ is open and the map $m\cdot - : M \to M$ given 
by the multiplication by $m$ from the left 
induces a homeomorphism from $M$ to $mM$.
\item We say that $M$ is admissible if $M$ is left-integral
and $\mathfrak{V}(M)$ is non-empty.
\item We say that $M$ is locally prodiscrete if
$M$ is left-integral and there exists an open subgroup 
$U \subset M$ which is a prodiscrete group.
\end{enumerate}
\end{defn}

We note that any prodiscrete monoid $M$ is admissible,
since any prodiscrete open subgroup is an element of $\mathfrak{V}(M)$.

\subsubsection{The category $BM$}
For a topological monoid $M$, we let $BM$ denote the
category of discrete sets equipped with continuous
left $M$-actions.
\begin{lem}
Let $H$ be an open subgroup of an admissible 
topological monoid $M$. Then the quotient set $M/H$ 
equipped with the quotient topology is an object of $BM$.
\end{lem}

\begin{proof}
The left-integrality of $M$ implies that
$m H$ is an open subset of $M$ for any $m \in M$.
This shows that $M/H$ is discrete and that 
for any $m \in M$ there
exists an open neighborhood $U$ of $1 \in M$
satisfying $U m \subset m H$.
Hence $M/H$ is an object of $BM$.
\end{proof}

\subsection{The construction of a $Y$-site and a grid}
We can construct a $Y$-site and a grid, 
starting from an admissible topological monoid $M$.

\subsubsection{Statements}
We say that an object $V$ of $BM$ is strongly homogeneous if
$V$ is isomorphic to $M/H$ for some $H \in \mathfrak{V}(M)$.
Let $\cC = \cC_M$ denote the full subcategory of
$BM$ whose objects are the strongly homogeneous object of $V$.

The following proposition will be proved in
Section \ref{sec:CM_E-cat}:

\begin{prop} \label{prop:CM_E-cat}
$\cC$ is an $E$-category, i.e., any morphism
in $\cC$ is an epimorphism.
\end{prop}

\begin{defn} \label{defn:CM_type_J}
Let $V$, $W$ be objects of $\cC$ and
let $f: V \to W$ be a morphism in $\cC$.
We say that $f$ is of type $J$ if
$f$ is surjective.
\end{defn}

Let $\cT$ denote the collection of morphisms
of type $J$ in $\cC$.
The following theorem will be proved in
Section \ref{sec:CM_Y-site}:

\begin{thm} \label{thm:CM_Y-site}
The collection 
$\cT$ is semi-localizing and 
$\wh{\cT} = \cT$
holds.
Moreover, the pair $(\cC,J_\cT)$ is a $Y$-site.
\end{thm}

Let $\cC'_0$ denote the following category.
The objects of $\cC'_0$ are the pairs
$(V,a)$ of an object $V$ of $\cC$ and an element $a \in V$.
For two objects $(V,a)$, $(W,b)$ of $\cC'_0$,
the morphisms from $(V,a)$ to $(W,b)$ in $\cC'_0$
are the morphisms $f:V \to W$ in $\cC$ satisfying
$f(a)=b$.
Let $\iota'_0 : \cC'_0 \to \cC$ denote the functor
that sends an object $(V,a)$ of $\cC'_0$ to $V$.
The following proposition will be proved in 
Section \ref{sec:CM_quasi-poset}.

\begin{prop} \label{prop:CM_quasi-poset}
The category $\cC'_0$ is a quasi-poset.
\end{prop}

Let us choose a skeletal subcategory $\cC_0$ of
$\cC'_0$ such that the inclusion functor $\cC_0 \inj \cC'_0$
is an equivalence of categories.
Let $\iota_0 :\cC_0 \to \cC$ denote the restriction of 
$\iota'_0$ to $\cC_0$.

The following theorem will be proved in 
Section \ref{sec:CM_grid}

\begin{thm} \label{thm:CM_grid}
$(\cC_0,\iota_0)$ is a grid of $(\cC,J_\cT)$.
\end{thm}

\subsubsection{Preliminary}
For an admissible topological monoid $M$, we let
$M^*$ denote the set of elements of $M$ having
left inverses.

\begin{lem} \label{lem:M*}
Let $M$ be an admissible topological monoid.
Then $M^*$ is an open subgroup of $M$.
\end{lem}

\begin{proof}
Since $\mathfrak{V}(M)$ is non-empty,
there exists an open subgroup $H \subset M$.
We then have $H \subset M^*$.
The left-integrality of $M$ 
implies that $m H$ is an open subset of $M$ for any $m \in M$.
This shows that $M^*$ is an open submonoid of $M$.
It remains to prove that any element of $M^*$ is
invertible.
Let $m \in M^*$ and let $n$ be a left inverse of $m$.
Since $nm=1$ we have $nmn = n$.
Then the left-integrality of $M$ implies that
$mn=1$. Hence $m$ is invertible.
\end{proof}

Let $V$ be a strongly homogeneous object of $BM$.
An element $v \in V$ is called a generator of $V$
if the map $M \to V$ that sends $m \in M$ to $mv$
is surjective.
Let $V^* \subset V$ denote the set of generators of $V$.

\begin{lem}
Let $V$ be an object of $\cC$. Let us choose an element
$H \in \mathfrak{V}(M)$ and an isomorphism
$\alpha : M/H \xto{\cong} V$. Then we have $V^* = \alpha(M^*/H)$.
\end{lem}

\begin{proof}
Let $m\in M$. Then it follows from the definition
of a generator that $m H$ is a generator of $M/H$
if and only if $m$ has a left inverse. Hence
the claim follows.
\end{proof}

\begin{cor} \label{cor:CM_U_open}
Let $V$ be an object of $\cC$ and let 
$v \in V^*$. Then the set
$U = \{m \in M\ |\ mv =v \}$ is an open subgroup of
$M$ that belongs to $\mathfrak{V}(M)$.
\end{cor}

\begin{proof}
We may assume that $V= M/H$ for some $H \in \mathfrak{V}(M)$.
Choose an element $m\in M^*$ satisfying $v = m H$.
Then we have $U = m H m^{-1}$.
This shows that $U$ is an open subgroup of $M$
that belongs to $\mathfrak{V}(M)$.
\end{proof}

\begin{lem} \label{lem:CM_f_property}
Let $f:V \to W$ be a morphism in $\cC$.
Let $w \in W^*$ and suppose that there exists
an element $v \in V$ satisfying $f(v)=w$.
Then $f$ is of type $J$ and $v \in V^*$.
\end{lem}

\begin{proof}
It is clear the $f$ is of type $J$.
Let us choose $v_0 \in V^*$ and let us
choose $m, m' \in M$ satisfying
$v =m v_0$ and $f(v_0) = m' w$.
Then we have $m m' w = w$.
Hence it follows from Corollary \ref{cor:CM_U_open}
that $m m' \in M^*$.
This shows that $m'$ has a left inverse.
Hence $m' \in M^*$. Since $M^*$ is a group,
we have $m \in M^*$. Hence $v = mv_0 \in M^*$.
\end{proof}

\begin{lem} \label{lem:CM_type_J_B-site}
Let $f:V \to W$ and $g:W \to Z$ be morphisms in $\cC$.
Then the composite $g \circ f$ is of type $J$
if and only if $f$ and $g$ are of type $J$.
\end{lem}

\begin{proof}
The ``if" part is clear.
We prove the ``only if" part.

Suppose that $g \circ f$ is of type $J$.
It is clear that $g$ is of type $J$.
We prove that $f$ is of type $J$.
Let us choose $z \in Z^*$ and
$v \in V$ satisfying $g(f(v)) =z$.
It then follows from Lemma \ref{lem:CM_f_property} that
$f(v) \in W^*$. This shows that $f$ is of type $J$.
\end{proof}

\begin{lem} \label{lem:CM_W^*}
Let $(V,a)$ be an object of $\cC'_0$.
Then there exists a morphism 
$f:(W,b) \to (V,a)$ in $\cC'_0$ satisfying
$b \in W^*$.
\end{lem}

\begin{proof}
We may assume that $V=M/H$ for some $H \in \mathfrak{V}(M)$.
Let us choose an element $m_0 \in M$ satisfying $a = m_0 H$.
Since the set $U = \{m \in M \ |\ m a = a\}$
is an open submonoid of $M$, there exists an open
subgroup $K \subset M$ satisfying $K \subset U \cap H$.
It is easy to check that $K \in \mathfrak{V}(M)$.
Let $W = M/K$. Then the map $M \to M$ given by
the multiplication by $m_0$ from the right
induces a morphism $f:W \to V$ in $\cC$.
Then the object $(W,1\cdot K)$ of $\cC'_0$ has
the desired property.
\end{proof}

\subsubsection{Proof of Proposition \ref{prop:CM_E-cat}} \label{sec:CM_E-cat}

\begin{proof}
Let $f:V \to W$ be a morphism in $\cC$,
and let $g_1, g_2 : W \to Z$ be morphisms in $\cC$
satisfying $g_1 \circ f = g_2 \circ f$.
It then suffices to prove that $g_1 = g_2$.
We may assume that there exist
$H,H',H'' \in \mathfrak{V}(M)$
satisfying $V=M/H$, $W= M/H'$, and $Z =M/H''$.
Let us choose $m,m_1,m_2 \in M$ such that
$f$ sends $1 \cdot H$ to $m \cdot H'$
and that $g_i$ sends $1 \cdot H'$ to
$m_i \cdot H''$ for $i \in \{1,2\}$.
Since $g_1 \circ f = g_2 \circ f$, we have
$m m_1 H'' = m m_2 H''$.
Since $M$ is left-integral, we have
$m_1 H'' = m_2 H''$. Hence we have $g_1=g_2$.
\end{proof}

\subsubsection{Proof of Theorem \ref{thm:CM_Y-site}} \label{sec:CM_Y-site}

\begin{proof}
First we prove that $\cT$ is semi-localizing.
Let $V_1 \xto{f_1} V \xleftarrow{f_2} V_2$
be a diagram in $\cC$ such that $f_2$ is of type $J$.
We prove that there exists an object $W$ of $\cC$
and morphisms $g_i: W \to V_i$ for $i\in \{1,2\}$
such that $g_1$ is of type $J$ and $f_1 \circ g_1 
= f_2 \circ g_2$.
To prove this, we may assume that there exist
$H,H_1,H_2 \in \mathfrak{V}(M)$
with $H_2 \subset H$ satisfying $V=M/H$, $V_1= M/H_1$, 
$V_2 =M/H_2$ and that $f_2 : V/H_2 \to V/H$ 
is the canonical quotient map.
Let us choose $m \in M$ such that
$f_1 : M/H_1 \to M/H$ sends $1 \cdot H_1$ to $m H$.
Since $M$ is left integral, the set
$U = \{ n \in M\ |\ nm \in m H_2\}$ is an open
mononid of $M$. Hence there exists an open subgroup
$H'$ of $M$ satisfying $H' \subset H_1 \cap U$.
Set $W = M/H'$ and write $g_1$ for the canonical
quotient map $W \to V_1$.
Let $g_2 : W \to V_2$ denote the map induced by
the map $M \to M$ given by the multiplication by $m$
from the right.
Then $g_2$ is well-defined and $W$, $g_1$, and $g_2$
have the desired property.

Next we prove that $\cT =\wh{\cT}$.
Let $f: V \to W$ and $g: Z \to V$ be morphisms in $\cC$.
Then it is clear from definition that if
$f \circ g$ is of type $J$, then $f$ is also of type $J$.
Hence we have $\cT = \wh{\cT}$.
Next we prove that $(\cC,J_\cT)$ is a $B$-site.
It follows from Proposition \ref{prop:CM_E-cat}
and Lemma \ref{lem:CM_type_J_B-site} that $(\cC,J_\cT)$ is a $B$-site.

Next we prove that $\cC(\cT)$ is $\Lambda$-connected.
Let $V_1$ and $V_2$ be objects of $\cC$.
We prove that there exist an object $V$ of $\cC$
and morphisms $f_1: V \to V_1$ and $f_2:V \to V_2$.
We may assume that $V_1= M/H_1$ and $V_2 = M/H_2$
for some $V_1, V_2 \in \mathfrak{V}(M)$.
Set $H = H_1 \cap H_2$. Then we have $H \in \mathfrak{V}(M)$.
Hence $V := M/H$ is an object of $\cC$ and we have
canonical projections $V \to V_i$ for $i \in \{1,2\}$.

It remains to prove that $\cT$ has enough Galois
coverings. Let $f: V \to W$ be a morphism in $\cC(J_\cT)$.
We prove that there exists a morphism 
$g : V' \to V$ in $\cC(J_\cT)$ such that
the composite $f\circ g$ is a Galois covering in $\cC$.
We may assume that there exist $H, H' \in \mathfrak{V}(M)$
with $H \subset H'$ satisfying $V=M/H$, $W=M/H'$, and
that $f:M/H \to M/H'$ is the canonical quotient map.
Since $H' \in \mathfrak{V}$, there exists an open
subgroup $K \subset M$ such that $K \subset H$ and
$K$ is normal in $H'$.
We set $V' = M/K$ and write $g'$ for the canonical
quotient map $V' \to V$. Then $g$ satisfies the
desired property.
\end{proof}

\subsubsection{Proof of Proposition \ref{prop:CM_quasi-poset}} \label{sec:CM_quasi-poset}

\begin{proof}
It is clear that the category $\cC'_0$ is 
essentially $\frU$-small.
Let $(V,a)$ and $(W,b)$ be two objects of $\cC'_0$.
It suffices to show that there exists at most one
morphism from $(V,a)$ to $(W,b)$.
Let $f_1,f_2 : (V,a) \to (W,b)$ be two morphisms in $\cC'_0$.
We prove that $f_1 = f_2$.
It follows from Lemma \ref{lem:CM_W^*} that
there exists a morphism $g:(Z,c) \to (V,a)$
satisfying $c \in Z^*$.
The two morphisms $f_1 \circ g$ and $f_2 \circ g$
send $c \in Z$ to $b \in W$.
Since $c \in Z^*$, we have $f_1 \circ g = f_2 \circ g$.
It follows from Proposition \ref{prop:CM_E-cat}
that $g$ is an epimorphism as a morphism in $\cC$.
Hence we have $f_1 = f_2$.
\end{proof}

\subsubsection{Proof of Theorem \ref{thm:CM_grid}} \label{sec:CM_grid}

\begin{lem} \label{lem:CM_edge_objects}
Let $(V,a)$ is an object of $\cC'_0$.
Then $(V,a)$ is an edge object if and only if $a \in V^*$.
\end{lem}

\begin{proof}
First we prove the ``if" part.
Suppose that $a \in V^*$.
Let $f: (W,b) \to (V,a)$ be a morphism in $\cC'_0$.
Since $f(b)=a$ and $a \in V^*$, it follows that
$f$ is surjective. Hence $f$, regarded a morphism in $\cC$,
is of type $J$.

Next we prove the ``only if" part.
Suppose that $a \not \in V^*$.
We prove that there exists a morphism
$f: (W,b) \to (V,a)$ such that $f$, regarded as a morphism
in $\cC$, is not of type $J$.
Let $U = \{ m \in M\ |\ ma=a \}$.
Then $U$ is an open submonoid of $M$.
Since $\mathfrak{V}(M)$ is non-empty,
there exists $H \in \mathfrak{V}(M)$
satisfying $H \subset U$.
Set $(W,b) = (M/H,1 \cdot H)$.
Then it follows from the construction of $H$
that there exists a morphism
$f: W \to V$ that sends $b$ to $a$.
Since $b \in W^*$, and $a \not\in V^*$ it follows that
$f$ is not surjective. Hence $f$ is a morphism from
$(W,b)$ to $(V,a)$ in $\cC'_0$ which is,
regarded as a morphism in $\cC$, is not of type $J$.
This completes the proof.
\end{proof}

\begin{proof}[{Proof of Theorem \ref{thm:CM_grid}}]
It is clear from the definition of $\cC_0$ that
$\cC_0$ is a poset.
Hence it suffices to prove that $\cC'_0$ is $\Lambda$-connected
and the pair $(\cC'_0,\iota'_0)$ satisfies Conditions (2), (3), and (4)
in Definition \ref{defn:grids}.

First we prove that $\cC'_0$ is $\Lambda$-connected.
Let $(V_1,a_1)$ and $(V_2,a_2)$ be two objects of $\cC'_0$.
We prove that there exist an object $(V,a)$ of $\cC'_0$
and morphisms $f_1 :(V,a) \to (V_1,a_1)$ and
$f_2 : (V,a) \to (V_2,a_2)$ in $\cC'_0$.
For $i =1,2$, the set
$\{ m \in M\ |\ ma_i = a_i \}$ is an open submonoid
of $M$. Since $\mathfrak{V}(M)$ is non-empty,
there exists $H \subset \mathfrak{V}(M)$ satisfying
$H \subset U_1 \cap U_2$.
Set $V = M/H$ and $a =1 \cdot H \in V$. 
It follows from the construction
of $H$ that, for $i \in \{1,2\}$, 
there exists a morphism $f_i : V \to V_i$ in $\cC$
that sends $a$ to $a_i$.
Since $f_i$ is a morphism from $(V,a)$ to $(V_i,a_i)$,
the object $(V,a)$, and the morphisms $f_1$ and $f_2$
have the desired property.

Next we prove that $(\cC'_0,\iota'_0)$ satisfies Condition (2).
Let $V$ be an object of $\cC$.
Let us choose $a \in V^*$.
Then it follows from Lemma \ref{lem:CM_edge_objects} that
$(V,a)$ is an edge object of $\cC'_0$.
Hence $(\cC'_0,\iota'_0)$ satisfies Condition (2).

Next we prove that $(\cC'_0,\iota'_0)$ satisfies Condition (3).
Let $(V,a)$ be an object of $\cC'_0$ and let
$f:W \to V$ be a morphism of type $J$ in $\cC$.
Since $f$ is surjective, there exists $b \in W$
satisfying $f(b)=a$. Hence $f$ is a morphism
from $(W,b)$ to $(V,a)$, which proves that
$(\cC'_0,\iota'_0)$ satisfies Condition (3).

Next we prove that $(\cC'_0,\iota'_0)$ satisfies Condition (4).
Let $(V,a)$ be an object of $\cC'_0$.
Let $\iota'_{0,(V,a)/}$ 
denote the functor $\cC'_{0,(V,a)/} \to \cC_{V/}$
induced by $\iota'_0$.
We prove that $\iota'_{0,(V,a)/}$ is an equivalence of categories.
We note that, since $\cC'_0$ is a quasi-poset
and $\cC$ is an essentially $\frU$-small $E$-category, 
both $\cC'_{0,(V,a)/}$ and $\cC_{V/}$ are quasi-posets.
Let $f:V \to W$ be a morphism in $\cC$.
Then $f$ is a morphism from $(V,a)$ to $(W,f(a))$
in $\cC'_0$. This prove that $\iota'_{0,(V,a)/}$ is essentially surjective.
It remains to prove that $\iota'_{0,(V,a)/}$ is fully faithful.
For $i \in \{1,2\}$ let $f_i: (V,a) \to (W_i,b_i)$
be two morphisms if $\cC'_0$.
Suppose that there exists a morphism 
$g: W_1 \to W_2$ in $\cC$ satisfying $g \circ f_1 = f_2$
in $\cC$. Since $f_i(a) = b_i$ for $i=1,2$,
we have $g(b_1) = b_2$, which shows that
$g$ is a morphism from $f_1$ to $f_2$ in
$\cC'_{0,(V,a)/}$.
This proves that $\iota'_{0,(V,a)/}$ is fully faithful.

This completes the proof.
\end{proof}

\subsection{The associated topological monoid}
Given an admissible topological monoid $M$,
we construct a locally prodiscrete monoid $\wh{M}$
as follows.

We set 
$\wh{M}=\lim_{H \in \mathfrak{V}(M)} M/H$
to be the limit of discrete sets $M/H$ in the category
of topological spaces. We can equip $\wh{M}$
with the structure of a topological monoid as follows. 
For two elements $m_1 = (m_{1,H})_{H \in \mathfrak{V}(M)}$,
$m_2 = (m_{2,H})_{H \in \mathfrak{V}(M)}$, we define the product
$m_1 m_2 \in \wh{M}$ as follows. 
Let $H \in \mathfrak{V}(M)$.
Since $m_2 H$ is an open subset of $M$,
the set $U = \{ m \in M\ | mm_2 \in m_2 H\}$ is
an open submonoid of $M$.
%
Since $H \in \mathfrak{V}(M)$, there exists
an open subgroup $K \subset M$ such that
$K \subset H \cap U$. It is easy to see that
$K \in \mathfrak{V}(M)$.
Let us take a lift $\wt{m}_{1,K} \in M$ 
of $m_{1,K} \in M/K$.
We set the H-component $(m_1 m_2)_H$ of $m_1 m_2$ to be 
$\wt{m}_{1,K} m_{2,H} \in M/H$.
The element $(m_1 m_2)_H$ of $M/H$ does
not depend on the choice of $K$ 
or the lift $\wt{m}_{1,K}$ and, when $H$ varies,
the elements $(m_1 m_2)_H$ gives an element
$m_1 m_2$ of $\wh{H}$.

It is not obvious but one can check
that $\wh{M}$ is a prodiscrete monoid.
There is a canonical morphism of topological monoids
$M \to \wh{M}$.
It is clear from the definition of $\wh{M}$ that,
if $M$ is a prodiscrete monoid, then the morphism
$M \to \wh{M}$ is an isomorphism.
We note that the morphism $M \to \wh{M}$ induces 
an equivalence $B\wh{M} \xto{\cong} BM$ of categories.

\begin{prop}
\label{prop:Galois monoid locally prodiscrete}
A monoid $M$ 
is the absolute Galois monoid $M_{(\cC_0, \iota_0)}$
associated to a grid $(\cC_0, \iota_0)$ of some $Y$-site
if and only if 
$M$ is locally prodiscrete.
\end{prop}
\begin{proof}
We have already mentioned that the ``if" part" is true. 
We prove the ``only if " part.
Suppose that a $Y$-site  $(\cC,J)$ has a grid $(\cC_0,\iota_0)$. 
We prove that the absolute Galois monoid $M_\Cip$ is locally prodiscrete.
It follows from Lemma \ref{lem:MC2} 
and the definition of the topology on $M_\Cip$ given in Section \ref{sec:topological monoid structure} 
that $M_\Cip$ is left integral.
Let us choose an edge object $X$ of $\cC_0$.
It suffices to prove that $\bK_X$ is a prodiscrete group. 
In Section \ref{sec:Gal gp surjective} 
we have introduced a prodiscrete group $H_X$ 
and constructed an isomorphism $\psi_X : H_X \xto{\cong} \bK_X$ as abstract groups. 
It suffices to show that $\psi_X$ is an isomorphism of topological groups.
It follows from the construction of $\psi_X$ that, 
for any object $f:Y\to X$ of $I_X$,
the image under $\psi_X$ of the kernel of $H_X \to \Gal(f)$ 
is equal to $\bK_Y$. 
Since $\{ \Ker\, H_X \to \Gal(f)\ | \ f \in I_X \}$ and $\{\bK_Y\ |\ (f: Y \to X) \in I_X\}$ 
form fundamental systems of neighborhoods of $1$ in $H_X$ and $\bK_X$, respectively, 
the isomorphism $\psi_X$ is an isomorphism of topological groups.
This completes the proof.
\end{proof}

The following is the
main result of this paragraph:
 \begin{thm}
\label{thm:locally prodiscrete monoid}
Let $M$ be an admissible topological monoid.
Then the absolute Galois monoid $M_\Cip$ 
associated with the grid $(\cC_0,\iota_0)$
constructed above is isomorphic to $\wh{M}$.
\end{thm}

\begin{proof}
First we construct a continuous homomorphism
$\xi : \wh{M} \to M_\Cip$.
Let $m=(m_H)_{H \in \mathfrak{V}(M)}$ be an element of 
$\wh{M}$.
Let $v = (V,a)$ be an object of $\cC_0$.
Let us choose $H \in \mathfrak{V}(M)$ satisfying
$H \cdot a =a$ and set $m a = m_H a$. 
Note that $m a \in V$ is independent of
the choice of $H$.
Let $\alpha(v)$ 
denote the unique object of $\cC_0$
that is isomorphic to $(V, m a)$ in $\cC'_0$.
Let $\gamma_v : (V,ma) \xto{\cong} \alpha(v)$
denote the unique isomorphism in $\cC'_0$.
If $g: v'= (V',a') \to v$ is a morphism in $\cC_0$,
it is then easy to see that $g$ is also a morphism
from $(V', m a')$ to $(V,ma)$ in $\cC'_0$.
By sending $v$ to $\alpha(v)$ and by sending
$f: v \to v'$ to $\gamma_v \circ f \circ \gamma_{v'}^{-1}$,
we obtain a functor $\alpha : \cC_0 \to \cC_0$.
By associating any object $v$ of $\cC_0$ to an 
isomorphism $\gamma_v$, we obtain a natural 
isomorphism $\gamma_\alpha : \iota_0 \to \iota_0 \circ \alpha$.
By sending $m \in \wh{M}$ to the pair
$(\alpha,\gamma_\alpha)$ constructed above, we obtain a
map $\xi : \wh{M} \to M_\Cip$.
It is straightforward to check that
$\xi$ is a continuous homomorphism of topological monoids.

Next we construct a continuous homomorphism
$z: M_\Cip \to \wh{M}$.
Let $(\alpha,\gamma_\alpha) \in M_\Cip$.
For $H \in \mathfrak{V}(M)$, let $v_H$ denote the
unique object of $\cC_0$ such that there exists
an isomorphism $\beta: (M/H,1 \cdot H) \xto{\cong} v_H$
in $\cC'_0$. We note that such an isomorphism $\beta$
is unique since $\cC'_0$ is a quasi-poset.
Set $\alpha(v_H) = (W,b)$ 
and let $m_H \in M/H$ denote the unique coset that is
mapped to $b$ under the composite 
$\gamma_\alpha(v_H) \circ \beta: M/H \xto{\cong} W$.
When $H,H' \in \mathfrak{V}(M)$ with $H' \subset H$,
the canonical quotient map
$M/H' \to M/H$ sends $m_{H'}$ to $m_H$.
Hence the family $(m_H)_{H \in \mathfrak{V}(M)}$
gives an element $m \in \wh{M}$.
By sending $(\alpha,\gamma_\alpha)$ to $s$ we obtain
a map $z : M_\Cip \to \wh{M}$.
It is straightforward to see that $z$ is a 
continuous homomorphism of topological monoids
and that the composite $z \circ \xi$ is equal to
the identity map.

It remains to prove that the composite
$\xi \circ z$ is equal to the identity map.
Let $(\alpha,\gamma) \in M_\Cip$ and set
$m = (m_H)_{H \in \mathfrak{V}(M)} = z(\alpha,\gamma)$
and $(\alpha',\gamma'_{\alpha'}) = \xi(m)$.
We prove that $(\alpha',\gamma'_{\alpha'})
= (\alpha,\gamma_\alpha)$.
Let $v = (V,a)$ be an edge object of $\cC_0$.
Then $v = v_H$ for some $H \in \mathfrak{V}(M)$.
Let us write $\alpha(v) = (W,b)$.
It follows from the construction of $m$ that
$\gamma_\alpha(v) : V \to W$ sends $ma$ to $b$.
Set $\alpha'(v) = (W',b')$.
It follow from the construction of $(\alpha',\gamma'_{\alpha'})$
that $(V,ma)$ is isomorphic to $(W',b')$ and
$\gamma'_{\alpha'}(v) : V \to W'$ sends $ma$ to $b'$.
Since $\gamma'_{\alpha'}(v) \circ \gamma_\alpha(v)$
gives an isomorphism $(W,b) \xto{\cong} (W',b')$
in $\cC_0$ and $\cC_0$ is a poset, we have
$(W,b) = (W',b')$. Since both $\gamma_\alpha(v)$
and $\gamma'_{\alpha'}(v)$ sends $ma \in V$ to $b$,
we have $\gamma_\alpha(v) = \gamma'_{\alpha'}(v)$.

Now let $v=(V,a)$ be an arbitrary object of $\cC_0$.
It follows from Lemma \ref{lem:yama2} that there exists an edge object
$v_1=(V_1,a_1)$ of $\cC_0$ and a morphism
$f:v_1 \to v$ in $\cC_0$.
Set $\alpha(v) = (W,b)$, $\alpha'(v)=(W',b')$ 
and $\alpha(v_1) = \alpha'(v_1) = (W_1,b_1)$
Let us consider the commutative diagram 
$$
\begin{CD}
V @<{f}<< V_1 @>{f}>> V \\
@V{\gamma_\alpha(v)}V{\cong}V @V{\gamma_\alpha(v_1)}V{\cong}V 
@V{\cong}V{\gamma'_{\alpha'}(v)}V \\
W @<{\alpha(f)}<< W_1 @>{\alpha'(f)}>> W'
\end{CD}
$$
It follows from the definition of $m$ that
$\gamma_\alpha(v_1)$ sends $m a_1$ to $b_1$.
Since $f(a_1)=a$ and $\alpha(f)(b_1)=b$,
we have $\alpha(v)(ma) = b$.
By construction of $(\alpha', \gamma'_{\alpha'})$,
$\alpha'(v)$ sends $ma \in V$ to $b' \in W'$.
Hence both $\alpha(v)$ and $\alpha'(v)$ are objects
of $\cC'_0$ that are isomorphic to $(V,ma)$ in $\cC'_0$.
Since $\cC_0$ is a skeletal full subcategory of $\cC'_0$,
we have $\alpha(v) = \alpha'(v)$.
Let us choose an edge object $v_2 = (V_2,a_2)$
in $\cC'_0$ and a morphism $g: v_2 \to (V,ma)$ in $\cC'_0$.
Since both $\gamma_\alpha(v) \circ g$ and
$\gamma'_{\alpha'}(v) \circ g$ are morphisms $V_2 \to W$
in $\cC$ that sends $a_2$ to $b$, we have
$\gamma_\alpha(v) \circ g = \gamma'_{\alpha'}(v) \circ g$.
Since $\cC$ is an $E$-category, we have 
$\gamma_\alpha(v) = \gamma'_{\alpha'}(v)$.
Since $v$ is arbitrary, we have $\alpha = \alpha'$
and $\gamma_\alpha = \gamma'_{\alpha'}$.
This completes the proof.
\end{proof}

\begin{prop}
\label{prop:equiv locally prodiscrete monoids}
Let $M$ be an admissible topological monoid.
Consider the $Y$-site and grid that were
constructed above.
Then the fiber functor
$\omega_\Cip:BM \to BM_\Cip$
is an equivalence of categories.
\end{prop}
\begin{proof}
Let us consider the following diagram of categories and functors:
\[
BM \xto{\omega_\Cip} BM_\Cip \xleftarrow{\cong} B\wh{M} \xto{\cong} BM.
\]
Here the second functor is the equivalence of categories 
given by the isomorphism $\xi : \wh{M} \to M_\Cip$ 
constructed in the proof of Theorem \ref{thm:locally prodiscrete monoid},
and the last functor is the equivalence of categories 
given by the homomorphism $M \to \wh{M}$.
Let $F : BM \to BM$ denote the composite. 
It suffices to prove that $F$ is isomorphic to the identity functor on $BM$. 
Let $V$ of $BM$. By definition of $F$, we have 
$F(V) = \varinjlim_{H \in \mathfrak{V}(M)} V^H$. 
Hence we have a bijection 
$\alpha_V : F(V) \xto{\cong} V$ that is functorial with respect to $V$.
It is straightforward from the definition of the homomorphism $\xi$ to see
that the bijection $\alpha_V$ is compatible with the actions of $M$.
Hence the collection $(\alpha_V)_V$ of bijections
gives an isomorphism of functors from $F$ to the identity functor on $BM$.
This completes the proof.
\end{proof}

\section{Construction of a grid from a fiber functor}
\label{sec:fiber to grid}
Given a $Y$-site and a grid, one can construct the absolute Galois monoid $M$ and the fiber functor to the category of smooth $M$-sets.
In this section, we give a statement in the converse direction.   

Suppose we are given a $Y$-site, an admissible topological monoid $M$,
and a functor from the associated topos to the 
category of continuous $M$-sets which is an equivalence.
The goal of this section is to construct a grid for the $Y$-site
whose associated fiber functor is the given functor.

\subsection{Setting and notation}
Let $(\cC, J)$ be a $Y$-site.
Let $M$ be an admissible topological monoid.
Let $BM$ denote the category of discrete left $M$-sets.
Suppose one has a functor
$\omega: \Shv(\cC, J) \to BM$.
We assume that $\omega$ is an equivalence.
Then we show that there exists a grid whose 
associated absolute Galois monoid is $M$ and 
the given equivalence is given by the fiber 
functor associated with the grid.

We use the following terminology: for two open subgroups
$H, H' \subset M$ with $H,H' \in \mathfrak{V}(M)$ and $H' \subset H$, 
the unique morphism $M/H' \to M/H$ in $BM$ that sends
the coset $1 \cdot H'$ to $1 \cdot H$ is called the
canonical quotient map.

\subsection{Construction of a grid}
We let $\omega_\cC :\cC \to BM$ denote the functor
that sends an object $X$ of $\cC$ to
$\omega(a_J(\frh_\cC(X)))$.
We define $\cC'_0$ to be the following category.
An object of $\cC'_0$ is a pair $(X, a)$ of an
object $X$ of $\cC$ and an element $b \in \omega_\cC(X)$.
A morphism $(X,a) \to (Y,b)$
is a morphism $f: X \to Y$ in $\cC$ such that
$\omega_\cC(f)$ sends $a$ to $b$.
The functor $\iota'_0:\cC'_0 \to \cC$ is defined by
$\iota'_0(X, a)=X$.
Let us choose a skeletal subcategory $\cC_0$ of $\cC'_0$
and let $\iota_0$ denote the restriction of $\iota'_0$
to $\cC_0$.
The following is the main result of this paragraph:

\begin{thm} \label{thm:15-Cip}
The pair $(\cC_0, \iota_0)$ is a grid for 
the $Y$-site $(\cC, J)$.
\end{thm}

\subsection{Preliminary}
\begin{lem} \label{lem:ac faithful}
The functor $\omega_\cC$ is faithful.
\end{lem}

\begin{proof}
Let $X,Y$ be objects of $\cC$.
It suffices to show that the map
\begin{equation} \label{eq:15-1}
\Hom_\cC(X,Y) \to \Hom_{\Shv(\cC,J)}
(a_J(\frh_\cC(X)), a_J(\frh_\cC(Y)))
\end{equation}
is injective. Let $F=\frh_\cC(Y)$.
By Yoneda's lemma and adjunction we have 
bijections $\Hom_\cC(X,Y) \cong F(X)$ and
$$
\Hom_{\Shv(\cC,J)}
(a_J(\frh_\cC(X)), a_J(F))
\cong \Hom_{\Presh(\cC)}
(\frh_\cC(X), a_J(F))
\cong a_J(F)(X),
$$
and the map \eqref{eq:15-1} can be identified with the
map $F(X) \to a_J(F)(X)$ given by the adjunction morphism
$F \to a_J(F)$ of presheaves.
Since $\cC$ is an $E$-category, it follows that
$F(f) : F(X'_2) \to F(X'_1)$ is injective for any
morphism $f:X'_1 \to X'_2$ in $\cC$.
Hence it follows from the description \eqref{eq:a_J}
of the sheafification functor $a_J$ that the
map $F(X) \to a_J(F)(X)$ is injective,
which proves the claim.
\end{proof}

\begin{prop} \label{prop:B-site_epi}
Let $(\cC,J)$ be a $B$-site and let $f:Y \to X$ be
a morphism in $\cC$. Then $f$ belongs to $\cT(J)$ if and
only if $a_J \frh_\cC(f) : a_J \frh_\cC(Y) \to
a_J \frh_\cC(X)$ is an epimorphism in $\Shv(\cC,J)$.
\end{prop}

\begin{proof}
First we prove the ``only if" part.
Suppose that $f$ belongs to $\cT(J)$.
Then for an arbitrary sheaf $\cF$ on $(\cC,J)$,
the restriction map $f^* : \cF(X) \to \cF(Y)$
is injective.
Note that we have natural bijections
$\cF(X) \cong \Hom_{\Shv(\cC,J)}(a_J \frh_\cC(X),\cF)$.
It follows that $a_J \frh_\cC(f)$ is an epimorphism.

Next we prove the ``if" part.
Suppose that $a_J \frh_\cC(f)$ is an epimorphism.
Let $a \in (a_J \frh_\cC(X))(X)$ denote the
section of the sheaf $a_J \frh_\cC(X)$
given by the identity morphism $\id:X \to X$
in $\cC$.
Since $a_J \frh_\cC(f)$ is an epimorphism,
there exists a morphism $g: Z \to X$ that belongs to $\cT(J)$
and an element $b \in (a_J \frh_\cC(Y))(Z)$ such that
The map $(a_J \frh_\cC(Y))(Z)
\to (a_J \frh_\cC(X))(Z)$ given by $f$ sends $b$ to the
pullback of $a$ with respect to $g$.
Let us choose a diagram $Z \xleftarrow{h} W \xto{h'} Y$
in $\cC$ with $h$ in $\cT(J)$ that represents $b$.
Then there exists a morphism $t: W' \to W$ in $\cT(J)$
which makes the diagram
$$
\begin{CD}
W' @>{h' \circ t}>> Y \\
@V{h \circ t}VV @VV{f}V \\
Z @>{g}>> X
\end{CD}
$$
commutative. Since $g$, $h'$ and $t$ belong to $\cT(J)$,
the composite $f \circ (h' \circ t)$ belongs to $\cT(J)$.
Since $(\cC,J)$ is a $B$-site, it follows that
$f$ belongs to $\cT(J)$.
\end{proof}

\begin{defn}
Let $\cD$ be a category having small colimits.
Let $X$ be an object of $\cD$. 
We say that $X$ is a cyclic object 
in $\cD$
if for any small (not necessarily filtered) 
inductive system $(Y_i)_{i \in I}$
of objects whose transition maps are monomorphisms,
the map $\varinjlim_{i \in I} \Hom_\cD(X,Y_i)
\to \Hom_{\cD}(X,\varinjlim_i Y_i)$ is bijective.
\end{defn}
\begin{rmk}
We give a remark to justify the terminology.
Let $BM$ be the category of continuous left $M$-sets.
An object $V$ of $BM$ is said to be generated by one element
if there exists a surjection $M \to V$ 
in the category of left $M$-sets, where 
we view $M$ canonically as a left $M$-set.
Then, an object in $BM$ is cyclic if and only if it is generated by one element.
In the proof of Lemma~\ref{lem:15-isom}, we see the proof 
of the `if' part.   The proof of the `only if' part is omitted
as we will not use the claim.

It is immediate that a cyclic object is connected,
in the sense that 
$\Hom_\cD(X, \coprod_i Y_i) =\coprod_i \Hom_\cD(X, Y_i)$
holds for a family of objects $Y_i$ indexed by $i$.
A cyclic object $X$ need not be compact or finitely presented in the 
sense that the canonical map
$\varinjlim \Hom_\cD(X, Y_i) \to \Hom_\cD(X, \varinjlim Y_i)$
for a filtered inductive system $Y_i$ need not be bijective in general.
For example, consider the category of 
$R$-modules for a commutative ring $R$ that has an ideal $\mathfrak{a}$ 
which is not finitely generated.   Then the module $R/\mathfrak{a}$ is 
cyclic, connected, not compact, and not finitely presented.
\end{rmk}

\begin{prop}\label{prop:cyclic}
Let $(\cC,J)$ be a $Y$-site. Then for any object
$X$ of $\cC$, the sheaf $a_J \frh_\cC(X)$ is a cyclic
object of $\Shv(\cC,J)$.
\end{prop}

\begin{rmk}
The main reason that we restrict ourselves to the case
where $(\cC,J)$ is a $Y$-site is that we use 
Corollary \ref{cor:sheaf_criterion3}
where $J$ is assumed to have enough Galois coverings.
It seems very likely that we can generalize
Corollary \ref{cor:sheaf_criterion3} 
to an arbitrary category with an $A$-topology
by proving a statement 
analogous to that of Lemma 2
in \cite[p.\ 126]{MM}.
By using such a generalization,
it is very probable that one can extend
Proposition \ref{prop:cyclic} 
for an arbitrary $B$-site $(\cC,J)$.
\end{rmk}

\begin{proof}
Let $(\cF_i)_{i \in I}$
be an arbitrary small inductive system of
sheaves on $(\cC,J)$
whose transition maps are monomorphisms.
We prove that the map
$\varinjlim_i \cF_i(X) \to 
(\varinjlim_i \cF_i)(X)$ is bijective.
Let $\cF$ denote the colimit of $\cF_i$
in the category $\Presh(\cC)$.
It suffices to show that $\cF$ is a sheaf on $(\cC,J)$.
Observe that, for any group $G$ and for any
filtered inductive system $(V_j)_{j \in J}$ of left $G$-sets
whose transition maps are injective,
the map $\varinjlim_j V_j^G \to (\varinjlim_j V_j)^G$
is bijective.
Hence by using Corollary \ref{cor:sheaf_criterion3}, 
we conclude that $\cF$ is a sheaf on $(\cC,J)$.
This completes the proof.
\end{proof}

\begin{lem} \label{lem:15-isom}
Let $X$ be an object of $\cC$. Then there exists an
open subgroup $H \in \mathfrak{V}(M)$ and
a surjective map $M/H \to \omega_\cC(X)$ of
left $M$-sets.
\end{lem}

\begin{proof}
It follows from Proposition \ref{prop:cyclic}
that $\omega_\cC(X)$ is a cyclic object of $BM$.
Thus it suffices to prove that for any
non-empty cyclic object $V$ of $BM$,
there exists an
open subgroup $H \in \mathfrak{V}(M)$ and
a surjective map $M/H \to V$ of
left $M$-sets.

Let $V$ be a non-empty cyclic object of $BM$.
Let $I$ denote the set of subsets
$S \subset V$ of the form
$S = \{ m v |\ m \in M \}$ for some
$v \in V$.
Then any $S \in I$ is a subobject of $V$ in $BM$.
Moreover $I$ is a poset with respect to
inclusions and we have $V = \varinjlim_{S \in I} S$.
Since $V$ is a cyclic object, there exists
$S_0 \in I$ and a morphism $V \to S_0$ in $BM$
whose composite with the inclusion $S_0 \to V$
is equal to the identity map $V \to V$.
This implies $S_0 = V$. Choose $v \in V$ satisfying
$S_0 = \{ m v |\ m \in M \}$ and an open subgroup
$H \in \mathfrak{V}(M)$ satisfying $Hv =v$.
Then the map $M/H \to V$ that sends $1\cdot H$
to $v$ satisfies the desired property.
\end{proof}

\begin{lem} \label{lem:above-below}
Let $x = (X, a)$ be an object of $\cC'_0$.
Let $X_1 \xto{f_1} X \xto{f_2} X_2$ be a diagram in $\cC$.
Suppose that $f_1$ belongs to $\cT(J)$.
Then there exist objects
$x_1 = (X_1, a_1)$ and $x_2 = (X_2, a_2)$
of $\cC'_0$ such that $f_1$ is a morphism from $x_1$ to 
$x$ in $\cC'_0$ and that $f_2$ is a morphism from 
$x$ to $x_2$
in $\cC'_0$.
\end{lem}

\begin{proof}
If follows from Proposition \ref{prop:B-site_epi} that
the map $\omega_\cC(f_1)$ is surjective. 
Let us choose an element $a_1 \in \omega_\cC(X_1)$
which is sent to $a$ under the map $\omega_\cC(f_1)$.
We set $a_2 = \omega_\cC(f_2)(a)$.
Then the objects $x_1 = (X_1,a_1)$ and $x_2 = (X_2,a_2)$ 
of $\cC'_0$
satisfy the desired property.
\end{proof}

\begin{lem} \label{lem:15-cover}
Let $V$ be an object of $BM$ and let $a \in V$.
Then there exist an object $X$ of $\cC$ and 
a morphism $\omega_\cC(X) \to V$ in $BM$
such that $a$ belongs to its image.
\end{lem}

\begin{proof}
Let $F$ be a sheaf on $(\cC,J)$ such that
$\omega(F)$ is isomorphic to $V$.
Since
 $\cC$ is essentially $\frU$-small
and any presheaf is a colimit of representables,
there exists a family $((X_i,a_i))_{i \in I}$ of
pairs $(X_i,a_i)$ of objects $X_i$ of $\cC$ 
and elements $a_i \in F(X_i)$ such that,
if we denote by $f_i$ the morphisms
$a_J(\frh_\cC(X_i)) \to F$ in $\Shv(\cC,J)$
given by $a_i$, then the morphism
$\coprod_i f_i : \coprod_i a_J(\frh_\cC(X_i)) \to F$
is an epimorphism.
It follows that there exists
$i_0 \in I$ such that $a$ is contained in
the image of the composite
$\omega_\cC(X_{i_0}) \xto{\omega(f_{i_0})}
\omega(F) \cong V$.
Set $X=X_{i_0}$ and $f=f_{i_0}$. Then the composite
$\omega_\cC(X) \xto{\omega(f)} \omega(F) \cong V$
has the desired property.
\end{proof}

\begin{lem} \label{lem:15-edge}
Let $(X,a)$ be an object of $\cC'_0$.
Then $(X,a)$ is an edge object of $\cC'_0$
if and only if the map
$f:M \to \omega_\cC(X)$ that sends
$m \in M$ to $ma$ is surjective.
\end{lem}

\begin{proof}
The ``if" part follows from Proposition Lemma \ref{prop:B-site_epi}.
We prove the ``only if" part.
Suppose that the map $f$ is not surjective.
Let us choose an open subgroup 
$H \in \mathfrak{V}(M)$ satisfying $Ha = \{a\}$
and set $V=M/H$.
Let $g: V \to \omega_\cC(X)$ denote the
morphism in $BM$ that sends $m H$ to $ma$.
It follows from Lemma \ref{lem:15-cover} that
there exists an object $Y$ of $\cC$ and
a morphism $h: \omega_\cC(Y) \to V$ in $BM$
such that $1 \cdot H$ belongs to the image of $h$.
Let us choose $b \in \omega_\cC(Y)$
satisfying $h(b) = 1 \cdot H$.
Let $s:a_J\frh_\cC(Y) \to a_J\frh_\cC(X)$ denote
the morphism in $\Shv(\cC,J)$ such that
$\omega(s)$ is equal to the composite 
$g \circ h : \omega_\cC(Y) \to \omega_\cC(X)$.
Let us choose a diagram
$Y \xleftarrow{t} Y' \xto{u} X$ in $\cC$ 
with $t \in \cT(J)$ which represents $s$.
It follows from Proposition \ref{prop:B-site_epi} that
there exists $b' \in \omega_\cC(Y')$
satisfying $b = \omega_\cC(t)(b')$.
Then $u$ is a morphism from $(Y',b')$
to $(X,a)$ in $\cC'_0$.
Since $g$ is not surjective, it follows that
$\omega_\cC(u)$ is not surjective.
Hence $(X,a)$ is not an edge object of $\cC'_0$.
\end{proof}

\subsection{Proof of Theorem \ref{thm:15-Cip}}

\begin{proof}
It suffices to prove that the pair $(\cC'_0,\iota'_0)$
satisfies the conditions in the definition of a grid
with the word ``poset" replaced by the word ``quasi-poset".

First we prove that $\cC'_0$ is a quasi-poset.
The existence of $\omega$ implies that
the category $\cC$ is non-empty. 
Hence $\cC'_0$ is non-empty.
Let 
$(X_1, a_1)$ and $(X_2, a_2)$
be two objects of $\cC'_0$ and let
$f_1,f_2: X_1 \to X_2$ be two morphisms in $\cC'_0$.
Let us choose an open subgroup $H \subset \mathfrak{V}(M)$
satisfying $H a_1 = \{a_1\}$.
Let $g: M/H \to X_1$ denote the morphism in $BM$
that sends $m H$ to $m a_1$ for any $m \in M$.
It follows from Lemma \ref{lem:15-cover} that
there exist an object $Y$ of $\cC$ and
a morphism $h: \omega_\cC(Y) \to M/H$ in $BM$ 
such that $1 \cdot H$ belongs to the image of $h$.
Let $s : a_J\frh_\cC(Y) \to a_J \frh_\cC(X)$ denote
the morphism satisfying $\omega(s) = g \circ h$.
Let us choose a diagram
$Y \xleftarrow{t} Z \xto{u} X$
in $\cC$ with $t \in \cT(J)$ which represents $s$.
Then we have $\omega_\cC(u) = g \circ h \circ \omega_\cC(t)$.
Since $t \in \cT(J)$, it follows from Proposition \ref{prop:B-site_epi} that
the image of $\omega_\cC(u)$ is equal to $\{ m a_1\ |\ m \in M \}$.
Since both $\omega_\cC(f_1) = \omega_\cC(f_2)$
sends $a_1$ to $a_2$, we have
$\omega_\cC(f_1 \circ u) = \omega_\cC(f_2 \circ u)$.
Hence it follows from Lemma \ref{lem:ac faithful} that
$f_1 \circ u = f_2 \circ u$. 
Since $\cC$ is an $E$-category, we have $f_1 = f_2$.
Thus $\cC'_0$ is a quasi-poset.

Let $X$ be an arbitrary object of $\cC$.
It follows from Lemma \ref{lem:15-isom} that there exists
$a \in \omega_\cC(X)$ satisfying $M a = \omega_\cC(X)$.
It then follows from Lemma \ref{lem:15-edge} that the pair
$(X,a)$ is an edge object of $\cC'_0$.
Hence $(\cC'_0,\iota'_0)$ satisfies Condition (2)
in Definition \ref{defn:grids}.

Let $x=(X,a)$ be an object of $\cC'_0$
and let $f: Y \to X$ be a morphism in $\cC$
that belongs to $\cT(J)$.
It then follows from Proposition \ref{prop:B-site_epi} that
$\omega_\cC(f)$ is surjective.
Hence it follows from Lemma \ref{lem:15-cover}
that there exists an object $y$ of $\cC'_0$
satisfying $\iota'_0(y)=Y$ and
$f$ is a morphism from $y$ to $x$ in $\cC'_0$.
Hence $(\cC'_0,\iota'_0)$ satisfies Condition (3)
in Definition \ref{defn:grids}.

Let $x=(X,a)$ be an object of $\cC'_0$.
We claim that the functor
$\iota'_{0,x/} : \cC'_{0,x/} \to \cC_{X/}$ induced by $\iota'_0$
is an equivalence of categories.
It follows from Lemma \ref{lem:above-below} that
this functor is essentially surjective.
Since both $\cC'_{0,x/}$ and $\cC_{X/}$ are
quasi-posets, it suffices to prove that for any two objects
$f_1:x \to x_1$ and $f_2: x \to x_2$ of $\cC'_{0,x/}$
such that there exists a morphism from $\iota'_{0,x/}(f_1)$ 
to $\iota'_{0,x/}(f_2)$ in $\cC_{X/}$, 
there exists a morphism $h$ from $f_1$ to $f_2$ in $\cC'_{0,x/}$.
Let us write
$x_1=(X_1,a_1)$ and $x_2=(X_2,a_2)$.
Then we have a commutative diagram
$$
\begin{CD}
\omega_\cC(X) @>{\omega_\cC(f_1)}>> 
\omega_\cC(X_1) \\
@| @VV{\omega_\cC(h)}V \\
\omega_\cC(X) @>{\omega_\cC(f_2)}>> 
\omega_\cC(X_2).
\end{CD}
$$
in $BM$.
Since the map $\omega_\cC(f_i)$ sends $a$ to $a_i$
for $i \in \{1,2\}$, the map
$\omega_\cC(h)$ sends $a_1$ to $a_2$.
Hence $h$ is a morphism from $x_1$ to $x_2$
in $\cC'_0$.
This proves the claim
that $\iota'_{0,x/}$ is an equivalence of categories.

It remains to prove that $\cC'_0$ is $\Lambda$-connected.
Let 
$x_1 =(X_1, a_1)$ and $x_2 =(X_2, a_2)$
be two objects of $\cC'_0$.
For $i=1,2$ let us choose an open subgroup 
$H_i \in  \mathfrak{V}(M)$ satisfying $H_i a_i = \{a_i\}$.
It follows from Lemma \ref{lem:15-isom} that
the map $M \to \omega_\cC(X_i)$ that sends
$m \in M$ to $m a_i$ gives a surjective morphism
$\epsilon_{i} : M/H_i \to \omega_\cC(X_i)$.
It follows from Lemma \ref{lem:15-cover} that
there exist an object $Y$ of $\cC$ and a 
surjective morphism
$\epsilon: \omega_\cC(Y) \to M/(H_1 \cap H_2)$ in $BM$.
Let us choose an element $b \in \omega_\cC(Y)$
satisfying $\epsilon(b) = 1 \cdot (H_1 \cap H_2)$
and set $y=(Y, b)$, which is an object of $\cC'_0$.
Since $\omega$ is an equivalence of category, there
exists a diagram
$$
a_J(\frh_\cC(X_1)) \xleftarrow{h_1} a_J(\frh_\cC(Y))
\xto{h_2} a_J(\frh_\cC(X_2))
$$
in $\Shv(\cC,J)$ such that the diagram
$$
\begin{CD}
\omega_\cC(X_1) @<{\omega(h_1)}<< \omega_\cC(Y) 
@>{\omega(h_2)}>>
\omega_\cC(X_2) \\
@A{\epsilon_1}AA @V{\epsilon}VV @AA{\epsilon_2}A \\
M/H_1 @<<< M/(H_1 \cap H_2) @>>> M/H_2
\end{CD}
$$
is commutative in $BM$, where the bottom arrows are the 
canonical quotient maps.
This commutativity implies that the map
$\omega(h_i)$ sends $b$ to $a_i$ for $i \in \{1,2\}$.

Note that $\Hom_{\Shv(\cC,J)}(a_J(\frh_\cC(Y)),
a_J(\frh_\cC(X_i))) \cong a_J(\frh_\cC(X_i))(Y)$
for $i \in \{1,2\}$.
Hence it follows from the description 
\eqref{eq:a_J} of the sheafification functor $a_J$ that
there exist a morphism $f: Z \to Y$ in $\cC$ and morphisms
$h'_i : Z \to X_i$ for $i \in \{1,2\}$ 
satisfying the following conditions:
\begin{itemize}
\item $f$ is a Galois covering in $\cC$,
\item $h'_i = h'_i \circ \sigma$ for $i \in \{1,2\}$ 
and for any $\sigma \in \Gal(f)$,
\item $\omega_\cC(h'_i) = \omega(h_i) \circ \omega_\cC(f)$
holds for $i \in \{1,2\}$.
\end{itemize}
It follows from Lemma \ref{lem:above-below} that
there exists an object $z = (Z,c)$
of $\cC'_0$ such that
$f$ is a morphism from $z$ to $y$ in $\cC'_0$.
By construction $h'_i$ is a morphism from
$z$ to $x_i$ in $\cC'_0$ for $i=1,2$.
This shows that $\cC'_0$ is $\Lambda$-connected.
\end{proof}

\subsection{Computation of the absolute Galois monoid}
Let $m \in M$. 
%
%
Let $\alpha'_m : \cC'_0 \to \cC'_0$ denote the
following functor: for an object
$x=(X,a)$ of $\cC'_0$, the object $\alpha'_m(x)$ is $(X,ma)$.
For another object $x'=(X',a')$, 
any morphism $f:X' \to X$ from $x'$ to $x$ in $\cC'_0$
can also be regarded as a morphism from $\alpha'_m(x')$ to $\alpha'_m(x)$,
since $\omega_\cC(f)$ sends $ma'$ to $ma$.
We define $\alpha'_m(f)$ to be $f$ regarded
as a morphism from $\alpha'_m(x')$ to $\alpha'_m(x)$.
By associating the identity map of
$\iota'_0(x) = \iota'_0(\alpha'_m(x))$
for each object $x$ of $\cC'_0$, we obtain
a natural isomorphism 
$\gamma'_m : \iota'_0 \to \iota'_0 \circ \alpha'_m$ of functors.

Let $(\cC_0,\iota_0)$ be the grid of $(\cC,J)$ in the statement of
Theorem \ref{thm:15-Cip}.
Let $\alpha_m : \cC_0 \to \cC_0$ denote the composite
of the inclusion $\cC_0 \inj \cC'_0$, the functor
$\alpha'_m$, and the unique quasi-inverse of the inclusion
$\cC_0 \to \cC'_0$. (Note that, since $\cC_0$ is thin, 
the quasi-inverse of 
the inclusion functor, which is an isomorphism,
is unique.)
The natural isomorphism $\gamma'_m$ induces a
natural isomorphism 
$\gamma_m : \iota_0 \to \iota_0 \circ \alpha_m$ of functors.

By construction, for two elements $m,m' \in M$,
the composite $\alpha'_m \circ \alpha'_{m'}$ is equal to
the functor $\alpha'_{mm'}$.
This implies that $\alpha_m \circ \alpha_{m'} = \alpha_{mm'}$
and that $\alpha_m$ is an isomorphism of functors.
Hence by sending $m \in M$ to
the pair $(\alpha_m,\gamma_m)$, we obtain a homomorphism
$\xi : M \to M_{(\cC_0,\iota_0)}$ of monoids.
The following is the main result in this paragraph:

\begin{thm} \label{thm:15-G}
Let the notation be as above. Then
the homomorphism $\xi$ induces an isomorphism
$\wh{M} \xto{\cong} M_{(\cC_0,\iota_0)}$ of monoids.
\end{thm}

Before proving Theorem \ref{thm:15-G},
we need the following proposition.

\begin{prop} \label{prop:M/H}
Let $X$ be an object of $\cC$.
Then there exists an open subgroup $H \subset M$
such that $\omega_\cC(X)$ is isomorphic to $M/H$.
\end{prop}

\begin{proof}
It follows from Lemma \ref{lem:15-isom} that
there exists $a \in \omega_\cC(X)$ satisfying
$M a = \omega_\cC(X)$.
Let us choose $H \in \mathfrak{V}(M)$
satisfying $H a = \{a\}$ and let
$f: M/H \to \omega_\cC(X)$ denote the morphism
in $BM$ that sends $m H$ to $ma$ for any
$m \in M$.
It follows from Lemma \ref{lem:15-cover} that
there exist an object $Y$ of $\cC$
and a morphism $g: \omega_\cC(Y) \to M/H$
such that $1\cdot H$ belongs to the image of $g$.
Let $s: a_J \frh_\cC(Y) \to a_J \frh_\cC(X)$
denote the morphism in $\Shv(\cC,J)$
satisfying $\omega(s) = g\circ f$.
Let $Y \xleftarrow{t} Z \xto{s'} X$ denote the
diagram in $\cC$ with $t \in \cT(J)$ that represents $s$.
By construction the composite $g\circ f$ is surjective.
It follows from Proposition \ref{prop:B-site_epi} that $\omega_\cC(t)$
is surjective. It follows that $\omega_\cC(t)$
is surjective. Hence it follows again from Proposition \ref{prop:B-site_epi}
that $s'$ belongs to $\cT(J)$.
Hence there exists a morphism $u:Z' \to Z$ in $\cC$
with $u \in \cT(J)$ such that $s' \circ u$ is a Galois covering in $\cC$.
Set $G = \Gal(s' \circ u)$.
Let us choose $b \in \omega_\cC(Z')$ satisfying
$M b = \omega_\cC(Z')$.
Since $M g(\omega_\cC(t \circ u)(b)) = M/H$,
we have $g(\omega_\cC(t \circ u)(b)) \in M^*/H$.
By replacing $b$ by $m' b$ for some suitable $m' \in M^*$,
we may and will assume that
$g(\omega_\cC(t \circ u)(b)) \in 1 \cdot H$.
Set $H' = \{ m \in M\ |\ m b =b \}$.
Our assumption implies $H' \subset H$.
Hence $H'$ is equal to the stabilizer of $b$
with respect to the $H$-action on $\omega_\cC(Z')$.
In particular $H'$ be an open subgroup of $H$.

Let $\sigma \in G$. Let $m_\sigma \in M$ be an element
satisfying $\omega_\cC(\sigma)(b) = m_\sigma b$.
Since $\omega_\cC(\sigma)$ is an automorphism in $\omega_\cC(Z')$,
we have $M m_\sigma b = \omega_\cC(Z')$.
In particular we have $M m_\sigma H = M$.
This shows that $m_\sigma \in M^*$.
Since $\omega_cC(\sigma)$ is an automorphism of 
$\omega_\cC(Z')$, the stabilizer of $m_\sigma b$
with respect to the $M^*$-action on $\omega_\cC(Z')$
is equal to $H'$.
This shows that $m_\sigma$ belongs to the normalizer
$N_{M^*}(H')$ of $H'$ in $M^*$.
\newcommand{\mbar}{\overline{m}}
Set $J = N_{M^*}(H')$ and let $\mbar_\sigma$ denote the
class of $m_\sigma$ in $J$.
Then $\mbar_\sigma$ is independent on the choice of $m_\sigma$
and the map $\mbar : G^\op \to J$ that sends
$\sigma$ to $\mbar_\sigma$ is a homomorphism of groups.

Since $s'\circ u$ is a Galois covering in $\cC$ that belongs to
$\cT(J)$, it follows that $\omega_\cC(X)$ is a quotient object
of $\omega_\cC(Z')$ with respect to the action of $G$.
This in particular shows that the set
$U = \{ m \in M\ | ma =a \}$ is equal to
the inverse image of $\mbar(G) \subset J$ 
under the quotient map $N_{M^*}(H') \surj J$.
Since the quotient map
$\omega_\cC(s' \circ u)$ factors through $f$,
we have $H \subset U \subset N_{M^*}(H')$.

We claim that, for two element $m_1, m_2 \in M$,
we have $m_1 a = m_2 a$ if and only if $m_1 U = m_2 U$.
The ``if" part is clear. We prove the ``only if" part.
Suppose that $m_1 a = m_2 a$. Then there exists
$\sigma \in G$ satisfying $m_2 b = m_1 \sigma(b)$.
This implies $m_2 H = m_1 m_\sigma H$.
Hence we have $m_2 U = m_1 U$.
This shows that $\omega_\cC(X)$ is isomorphic to
$M/U$, which proves the claim.
\end{proof}

\begin{proof}[{Proof of Theorem \ref{thm:15-G}}]
First we introduce categories $\cD$, $\cD'_0$
and a functor $\jmath'_0: \cD'_0 \to \cD$.
Let $\cD$ denote the full subcategory of
$BM$ whose objects are those objects $V$ of $BM$
that are isomorphic to $\omega_\cC(X)$ 
for some objects $X$ of $\cC$.
Let $\cD'_0$ denote the following category.
The objects of $\cD'_0$ are the pairs $(V,a)$
of objects $V$ of $\cD$ and elements
$a \in V$.
For two objects $(V,a)$ and $(V',a')$ of $\cD_0$,
a morphism from $(V',a)$ to $(V,a)$
is a morphism $V' \to V$ in $BM$ that sends
$a'$ to $a$.
Let $\jmath'_0 : \cD'_0 \to \cD$ denote the
functor that sends an object $(V,a)$ of
$\cD_0$ to $V$.
Let us choose a skeletal subcategory $\cD_0$ of $\cD'_0$
and let $\jmath_0$ denote the restriction of $\jmath'_0$
to $\cD_0$.

It follows from Proposition \ref{prop:M/H} that any object $V$ of $\cD$
is isomorphic to $M/H$ for some open subgroup $H$
of $M$.

It follows from Lemma \ref{lem:15-cover} that for any 
$H \in \mathfrak{V}(M)$, there exists an
open subgroup $H' \subset H$ such that
$M/H'$ is an object of $\cD$.
Let $\cT_\cD$ denote the collection of surjective
morphisms in $\cD$.
By the argument similar to that in the proofs of
Theorem \ref{thm:CM_Y-site} and Theorem \ref{thm:CM_grid},
one can check that $\cT_\cD$ is semi-localizing, 
that $\cD$ equipped with the topology $J_\cD := J_{\cT_\cD}$ is a $Y$-site,
and that $(\cD_0,\jmath_0)$ is a grid for $(\cD,J_D)$.

Let $M_{(\cD_0,\jmath_0)}$ denote the
absolute Galois monoid associated with the grid $(\cD_0,\jmath_0)$.
We construct a homomorphism
$M \to M_{(\cD_0,\jmath_0)}$.
Let $m \in M$. 
Let $\beta_m : \cD_0 \to \cD_0$ denote the functor
that sends an object $(V,a)$ of $\cD_0$ to 
the unique object of $\cD_0$ isomorphic to $(V,ma)$
in $\cD'_0$.
For an object $(V,a)$ of $\cD_0$, we denote by $\delta_{m,H}$
the unique isomorphism from $(V,ma)$ to $\beta_m((V,a))$,
regarded as a morphism from $V$ to $\jmath_0(\beta_m(V,a))$
in $BM$.
When $(V,a)$ varies, the isomorphisms $\delta_{m,H}$
give a natural isomorphism $\delta_m : \jmath_0 \xto{\cong}
\jmath_0 \circ \beta_m$.
By sending $m \in M$ to $(\beta_m,\delta_m) \in M_{(\cD_0,\jmath_0)}$,
we obtain a map $\xi'':M \to M_{(\cD_0,\jmath_0)}$.
It is straightforward to check that $\xi''$ is a continuous
homomorphism of monoids, and moreover,
by the argument similar to that in the proof of
Theorem~\ref{thm:locally prodiscrete monoid},
one can check that the map $\xi''$ gives an isomorphism from
$\wh{M}$ to $M_{(\cD_0,\jmath_0)}$.

Next we construct a homomorphism
$\xi': M_{(\cD_0,\jmath_0)} \to M_\Cip$ of topological monoids.
Let $p' : \cC'_0 \to \cD_0$ denote the functor 
that sends an object $(X,a)$ of $\cC'_0$ to
the unique object $p'(x)$ of $\cD_0$
which is isomorphic to $(\omega_\cC(X),a)$ in $\cD'_0$.
It is easy to see that $p'$ is essentially
surjective.
Let $x=(X,a)$ be an object of $\cC'_0$.
Let $\epsilon_x:(\omega_\cC(X),a) \to p'(x)$ denote the
unique isomorphism in $\cD'_0$.
Let $(\beta,\delta_\beta) \in M_{(\cD_0,\jmath_0)}$.
For an object $x=(X,a)$ of $\cC'_0$, 
set $p'(x) = (V,b)$ and $\beta(p'(x)) = (V',b')$,
and let $c$ denote the
unique element of $\omega_\cC(X)$ that is mapped to
$b'$ under the composite
$$
\omega_\cC(X) \xto{\epsilon_x} V
\xto{\delta_\beta(p(x))} V'
$$
and set $\alpha'(x) = (X,c)$.
Let $x'=(X',a')$ be another object of $\cC'_0$.
If $f:X' \to X$ is a morphism from $x'$ to $x$ in
$\cC'_0$, then as is easily seen, $f$ is also a
morphism from $\alpha'(x')$ to $\alpha'(x)$
in $\cC'_0$. Hence we obtain a functor 
$\alpha' : \cC'_0 \to \cC'_0$.
By associating the identity morphism
$X \to X$ for any object $x=(X,a)$ of $\cC'_0$,
we obtain a natural isomorphism 
$\gamma'_{\alpha'} : \iota'_0 \xto{\cong}
\iota'_0 \circ \alpha'$.
Let $\alpha : \cC_0 \to \cC_0$ denote the
composite of the inclusion functor 
$\cC_0 \inj \cC'_0$, the functor $\alpha'$,
and the unique quasi-inverse of the inclusion
functor $\cC_0 \inj \cC'_0$.
Then $\gamma'_{\alpha'}$ induces a
natural isomorphism $\gamma_\alpha : \iota_0
\xto{\cong} \iota_0 \circ \alpha$
and the pair $(\alpha,\gamma_\alpha)$
is an element of $M_\Cip$.
By sending $(\beta,\delta_\beta)$ to
$(\alpha,\gamma_\alpha)$, we obtain a map
$\xi': M_{(\cD_0,\jmath_0)} \to M_\Cip$.
It is straightforward to check that $\xi'$ is
a continuous homomorphism of topological monoids.

It is easy to see that the map
$\xi$ is equal to the composite $\xi' \circ \xi''$.
Hence to prove the theorem, it suffices to
show that $\xi'$ has a continuous
inverse homomorphism.

We construct a homomorphism
$\zeta: M_\Cip \to M_{(\cD_0,\jmath_0)}$ of topological monoids
which is a continuous inverse to $\xi'$.
We let $p: \cC_0 \to \cD_0$ denote the restriction
of $p'$ to $\cC_0$.
For an object $x=(X,a)$ of $\cC'_0$,
we introduced an isomorphism
$\epsilon_x : \omega_\cC(X) \xto{\cong} \jmath_0(p'(x))$.
When $x$ varies, the isomorphisms $\epsilon_x$ give 
rise to natural isomorphisms 
$\epsilon': \iota'_0 \xto{\cong} \jmath_0 \circ p'$
and $\epsilon : \iota_0 \xto{\cong} \jmath_0 \circ p$.

Let $(\alpha, \gamma_\alpha)$ be an element of the
absolute Galois monoid $M_\Cip$.
Let $x = (X,a)$
be an object of $\cC_0$.
Let us write $\alpha(x) = (X',a')$.
%
Let us consider the composite
$$
\delta_x : \jmath_0(p(x)) \xto{\epsilon_x^{-1}} \omega_\cC(X)
\xto{\omega_\cC(\gamma_\alpha(X))} 
\omega_\cC(X') \xto{\epsilon_{\alpha(x)}}
\jmath_0(p(\alpha(x))).
$$
We claim that the isomorphism $\delta_x$ 
depends only on $p(x)$, i.e. if $y$ is an object
of $\cC_0$ with $p(x)=p(y)$, then
$p(\alpha(x)) = p(\alpha(y))$ and
$\delta_x = \delta_y$.
In fact, since $\cC_0$ is $\Lambda$-connected,
there exists a diagram $x \xleftarrow{f} z \xto{g} y$
in $\cC_0$. Then we have a commutative diagram
$$
\begin{CD}
\iota_0(x) @<{\iota_0(f)}<< \iota_0(z) @>{\iota_0(g)}>> \iota_0(y) \\
@V{\gamma_\alpha(x)}V{\cong}V @V{\gamma_\alpha(z)}V{\cong}V 
@V{\cong}V{\gamma_\alpha(y)}V \\
\iota_0(\alpha(x)) @<{\iota_0(\alpha(f))}<< \iota_0(\alpha(y)) 
@>{\iota_0(\alpha(g))}>> \iota_0(\alpha(z))
\end{CD}
$$
in $\cC$ that induces a commutative diagram 
$$
\begin{CD}
\jmath_0(p(x)) @<<< \jmath_0(p(z)) @>>> \jmath_0(p(y)) \\
@V{\delta_x}V{\cong}V @V{\delta_z}V{\cong}V @V{\cong}V{\delta_y}V \\
\jmath_0(p(\alpha(x))) @<<< \jmath_0(p(\alpha(z))) @>>>
\jmath_0(p(\alpha(y)))
\end{CD}
$$
in $BM$
where the horizontal arrows are the canonical
quotient maps.
This shows that $p(\alpha(x))=p(\alpha(y))$
and $\delta_x = \delta_y$.
For an object $(V,a)$ of $\cD_0$, choose an
element $x$ of $\cC_0$ satisfying
$p(x)=(V,a)$ and set $\beta((V,a)) = p(\alpha(x))$
and $\delta_\beta((V,a)) = \delta_x$.
Then $\beta((V,a))$ and $\delta_\beta((V,a))$ 
do not depend on the choice of $x$.
Since $\cC_0$ is cofiltered, it follows that
$\delta_\beta((V,a))$ for various $(V,a)$ gives
a natural isomorphism 
$\delta_\beta : \jmath_0 \xto{\cong} \jmath_0\circ \beta$.
Hence $(\beta,\delta_\beta)$ is an element of
$M_{(\cD_0,\jmath_0)}$. By sending $(\alpha,\gamma_\alpha)$
to $(\beta,\delta_\beta)$, we obtain a map
$\zeta : M_\Cip \to M_{(\cD_0,\jmath_0)}$.
It is straightforward to check that $\zeta$ is
a continuous homomorphism of topological monoids.

It is easy to see that both $\xi' \circ \zeta$
and $\zeta \circ \xi$ are the identity maps.
Hence $\xi'$ is an isomorphism of monoids.
This completes the proof of Theorem \ref{thm:15-G}
\end{proof}

\begin{prop}
Let $\omega_\Cip : \Shv(\cC,J) \to BM_\Cip$ 
denote the fiber functor associated with the 
grid $(\cC_0,\iota_0)$ constructed above.
Let $\xi: M \to M_\Cip$ be the homomorphism of 
topological monoids introduced above.
Let $\xi_*: BM_\Cip \xto{\cong} BM$ denote 
the isomorphism of categories induced by $\xi$.
Then we have a canonical natural isomorphism from 
$\Theta : \xi_* \circ \omega_\Cip \xto{\cong} \omega$.
\end{prop}

\begin{proof}
Let $F$ be a sheaf on $(\cC,J)$.
By definition we have
$\omega_\Cip(F) = \varinjlim_{x \in \cC_1} F(\iota_0(x))$,
where $\cC_1$ denotes the full subcategory of $\cC_0$
whose objects are the edge objects of $\cC_0$.
Let $x=(X,a) \in \cC_0$ be an edge object.
Set $H_x = \{ m \in M\ |\ ma = a \}$.
Then it follows from Proposition \ref{prop:M/H} that 
the map $M \to \omega_\cC(X)$ that sends $m \in M$
to $ma$ induces an isomorphism
$M/H_x \xto{\cong} \omega_\cC(X)$ in $BM$.
Then we have isomorphisms
\begin{align*}
F(\iota_0(x))
& = \Hom_{\Shv(\cC,J)}(a_J(\frh_\cC(X)),F)
\cong \Hom_{BM}(\omega_\cC(X),\omega(F)) \\
& \cong \Hom_{BM}(\jmath_0(p(x)), \omega(F))
\cong \omega(F)^{H_x},
\end{align*}
where 
$p : \cC_0 \to \cD_0$ is the functor
introduced in the proof of Theorem \ref{thm:15-G}.
Hence
$$
\omega_\Cip(F) \cong \varinjlim_{x \in \cC_1} \omega(F)^{H_x}.
$$

Let $I$ denote the set of open subgroups $H$ of $M$ such that
$M/H$ is an object of $\cD$.
Since $\cC_1$ is cofiltered, the functor $p$ induces a bijection
$$
\varinjlim_{x \in \cC_1} \omega(F)^{H_x}
\cong \varinjlim_{H \in I} \omega(F)^{H}
$$
It follows from Lemma \ref{lem:15-cover} that 
$I$ is cofinal in the poset of all open subgroups
of $M$. Hence we have
$$
\varinjlim_{H \in I} \omega(F)^{H}
= \omega(F).
$$
Hence we obtain a bijection
$$
\Theta_F : \omega_\Cip(F) \cong \omega(F).
$$
It is then straightforward to check that
for any $m \in M$ and for any $s \in \omega_\Cip(F)$
we have $\Theta_F (\xi(m) s) = m \Theta_F(s)$.
Hence $\Theta_F$ is an isomorphism
from $\xi_* \circ \omega_\Cip(F)$ to $\omega(F)$ in $BM$.
By associating $\Theta_F$ to each object $F$ of
$\Shv(\cC,J)$, we obtain a desired
natural isomorphism $\Theta: \xi_* \circ \omega_\Cip \xto{\cong}
\omega$.
\end{proof}

\section{On the Enough Galois property}
\label{sec:on enough Galois}
The referee asks if the property ``enough Galois" is necessary.
Our answer is negative.
Below, we give examples of 
$B$-sites which do not have
enough Galois coverings (i.e., not a $Y$-site) 
yet the topos is equivalent to 
the category of continuous $G$-sets
for some profinite group $G$.

In the earlier subsections of this section,
we give conditions on a profinite group $G$ such that
if the conditions are met, then  
there is a $B$-site which does not have
enough Galois coverings yet the topos is equivalent to 
the category of continuous $G$-sets.
In Section~\ref{sec:construction profinite}, 
we provide such profinite groups.
We thank Yuichiro Hoshi for an idea on the construction.

\subsection{Conditions on a profinite group}
\label{sec:conditions on profinite}
We give a condition on a profinite group in this section.
Let $G$ be a profinite group.  
For an open subgroup $H \subset G$, 
we consider the following condition:

(*1) For any open subgroup $U$ of $G$,
there exists an open normal subgroup
$N$ of $H$ such that $N \subset U$ 
and that $N$ is not normal in $G$.

In other words, there exists a fundamental system
of neighborhoods of 1 in $G$ 
where each is a normal subgroup of $H$ 
which is not normal in $G$.

Throughout this section, 
we assume that the following properties (*2)(*3) 
hold for $G$:

(*2) For any 
$g_1, g_2 \in G$,
there exists 
an open subgroup $H \subsetneq G$ which is not normal
such that $H$ satisfies (*1) and $g_1, g_2 \in H$.

(*3)
For any proper open subgroup $H \subset G$,
the property (*1) for $H$ holds.

We remark that the condition ``$H$ satisfies (*1)'' in (*2) 
is superfluous under (*3).    
We leave it as it is because the argument below may work under 
weaker assumptions.  For example, 
it is possible that we may weaken (*3) to a certain class
of open subgroups, in which case we need the condition (*2)
as stated.

\subsection{}
Let $\BG$ denote the category of 
continuous left $G$-sets.
Let $\cC$ be the full subcategory of $\BG$
consisting of object $G/G$ and the objects of the form 
$G/\bK$ where $\bK\subsetneq G$ runs over proper open
subgroups that are not normal in $G$.
Using the assumptions on $G$, 
one can check that the category is semi-cofiltered, hence
it can be equipped with the atomic topology $J$.

\begin{lem}
The site $(\cC, J)$ is a $B$-site.  It does not have enough 
Galois coverings.
\end{lem}
\begin{proof}
It is straightforward to check that it is a $B$-site.
Because we take $\bK$'s to be non-normal subgroups, 
there is no Galois covering of the object 
$G/G$, except the identity map.
\end{proof}

\subsection{}
Let $\cC_0$ denote the subcategory of $\cC$
whose objects are the objects of $\cC$, and the 
morphisms are given by
\[
\Hom_{\cC_0}
(G/\bK_1, G/\bK_2)=
\left\{
\begin{array}{ll}
f & \text{if }\bK_1 \subset \bK_2,
\\
\emptyset
& \text{otherwise},
\end{array}
\right.
\]
where $f:G/\bK_1 \to G/\bK_2$
is the map induced by the identity map
$G \to G$.
Let $\iota_0: \cC_0 \to \cC$ denote
the inclusion functor.

\subsection{}
We construct a functor 
$\omega: \Shv(\cC, J) \to \BG$ as follows.
Let $F \in Shv(\cC, J)$ be a sheaf.
We set 
$\omega(F)=\varinjlim_{X \in \cC_0}F(X)$ as a set.

Let us define a continuous left action of $G$ on 
$\omega(F)$ as follows.
Let $g \in G$.   
Take an auxiliary element $g_1 \in G$ and 
take an open $H_{g, g_1} \subset G$, which meets (*1) and $g, g_1 \in H_{g, g_1}$, 
that exists by (*2)
for $g$ and $g_1$.
Let $U \subset G$ run over the set of open subgroups of $G$.
To each $U$, there is an open subgroup $N_{U,g,g_1} \subset U \cap H_{g,g_1}$ which is normal in $H_{g, g_1}$ and not normal in $G$.
Let $N_{H, g, g_1}$ denote the set of such $N_{U, g, g_1}$'s.
Then we have
$\omega(F)=\varinjlim_{N \in N_{H_{g,g_1}}}F(N)$.
The action of an element $h \in H_{g,g_1}$ is given
as the composite
\[
\varinjlim F(N) \xto{h^*} \varinjlim F(hNh^{-1}) \stackrel{f}{=} 
\varinjlim F(N).
\]
Here, the three colimits are over $N \in N_{H_{g,g_1}}$.
There is a map of left $G$-sets $G/(hNh^{-1}) \to G/N$
given by sending the coset $1 \cdot (hNh^{-1})$
to the coset $h \cdot N$.    The map $h^*$ is the limit 
of the pullback maps with respect to this map at each $N$.
The map $f$ is given by the identification $hNh^{-1}=N$
since $N$ is normal in $H_{g, g_1}$.
This gives the left action of $H_{g, g_1}$ on $\omega(F)$,
and in particular the action of the element $g$ on $\omega(F)$.

One can check that the action of $g$ does not depend on the choice of $g_1$ and of $H_{g, g_1}$.

To see that it actually defines a left group action, 
one needs to check that 
$g_1 \circ g_2=g_1\cdot g_2$.   
Note that $g_1, g_2, g_1 \cdot g_2$ all 
belong to $H_{g_1, g_2}$.   Hence the equality above holds true.

One can check that this map actually gives a functor.

\subsection{}
We recall the construction (slightly modified) 
in Section~\ref{sec:site from G}.
Let $\cD$ be the full category of the 
category of continuous left $G$-sets 
where the objects are 
$G/\bK$ for all open subgroups $\bK \subset G$.
The category is cofiltered and may be equipped with
the atomic topology $J_\cD$.
Then $(\cD, J_\cD)$ is a $Y$-site.   
A grid $\cD_0$ can be defined in a manner similar to that 
of $\cC_0$.   The fiber functor will be denoted 
$\omega_\cD$.   Since $G$ is profinite, the hom sets of
the category $\cD$ are finite.   It follows from Theorem~\ref{thm:Galois_main} 
that $\omega_\cD: \Shv(\cD, J_\cD) \to BG$
is an equivalence of categories.

\subsubsection{}
There is a natural inclusion functor $\cC \subset \cD$.
This defines a restriction functor
$r: \Shv(\cD, J_\cD) \to \Pre(\cC)$
to the category of presheaves on $\cC$.
\begin{lem}
For $F \in \Shv(\cD, J_\cD)$, the presheaf
$rF$ is a sheaf on $\cC$.
\end{lem}
\begin{proof}
This follows directly from 
\cite[Lemma 2, p.126]{MM}
that gives a sheaf criterion in 
atomic topology,
since $\cC \subset \cD$ 
is a full subcategory.
\end{proof}
By abuse of notation, we denote the restriction 
functor
$r: \Shv(\cD, J_\cD)
\to \Shv(\cC, J_\cC)$
also by $r$.

\subsubsection{}
We define a functor 
$\tau:\BG
\to \Shv(\cD)$
by 
$\tau(V)(G/\bK)=V^\bK$
(the part fixed by the action of $\bK$).
This is a quasi-inverse of $\omega_\cD$.

\subsection{}
The goal is to prove the following.  The following corollary 
means that 
there exists (depending on the existence of the profinite
group $G$ satisfying some conditions) 
a $B$-site with hom finite underlying category
which does not have enough Galois covering yet 
the topos is equivalent to the category $\BG$.
\begin{prop}
\label{prop:restriction equivalence}
The functor $r$ is an equivalence of categories.
\end{prop}

Let us prove that 
$\tau \omega$ is a quasi-inverse of $r$.
We have 
$\tau \omega r = \tau \omega_\cD \cong \id_{\Shv(\cD, J_\cD)}$.
Let us construct a morphism of functors
$\id_{\Shv(\cC, J_\cC)} \to r\tau \omega$
and show that it is an isomorphism.

\subsubsection{}
Let $F \in \Shv(\cC, J_\cC)$ be a sheaf
and $G/\bK \in \cC$ be an object.
The universality of colimit gives a 
map
$f: F(G/\bK) \to
\omega(F)=\varinjlim_{\bK' \in \cC}
F(G/\bK')$.
\begin{lem}
The image of $f$ is contained in
the fixed part $\omega(F)^\bK$.
\end{lem}
\begin{proof}
Let $\bK \subset G$ be an open subgroup
that satisfies (*1).
Then there is a set $T$ of open subgroups of $G$ 
that is contained in $\bK$ and normal in $\bK$
and is a fundamental system of neighborhoods.
In this case, $\omega(F)^\bK=(\varinjlim_{N \in T}F(G/N))^\bK$.
Now, since $F$ is a sheaf, $F(G/\bK)=F(G/N)^{\bK/N}$ for all $N \in T$.
By taking the limit, one obtains the claim in this case.

Let $\bK$ be general (by assumption (*3), only the case $\bK=G$ remains).
Take a set of (proper) open subgroups $\{\bK_i\}_{i \in I}$ indexed by a set $I$
such that $G$ is generated by $\cup_{i\in I}\bK_i$ and each $\bK_i$ satisfies (*1).
Then since $\omega(F)^G=\bigcap_{i \in I}
\omega(F)^{\bK_i}$
and $F(G/G) \subset \bigcap_{i \in I} F(G/\bK_i)$,
the claim holds true.
\end{proof}

Let us write 
$f': F(G/\bK) \to \omega(F)^\bK$
for the map given by $f$ and the lemma above.

\begin{lem}
\label{lem:rtw isom}
The map $f'$ is an isomorphism.
\end{lem}

\begin{proof}
The map $f'$ is injective since $F$ is a sheaf and the topology is atomic so that 
$F(G/\bK') \to F(G/\bK'')$ injective for any 
morphism $G/\bK'' \to G/\bK'$.

Let us prove that the map $f'$ is surjective.
It suffices to treat the case 
$\bK=G$.
Take an element $x \in \omega(F)^G$.
We want to show that there exists an element $x_0
\in F(G/G)$ such that $f'(x_0)=x$.
Let us take a representative 
$y \in F(G/\bK)$ of $x$ for some $\bK$.

By \cite[p.126, Lemma 2]{MM},
it suffices to show the following:
For any two morphisms
$\varphi_1, \varphi_2: G/\bK' \to G/\bK$,
the equality 
$\varphi_1^*(y)=\varphi_2^*(y)$
holds.

We note here that if there is a map $G/\bK' \to G/\bK$
that sends $1\cdot \bK'$ to $\sigma \cdot \bK$,
then $\bK' \subset \sigma \bK \sigma^{-1}$ 
must hold.

We may without loss of generality assume that 
$\bK' \subset \bK$.
We may moreover assume 
(by replacing $\bK'$ by $\alpha^{-1} \bK' \alpha$ 
for some $\alpha \in G$)
that $\varphi_1$ sends $1\cdot \bK'$ to $1 \cdot \bK$.
Take $\sigma \in G$ such that 
$\varphi_2(1 \cdot \bK')=\sigma\cdot \bK$.
The map $\varphi_2$ 
factors as
$G/\bK' \to G/\sigma \bK \sigma^{-1} 
\cong G/\bK$
where the first map is the map induced by the identity map on $G$ 
and the second map is the map that sends
$1 \cdot (\sigma \bK \sigma^{-1})$ 
to $\sigma \cdot \bK$.
Hence the map $\varphi_2^*$ factors as
$F(G/\bK) \to F(G/\sigma^{-1}\bK \sigma)
\to F(G/\bK')$.
This means that the images of $\varphi_1^*(y)$ 
and $\varphi_2^*(y)$ in $\omega(F)$
differ by the action of $\sigma$.   Since they both represent $x \in 
\omega(F)$ and $x$ is $G$-invariant, they must be equal.
This proves the claim.
\end{proof}

\begin{proof}[Proof of Proposition~\ref{prop:restriction equivalence}]
From Lemma~\ref{lem:rtw isom} it follows that 
$F(G/\bK) \cong \omega(F)^\bK$
for all $\bK$.
By definition the right hand side equals
$(r \tau \omega(F))(G/\bK)$.
This implies that
$\id_{\Shv(\cC, J_\cC)} \to r \tau \omega$
is an isomorphism.
\end{proof}

\begin{cor}
$\omega$ is an equivalence.
\end{cor}
\begin{proof}
By Proposition~\ref{prop:restriction equivalence} 
and by that $\omega_\cD$ is 
an equivalence, 
we see that $r$ is an equivalence of categories.
Since $\omega_\cD$ is an equivalence of categories,
we see that $\omega$ is also an equivalence of categories.
\end{proof}

\subsection{Construction}
\label{sec:construction profinite}
\newcommand{\etale}{\text{\'{e}t}}
\newcommand{\gbar}{\overline{g}}
\newcommand{\Hbar}{\overline{H}}
In this subsection, we give a construction of profinite groups
$G$ meeting the conditions in Section~\ref{sec:conditions on profinite}.

Let $G$ be a topologically finitely generated free profinite group.
A finite subset $S$ of $G$ is called a set of free generators of $G$
if the homomorphism from the free profinite group generated by $S$ to
$G$ that sends any $s \in S$ to $s$ 
is an isomorphism of topological groups.
A finite subset $T \subset G$ is called a part of free generators of $G$
if there exists a set $S$ of free generators of $G$ that contains $T$.

Let $G$ be a topologically finitely generated free profinite group.
Note that any open subgroup of $G$ is a topologically finitely generated 
free profinite group.
For an open subgroup $H$ of $G$, let us consider the following condition:
\begin{description}
\item[(*)] there exist elements $h_1,h_2 \in H$ with $h_1 \neq h_2$ 
and positive integers $n_1,n_2 \ge 1$
such that $h_1^{n_2}$ and $h_2^{n_1}$ are conjugate in $G$ 
and that $\{h_1,h_2\}$ is a part of free generators of $H$.
\end{description}

\begin{lem} \label{lem:condition*1}
Let $H$ be an open subgroup of $G$ satisfying Condition (*).
Then $H$ satisfies Condition (*1) in 
Section~\ref{sec:conditions on profinite}, i.e.,
for any open subgroup $U \subset G$, there exists
an open normal subgroup $N \subset H$ with $N \subset U$ 
such that $N$ is not normal in $G$.
\end{lem}

\begin{proof}
We may and will assume that $U \subset H$ and $U$ is normal in $H$.
Let us choose $h_1,h_2 \in H$ and $n_1,n_2 \ge 1$ satisfying Condition (*)
and a set $S$ of generators of $H$ with $h,h' \in S$.
Let $p$ be a prime number which does not divide $n_1 n_2 [H:U]$ 
and let $f: H \to \Z/p\Z$ denote the homomorphism
of groups that sends $h_1$ to $1 \mod{p\Z}$
and sends any element in $S \setminus \{h_1\}$
to $0 \mod{p\Z}$.
Let $N$ denote the kernel of the homomorphism
$H \to H/U \times \Z/p\Z$ that sends
$g \in H$ to $(g U, f(g))$.
Then the order of $h_1 N$ in $H/N$ is a multiple of $p$
and the order of $h_2 N$ in $H/N$ is prime to $p$.
Since $h_1^{n_2}$ and $h_2^{n_1}$ are conjugate in $G$, 
this shows that $N$ is not a normal subgroup of $G$.
Thus $N$ has the desired property.
\end{proof}

\begin{prop} \label{prop:condition*1}
Let $G$ be a topologically finitely generated
free profinite group, and let $H \subset G$ be an
open subgroup. Suppose that $G$ is non-abelian and
$H \neq G$. Then 
\begin{enumerate}
\item There exists an element $s \in G$
such that $\{s\}$ is a part of free generators of $G$
and that $s$ does not act transitively on $G/H$.
\item $H$ satisfies Condition (*).
\end{enumerate}
\end{prop}

\begin{proof}
First we prove the claim (1).
Let us choose a set $S$ of free generators of $G$.
Since $G$ is non-abelian, $S$ is not a singleton.
Choose elements $s_1,s_2 \in S$ with $s_1 \neq s_2$.
If there exists $i \in \{1,2\}$ such that $s_i$ does
not act transitively on $G/H$, then we may choose $s=s_i$.
Suppose otherwise. There exists an integer $i \in H$
satisfying $s_1 H = s_2^i H$. Then $s = s_1 s_2^{-i}$
satisfies the desired property.

Next we prove the claim (2).
Let us choose $s \in G$ satisfying the condition in (1)
and a set $S$ of free generators of $G$ with $s \in S$.
Let us consider the affine line $\A^1_\C$ over the field
$\C$ of complex number.
Let $\xi$ be a generic geometric point of $\A^1_\C$.
Let $\iota : S \to \A^1_\C(\C) = \C$ be an injective map
and set $C = \A^1_\C \setminus \iota(S)$.
Then there exists an isomorphism $\alpha: G \xto{\cong} \pi_1^{\etale}(C,\xi)$
of topological groups such that for every $h \in S$,
the closed subgroup generated by $\alpha(h)$ is the inertia 
subgroup at
$\iota(h)$.
Let $f: X \to \A^1_\C$ denote the finite morphism of
smooth complex curves, \'{e}tale outside $\iota(S)$,
corresponding to the open subgroup $\alpha(H)$ of $\pi_1^{\etale}(C,\xi)$.
Then the condition in (1) implies that $f^{-1}(\iota(s))$ consists
of at least two points.
Set $U = f^{-1}(C)$ and identify $\alpha(H)$ with $\pi_1^{\etale}(U,\xi)$.
Let us choose $x_1,x_2 \in f^{-1}(\iota(s))$ with $x_1 \neq x_2$,
and let $n_1$ and $n_2$ denote the ramification index of $f$ at
$x_1$ and $x_2$, respectively.
For $i \in \{1,2\}$, one can choose a conjugate $g_i$ of $\alpha(s)$ 
in $\pi_1^\etale(C,\xi)$ in such a way that 
$g_i^{n_i}$ belongs to $\pi_1^\etale(U,\xi)$, that $g_i^{n_i}$
is a generator of a inertia subgroup at $x_i$, and that
$\{g_1^{n_1}, g_2^{n_2}\}$ is a part of free generators of
$\pi_1^\etale(U,\xi)$.
Set $h_i = \alpha^{-1}(g_i^{n_i})$ for $i \in \{1,2\}$.
Then the elements $h_1,h_2 \in H$ and
the integers $n_1,n_2 \ge 1$ have the desired properties.
\end{proof}

\begin{cor}
Let $G$ be a non-abelian topologically finitely generated
free profinite group, and let $H \subsetneqq G$ be a
proper open subgroup. Then $H$ 
satisfies Condition (*1) in Section~\ref{sec:conditions on profinite},
\qed
\end{cor}

\begin{lem}
Let $G$ be a topologically finitely generated
free profinite group. Let $S$ be a set of free generators
of $G$. Suppose that $S$ consists of at least three elements.
Then $G$ satisfies Condition (*2) in 
Section~\ref{sec:conditions on profinite}, i.e.,
for any $g_1,g_2 \in G$, there exists a non-normal open
subgroup $H \subset G$ satisfying $g_1,g_2 \in H$.
\end{lem}

\begin{proof}
Let $G_\ab$ denote the quotient of $G$ by the closure of
$[G,G]$. For $g \in G$, let $\gbar$ denote the class of $g$
in $G_\ab$. Since $G_\ab$ is isomorphic to $\prod_{s \in S}\wh{\Z}$,
our assumption on $S$ shows that 
there exist an element $h \in G$ satisfying the
following properties:
\begin{itemize}
\item $\{h\}$ is a part of free generators of $G$,
\item the closed subgroup of $G_\ab$ generated by
$\gbar_1$, $\gbar_2$ and $\overline{h}$ is not open in $G_\ab$.
\end{itemize}
Hence there exists a proper open subgroup $\Hbar' \subsetneqq G_\ab$
of index at least four such that $\gbar_1,\gbar_2, \overline{h} \in \Hbar'$.
Let $H' \subset G$ denote the inverse image of 
$\Hbar'$ under the quotient 
homomorphism $G \to G_\ab$.
Then there exist four conjugates  
$h_1, h_2, h_3, h_4$, pairwise distinct, 
of $h$ in $G$
such that $\{h_1, h_2, h_3, h_4 \}$ is a part of free generators of $H'$.
Let $H'_\ab$ denote the quotient of $H'$ by the closure of
$[H',H']$. For $g \in H'$, let $\wt{g}$ denote the class of $g$
in $H'_\ab$. 
Then there exists $i,j \in \{1,2,3,4\}$ such that
$\wt{h}_j$ does not belong to the closed subgroup of $H'_\ab$ generated by
$\wt{g}_1, \wt{g}_2$ and $\wt{h}_i$.
Hence one can choose an open subgroup
$\wt{H} \subset H'_\ab$ satisfying
$\wt{g}_1, \wt{g}_2, \wt{h}_i \in \wt{H}$
and $\wt{h}_j \not\in \wt{H}$.
Let $H$ denote the inverse image of $\wt{H}$ under the
quotient homomorphism $H' \to H'_\ab$.
Then we have $g_1, g_2, h_i \in H$ and $h_j \not\in H$.
Since $h_i$ and $h_j$ are conjugate in $G$, this shows that
$H$ is not a normal subgroup of $G$.
\end{proof}

\section{Higher derived limits}
\label{sec:higher derived limits}
In this section, we formulate our results using higher derived limits 
of pro-groups.    The referee pointed out that our cardinality conditions
are merely sufficient conditions for some statements to hold 
and asked for necessary conditions.  
Using higher derived limits,
we are able to give necessary and sufficient 
conditions in some cases (e.g.\ for the existence and uniqueness of grids).
It then becomes clear that our cardinality conditions are sufficient 
conditions for some higher derived limits to be trivial (in the sense to be 
made precise below).
We also provide an example of a $Y$-site such that a grid does not exist.
One problem that remains is to find a necessary condition for 
the fiber functor to be full and essentially surjective.

\subsection{Higher derived limits of pro-groups}

\subsubsection{ }
Let $I$ be a cofiltered quasi-poset and $(G_i)_{i \in I}$
a projective system of groups.
For an integer $n \ge 0$, let $N_n(I)$ denote the set of
$(n+1)$-tuples $(i_0,\ldots,i_n)$ of elements of $I$
such that for $j=1,\ldots,n$, there exists
a morphism from $i_{j-1}$ to $i_j$ in $I$.
For $(i,j) \in N_1(I)$,
we let $\phi_{ij} : G_i \to G_j$ denote the transition homomorphism.
\subsubsection{First derived limits}
We have introduced this notion already in Section~\ref{sec:torsors}
but since we wish to use $I^{op}$ instead of $I$,
we reintroduce it here again.
Let $Z^1((G_i)_{i \in I})$ denote the set of elements
$(g_{ij})_{(i,j) \in N_1(I)} \in \prod_{(i,j) \in N_1(I)} G_j$ 
satisfying $g_{ik} = g_{jk} \phi_{jk}(g_{ij})$ for any
$(i,j,k) \in N_2(I)$.
Two elements $(g_{ij})_{(i,j) \in N_1(I)}$
and $(g'_{ij})_{(i,j) \in N_1(I)}$ of
$Z^1((G_i)_{i \in I})$ are called equivalent
if there exists an element $(h_i)_{i \in I} \in \prod_{i \in I} G_i$
such that $g'_{ij} = h_i g_{ij} \phi_{ij}(h_j)^{-1}$
holds for any $(i,j) \in N_1(I)$.
This gives an equivalence relation on the set
$Z^1((G_i)_{i \in I})$.
We write $R^1 \varprojlim_{i \in I} G_i$ for the
quotient of $Z^1((G_i)_{i \in I})$ with respect to this equivalence
relation.

If $G_i$ is an abelian group for any $i \in I$, then
this pointed set 
$R^1\varprojlim_i G_i$
can be identified with the
derived projective limit $\varprojlim^1_i G_i$,
regarded as a pointed set by forgetting the group structure.

\subsubsection{Surjectivity of 
the canonical projection and triviality of the 
first derived limit}\label{sec:surj and R1}
Let $I$ be a cofiltered quasi-poset and
let $(G_i)_{i \in I}$ be an $I$-projective system
of groups with surjective transition homomorphisms.
%

Now let us fix an object $i \in I$ and
let $J = I_{/i} \subset I$ denote the full subcategory of
objects $j \in I$ that admit a morphism to $i$.
For $j \in J$, set $H_j = \Ker\, \phi_{ji}$.
Then $J$ is a cofiltered poset and 
$(H_j)_{j \in J}$ is a $J$-projective
system of groups.

\begin{lem}
Under the notation above, 
the surjectivity of the canonical projection 
homomorphism
$$
p_i :\varprojlim_{j \in I} G_j \to G_i
$$
is equivalent to the non-emptiness of the limit
$\varprojlim_{j \in J} \phi_{ji}^{-1}(g)$
of the $(H_j)_{j \in J}$-torsor
$(\phi_{ji}^{-1}(g))_{j \in J}$
for every $g \in G_i$.
In particular, if $R^1 \varprojlim_{j \in J} H_j$
is trivial, then the homomorphism $p_i$ is surjective.
\end{lem}
\begin{proof}
Since $J$ is cofinal in $I$, the map
$\varprojlim_{j \in I} G_j \to \varprojlim_{j \in J} G_j$
is bijective.
Hence the inclusion $\phi_{ji}^{-1}(g) \subset G_j$
for each object $j$ of $I_{/i}$ induces a
bijection $\varprojlim_{j \in J} \phi_{ji}^{-1}(g)
\xto{\cong} p_i^{-1}(g)$.
Hence the claim follows.
\end{proof}

\subsubsection{Second derived limits}
\newcommand{\Gpbar}{\overline{Grp}}
Let $(\Gpbar)$ denote the following category.
An object of $(\Gpbar)$ is a group.
For two groups $G$, $H$, a morphism from $G$
to $H$ in $(\Gpbar)$ is an $H$-conjugacy class of
homomorphisms from $G$ to $H$ of groups.

Let $I$ be a cofiltered quasi-poset and $(G_i)_{i \in I}$
a projective system in the category $(\Gpbar)$.
In this paragraph we introduce a set
$R^2 \varprojlim_i G_i$
and a subset 
$(R^2 \varprojlim_i G_i)_*
\subset 
R^2 \varprojlim_i G_i$.
If $G_i$ is abelian for any $i \in I$, then
the set $R^2\varprojlim_i G_i$ can be identified with the
derived projective limit $\varprojlim^2_i G_i$,
and, under this identification, the subset
$(R^2\varprojlim_i G_i)_*$
is the singleton that consists of the unit element 
of the abelian group $\varprojlim^2_i G_i$.

For $(i,j) \in N_1(I)$,
let $\overline{\phi}_{ij} : G_i \to G_j$ 
denote the transition homomorphism in $(\Gpbar)$.
Let $Z^2((G_i)_{i \in I})$ denote the set of 
pairs 
$((\phi_{ij})_{(i,j) \in N_1(I)}, (g_{ijk})_{(i,j,k) \in N_2(I)})$
satisfying the following properties:
\begin{itemize}
\item $\phi_{ij}: G_i \to G_j$ is a homomorphism of groups
whose $G_j$-conjugacy class is equal to $\overline{\phi}_{ij} $
for any $(i,j) \in N_1(I)$.
\item $g_{ijk}$ is an element of $G_k$ for any
$(i,j,k) \in N_2(I)$.
\item We have $\phi_{jk} \circ \phi_{ij}
= g_{ijk} \phi_{ik} g_{ijk}^{-1}$
for any $(i,j,k) \in N_2(I)$.
\item We have $g_{jkl} g_{ijl} = \phi_{kl}(g_{ijk}) g_{ikl}$
for any $(i,j,k,l) \in N_3(I)$.
\end{itemize}

Two elements 
$((\phi_{ij})_{(i,j) \in N_1(I)}, (g_{ijk})_{(i,j,k) \in N_2(I)})$
and
$((\phi'_{ij})_{(i,j) \in N_1(I)}, (g'_{ijk})_{(i,j,k) \in N_2(I)})$
of $Z^2((G_i)_{i \in I})$ are called equivalent
if there exists an element $(h_{ij})_{(i,j) \in N_1(I)} 
\in \prod_{(i,j) \in N_1(I)} G_j$ satisfying the following
properties:
\begin{itemize}
\item We have $\phi'_{ij}(g) = h_{ij} \phi_{ij}(g) h_{ij}^{-1}$
for any $(i,j) \in N_1(I)$ and for any $g \in G_i$.
\item We have $g'_{ijk} = h_{jk} \phi_{jk}(h_{ij}) g_{ijk} h_{ik}^{-1}$
for any $(i,j,k) \in N_2(I)$.
\end{itemize}

This gives an equivalence relation on the set
$Z^2((G_i)_{i \in I})$.
We write $R^2 \varprojlim_{i \in I} G_i$ for the
quotient of $Z^2((G_i)_{i \in I})$ with respect to this equivalence
relation.

The set $R^2 \varprojlim_{i \in I} G_i$ does not always have
a canonical structure of pointed set. However
one can define a canonical subset 
$(R^2 \varprojlim_{i \in I} G_i)_* \subset R^2 \varprojlim_{i \in I} G_i$
as the set of equivalence classes in
$Z^2((G_i)_{i \in I})$ that contains
an element 
$((\phi_{ij})_{(i,j) \in N_1(I)}, (g_{ijk})_{(i,j,k) \in N_2(I)})$
of $Z^2((G_i)_{i \in I})$
with $g_{ijk}=1$ for any $(i,j,k) \in N_2(I)$.

\subsection{Existence of a grid and the second derived limit}
\label{sec:grid and R2}
\subsubsection{}
Let  $(\cC,J)$ be a $Y$-site.
Let $X$ be an object of $\cC$ and let us consider
the category $I=\Gal/X$ introduced in Section \ref{sec:Gal/X}.
Recall that $I$ is a cofiltered quasi-poset.
For an object $i = (Y,f)$ of $I$, let us write 
$G_i = \Aut_X(Y)$. Then any morphism
$i \to j$ in $I$ induces a morphism $\overline{\phi}_{ij}$ from
$G_i$ to $G_j$ in $(\Gpbar)$.
Thus $(G_i)_{i \in I}$ forms a projective system in $(\Gpbar)$.

For each $(i,j) \in N_1(I)$, let us choose
a morphism $f_{ij}$ from $i$ to $j$ in $\cC_{/X}$.
Then it follows from Lemma \ref{lem:descends} that
for any element $g\in G_i$, there exists a unique element
$\phi_{ij}(g)$ satisfying 
$f_{ij} \circ g = \phi_{ij}(g) \circ f_{ij}$.
The uniqueness shows that the map $\phi_{ij} :G_i \to G_j$
that sends $g \in G_i$ to $\phi_{ij}(g)$ is a
homomorphism of groups.
For each $(i,j,k) \in N_2(I)$, let $g_{ijk} \in G_k$ denote
the unique element satisfying $f_{jk} \circ f_{ij} = g_{ijk} \circ f_{ik}$.
Then the pair 
$$
z((f_{ij})_{(i,j) \in N_1(I)}) 
:= ((\phi_{ij})_{(i,j) \in N_1(I)}, (g_{ijk})_{(i,j,k)\in N_2(I)})
$$
is an element of $Z^2((G_i)_{i \in I})$.
If $(f'_{ij})_{(i,j) \in N_1(I)}$ is another choice of morphisms
in $\cC_{/X}$, then it is straightforward to check that
$z((f_{ij})_{(i,j)\in N_1(I)})$ and
$z((f'_{ij})_{(i,j) \in N_1(I)})$ are equivalent.
Hence the class $z$ of $z((f_{i,j})_{(i,j) \in N_1(I)})$
in $R^2 \varprojlim_{i \in I} G_i$ does not depend on
the choice of $f_{ij}$.
\begin{lem} \label{lem:criterion R2}
The class $z$ belongs to
the subset $(R^2 \varprojlim_{i\in I} G_i)_*$
if and only if there exists a grid of $(\cC,J)$.
\end{lem}

\begin{proof}
First we prove the ``only if" part.
Suppose that the class $z$ belongs to the subset 
$(R^2 \varprojlim_{i\in I} G_i)_*$.
We prove that there exists a grid of $(\cC,J)$.

In Proposition \ref{cor:grid existence}, we have proved the
existence of a grid of $(\cC,J)$ under the assumption that
the category $\cC(\cT(J))_{/X}$ satisfies at least one of
the two cardinality conditions in Section \ref{sec:cardinality}.
In the argument of the proof of Proposition \ref{cor:grid existence},
we have used the cardinality conditions only in the proof of
Lemma~\ref{lem:US non-empty}.
In the notation of 
Section~\ref{sec:construction 1},
we take $X_0 = X$ and choose $V$ as the set of objects of $I$.
It suffices to prove that the statement of Lemma~\ref{lem:US non-empty} 
is true for this choice of $X_0$ and $V$.
Then the remaining argument in the
proof of Proposition \ref{cor:grid existence}
proves the existence of a grid of $(\cC,J)$.

By assumption, there exists an element 
$(h_{ij})_{(i,j)\in N_1(I)} \in \prod_{(i,j) \in N_1(I)} G_j$
such that
the equality $1 = h_{jk} \phi_{jk}(h_{ij}) g_{ijk} h_{ik}^{-1}$
holds for any $(i,j,k) \in N_2(I)$.
For $(i,j) \in N_1(I)$, set $f'_{ij} = h_{ij} \circ f_{ij}$.
Then for any $g \in G_i$, we have 
\begin{equation} \label{eq:f'ij}
f'_{jk} \circ f'_{ij} = f'_{ik}.
\end{equation}

For a finite subset $S \subset V$, let us choose $j \in I$ such that
$(j,i) \in N_1(I)$ for any $i \in S$.
Then $(f_{ji})_{i \in S}$ gives an element $v_S$ of
the set $U(S)$ in introduced in
Section~\ref{sec:construction 1}.
It follows from \eqref{eq:f'ij} that the element $v_S$
does not depend on the choice of $j$ and,
when $S$ varies, the family $(v_S)_S$ is an element of
the set $U$ in Lemma~\ref{lem:US non-empty}.
Thus the statement of Lemma~\ref{lem:US non-empty}
is true for this choice of $X_0$ and $V$, 
as desired.

Next we prove the ``if" part.
Suppose that a grid $(\cC_0,\iota_0)$
of $(\cC,J)$ exists.
We prove the class $z$ belongs to 
$(R^2 \varprojlim_{i \in I} G_i)_*$.
Let us choose an edge object $X_0$ of $\cC_0$ and
an isomorphism $\beta_X : \iota_0(X_0) \xto{\cong} X$ in $\cC$.
Let $i=(Y,f)$ be an object of $I$. It follows from
Condition (3) in Definition \ref{defn:grids} that
there exist a morphism $f_{0,i} : Y_i \to X_0$
in $\cC_0$ and an isomorphism $\beta_i : \iota_0(Y_i) \xto{\cong} Y$
satisfying $f \circ \beta_i = \beta_X \circ \iota_0(f_{0,i})$.
Let $(i,j) \in N_1(I)$. Since $\beta_j^{-1} \circ f_{ij} \circ \beta_i$
is a morphism from $\iota_0(Y_i)$ to $\iota_0(Y_j)$, it follows
from Condition (3) in Definition \ref{defn:grids} that
there exists a unique morphism $f_{0,ij}$
from $Y_i$ to $Y_j$ in $\cC_0$. Set
$f'_{ij} = \beta_j \circ \iota_0(f_{0,ij}) \circ \beta_i^{-1}$.
By construction we have $f'_{jk} \circ f'_{ij} = f'_{ik}$
for $(i,j,k) \in N_2(I)$.
Since $z$ is equal to the class of $z((f'_{ij})_{(i,j) \in N_1(I)})$,
the class $z$ belongs to $(R^2 \varprojlim_{i \in I} G_i)_*$.
This completes the proof.
\end{proof}

\subsubsection{A $Y$-site from a pro-group}
\label{sec:pro-group to Y-site}
Let $I$ be a cofiltered poset with a final object $i_0$ and
$(G_i)_{i \in I}$ a projective system in $(\Gpbar)$ with
$G_{i_0} = \{1\}$.
Suppose that the 
transition morphism $G_i \to G_j$ is  
the class of a surjective homomorphism of groups 
for each $(i,j) \in N_1(I)$.
Let $\wt{z}=((\phi_{ij})_{(i,j) \in N_1(I)}, (g_{ijk})_{(i,j,k) \in N_2(I)})$
be an element of $Z^2((G_i)_{i \in I})$.
Let us consider the following category $\cC$.
The objects of $\cC$ are the objects of $I$.
For two objects $i,j$ of $I$, the set of morphisms from
$i$ to $j$ in $\cC$ is given as follows:
$$
\Hom_\cC(i,j) = \begin{cases}
G_j, & \text{ if }(i,j) \in N_1(I) \\
\emptyset, & \text{ otherwise}.
\end{cases}
$$
For $(i,j,k) \in N_2(I)$, the composition
$\Hom_{\cC}(j,k) \times \Hom_\cC(i,j) \to \Hom_\cC(i,k)$
is the map $G_k \times G_j \to G_k$ that sends
$(g,h) \in G_k \times G_j$ to $g \phi_{jk}(h) g_{ijk}$.
Let $J$ denote the atomic topology on $\cC$.
Then $(\cC,J)$ is a $Y$-site.

\subsubsection{$Y$-sites with no grid}
One can construct a $Y$-site
for which there does not exist a grid 
as follows.
Let $(G_i)_{i \in I}$ be 
a pro-group as in 
Section~\ref{sec:pro-group to Y-site}.
Suppose that 
$R^2\varprojlim_i ((G_i)_{i \in I}) \supsetneq 
R^2\varprojlim_i ((G_i)_{i \in I})_*$.
Take $\tilde{z} \in 
R^2\varprojlim_i ((G_i)_{i \in I}) \setminus
R^2\varprojlim_i ((G_i)_{i \in I})_*$.
Construct a $Y$-site as in the procedure above 
using this $\tilde{z}$.
For $(i,j) \in N_1(I)$, 
Let $f_{ij} \in \Hom_\cC(i,j)=G_j$ be the element 
corresponding to the identity element in $G_j$.
Then, the class of $z((f_{i,j})_{(i,j) \in N_1(I)})$ equals the given $\tilde{z}$.
It follows from Lemma~\ref{lem:criterion R2}
that a grid for this $Y$-site does not exist.

For the existence of a surjective pro-group with
$R^2\varprojlim_i ((G_i)_{i \in I}) \supsetneq 
R^2\varprojlim_i ((G_i)_{i \in I})_*$, we could not find a 
published reference
but there is an online note by Ziegler \cite{Ziegler}
which mentions the existence under the diamond principle.

\subsubsection{}
We see below that our 
cardinality conditions imply
$R^2 \varprojlim_{i \in I} G_i
= (R^2 \varprojlim_{i \in I} G_i)_*$.
For example, any locally profinite groups
have this property.
 
\begin{lem}
Let $I$ be a cofiltered poset with a final object $i_0$ and
$(G_i)_{i \in I}$ a projective system in $(\Gpbar)$ with
$G_{i_0} = \{1\}$.
Suppose either that $G_i$ is finite for every $i \in I$, 
or that $I$ is at most countable, 
and the transition morphism $G_i \to G_j$ are 
the class of a surjective homomorphism of groups 
for each $(i,j) \in N_1(I)$.
Then we have $R^2 \varprojlim_{i \in I} G_i
= (R^2 \varprojlim_{i \in I} G_i)_*$.
\end{lem}

\begin{proof}
Let $\wt{z}=((\phi_{ij})_{(i,j) \in N_1(I)}, (g_{ijk})_{(i,j,k) \in N_2(I)})$
be an element of $Z^2((G_i)_{i \in I})$. We prove that
the class of $\wt{z}$
in $R^2 \varprojlim_{i \in I}  G_i$ belongs to
$(R^2 \varprojlim_{i \in I}  G_i)_*$.

Construct a $Y$-site $(\cC,J)$ as in Section~\ref{sec:pro-group to Y-site}.
Then $\cC = \cC_{/i_0}$ satisfies at least one 
of the two cardinality conditions
in Section \ref{sec:cardinality}.
Hence Proposition \ref{cor:grid existence}
shows that there exists a grid $(\cC_0,\iota_0)$
of $(\cC,J)$.
It follows from the construction of $(\cC,J)$
that the class of $\wt{z}$ 
in $R^2 \varprojlim_{i \in I}  G_i$
is equal to the class $z$ in the statement of
Lemma \ref{lem:criterion R2} associated
with the category $\cC = \cC(\cT(J))_{/i_0}$.
Hence Lemma \ref{lem:criterion R2} implies
that the class of $\wt{z}$ belongs to
$(R^2 \varprojlim_{i \in I}  G_i)_*$.
Since $\wt{z}$ is arbitrary, we have
$R^2 \varprojlim_{i \in I}  G_i
= (R^2 \varprojlim_{i \in I}  G_i)_*$.
\end{proof}


\subsection{Uniqueness of a grid and the first derived limit}
\label{sec:uniqueness and R1}
In Section~\ref{sec:grid uniqueness}, 
we showed the uniqueness of grids under
our cardinality conditions,
and gave a counterexample when the condtions are not satisfied.
In this section, we give a more precise criterion on the 
uniqueness of grids in terms of the first derived limit
by introducing the notion of a grid pinned at an edge object.

\subsubsection{}
We say that two grids $(\cC_{0,1},\iota_{0,1})$ and
$(\cC_{0,2},\iota_{0,2})$ of $(\cC,J)$ are equivalent if
there exist an isomorphism $\alpha:\cC_{0,1} \xto{\cong} \cC_{0,2}$ 
of categories and a natural isomorphism $\gamma : \iota_{0,1} 
\xto{\cong} \iota_{0,2} \circ \alpha$ of functors.

Suppose that the $Y$-site $(\cC,J)$ has a grid $(\cC_0,\iota_0)$.
Let us fix an edge object $X_0$ of $\cC_0$ and set $X= \iota_0(X_0)$.
Let $\iota_{0,X_0} : \cC_{0,/X_0} \to \cC_{/X}$ denote the
functor induced by $\iota_0$.
The axioms in the definition of a grid imply that
for any object $i$ of $\cC_{/X}$, there exists a unique
object $s(i)$ of $\cC_{0,/X_0}$ such that $\iota_{0,X_0}(i_0)$
is isomorphic to $i$.
Moreover, for any two objects $i$, $j$ of $\cC_{/X}$ such
that there exists a morphism from $i$ to $j$,
the image under $\iota_{0,X_0}$ of the unique morphism
from $s(i)$ to $s(j)$ gives a specified morphism
$f_{i,j}$ from $i$ to $j$ in $\cC_{/X}$.

Let us consider the cofiltered quasi-poset $I = \Gal/X$
introduced in Section \ref{sec:Gal/X}.
For an object $i = (Y,f)$ of $I$, let us write 
$G_i = \Aut_X(Y)$. Then any morphism
$i \to j$ in $I$ induces a homomorphism 
$\phi_{ij} : G_i \to G_j$ of groups characterized by
the following property:
we have $f_{ij} \circ g = \phi_{ij}(g)\circ f_{ij}$
for any $g \in G_i$.
The morphisms $\phi_{ij}$ for each $(i,j) \in N_1(I)$
gives $(G_i)_{i \in I}$ a structure of projective
system of groups.

\begin{defn}[pinned grid]
A grid of $(\cC,J)$ pinned at $X$ is a triple
$\wt{\cC}'_0 = ((\cC'_0, \iota'_0), X'_0, \beta')$ of a grid
$(\cC'_0,\iota'_0)$ of $(\cC,J)$,
an edge object $X'_0$ of $\cC'_0$, and
an isomorphism $\beta' : \iota'_0(X'_0) \xto{\cong} X$.

We say that two grids
$((\cC'_0, \iota'_0), X'_0, \beta')$ and
$((\cC''_0, \iota''_0), X''_0, \beta'')$
of $(\cC,J)$ pinned at $X$ are equivalent
if there exists an isomorphism
$\alpha : \cC'_0 \xto{\cong} \cC''_0$
of categories and a natural isomorphism
$\gamma_\alpha : \iota'_0 \xto{\cong}
\iota''_0 \circ \alpha$ satisfying
$\alpha(X'_0) = X''_0$ and
$\gamma_\alpha(X'_0) = {\beta''}^{-1} \circ \beta'$.
\end{defn}
For example, $\wt{\cC}_0 = ((\cC_0,\iota_0),X_0,\id_X)$
is a grid of $(\cC,J)$ pinned at $X$.

Let $\wt{\cC}'_0 = ((\cC'_0, \iota'_0),X'_0,\beta')$ be 
another grid of $(\cC,J)$ pinned at $X$.
In the notation in the proof of Proposition \ref{prop:grid uniqueness},
we choose $(\cC_{0,1},\iota_{0,1})$,
$(\cC_{0,2},\iota_{0,2})$, $X_{0,1}$, $\beta_1$,
$X_{0,2}$, $\beta_2$ to be
$(\cC_0,\iota_0)$, $(\cC'_0,\iota'_0)$,
$X_0$, $\id_X$, $X'_0$, and $\beta'$,
respectively.
Observe that the functor $\iota_0$ induces an
equivalence of categories from the poset $\cC'_{0,1}$
in the proof of Proposition \ref{prop:grid uniqueness}
to the quasi-poset $I$.
Via this equivalence, we obtain
a $(G_i)_{i \in I}$-torsor 
$S(\wt{\cC}'_0)$ in the sense of
Section \ref{sec:torsors}
from the filtered projective system $(S_f)$ introduced
in the proof of Proposition \ref{prop:grid uniqueness}.
An explicit description of $S(\wt{\cC}'_0)$
is given as follows: for each object $i=(Y,f)$ of $I$,
the $i$-th component $S(\wt{\cC}'_0)_i$ is
the left $G_i$-torsor of pairs $(f',\gamma')$ 
of a morphism $f':Y' \to X'_0$ in $\cC'_0$
and an isomorphism $\gamma' : Y \xto{\cong} \iota'_0(Y')$
satisfying $f = \beta' \circ \iota'_0(f') \gamma'$.
Here the action of $G_i$ on $S(\wt{\cC}'_0)$ is given by
$g(f',\gamma') = (f', \gamma \circ g^{-1})$
for $g \in G_i$.
Then, the argument in the proof of Lemma \ref{lem:non-empty}
gives an element $z(\wt{\cC}'_0)$ of $R^1 \varprojlim_{i \in I} G_i$
corresponding to the $(G_i)_{i \in I}$-torsor $S(\wt{\cC}'_0)$.

It is straightforward to check that,
if two grids $\wt{\cC}'_0$ and 
$\wt{\cC}''_0$ of $(\cC,J)$ pinned at $X$ are
equivalent, then
$z(\wt{\cC}'_0)$ and $z(\wt{\cC}''_0)$ are equal.

Thus by sending $\wt{\cC}'_0$ to $z(\wt{\cC}'_0)$,
we obtain a map $z$ from the set of equivalence classes of grids
pinned at $X$ to $R^1 \varprojlim_{i \in I} G_i$.

\begin{prop}
\label{prop:pinned grid}
The map $z$ from the set of equivalence classes of 
grids pinned at $X$
to 
$R^1\varprojlim_{i \in I}G_i$ is bijective.
\end{prop}

\begin{proof}
First we prove that the map $z$ is injective.
Let $\wt{\cC}'_0 = ((\cC'_0,\iota'_0),X'_0,\beta')$ and 
$\wt{\cC}''_0 = ((\cC''_0,\iota''_0),X''_0,\beta'')$
be two grids of $(\cC,J)$ pinned at $X$, satisfying
$z(\wt{\cC}'_0) = z(\wt{\cC}''_0)$.
We prove that $\wt{\cC}'_0$ and $\wt{\cC}''_0$
are equivalent.

Let us consider the posets $I_{X'_0}$ and $I_{X''_0}$
introduced at the beginning of Section \ref{sec:IX}.
Then the functor $\iota'_0$ and $\iota''_0$ induce
equivalences of categories $I_{X'_0} \to I$ and
$I_{X''_0} \to I$. Since $I_{X''_0}$ is a poset,
these equivalences of categories induces an
isomorphism $\alpha': I_{X'_0} \xto{\cong} I_{X''_0}$ 
of posets satisfying $\alpha'(X'_0) = X''_0$.

By assumption, the $(G_i)_{i \in I}$-torsors $S(\wt{\cC}'_0)$
and $S(\wt{\cC}''_0)$ are isomorphic.
Let us choose an isomorphism $\eta : S(\wt{\cC}'_0) \cong
S(\wt{\cC}''_0)$ of $(G_i)_{i \in I}$-torsors.
%
%
The isomorphism $\eta$ gives, for each $i=(Y,f) \in I$,
an isomorphism $\eta_i : S(\wt{\cC'}_0)_i \to S(\wt{\cC}''_0)_i$ 
of left $G_i$-torsors.
Choose an element $(f',\gamma') \in S(\wt{\cC}'_0)_i$ and
let $(f'',\gamma'') = \eta_i(f',\gamma')$.
Then $f'$, $f''$, and the composite 
$\gamma_i =\gamma'' \circ {\gamma'}^{-1}$
depends only on $\eta$ and $i$, and is independent of
the choice of $\gamma'$.
It follows from the construction of $\eta_i$ that
we have $f'' = \alpha'(f')$.
Hence the isomorphisms $\gamma_i$ for various $i$
gives a natural isomorphism $\gamma_{\alpha'} 
: \iota'_0|_{I_{X'_0}}  \xto{\cong}
\iota''_0|_{I_{X''_0}} \circ \alpha'$ of functors
satisfying
$\gamma_{\alpha'}(X'_0) ={\beta''}^{-1} \circ \beta'$.

Note that $\cC'_0$, $\cC''_0$ are isomorphic to
the $2$-colimits 
$$\varinjlim_{(f':Y' \to X'_0) \in I_{X'_0}} \cC'_{0,/Y'},
\varinjlim_{(f:Y''\to X''_0) \in I_{X''_0}} \cC'_{0,/Y''},$$ 
respectively.
By Condition (4) in Definition \ref{defn:grids}, we have
equivalences
$$
\cC'_{0,/Y'} \xto{\cong}
\cC_{/\iota'_0(Y')}
\xleftarrow[\cong]{-\circ \gamma_{\alpha'}(Y')}
\cC_{/\iota''_0(\alpha'(Y'))}
\xleftarrow{\cong}
\cC''_{0,/\alpha'(Y')}
$$
of categories for any object $Y' \to X'_0$ of $I_{X'_0}$.
These equivalences induce an isomorphism
$\alpha_{Y'} :\cC'_{0,/Y'} \xto{\cong} \cC''_{0,/\alpha'(Y')}$
of posets and a natural isomorphism
$\iota'_0|_{\cC'_{0,/Y'}} \xto{\cong} 
\iota''_0|_{\cC''_{0,/\alpha'(Y')}} \circ \alpha_{Y'}$
of functors.
Passing to the colimit with respect to $Y'$,
we obtain an isomorphism $\alpha: \cC'_0 \xto{\cong}
\cC''_0$ and a natural isomorphism
$\gamma_\alpha : \iota'_0 \xto{\cong} 
\iota''_0 \circ \alpha$ of functors satisfying
$\alpha(X'_0) = X''_0$ and 
$\gamma_\alpha(X'_0) = {\beta''}^{-1} \circ \beta'$.
This proves that
$\wt{\cC}'_0$ and $\wt{\cC}''_0$
are equivalent.

Next we prove that $z$ is surjective.
Let $z$ be an arbitrary element of $R^1 \varprojlim_{i \in I} G_i$.
Choose an element $(g_{ij}) \in Z^1((G_i)_{i \in I})$ that represents $z$.
Let us consider the posets $I_{X_0}$
introduced at the beginning of Section \ref{sec:IX}.
Let us introduce the functor $\iota' : I_{X_0} \to \cC$
defined as follows: for any object $f:Y \to X_0$ of $I_{X_0}$, 
we set $\iota'(f) = \iota_0(Y)$.
Let $f_1: Y_1 \to X_0$ and $f_2 : Y_2 \to X_0$ be two
objects of $I_{X_0}$ and suppose that there exists a
morphism $g:Y_1 \to Y_2$ from $f_1$ to $f_2$ in $I_{X_0}$.
Set $i_1 = \iota_0(f_1)$ and $i_2 = \iota_0(f_2)$.
Then $i_1$ and $i_2$ are objects of $I$ and we have
$(i_1,i_2) \in N_1(I)$.
We then set $\iota'(g) = g_{i_1 i_2} \circ \iota_0(g)$.
Let $f_1$, $f_2$, $f_3$ be three objects of $I_{X_0}$
such that there exist morphisms $g:f_1 \to f_2$ and $h:f_2 \to f_3$.
Set $i_j = \iota_0(f_j)$ for $j \in \{1,2,3\}$.
Then the equality $g_{i_1 i_3} = g_{i_2 i_3} \phi_{i_2 i_3}(g_{i_1 i_2})$
implies that $\iota'(h \circ g) = \iota'(h) \circ \iota'(g)$.

Note that $\cC_0$ is isomorphic to
the $2$-colimit $\varinjlim_{(f:Y \to X_0) \in I_{X_0}} \cC_{0,/Y}$.
By Condition (4) in Definition \ref{defn:grids}, we have
equivalences
$\cC_{0,/Y} \xto{\cong}
\cC_{/\iota_0(Y)}$
of categories for any object $Y \to X_0$ of $I_{X_0}$.
Then one can extend the functor $\iota'$
to obtain a functor $\iota'_0 : \cC_0 \to \cC$
satisfying $\iota'_0(X_0) =X$.
Set $\wt{\cC}'_0 = ((\cC_0,\iota'_0),X_0,\id_X)$.
Then it is straightforward to check that $z(\wt{\cC}'_0) = z$.
This proves that the map $z$ is surjective.
\end{proof}

\subsection{Fullness and essential surjectivity of $\omega_\Cip$ 
and the first derived limit}
\label{sec:full ess surj}
Let us discuss below a sufficient condition for the fiber functor 
$\omega_{\cC_0, \iota_0}$ to be an equivalence.   We do 
not have a necessary condition.
\subsubsection{}
Suppose that the $Y$-site $(\cC,J)$ has a grid $(\cC_0,\iota_0)$.
Then by Theorem \ref{thm:Galois_main}, we 
have a faithful functor $\omega_\Cip : \Shv(\cC,J) \to 
(M_\Cip\text{-set})_{\text{sm}}$.

Moreover, $\omega_\Cip$ gives an equivalence of categories
under the assumption that for any object $X$ of $\cC$, the category
$\cC(\cT(J))_{/X}$ satisfies at least one of the two conditions 
in Section \ref{sec:cardinality}.

In the proof of Theorem \ref{thm:Galois_main},
we have used this assumption 
only in Lemma \ref{lem:Gal gp surjective}
and Lemma \ref{lem:sufficiently many}.
Let us choose an edge object $X_0$ of $\cC_0$ 
such that $\iota_0(X)$ is isomorphic to $X$
and fix an isomorphism $\iota_0(X) \cong X$.
Let us consider the cofiltered quasi-poset $I = \Gal/X$
introduced in Section \ref{sec:Gal/X}.
For each object $i=(Y,f)$ of $I$, set $G_i = \Aut_X(Y)$.
Then, as we have seen in Section \ref{sec:uniqueness and R1},
the grid $(\cC_0,\iota_0)$ gives $(G_i)_{i \in I}$ a structure
of an $I$-projective system of groups with surjective
transition homomorphisms.
In the proof of Lemma \ref{lem:Gal gp surjective},
we have used the assumption to assure that the
canonical projection
\begin{equation} \label{eq:canproj_surj}
\varprojlim_{j \in I} G_j \to G_i
\end{equation}
is surjective for any object $i$ of $I$.
In the proof of Lemma \ref{lem:sufficiently many},
we have used the assumption to assure that 
the first derived projective limit
\begin{equation} \label{eq:R1_triviality}
R^1 \varprojlim_{j \in I} G_j
\end{equation}
is trivial. As we have explained in
Section \ref{sec:surj and R1}, one can deduce the surjectivity
of \eqref{eq:canproj_surj} from the triviality
of \eqref{eq:R1_triviality} for some other $X$.
Moreover, the statement of Lemma \ref{lem:Gal gp surjective}
can be regarded as a special case of
Lemma \ref{lem:sufficiently many}.
As summarized above, the argument in the proof
of Theorem \ref{thm:Galois_main} shows that one can prove
that $\omega_\Cip$ is an equivalence of categories
without assuming any of the conditions 
in Section \ref{sec:cardinality}, if we
assume that the first derived limit \eqref{eq:R1_triviality} is trivial
for any object $X$ of $\cC$.

\subsubsection{}
We have cases where the fiber functor is an equivalence without
knowing the triviality of the first derived limits as discussed above.
These are the cases of $Y$-sites constructed from 
locally prodiscrete monoids (see Proposition~\ref{prop:equiv locally prodiscrete monoids}).   

As seen in Proposition~\ref{prop:Galois monoid locally prodiscrete},
when the first derived limit is trivial and the equivalence holds true,
the absolute Galois monoid is locally prodiscrete.    We do not know 
wheather the converse, that is, locally prodiscrete implies triviality of 
the first derived limit, holds true.

\subsection{An application: weakly flasque projective systems of abelian groups}
When the groups in a projective system are abelian
and the system is weakly flasque (\cite[p.6]{Jensen}) 
we can use the 
criterion given in this section to obtain some results as follows.
We note that this assumption is weaker than the cardinality conditions.

Let $(\cC, J)$ be a $Y$-site.
Let $X$ be an object of $\cC$.
Then we obtain a projective system of groups indexed by 
$I=\Gal/X$.
Suppose all the groups in the system are abelian.
Assume that the projective system is weakly flasque.
Then all the higher derived limits are trivial 
(\cite[p.9, Thm.\ 1.8.]{Jensen}).
Hence from Lemma~\ref{lem:criterion R2} and 
Proposition~\ref{prop:pinned grid}, it follows that
 there exists a unique grid.
From the argument given in Section~\ref{sec:full ess surj},
it follows that the fiber functor is an equivalence.

\end{document}